\documentclass[10.9pt,a4paper]{amsart}
\usepackage{a4wide}
\usepackage[utf8]{inputenc}
\usepackage[english]{babel}

\usepackage{amsfonts,amssymb,amsmath}
\usepackage{xcolor}
\usepackage[colorlinks,
    linkcolor={red!50!black},
    citecolor={blue!50!black},
    urlcolor={blue!80!black}]{hyperref}
\usepackage[capitalize]{cleveref}

\usepackage{tikz}
\usepackage[all]{xy}
\usepackage{enumitem}
\makeatletter
\newcommand{\mylabel}[2]{#2\def\@currentlabel{#2}\label{#1}}
\usetikzlibrary{calc, matrix, arrows, cd}

\newcommand{\bsm}{\left(\begin{smallmatrix}}
\newcommand{\esm}{\end{smallmatrix}\right)}
\newtheorem{innercustomthm}{Theorem}
\newenvironment{customthm}[1]
  {\renewcommand\theinnercustomthm{#1}\innercustomthm}
  {\endinnercustomthm}
  
  \newtheorem{innercustomprop}{Proposition}
\newenvironment{customprop}[1]
{\renewcommand\theinnercustomprop{#1}\innercustomprop}
{\endinnercustomprop}

\numberwithin{equation}{section}

\newtheorem{theorem}[equation]{Theorem}

\newtheorem{corollary}[equation]{Corollary}
\newtheorem{lemma}[equation]{Lemma}
\newtheorem{proposition}[equation]{Proposition}

\theoremstyle{definition}
\newtheorem{definition}[equation]{Definition}

\newtheorem{remark}[equation]{Remark}

\newtheorem{convention}[equation]{Convention}
\newtheorem{construction}[equation]{Construction}
\newtheorem{notation}[equation]{Notation}
\newtheorem{claim}{Claim}
\newtheorem*{claim*}{Claim}

  \newcommand{\unaryminus}{\scalebox{0.75}[1.0]{\( - \)}}

\newcommand{\Z}{\mathbb{Z}}
\newcommand{\Q}{\mathbb{Q}}
\newcommand{\R}{\mathbb{R}}
\newcommand{\C}{\mathbb{C}}

\newcommand{\RP}{\mathbb{RP}}

\newcommand{\Hom}{\operatorname{Hom}}
\newcommand{\Surf}{\operatorname{Surf}}
\newcommand{\Emb}{\operatorname{Emb}}
\newcommand{\closed}{\operatorname{closed}}
\newcommand{\Homeo}{\operatorname{Homeo}}
\newcommand{\charac}{\operatorname{char}}
\newcommand{\ord}{\operatorname{ord}}
\newcommand{\Mod}{\operatorname{Mod}}
\newcommand{\Ad}{\operatorname{Ad}}
\newcommand{\Iso}{\operatorname{Iso}}
\newcommand{\fr}{\operatorname{fr}}
\newcommand{\incl}{\operatorname{incl}}
\newcommand{\rel}{\operatorname{rel}}

\newcommand{\Spin}{\operatorname{Spin}}
\newcommand{\proj}{\operatorname{proj}}

\newcommand{\cd}{\operatorname{cd}}
\newcommand{\pr}{\operatorname{pr}}
\newcommand{\PD}{\operatorname{PD}}

\newcommand{\coker}{\operatorname{coker}}
\newcommand{\ks}{\operatorname{ks}}
\newcommand{\Ext}{\operatorname{Ext}}

\newcommand{\Herm}{\operatorname{Herm}}
\newcommand{\id}{\operatorname{id}}

\newcommand{\Rot}{\operatorname{Rot}}

\newcommand{\ev}{\operatorname{ev}}
\newcommand{\Aut}{\operatorname{Aut}}
\newcommand{\im}{\operatorname{im}}
\newcommand{\wt}{\widetilde}
\newcommand{\wh}{\widehat}

\begin{document}
\title{Homeomorphisms of surfaces in~$4$-manifolds}

\author[A.~Conway]{Anthony Conway}
\address{The University of Texas at Austin, Austin TX}
%%%
\email{anthony.conway@austin.utexas.edu}
\author[D.~Kasprowski]{Daniel Kasprowski}
\address{University of Southampton, United Kingdom}
\email{d.kasprowski@soton.ac.uk }

%%%

\begin{abstract}
This paper establishes necessary and sufficient conditions for locally flat knotted surfaces in simply-connected $4$-manifolds to be equivalent.
For surfaces with knot group $\Z_d$, we extend results of Lee-Wilczynski from spheres to surfaces of arbitrary genus; the surfaces are permitted to be nonorientable and have boundary.
We prove that most projective planes with knot group $\Z_2$ and the same Euler number are determined by the equivariant intersection form of their exterior.
We also prove that knots with prime power determinants bound
at most one Moebius band in~$D^4$ with knot group $\Z_2$ and a given Euler number.
Cancellation results lead to new criteria for homologous discs to be equivalent rel.\ boundary.
Finally,  we determine the topological extendable mapping class group of knotted surfaces with abelian knot group.
\end{abstract}

\maketitle

\section{Introduction}
\label{sec:Introduction}

This article concerns the classification of locally flat surfaces in simply-connected~$4$-manifolds.
We fix a simply-connected~$4$-manifold~$X$ and briefly review what is known on this topic; for brevity a~\emph{$G$-surface} will refer to a locally flatly embedded surface~$S \subset X$ with~$\pi_1(X \setminus S) \cong G$.
Freedman proved that~$\Z$-spheres in~$S^4$ are unknotted~\cite{Freedman}.
For~$d>0$, Lee and Wilczynski 
%\purple{determined situations where homologous $\Z_d$-spheres are isotopic}~\cite{LeeWilczyOdd,LeeWilczy},
showed that homologous~$\Z_d$-spheres are determined by the pointed equivariant intersection form of their branched cover~\cite{LeeWilczyOdd,LeeWilczy}, 
whereas Boyer proved that when the complement is simply-connected, closed knotted surfaces are determined by their genus and homology class~\cite{BoyerRealization}; this was recently generalised to surfaces with boundary by Pyronneau~\cite{Pyronneau}.
Closed~$\Z$-surfaces are determined by the equivariant intersection form of their exterior~\cite{ConwayPowell};
when the boundary is nonempty,  an additional invariant is needed to classify~$\Z$-surfaces~\cite{ConwayPowell,
ConwayPiccirilloPowell,ConwayDaiMiller}.
Nonorientable~$\Z_2$-surfaces in~$S^4$ were studied in~\cite{Lawson, FinashinKreckViro,KreckOnTheHomeomorphism, ConwayOrsonPowell,ConwayGalvin}.
Results for nonabelian knot groups include~\cite{Hillman2KnotsBooks,FriedlTeichner,ConwayPowellDiscs,ConwayDiscs}.

Our main results (Theorems~\ref{thm:CompatiblePairClosedIntro} and~\ref{thm:CompatiblePairIntro}) unify these statements (see Remark~\ref{rem:Closed}) and yield new classifications.
For $d>0$,  they extend Lee and Wilczynski's work from $\Z_d$-spheres to~$\Z_d$-surfaces of arbitrary genus,  possibly with nonempty boundary,  and allowing for nonorientable surfaces.
For~$d=0$, the outcome recovers previous results on~$\Z$-surfaces from~\cite{ConwayPowell}.
In the nonorientable case,  we deduce that most $\Z_2$-projective planes are determined by the equivariant intersection form of their exterior (Theorem~\ref{thm:ProjectivePlanesIntro}); this generalises Lawson's theorem that~$\Z_2$-projective planes in~$S^4$ are unknotted~\cite{Lawson}.
As another application, we prove that knots with prime power determinants bound at most one~$\Z_2$-Moebius band in~$D^4$ with a given Euler number (Theorem~\ref{thm:Moebius}).
Applications to~$\Z_d$-discs include Theorems~\ref{thm:Discs} and~\ref{thm:SpheresDiscsNoCancellation}.
Our main results also give necessary and sufficient conditions for a homeomorphism~$S \to S$ of a simple surface to extend to a homeomorphism~$(X,S) \to (X,S)$.
This leads to a description of the topological extendable mapping class group 
%the subgroup of mapping classes of~$S$ that extend over~$X$
(Theorem~\ref{thm:K3Intro}); this is related to a topological version of Problem~4.39 of the K3 problem list~\cite{K3}.
Here, a surface~$S \subset X$ is said to be \emph{simple} if its exterior $X_S$ has abelian fundamental group.

\medbreak
In what follows,~$4$-manifolds are assumed to be compact, connected and oriented unless mentioned otherwise.
Surfaces are assumed to be compact and connected but not necessarily orientable.
We work in the topological category.

\subsection{Sample applications}
Before describing the general theory, we begin by listing some applications.
Our first result concerns simple Moebius bands in the $4$-ball.
The knot group of such Moebius bands is necessarily of order two.

\begin{theorem}
\label{thm:MoebiusIntro}
Up to isotopy rel.\ boundary, a knot~$K \subset S^3$ with~$\det(K)$ a prime power bounds at most one~$\Z_2$-Moebius band in~$D^4$ with a given Euler number.
\end{theorem}

It is known that $\Z_2$-Moebius bands in $D^4$ with boundary a determinant one knot are necessarily isotopic~\cite[Theorem B]{ConwayOrsonPowell}, but no result is available when the determinant is nontrivial.

Our next application concerns projective planes.
Whereas Lawson proved that $\Z_2$-projective planes in $S^4$ are unknotted~\cite{Lawson}, in more general $4$-manifolds, it also becomes necessary to take into account the equivariant intersection forms of the exteriors and whether the surfaces are characteristic or ordinary.
For brevity,  we refer to this latter property as the \emph{type} of the surface.
We also also note that the knot group of a nonorientable simple surface in a simply-connected~$4$-manifold has order at most two.

\begin{theorem}
\label{thm:ProjectivePlanesIntro}
Let $X$ be a closed simply-connected~$4$-manifold.
% and let $e \in \Z$ be an integer.
\begin{itemize}
\item Two projective planes in $X$ with simply-connected complements and the same odd Euler number and type are equivalent if and only if their exteriors have isometric intersection forms.
The intersection forms are isometric if the projective planes are homologous.
\item Two $\Z_2$-projective planes in $X$ with the same Euler number are equivalent if and only if their exteriors have isometric equivariant intersection forms.
\end{itemize}
The same assertions hold for Moebius bands with boundary a knot $K \subset S^3$ satisfying~$\det(K)=1$.
\end{theorem}

%The Euler number $e(S)$ of a nonorientable surface $S \subset X$ has same parity as the signature~$\sigma(X)$~\cite{Rokhlin}; see also~\cite[above Corollary 1.1]{Yasuhara}.
In $X=S^4$, any two~$\Z_2$-projective plane exteriors are seen without too much difficulty to have isometric equivariant intersection forms; see e.g.~\cite[Proposition 5.11]{ConwayOrsonPowell}.
%nonorientable surfaces are necessarily nullhomologous and 
%the equivariant intersection form of $\Z_2$-projective plane exteriors is seen without too much difficulty to be isometric to $\Z_- \times \Z_- \to \Z[\Z_2],(x,y) \mapsto 4(1-T)xy$; see e.g.~\cite[Proposition 5.11]{ConwayOrsonPowell}.
Thus, in this case, Theorem~\ref{thm:ProjectivePlanesIntro} recovers Lawson's result that simple projective planes in $S^4$ are unknotted~\cite{Lawson}.

Next, we consider discs and spheres.
Our first theorem on discs is an analogue of Lee and Wilczynski's uniqueness 
%(pre-cancellation) 
result for spheres~\cite{LeeWilczyOdd,LeeWilczy}.
%, with the difference that we consider the discs up to equivalence instead of isotopy.
%{AC: Can we make it be isotopy as they do?}
In order to state this result,  given an integer~$d>0$ and a~$\Z_d$-surface~$S \subset X$,  we write~$\lambda_{\Sigma_d(S)}$ for the equivariant intersection form of the~$d$-fold branched cover and refer to an isometry~$F \colon \lambda_{\Sigma_d(S_0)} \cong \lambda_{\Sigma_d(S_1)}$ as \emph{pointed} if it satisfies~$F([\widetilde{S}_0])=([\widetilde{S}_1])$, where~$\widetilde{S}_i \subset \Sigma_d(S_i)$ denotes the branch set.

\begin{theorem}
\label{thm:PointedFormsDiscsIntro}
Let~$d>0$, let~$X$ be a simply-connected~$4$-manifold with boundary~$\partial X=S^3$,  let~$K \subset~S^3$ be a knot with~$H_1(\Sigma_d(K))=0$ and let~$D_0,D_1$ be simple discs with boundary~$K$,  
nonzero Euler number $e$, and the same type.
%that represent a divisibility~$d$ homology class~$x \in H_2(X,\partial X)$.
The following assertions are equivalent:
\begin{enumerate}
\item the discs~$D_0$ and~$D_1$ are equivalent rel.\ boundary,
\item there is a pointed isometry~$\lambda_{\Sigma_d(D_0)} \cong \lambda_{\Sigma_d(D_1)}$.
\end{enumerate}
When the discs represent a divisibility~$d$ homology class~$x \in H_2(X,\partial X)$,
a pointed isometry exists if~$d=1$ or if~$b_2(X)>6$ and 
\begin{equation}
\label{eq:UniquenessIneq}
b_2(X)> \max_{0\leq j<d}\Big|\sigma (X)-\frac{2j(d-j)}{d^2}\, x\cdot x +\sigma_K (e^{\frac{2\pi i j}{d}})\Big|.
\end{equation}
\end{theorem}
Here $\sigma_K(\omega)$ denotes the Levine--Tristram signature of the knot $K$ at $\omega \in S^1$.
%Lee-Wilczynski prove that when~\eqref{eq:UniquenessIneq} holds and $b_2(X)>6$ (this latter condition can be replaced by one of the conditions in~\cite[Addendum 1]{LeeWilczy}), homologous~$\Z_d$-spheres in such manifolds are isotopic.
Verifying that an isometry is pointed is typically challenging and, when~$d>1$,  the inequality in~\eqref{eq:UniquenessIneq} precludes~$X$ from being definite.
These limitations are also present in Lee and Wilczynski's work on spheres but not in the following result.

\begin{theorem}
\label{thm:SpheresDiscsNoCancellationIntro}
%%Phrasing it with homology classes to compare with LW.
Let~$d\neq 0$ and~$e \neq 0$ be integers such that~$n=e/d$ is an odd prime power or~$1,2,4,2p^k$ with~$p$ an odd prime power.
\begin{itemize}
\item Given a closed simply-connected~$4$-manifold~$X$,  two~$\Z_d$-spheres in~$X$ with Euler number $e$ and the same type are equivalent if and only if their exteriors have isometric equivariant intersection forms.
\item Given a simply-connected~$4$-manifold~$X$ with~$\partial X = S^3$ and a knot~$K \subset \partial X = S^3$ with~$H_1(\Sigma_d(K))=0$,  two~$\Z_d$-discs in~$X$ with boundary~$K$,  Euler number $e$, and the same type are equivalent rel.\ boundary if and only if their exteriors have isometric equivariant intersection forms.
\end{itemize}
\end{theorem}

When~$X=\C P^2$, surfaces with knot group of order $d=2$ are ordinary,  have Euler number $\pm 4$,  and the equivariant intersection form of their exteriors are seen to be isometric without too much difficulty~\cite[Corollary 2.4]{ConwayOrson}.
%then~$e=4$ and the equivariant intersection is seen without too much difficulty to be isometric to~$\Z_- \times \Z_- \to \Z[\Z_2],(x,y) \mapsto -2(1-T)xy$~\cite[Corollary 2.4]{ConwayOrson}; here~$\Z_2=\langle T\mid T^2=1\rangle$ and~$\Z_-$ denotes the abelian group~$\Z$ with~$\Z_2$-action given by~$Tx=-x$ for every~$x \in \Z$.
Thus in this case, Theorem~\ref{thm:SpheresDiscsNoCancellation} recovers~\cite[Theorem 1.1]{ConwayOrson} with equivalence instead of isotopy.
Finally, we note that  the reason for which we focus on~$\Z_d$-discs and~$\Z_d$-spheres with~$d>0$ is that~$\Z$-discs and~$\Z$-spheres in simply-connected~$4$-manifolds were classified in~\cite{ConwayPowell,ConwayPiccirilloPowell}.

\subsection{Closed surfaces}

We now describe the general theory, beginning with the closed case.
Given closed surfaces~$S_0,S_1 \subset X$ whose exteriors~$X_{S_0},X_{S_1}$ have fundamental group $\Z_d$,  an isometry~$
F \colon H_2(X_{S_0};\Z[\Z_d]) \to H_2(X_{S_1};\Z[\Z_d])$ induces (via Poincar\'e duality and the evaluation map) an isomorphism~$H_2(X_{S_0},\partial X_{S_0};\Z[\Z_d]) \to H_2(X_{S_1},\partial X_{S_1};\Z[\Z_d])$ and thus an isomorphism
$$
\partial F \colon H_1(\partial X_{S_0};\Z[\Z_d]) \to H_1(\partial X_{S_1};\Z[\Z_d]).
$$
When the $S_i$ are characteristic,  the \emph{Freedman--Kirby quadratic form} (resp.\ the \emph{Guillou--Marin quadratic form} in the nonorientable case) refers to the quadratic refinement~$H_1(S_i;\Z_2) \to \Z_2$ (resp.~$H_1(S_i;\Z_2) \to \Z_4$) of the $\Z_2$-intersection form on $H_1(S_i;\Z_2)$ that maps a loop~$\gamma \subset S_i$ to~$q_{S_i}([\gamma])=D \cdot S_i+W(D)$, where $D \subset X$ is an embedded disc with boundary $\gamma$ that intersects~$S_i$ transversally, and $W(D)$ denotes a certain framing obstruction~\cite{FreedmanKirby,GuillouMarin}.
For brevity,  we refer to the \emph{the quadratic form} to mean either the Freedman--Kirby quadratic form (in the orientable case) or the Guillou--Marin quadratic form (in the nonorientable case).

Our first result states necessary and sufficient conditions for closed~$\Z_d$-surfaces to be equivalent; here for the definition of the relative $k$-invariant, we refer to~\cite{ConwayKasprowskiKinvariant}.
\begin{theorem}
\label{thm:CompatiblePairClosedIntro}
Let~$X$ be a closed simply-connected~$4$-manifold and let~$S_0,S_1 \subset X$ be~$\Z_d$-surfaces 
with the same Euler number that are either both characteristic or both ordinary.
The following assertions are equivalent:
\begin{itemize}
\item There is a homeomorphism~$(X,S_0) \to (X,S_1)$.
\item There is an isometry $F\colon \lambda_{X_{S_0}} \cong \lambda_{X_{S_1}}$ of the equivariant intersection forms and a bundle isomorphism $H \colon \overline{\nu}(S_0) \to \overline{\nu}(S_1)$ such that~$(F,H|)$ preserves the relative~$k$-invariants,
$$\partial F=(H|)_* \colon H_1(\partial X_{S_0};\Z[\Z_d]) \to H_1(\partial X_{S_1};\Z[\Z_d]),$$
and such that 
\begin{itemize}
\item when $X$ is spin,  $d=1$, and $S_0,S_1$ are nonorientable,  $X_{S_0} \cup_{H|} X_{S_1}$ is spin,
\item when $S$ is characteristic $H|_{S_0}$ preserves the quadratic forms.
\end{itemize}
\end{itemize}
The~$k$-invariant condition can be omitted in the following circumstances: either~$\pi$ is trivial,  or~$\pi$ is infinite cyclic,  or the $S_i$ are spheres, or the $S_i$ are projective planes. 
The condition involving the quadratic forms is automatic when the $S_i$ are spheres or projective planes. 
\end{theorem}

A variation on Theorem~\ref{thm:CompatiblePairClosedIntro} provides conditions for a homeomorphism $\theta \colon S_0 \to S_1$ to extend to a homeomorphism $(X,S_0) \to (X,S_1)$; see Section~\ref{sec:ProofMain} for details.

%When the $S_i$ are characteristic,  the Freedman--Kirby quadratic form (resp.\ the Guillou--Marin quadratic form in the nonorientable case) refers to the quadratic refinement~$H_1(S_i;\Z_2) \to \Z_2$ (resp.~$H_1(S_i;\Z_2) \to \Z_4$) of the $\Z_2$-intersection form on $H_1(S_i;\Z_2)$ that maps a loop~$\gamma \subset S_i$ to~$q([\gamma])=D.F+W(D)$, where $D \subset X$ is an embedded disc with boundary $\gamma$ that intersects $F$ transversally, and $W(D)$ denotes a certain framing obstruction~\cite{FreedmanKirby,GuillouMarin}.

\begin{remark}
\label{rem:Necessary}
The surfaces having the same Euler number ensures that there is a bundle isomorphism~$H \colon \overline{\nu}(S_0) \to \overline{\nu}(S_1)$.
%extending~$\theta$.
The crux of Theorem~\ref{thm:CompatiblePairClosedIntro} is therefore whether it is possible to simultaneously find such an~$H$ and an isometry~$F$ with~$\partial F=(H|)_*$.
As described below,  there are cases where, given an~$F$,  any $H$ will satisfy $\partial F=(H|)_*$.
Note that for $(H|)_*$ to be defined on Alexander modules,~$(H|_\partial)_* \colon \pi_1(\partial X_{S_0}) \to \pi_1(\partial X_{S_1})$ must commute with the coefficient systems to~$\Z_d$.
We refer to Remark~\ref{rem:Intertwine} for further details on this technical point 
but note that this is implicitly assumed of $H$ in Theorems~\ref{thm:CompatiblePairClosedIntro} and~\ref{thm:CompatiblePairIntro}.
%\end{itemize}
\end{remark}

\begin{remark}
\label{rem:Closed}
We compare Theorem~\ref{thm:CompatiblePairClosedIntro} with prior literature on knotted surfaces.
For brevity, we refer to a pair $(F,H)$ satisfying~$\partial F=(H|_\partial)_*$ as an \emph{extendable compatible pair}.
In brief,  Theorem~\ref{thm:CompatiblePairClosedIntro} unifies all previous results in the area, but is not necessarily as sharp: in several cases, the existence of an extendable compatible can be established with further work.
\begin{itemize}
\item When $\pi$ is trivial and the surfaces are orientable, Theorem~\ref{thm:CompatiblePairClosedIntro} recovers a portion of~\cite{BoyerUniqueness,BoyerRealization}, but Boyer then establishes the existence of an extendable compatible pair,  leading (with additional work) to the result that homologous closed orientable surfaces of the same genus with simply-connected complements are isotopic~\cite[Theorem~F]{BoyerRealization}.
\item When $\pi=\Z_d$ is finite cyclic and the surfaces are spheres,  we show in Section~\ref{sec:CompatiblePairs} that Theorem~\ref{thm:CompatiblePairClosedIntro} recovers a result of Lee and Wilczynski, namely the combination of~\cite[Proposition 3.1 and Theorem 3.3]{LeeWilczyOdd} and~\cite[Proposition 4.6]{LeeWilczy} which asserts that homologous $\Z_d$-spheres are equivalent if and only if the pointed equivariant intersection forms of their branched covers are isometric.
Given a nonzero divisibility~$d$ class~$x \in H_2(X),$ when~$X$ satisfies~$b_2(X)>6$ and
\begin{equation*}
%\label{eq:UniquenessIneq}
 b_2(X) > \underset{0 \leq j <d}{\operatorname{max}} \ \Big| \sigma(X)-\frac{2j(d-j)}{d^2}x\cdot x \Big|,
\end{equation*}
Lee and Wilczynski prove that such an isometry necessarily exists,  leading (with additional work) to the result that spheres representing $x$ are isotopic.
The condition $b_2(X) >6$ can be replaced by one of the conditions in~\cite[Addendum 1]{LeeWilczy}.
%When $d>1$,  the inequality in~\eqref{eq:UniquenessIneq} precludes $X$ from being definite.
\item When $\pi=\Z$, Theorem~\ref{thm:CompatiblePairClosedIntro} recovers a portion of~\cite[Theorem 1.4]{ConwayPowell}, but the latter uses Blanchfield forms to show that an extendable compatible pair always exists so that closed~$\Z$-surfaces with isometric intersection forms are necessary equivalent.
\item When $\pi=\Z_2,X=S^4$,  the surfaces are of nonorientable genus $h$, and the Euler numbers of the surfaces satisfy~$e\neq 2h$,  Theorem~\ref{thm:CompatiblePairClosedIntro} follows from~\cite[Theorem A]{ConwayOrsonPowell}; the latter establishes that the surfaces are isotopic.
The same is true when $h \leq 5$ and $e=2h$ (with the case $h=1$ originally due to Lawson~\cite{Lawson} and the cases $h=4,5$ requiring~\cite{Pencovitch}).
%%Don't delete
%%If can get ride of k-invariant, then can write something like the following:
%When the $e=2h$ with~$h>6$, our results are new.
%\end{itemize}
%\end{remark}
%
%
%\begin{remark}
\item When~$b_2(X)>|\sigma(X)|+2$,  Sunukjian~\cite[Theorem 7.4]{Sunukjian} proves that homologous~$\Z_d$-surfaces of positive genus $g$ are equivalent provided 
$$
b_2(X)+2g > \operatorname{max}_{0 \leq j <d}
\Big|\sigma(X)-\frac{2j(d-j)}{d^2}Q_X(x,x)\Big|.
$$
Thanks to work of Lee and Wilczynski~\cite{LeeWilczyGenus}, this condition can be reformulated as requiring that the surfaces have nonminimal genus.
Since we are unable to control the relative~$k$-invariant for surfaces of positive genus,  we are currently unable to recover this result; see Remark~\ref{rem:Sunukjian} for details.
\end{itemize}
\end{remark}

%%Don't delete.
%%I'm writing Aut(\lambda) instead of Aut_{\Z[\pi]} to keep it coherent with the rest of the intro.
\subsection{Estimating the number of surfaces with fixed invariants}
In several cases,  we are able to provide an upper bound on the number of $\Z_d$-surfaces with the same genus $g$,  Euler number~$e$ and whose exteriors have equivariant intersection form isometric to a fixed hermitian form
$$\lambda \colon H \times H \to \Z[\Z_d].$$
We provide a brief outline for~$d>0$ since the case~$d=0$ was treated in~\cite{ConwayPiccirilloPowell}.
For simplicity,  we presently restrict to closed surfaces,  even though Section~\ref{sec:Injection} (and specifically Theorem~\ref{thm:InjectionSurf}) allows for surfaces with boundary.
In the nonorientable case,  recall that~$d \in \{1,2\}$ is forced.

%%Don't delete.
%Really meant to be a single surface here.

Given a closed~$\Z_d$-surface~$S \subset X$,  Poincaré duality on $X_S$ induces an isomorphism
$$
D_S \colon H_1(\partial X_S;\Z[\Z_d]) \xrightarrow{\cong} \coker(\Ad \lambda_{X_S}).
$$
Here,  $\Ad \lambda \colon H \to H^*:=\Hom_{\Z[\Z_d]}(H,\Z[\Z_d])$ denotes the adjoint of a hermitian form $\lambda$.

The choice of  an isometry~$F \colon \lambda_{X_S} \cong \lambda$ induces an isomorphism~$\coker(\Ad \lambda_{X_S}) \cong \coker(\Ad \lambda).$ 
When~$S$ has (possibly nonorientable) genus~$g$ and Euler number~$e$,  the choice of an identification~$\fr$ of the normal bundle~$\nu(S)$ with the Euler number~$e$ rank~$2$ vector bundle over the abstract genus~$g$ surface~$\Sigma$ induces an isomorphism~$H_1(\partial X_S;\Z[\Z_d]) \cong H_1(Y_e)$, where~$Y_e$ denotes the Euler number~$e$ circle bundle over~$\Sigma$.
 
In order for~$D_S$ to yield an invariant of~$S$, it is necessary to quotient out by the choice of the isometry~$F$,  the choice of the identification~$\fr$,  and the choice of the parametrisation of the surface.
More precisely, we show that the assignment~$S \mapsto D_S$ gives rise to an invariant that we call, as in~\cite{ConwayPiccirilloPowell}, the \emph{automorphism invariant}:
$$
\Surf^w_{d,e,\lambda}(g)(X) 
\to
\frac{\Iso_{\Z[\Z_d]}^{\ell k}(\coker(\Ad \lambda),H_1(Y_e;\Z[\Z_d]))}
{\Aut(\lambda) \times (d \cdot H^1(\Sigma;\Z^w) \rtimes \Homeo(\Sigma))}.
$$
Here,~$\Surf_{d,e,\lambda}^w(g)(X)$ denotes the set of $\Z_d$-surfaces of genus $g$, Euler number $e$ and whose exterior has prescribed equivariant intersection form $\lambda \colon H \times H \to \Z[\Z_d]$.
 The superscript $w \in \{+,-\}$ records whether or not the surface is orientable.
The set~$\Iso_{\Z[\Z_d]}^{\ell k}(\coker(\Ad \lambda),H_1(Y_e;\Z[\Z_d])$ consists of isometries $\coker(\Ad \lambda) \cong H_1(Y_e;\Z[\Z_d])$ whose restrictions to the $\Z$-torsion subgroups preserve the appropriate $\Q/\Z$-valued linking forms; we refer to Section~\ref{sec:Injection} for the details but note that this is where we use $d>0$ as, for the $d=0$ case,  equivariant linking forms (i.e.\ Blanchfield forms) would be involved instead.
Finally, $\Aut(\lambda)$ denotes the group of isometries of $\lambda$.
For a detailed description of the actions involved in the orbit set, we also refer to Section~\ref{sec:Injection}.

Decompose~$\Surf_{d,e,\lambda}^w(g)(X)$ according to whether or not the surfaces are characteristic in~$X$:
$$\Surf_{d,e,\lambda}^w(g)(X)=\Surf_{d,e,\lambda}^{w,\charac}(g)(X) \sqcup \Surf_{d,e,\lambda}^{w,\ord}(g)(X).$$
Similarly,  decompose the set of spin structures on $Y_e$ according to whether or not they extend over the Euler number $e$ disc bundle over $\Sigma$:
 $$
\Spin(Y_e)=
\Spin^{\operatorname{ext}}(Y_e)
\sqcup
\Spin^{\operatorname{not-ext}}(Y_e).
$$
The following result 
%(which is a particular case of Theorem~\ref{thm:InjectionSurf}) 
shows that the automorphism invariant,  the relative $k$-invariant~$k \in H^3(\Z_d,Y;H)$ and the spin structure determined by~$X_S$ on~$\partial X_S \cong Y_e$ form a complete set of invariants of~$\Z_d$-surfaces with Euler number~$e$ and equivariant intersection form $\lambda$.
%We now prove Theorem~\ref{thm:InjectionSurfIntro} from the introduction.
For brevity, we use the abbreviations~$\Iso:=\Iso_{\Z[\Z_d]}^{\ell k}(\coker(\Ad Q),H_1(Y_e;\Z[\Z_d])), H^3:=H^3(\Z_d,Y_e;H)$ and~$\Spin^{*}:=\Spin^{*}(Y_e)$.

\begin{theorem}
\label{thm:InjectionSurfIntro}
Let~$X$ be a closed simply-connected~$4$-manifold,  let~$d>0$ be an integer,  let $e \in \Z$ be another integer, and let $\lambda \colon H \times H \to \Z[\Z_d]$ be a hermitian form over $\Z[\Z_d]$.

In the characteristic case,  there are injections
%%Don't delete.
%Note in the nonspin case Rot(\alpha) acts trivially on Spin(Y_e) that doesn't extend.
%%
%%In the spin case, Rot(\alpha) acts nontrivially on Spin(Y_e) that extend.
%%B.14
$$
\Xi_{\Surf} \colon \Surf_{d,e,\lambda}^{w,\charac}(g)(X) 
\hookrightarrow
\begin{cases}
\frac{\Iso
\times H^3 \times\Spin^{\operatorname{not-ext}}
 }{\Aut(\lambda) \times d \cdot H^1(\Sigma;\Z^w)}\Big/\Homeo(\Sigma) \quad& \text{if $X$ is not spin,} \\
 %%%
 \frac{\Iso
\times H^3 }{\Aut(\lambda) \times d \cdot H^1(\Sigma;\Z^w)}\Big/\Homeo(\Sigma)  \quad & \text{if $X$ is spin.}
 \end{cases}
$$
In the ordinary case,  there are injections
$$
\Xi_{\Surf}\colon \Surf_{d,e,\lambda}^{w,\ord}(g)(X) 
\hookrightarrow
\begin{cases}
\frac{\Iso
\times H^3 }
{\Aut(\lambda) \times d \cdot H^1(\Sigma;\Z^w)}\Big/\Homeo(\Sigma) \quad& \text{if $X$ is not spin,} \\
 %%%
 \frac{\Iso
\times H^3  \times\Spin^{\operatorname{ext}}}{\Aut(\lambda) \times d \cdot H^1(\Sigma;\Z^w)}\Big/\Homeo(\Sigma)  \quad & \text{if $X$ is spin.}
 \end{cases}
$$
When the $k$-invariant condition is automatic, the $H^3(\Z_d,Y_e;H)$ factors can be omitted.
The action of $d \cdot H^1(\Sigma;\Z^w)$ on spin structures is trivial in the characteristic case and nontrivial in the ordinary case.
\end{theorem}

The cardinality of the targets of these injections provides an upper bound on the number of~$\Z_d$-surfaces with genus $g$,  Euler number $e$ and whose exteriors have equivariant intersection form $\lambda$.
Theorems~\ref{thm:ProjectivePlanesIntro}  and~\ref{thm:SpheresDiscsNoCancellation} are then obtained by analysing nonorientable analogues of these sets.

\begin{remark}
\label{rem:KirbyTaylorIntro}
Use $\operatorname{Quad}(H_1(S;\Z_2))$ to denote the set of $\Z_2$-quadratic refinements (resp.\ $\Z_4$-valued in the nonorientable case) of the $\Z_2$-intersection form of $S$.
Building on work of Kirby--Taylor~\cite{KirbyTaylor} (see also~\cite[Section 6]{ConwayOrsonPowell} and~\cite[Section 3]{Pyronneau}),  it is known that there is a bijection~$\operatorname{Spin}(Y_e) \cong \Spin(S^1) \times \operatorname{Quad}(H_1(S;\Z_2))$ that induces maps
\begin{align*}
\Surf_{d,e,\lambda}^{w,\charac} 
\to 
\operatorname{Spin}^{\operatorname{not-ext}}(Y_e)
\xrightarrow{\cong} 
\operatorname{Quad}(H_1(S;\Z_2))
&\quad \text{ if $X$ is nonspin,} \\
%%%%
\Surf_{d,e,\lambda}^{w,\ord} 
\to 
\operatorname{Spin}^{\operatorname{ext}}(Y_e)
\xrightarrow{\cong} 
\operatorname{Quad}(H_1(S;\Z_2))
&\quad \text{ if $X$ is spin.}
\end{align*}
In the first  case,  a characteristic~$\Z_d$-surface $S \subset X$ is mapped to its Freedman--Kirby (resp.\ Guillou--Marin quadratic form in the nonorientable case).
In the second case, the assignment depends on the choice of the identification of the normal bundle.
%Here, recall that the Freedman--Kirby quadratic form (resp.\ the Guillou--Marin quadratic form in the nonorientable case) refers to the quadratic refinement~$H_1(S_i;\Z_2) \to \Z_2$ (resp.~$H_1(S_i;\Z_2) \to \Z_4$) of the $\Z_2$-intersection form on $H_1(S_i;\Z_2)$ that maps a loop~$\gamma \subset S_i$ to~$q([\gamma])=D.F+W(D)$, where $D \subset X$ is an embedded disc with boundary $\gamma$ that intersects $F$ transversally, and $W(D)$ denotes a certain framing obstruction~\cite{FreedmanKirby,GuillouMarin}.
\end{remark}

\subsection{The extendable mapping class group of simple surfaces}
\label{sub:Extendable}
For $S \subset X$ a simple surface,  the proof of Theorem~\ref{thm:InjectionSurfIntro} also provides a framework to study the subgroup~$\Mod_X(S)$ of the mapping class group~$\Mod(S)$ consisting of those mapping classes that extend to homeomorphisms~$(X,S) \to (X,S)$.
Determining the smooth analogue of~$\Mod_X(S)$ is Problem 4.39 of the new Kirby problem list, but note that for a mapping class to extend smoothly, it must in particular extend as a homeomorphism.
Here and in what follows, we identify the abstract genus $g$ surface~$\Sigma$ with its image $S \subset X$ by an embedding
$$
\iota \colon \Sigma \hookrightarrow X.
$$
In other words,  this problem is asking to identify those (isotopy classes of) surface homeomorphisms~$\theta \colon S \to S$ such that~$\iota \circ \theta=\Phi \circ \iota$ for some homeomorphism~$\Phi \colon X \to X$, i.e.\ to determine those~$\theta$ such that the embeddings~$\iota \circ \theta$ and~$\iota$ are equivalent.

For~$\bullet=\ord,\charac$,  Theorems~\ref{thm:Injectionv1} and~\ref{thm:InjectionSurf} show that the maps~$\Xi_{\Surf}$ from Theorem~\ref{thm:InjectionSurfIntro} factor through the corresponding sets of embeddings which we denote~$\Emb_{d,e,\lambda}^{w,\bullet}$:
\begin{equation}
\label{eq:XiEmbIntro}
\xymatrix{
\Emb_{d,e,\lambda}^{w,\bullet} \ar@{^{(}->}[r]^-{\Xi_{\Emb}}\ar@{->>}[d]& \mathcal{A}\ar@{->>}[d] \\
\Surf_{d,e,\lambda}^{w,\bullet}  \ar@{^{(}->}[r]^-{\Xi_{\Surf}}& \mathcal{A}/\Homeo(S). 
}
\end{equation}
Here,  for $*\in \{\operatorname{ext},\operatorname{non-ext} \}$, we are using $\mathcal{A}$ to denote the targets of the maps in Theorems~\ref{thm:InjectionSurfIntro}:
%,~\ref{thm:Injectionv1} and~\ref{thm:InjectionSurf}:
$$
\mathcal{A}:=\begin{cases}
\frac{\Iso_{\Z[\Z_d]}^{\ell k}(\coker(\Ad Q),H_1(Y_e;\Z[\Z_d]))
\times H^3(\Z_d,Y_e;H) \times\Spin^*(Y_e)
 }{\Aut(\lambda) \times d \cdot H^1(\Sigma;\Z^w)} \quad& \text{or,} \\
 %%%
 \frac{\Iso_{\Z[\Z_d]}^{\ell k}(\coker(\Ad Q),H_1(Y_e;\Z[\Z_d]))
\times H^3(\Z_d,Y_e;H) }{\Aut(\lambda) \times d \cdot H^1(\Sigma;\Z^w)}.  \quad & \text{}
 \end{cases}
$$
The following theorem  provides an algebraic description of the subgroup~$\Mod_X(S)$ of extendable mapping classes for~$\Z_d$-surfaces with~$d>0$.
For~$d=0$, see Remark~\ref{rem:Hirose} below.
\begin{theorem}
\label{thm:K3Intro}
Let~$X$ be a closed simply-connected~$4$-manifold and let $d>0$.
%and let~$S \subset X$ be a closed~$\Z_d$-surface with~$d>0$.
The extendable mapping class group of a closed~$\Z_d$-surface agrees with the stabiliser of~$\Xi_{\Emb}(S)\in \mathcal{A}$ under the action of~$\Homeo(S)$:
$$
\Mod_X(S)=\operatorname{Stab}_{\Homeo(S)}(\Xi_{\Emb}(\iota)).
$$
\end{theorem}
\begin{proof}
If~$\theta$ extends to a homeomorphism~$\Theta \colon (X,S) \to (X,\theta(S))$, then the embeddings~$\iota$ and~$\iota \circ \theta$ share all the invariants listed in Theorem~\ref{thm:InjectionSurfIntro}.
This implies that~$\theta$ lies in the stabiliser of~$\Xi_{\Emb}$.
Conversely,  if~$\iota$ and~$\iota \circ \theta$ share the aforementioned invariants, then Theorem~\ref{thm:Injectionv1} (i.e.\ the diagram from~\eqref{eq:XiEmbIntro}) implies that the embeddings~$\iota$ and~$\iota \circ \theta$ are equivalent, i.e.\ that~$\theta$ is extendable.
The proof is analogous in the nonorientable case but one applies Theorem~\ref{thm:InjectionSurf} instead of Theorem~\ref{thm:InjectionSurfIntro}. 
\end{proof}

Using $q_S$ to denote the Freedman--Kirby or Guillou--Marin quadratic refinement of the $\Z_2$-intersection form of $S$, we relate Theorem~\ref{thm:K3Intro} to prior work in the area.
\begin{remark}
When~$d=1$ (so that $S \subset X$ has simply-connected complement), work of Boyer~\cite[Proof of Theorems F and G]{BoyerRealization} implicitly shows that if~$S$ is orientable, then
\begin{equation}
\label{eq:BoyerHirose}
\operatorname{Mod}_X(S)=
\begin{cases}
\left\{ 
[\theta]\in \operatorname{Mod}(S) \mid
q_S \circ \theta_*=q_S
%\text{ preserves the~$FK$-forms}
\right\} &\quad \text{if $S$ is characteristic} \\
%%%
\operatorname{Mod}(S) &\quad \text{if $S$ is ordinary}.
\end{cases}
\end{equation}
We explain how this relates to Theorem~\ref{thm:K3Intro}.
Firstly,  in the simply-connected case,  the $k$-invariant condition is automatic, i.e.\ the relative $k$-invariant is fixed by every surface homeomorphism.
%%We didnd't precisely describe the action at this point, but c'est la vie.
Secondly,   since Boyer showed that for such surfaces a compatible pair necessarily exists,  the automorphism invariant is fixed by every surface homeomorphism.
Finally, as explained in Remark~\ref{rem:KirbyTaylorIntro}, in the characteristic case, fixing the spin structure is equivalent to fixing the Freedman--Kirby quadratic form.
Thus, in this case, Theorem~\ref{thm:K3Intro} reduces to~\eqref{eq:BoyerHirose}.
\end{remark}

\begin{remark}
\label{rem:Hirose}
We discuss the analogue of Theorem~\ref{thm:K3Intro} for~$\Z$-surfaces, i.e.\ the case~$d=0$.
Since a framing induces an isomorphism~$H_1(\partial X_S;\Z[\Z])\cong H_1(S)$,  an isometry~$F \colon \lambda_{X_S} \cong \lambda_{X_S}$ induces an isomorphism of~$H_1(S)$.
The work from~\cite{ConwayPowell} implicitly shows that
\begin{equation}
\label{eq:Extendabled=0}
\operatorname{Mod}_X(S)=\left\{ 
[\theta] \in \operatorname{Mod}(S) \mid \theta_*=\partial F \text{ for } F \in \Aut(\lambda_{X_S})
\right\}.
\end{equation}
When~$X=S^4$ and~$S=U$ is unknotted,  Hirose~\cite[Theorem 1.2]{Hirose} (see also~\cite[Theorem~1.2]{HiroseNonorientable} for the nonorientable case) proved that 
$$\operatorname{Mod}_X(S)
=
\left\{ 
[\theta]\in \operatorname{Mod}(S) \mid
q_S \circ \theta_*=q_S
\right\}.
$$
In fact, Hirose proves this result for the smooth analogue of $\operatorname{Mod}(S)$.
We relate this description to the one from~\eqref{eq:Extendabled=0}.
Firstly,  Proposition~\ref{prop:Audrick} implies that for a characteristic~$\Z$-surface~$S\subset S^4$,  the condition~$q_S \circ \theta_*=q_S$ is equivalent to the requirement that the union~$X_S \cup_h X_S$ be spin, where~$h:=H|$ is obtained by extending~$\theta$ to a bundle isomorphism~$H$ and restricting the outcome to the boundary.
Secondly,  the hypothesis~$\theta_*=\partial F$ ensures that $(F,h)$ forms a compatible pair which,  thanks to~\cite[Theorem 3.12]{ConwayPowell}, implies that~$X_S \cup_h X_S$ is spin.
Thus, the condition that~$\theta_*=\partial F$ is typically stronger than~$q_S \circ \theta_*=q_S$ but Hirose's theorem implies that for unknotted surfaces they are in fact equivalent.
\end{remark}

We refer to~\cite{Lehman, BaykurSunukjianExtendable} for recent work on extendable mapping classes (in the smooth category) as well as for further references on the topic.

%%Don't delete
%We note that, given~$\Z$-surfaces~$S_0,S_1$, using~\cite{ConwayCrowleyPowell},  the condition for~$X_{S_0} \cup_h X_{S_1}$ to be spin can also be phrased as requiring that~$h$ preserve the quadratic refinements of the Blanchfield forms of~$\partial X_{S_0}$ and~$$\partial X_{S_1}$.

\subsection{Surfaces with boundary}

For $X$ a $4$-manifold with boundary $\partial X=S^3$ and $S \subset X$ a surface with boundary a knot $K\subset S^3$,  the boundary $\partial X_S$ of the surface exterior decomposes as~$E_K \cup (\partial \overline{\nu}(S)\setminus \nu(K))$, where $E_K$ denotes the exterior of the knot $K$.
As we note in Section~\ref{sec:Closed},  Theorem~\ref{thm:CompatiblePairClosedIntro} (which concerned closed surfaces) is a consequence of the following result for surfaces with nonempty boundary.

\begin{theorem}
\label{thm:CompatiblePairIntro}
Let~$X$ be a simply-connected~$4$-manifold with boundary~$\partial X=S^3$, and let~$K \subset~S^3$ be a knot.
When~$d>0$ is not a prime, assume that $\Sigma_d(K)$ is a rational homology sphere.
Let~$S_0,S_1 \subset X$ be~$\Z_d$-surfaces with boundary~$K$,
with the same Euler number, and that are either both characteristic or both ordinary.
The following assertions are equivalent:
\begin{itemize}
\item There is a rel.\ boundary homeomorphism~$(X,S_0) \to (X,S_1)$.
\item There is an isometry $F\colon \lambda_{X_{S_0}} \cong \lambda_{X_{S_1}}$ of the equivariant intersection forms and a bundle isomorphism $H \colon \overline{\nu}(S_0) \to \overline{\nu}(S_1)$ that is the identity on a tubular neighborhood of $K$ such that~$(F,\id_{E_K} \cup H|)$ preserves the relative~$k$-invariants,
$$\partial F=(\overbrace{\id_{E_K} \cup H|}^{:=h})_* \colon H_1(\partial X_{S_0};\Z[\Z_d]) \to H_1(\partial X_{S_1};\Z[\Z_d]),$$
and such that 
\begin{itemize}
\item when $X$ is spin,  $d=1$, and $S_0,S_1$ are nonorientable,~$X_{S_0} \cup_h X_{S_1}$ is spin,
\item when $S$ is characteristic $H|_{S_0}$ preserves the quadratic forms.
\end{itemize}
\end{itemize}
The~$k$-invariant condition can be omitted in the following circumstances: either~$\pi$ is trivial,  or~$\pi$ is infinite cyclic,  or the $S_i$ are discs with nonzero Euler number, or the $S_i$ are Moebius bands. 
The condition involving the quadratic forms is automatic when the $S_i$ are discs or Moebius bands. 
\end{theorem}

As in the closed case, a variation on Theorem~\ref{thm:CompatiblePairIntro} provides conditions for a rel.\ boundary homeomorphism $\theta \colon S_0 \to S_1$ to extend to a homeomorphism $(X,S_0) \to (X,S_1)$; see Section~\ref{sec:ProofMain} for the details.
We do not know whether the condition on~$H_1(\Sigma_d(K))$ being torsion is necessary; note that when~$d$ is a prime power,~$H_1(\Sigma_d(K))$ is necessarily torsion~\cite[Corollary 9.8]{LickorishIntroduction}.

\begin{remark}
We compare Theorem~\ref{thm:CompatiblePairIntro} with prior literature on the topic.
When $\pi$ is trivial and the surfaces are orientable, recent work of Pyronneau shows that homologous genus $g$ surfaces with the same boundary and with simply-connected complements are isotopic rel.\ boundary~\cite{Pyronneau}.
When $\pi=\Z$, Theorem~\ref{thm:CompatiblePairClosedIntro} recovers~\cite[Theorem~1.3]{ConwayPowell}.
When $\pi=\Z_2, X=D^4$,  the surfaces are of nonorientable genus $h$,  the relative Euler numbers satisfy~$e\neq 2h$ and $\det(K)=1$,  it is known that $\Z_2$-surfaces with boundary $K$ are isotopic~\cite[Theorem~B]{ConwayPowell}; again the same is true when $h \leq 5$ and $e=2h$ with the cases $h=4,5$ requiring~\cite{Pencovitch}.
\end{remark}

%%%Don't delete
%\begin{remark}
%We could try~$BS(1,n)$-surfaces, but it might be higher genus; so let's delay until later.
%Starting point: Determine the module structure of~$H_1(\partial X_{S_i};\Z[\pi])$ when~$\pi$ is a Baumslag-Solitar group.
%\end{remark}

The take away from Theorems~\ref{thm:CompatiblePairClosedIntro},~\ref{thm:InjectionSurfIntro} and~\ref{thm:CompatiblePairIntro} is that surfaces in simply-connected $4$-manifolds with abelian (and therefore cyclic) knot group are determined by their genus, their Euler number, the equivariant intersection form of their exteriors, a relative~$k$-invariant, the spin structure on the boundary of the exterior, and a secondary invariant that is related to the existence of compatible pair.
In several cases, the existence of a compatible pair is automatic, as is the condition involving spin structures and the relative $k$-invariant.

\subsection{Outlook}
\label{sub:Outlook}
%Outlook and strategy of the proof

The proofs of Theorems~\ref{thm:CompatiblePairClosedIntro} and~\ref{thm:CompatiblePairIntro} rely on our work on the classification of~$4$-manifolds with boundary for which the inclusion induced map $i_* \colon \pi_1(\partial X) \to \pi_1(X)$ is surjective~\cite{ConwayKasprowski4Manifolds}.
This latter condition constrains the types of (good) knot groups that our result apply to (e.g. for knotted spheres,  this condition forces the knot group to be cyclic).
For surfaces with boundary,  more knot groups become available and, for example,  Theorem~\ref{thm:BS12} describes how our framework leads to a generalisation of the classification of~$BS(1,2)$-homotopy ribbon discs in~$D^4$ from~\cite[Theorem~1.4]{ConwayPowellDiscs} to Euler number zero $BS(1,2)$-homotopy ribbon discs in more general~$4$-manifolds with boundary~$S^3$.

\medbreak

In conclusion,  it appears to us that further progress on the classification on surfaces in simply-connected~$4$-manifolds would either involve determining restrictions on the present invariants (specifically the relative $k$-invariant and the equivariant intersection form) or on broadening the classification of $4$-manifolds with boundary in order to weaken the condition that $i_*$ be surjective.

\subsection*{Organisation}
Section~\ref{sec:Setup} contains some background homological calculations involving surface exteriors.
Section~\ref{sec:CompatiblePairTheorem} generalises some results from~\cite{ConwayKasprowski4Manifolds} so that they apply to exteriors of simple knotted surfaces.
Section~\ref{sec:kInvariant} describes situations where the $k$-invariant condition can be omitted.
Section~\ref{sec:SpinUnion} assumes that the universal covers~$\widetilde{X}_{S_0}$ and~$\widetilde{X}_{S_1}$ are spin and studies when it is possible to extend a homeomorphism  $S_0 \to S_1$ to a bundle isomorphism~$H \colon \overline{\nu}(S_0) \to \overline{\nu}(S_1)$ such that the universal cover of~$X_{S_0} \cup X_{S_1}$ is also spin.
Section~\ref{sec:ProofMain} proves Theorem~\ref{thm:CompatiblePairIntro}.
Section~\ref{sec:CompatiblePairs} relates compatible pairs to isometries of pointed hermitian forms.
Section~\ref{sec:Injection} proves a version of Theorem~\ref{thm:InjectionSurfIntro} for surfaces with boundary. 
Section~\ref{sec:Closed} proves Theorems~\ref{thm:CompatiblePairClosedIntro} and~\ref{thm:InjectionSurfIntro} on closed surfaces.
Section~\ref{sec:DiscsSpheres}, which focuses on applications to spheres and discs, contains the proofs of Theorems~\ref{thm:PointedFormsDiscsIntro} and~\ref{thm:SpheresDiscsNoCancellation}.
Section~\ref{sec:Nonorientable} is concerned with projective planes and Moebius bands and proves Theorems~\ref{thm:MoebiusIntro} and~\ref{thm:ProjectivePlanesIntro}.
Section~\ref{sec:BS(12)} briefly discusses $BS(1,2)$-homotopy ribbon discs.

\subsection*{Conventions}
Spaces are assumed to be compact and connected unless mentioned otherwise.
Throughout this article, $X$ will denote a $4$-manifold that is either closed or has boundary homeomorphic to~$S^3$.
In the latter case, we fix an identification $\partial X \cong S^3$ once and for all and therefore assume that $\partial X=S^3$ to ease notation.

Given a ring~$R$ with involution, and a left~$R$-module~$H$, we write~$\overline{H}$ for the right~$R$-module whose underlying group agrees with that of~$H$ but with the~$R$-module structure~$x\cdot r:=\overline{r} x$ for~$x \in H$ and~$r \in R$.
Similarly, we write~$H^*:=\overline{\Hom_{\operatorname{left-}R}(H,R)}$ to indicate that we are using the involution on~$R$ to turn the right~$R$-module $\Hom_{\operatorname{left-}R}(H,R)$ into a left~$R$ module.

\subsection*{Acknowledgments}
AC was partially supported by the NSF grant DMS~2303674.

	\section{Simple surfaces }
\label{sec:Setup}	
	
	The goal of this section is to set up some conventions on knotted surfaces and to calculate the Alexander module of the boundary of the exterior of a~$\Z_d$-surface.
	Section~\ref{sub:Boundary} describes the homeomorphism type of the boundary of a surface exterior.
	Section~\ref{sub:HomologyCalculation} is concerned with the homology of surface exteriors and their boundaries.
	Section~\ref{sub:NiceIdentifications} describes issues related to identifications,  framings and coefficient systems.
	Finally, Section~\ref{sub:AlexanderModule} calculates the Alexander module of the boundary of the exterior.

\subsection{The boundary of surface exteriors}
\label{sub:Boundary}
	
We begin by describing the homeomorphism type of the boundary of surface exteriors.
This requires some notation.
Rank two vector bundles over a surface~$\Sigma$ with nonempty boundary are classified by their~$w_1$.
Fix a model~$\Sigma\mathbin{\wt{\times}}\R^2$ of the (total space of the) rank two vector bundle over~$\Sigma$ determined by~$w_1(\Sigma)\colon \pi_1(\Sigma)\to \Z_2.$
%\cong  \pi_1(BO(2))$.
Write~$\Sigma\mathbin{\wt{\times}}D^2$ for the disc bundle of~$\Sigma\mathbin{\wt{\times}}\R^2$ and~$\Sigma\mathbin{\wt{\times}}S^1$ for its circle bundle.
Again using that~$H^2(\Sigma)=0$,  the Euler class of~$\Sigma\mathbin{\wt{\times}}\R^2$ vanishes, leading to the existence of a nowhere zero section and thus to the existence of a section of~$\Sigma\mathbin{\wt{\times}}S^1$.
Fix a section~$r\colon \Sigma\to \Sigma\mathbin{\wt{\times}}S^1$. 
	\begin{convention}
	\label{conv:SectionIsFraming}
Later on, it will be helpful to work with explicit models of these bundles.
	\begin{itemize}
\item When $\Sigma$ is an orientable surface,  we choose the trivial bundle~$\Sigma \times \R^2$ as a model for $\Sigma\mathbin{\wt{\times}}\R^2$ and $r\colon \Sigma \to \Sigma \times S^1 \subset \Sigma \times \R^2, x \mapsto (x,\bsm 1\\ 0\esm)$ as the section.
\item In general,  we use~$p \colon \widehat{\Sigma} \to \Sigma$ to denote the orientation double cover of~$\Sigma$,  consider the~$\Z_2$-action on~$\R^2$ given by~$(x,y)\mapsto (x,-y)$ and make use of~$\widehat{\Sigma} \times_{\Z_2} \R^2$ as a model for~$ \Sigma\mathbin{\wt{\times}} \R^2.$
When~$\Sigma$ is orientable,  this is canonically isomorphic to~$\Sigma \times \R^2$.
The preferred section~$r\colon \Sigma \to \Sigma \mathbin{\wt{\times}} S^1 \subset \Sigma \mathbin{\wt{\times}} \R^2$ is~$r(b):= [b',\bsm 1\\ 0\esm]$ for any~$b'\in p^{-1}(b)$; since~$\bsm 1\\ 0\esm \in S^1$ is a fixed point under the~$\Z_2$ action, this does not depend on the choice of~$b'$. 
%%%Write p=diag(1,-1). Note that that p^{-1}=p. Write T for the generator of Z_2.
%%If b'=Tb'', then (P.I=PIP^{-1}=I) Flip on 1 \in S^1 doesn't do anything.
%%If [b',1]=[Tb',p \cdot 1]=[b'',1] =[b'',1]
	\end{itemize}
	\end{convention}

In what follows, we write~$E_K$ for the exterior of a knot~$K \subset S^3$.
The next lemma determines the homeomorphism type of the boundary of surface exteriors.
	
	\begin{lemma}
		\label{lem:WhatIsTheBoundary}
		For a surface~$S \subset X$ with boundary~$K$ and relative Euler number~$e$, there is a homeomorphism~$\partial X_S \cong E_K \cup (\Sigma \mathbin{\wt{\times}}S^1)$, where the gluing identifies the meridian of~$K$ with the~$S^1$-fibre and the~$e$-framed longitude of~$K$ with~$r(\partial \Sigma)$.
	\end{lemma}
	\begin{proof}
		Since~$X$ is orientable,~$\overline{\nu}(S)$ is bundle isomorphic to~$\Sigma\mathbin{\wt{\times}}D^2$.
%%%Don't delete.
		%Since X is orientable the bundle \nu(S) is orientable and so then it's determined by w_1 of the bundle which is w_1 of the surface.
		Hence there is a homeomorphism~$\partial X_S \cong E_K \cup (\Sigma \mathbin{\wt{\times}}S^1)$ for some gluing.
		The gluing identifies the~$n$-framed longitude (for some~$n$) with~$r(\partial \Sigma)$; the challenge is to show that~$n=e$.
Choose an orientable surface~$G \subset D^4$ with boundary $K$.
		Consider the closed~$4$-manifold~$\widehat{X}:=X \cup -D^4$ and the closed surface
		$$(\widehat{X},\widehat{S})=(X,S) \cup -(D^4,G).$$
We will calculate~$e(\widehat{S})$ in two different way, first showing it 
		
		We assert that~$e(\widehat{S})= e$.
		Pick a push-off~$S'' \subset X$ of~$S$ whose boundary is the~$0$-framed longitude~$\ell_K$ of the knot~$K$ as well as an orientable surface~$G'' \subset D^4$ whose boundary is also~$\ell_K$.
Since the union~$\widehat{S}'':=S'' \cup -G''$ is a push off of~$\wh{S}$,  so~$e(\wh{S})=\widehat{S}.\widehat{S}''$. 
		Here if~$S$ is nonorientable we are using the local orientations to make the intersection count well-defined; see e.g.~\cite[Remark~2.1]{ConwayOrsonPowell} for more details. 
		We deduce that
		\[e(\widehat{S})=\widehat{S}.\widehat{S}''=S.S''+G'.G''=e-\ell k(K,\ell_K)=e.\]
		%%Technically we glue along the reverse of~$K$ or something.
This concludes the proof of the assertion.
	
We now show that~$e(\widehat{S})=n$.
Consider~$S':=r(\Sigma) \subset \Sigma \mathbin{\wt{\times}} S^1 \subset \partial X_S \subset X$ as a push-off of~$S$.
Note that~$S'$ is disjoint from~$S$, where the latter is viewed as the~$0$-section of~$\nu(S) \cong \Sigma \mathbin{\wt{\times}} \R^2$.
%%Don't delete
%because s is nowhere vanishing.
Recall that~$r(\partial \Sigma)$ is identified with the~$n$-framed longitude of~$K$.
Let~$G' \subset D^4$ be a push-off of~$G$  with boundary this~$n$-framed longitude.
		Since the surface~$G' \subset D^4$ is orientable,  its Euler number relative to the Seifert framing is zero.
		It follows that~$\widehat{S}':=S'\cup G'$ is a push-off of~$\widehat{S}:=S \cup G$ with~$e(\widehat{S})=e(S)$, whence, using the assertion,
		\[
		e
		=e(\widehat{S})
		=\widehat{S}.\widehat{S}'
%		=e(S)
		=S.S'+G.G'=0+\ell k(\partial G,\partial G')
		=n,
		\]
		establishing the lemma.
	\end{proof}
	
	\subsection{Integral homology}
	\label{sub:HomologyCalculation}
	
This section describes the homology of surface exteriors and their boundary.	
The following lemma (combined with Lemma~\ref{lem:WhatIsTheBoundary}) calculates the homology of the boundary.
	To do so,  we write $E_K \cup_e (\Sigma \mathbin{\wt{\times}}S^1)$ for the gluing described in Lemma~\ref{lem:WhatIsTheBoundary}.

\begin{lemma}
\label{lem:MV}
Let $\Sigma$ be a surface with one boundary component.
\begin{itemize}
\item If~$\Sigma$ is orientable,  then there is
split short exact sequence
$$
\xymatrix{
0\ar[r]&
\Z_e \ar[r]&
H_1(E_K \cup_e (\Sigma \times S^1))\ar[r]&
H_1(\Sigma)\ar[r] \ar@/_1pc/[l]_-{r_*}&
0
}
$$
The generator of~$\Z_e$ maps to the $S^1$-fibre~$[S^1]$.
% an isomorphism
%$$H_1(E_K \cup_e (\Sigma \times S^1)) \xrightarrow{\cong} H_1(\Sigma) \oplus \Z_e$$
%that sends the meridian~$[S^1]$ to~$(0,1)$. 
\item If $\Sigma$ is nonorientable,  then there is a short exact sequence
$$
\xymatrix{
0\ar[r]&
\Z_2 \ar[r]&
H_1(E_K \cup_e (\Sigma \times S^1))\ar[r]&
H_1(\Sigma)/\langle [\partial \Sigma] \rangle \ar[r] \ar@{-->}@/_1pc/[l]_-{r_*}&
0.
}
$$
%$$0\to \Z_2\to H_1(E_K \cup_e (\Sigma \mathbin{\wt{\times}}S^1))\to H_1(\Sigma)/\langle [\partial \Sigma] \rangle\to 0.$$
The generator of~$\Z_2$ maps to the meridian~$[S^1]$ and~$H_1(\Sigma)/\langle [\partial \Sigma] \rangle \cong \Z_2\oplus \Z^{h-1}$, where~$h$ is the nonorientable genus of~$\Sigma$.
The sequence splits if and only if~$e$ is even,  so that
			$$ H_1(E_K \cup_e (\Sigma \mathbin{\wt{\times}}S^1))
			\cong
			\begin{cases}
				\Z_2 \oplus \Z_2 \oplus \Z^{h-1} &\quad \text{ if~$e$ is even}  \\
				\Z_4 \oplus \Z^{h-1} &\quad \text{ if~$e$ is odd}.
			\end{cases}
			$$
\end{itemize}
	\end{lemma}
	\begin{proof}
In the orientable case,~$\Sigma \mathbin{\wt{\times}} S^1 = \Sigma \times S^1$ and the result follows promptly from a Mayer--Vietoris argument.
We therefore focus on the nonorientable case.
A rapid Mayer--Vietoris calculation shows that~$H_1(E_K \cup_e (\Sigma \mathbin{\wt{\times}}S^1))$ is independent of the knot~$K$.
We can therefore assume that~$K=U$ and, in this case,  the Mayer--Vietoris sequence calculates the homology of the Euler number~$e$ circle bundle~$Y$ over the closed nonorientable genus~$h$ surface.
The remaining statements now follow from~\cite[Proposition~4.2]{ConwayOrsonPowell}.
%%The SES matches COP because H_1(closed surface)=H_1(surface)/H_1(boundary).
	\end{proof}
	
	The following completes the calculation of the integral homology of surface exteriors.

	\begin{lemma}
		\label{lem:H1NonOrientable}
Let $S\subset X$ be a surface with one boundary component. 
\begin{itemize}
\item If $S$ is orientable and $[S]\in H_2(X,\partial X)$ has divisibility $d$, then 
$$H_1(X_S) \cong \Z_d.$$
\item If $S$ is nonorientable, then
		$$
		H_1(X_S) =
		\begin{cases}
			0 &\quad \text{ if } [S] \in H_2(X,\partial X;\Z_2) \text{ is nonzero,} \\
			\Z_2 & \quad \text{ if } [S] \in H_2(X,\partial X;\Z_2) \text{ is zero.}
		\end{cases}
		$$
		When~$[S]$ is zero, the relative Euler number of~$e$ is even.
		\end{itemize}
	\end{lemma}
	\begin{proof}
The orientable case  is well documented (see e.g.~\cite{LeeWilczyOdd,ConwayOrsonPencovitch}) and follows from a Mayer--Vietoris argument, so we focus on the nonorientable case.
		Consider the Mayer--Vietoris sequence for~$X=X_S \cup \overline{\nu}(S)$.
		The gluing occurs along~$Y^\circ:=\partial X_S \setminus E_K \cong \Sigma \mathbin{\wt{\times}}S^1$ leading to the exact sequence
		$$\ldots  \to H_2(X) \xrightarrow{\partial} H_1(Y^\circ) \to H_1(X_S) \oplus H_1(\overline{\nu}(S)) \to 0.$$
		By~\cite[Proposition 4.2]{ConwayOrsonPowell}, the map~$H_1(Y^\circ) \to H_1(\overline{\nu}(S))$ is a surjection whose kernel is order~$2$, generated by the~$S^1$-fibre~$\mu$.
		It follows that there is an exact sequence
		$$\ldots  \to H_2(X) \xrightarrow{\partial} \Z_2\langle\mu\rangle \to H_1(X_S)  \to 0.$$
Thus $H_1(X_S)=0$ or $H_1(X_S)\cong \Z_2$ depending on whether or not $\partial$ is the nontrivial map.
Now, since~$\partial(x)=Q_X([S],x)\mu$,  it follows that~$\partial$ is nontrivial if and only if there exists an~$x \in H_2(X;\Z_2)$ such that~$Q_X([S],x)=1$.
		Since~$Q_X$ is nonsingular, this occurs if and only if~$[S] \in H_2(X,\partial X;\Z_2)$ is~nontrivial.

		We conclude by proving the assertion concerning the relative Euler number.
		Consider the inclusion induced map~$i_* \colon H_1(\partial X_S) \to H_1(X_S)$.
		It is surjective because it maps the~$S^1$-fibre to the meridian.
		Also,  since the class~$[S] \in H_2(X,\partial X;\Z_2)$ is trivial,~$H_1(X_S) \cong \Z_2$.
		If~$e$ were odd, we would have~$H_1(\partial X_S) \cong \Z_4 \oplus \Z^{h-1}$ where the meridian represents~$2 \in \Z_4$ (without loss of generality, we can take~$K=U$ and apply~\cite[Proposition~4.2]{ConwayOrsonPowell}).
		This would contradict~$i_*$ being surjective:~$1=i_*(\mu)=i_*(2)=0.$
		%%%%%Don't delete. This is how I got the intuition.
		%Consider the closed case when~$h=1$.
		%We argue by contradiction.
		%Mayer--Vietoris for~$X=X_S\cup \overline{\nu}(S)$ gives
		%$$ \ldots  \to H_2(X) \to \overbrace{H_1(Y)}^{\cong \Z_4} \to \overbrace{H_1(X_S)\oplus H_1(\overline{\nu}(S))}^{\cong \Z_2 \oplus \Z_2} \to 0~$$
		%To get the right hand side iso we  used~$d$ even whereas for the left hand side we used~$e$ odd.
		%This is not possible, because surjection of finite groups gives bijection.
	\end{proof}

	\subsection{Nice identifications}	
	\label{sub:NiceIdentifications}
	Next, we make the identification from Lemma~\ref{lem:WhatIsTheBoundary} more explicit.
	This will be helpful to better understand the map~$H_1(E_K \cup_e (\Sigma \mathbin{\wt{\times}}S^1)) \cong H_1(\partial X_S) \to H_1(X_S)$ induced by the inclusion.
In what follows, recall that we have fixed a section~$r\colon \Sigma\to \Sigma \mathbin{\wt{\times}}S^1 \subseteq \Sigma\mathbin{\wt{\times}}D^2$.

	\begin{definition}
		\label{def:NiceFraming-or}
		Let~$S \subset X$ be a surface with~$\partial S=K$.
		A  bundle isomorphism~$\fr\colon \nu(S) \xrightarrow{\cong} \Sigma \mathbin{\wt{\times}}\R^2$ is an~\emph{$e$-nice identification} if 
		\begin{itemize}
			\item the isomorphism~$\fr$ extends the $e$-framing of $\nu(K)$,  that is~$\ell k(\fr^{-1}\circ r(\partial \Sigma),K)=e$,
			%The things really are equivalent.
			\item pushing off any loop~$\gamma \subset S$ using~$\fr^{-1}\circ r$ results in a curve that is nullhomologous in~$X_S$.
		\end{itemize}
We refer to $\fr^{-1}\circ r \colon \Sigma \to \nu(S)$ as an \emph{$e$-nice section}.
	\end{definition}

	%%%Don't delete.
	%\begin{remark}
	%In the nonorientable case,  the normal bundle~$\nu(S)$ won't in general be trivial and so we will instead define~$e$-sections.
	%%%See COP23 although the terminology isn't used there
	%I feel that the ordering of COP23 should have maybe followed the ordering of this section; see in particular Remark~\ref{rem:IdentifCovers}.
	%\end{remark}

\begin{remark}
When $S$ is orientable,  Convention~\ref{conv:SectionIsFraming} implies that an $e$-nice identification is
%,  in particular,  
%Line break
a framing of $S \subset X$ that extends the $e$-framing on $K=\partial S$.
This is why we use the notation $\fr$ for nice identifications even though, in the nonorientable case,  the normal bundle of $S$ is nontrivial.
\end{remark}

In order to prove that~$e$-nice identifications exist, we introduce a bundle automorphism of $\Sigma \mathbin{\wt{\times}} \R^2$ that will be used frequently throughout this paper; details can be found in Appendices~\ref{sec:BundleIsos} and~\ref{sec:BundleIsosNonOri}.
We begin with orientable case.

\begin{customprop}{\ref{prop:RotAlphaAxioms}}
\label{prop:RotAlphaAxiomsSection2}
Let $\Sigma$ be an orientable surface with nonempty connected boundary.
For every~$\alpha\in H^1(\Sigma)$, there exists a bundle isomorphism
$$
\Rot(\alpha) \colon \Sigma \times \R^2 \to \Sigma \times \R^2
$$
that satisfies the following properties:
\begin{enumerate}
\item The homeomorphism $\Rot(\alpha)$ is orientation-preserving and restricts to the identity on the~$0$-section and on $(\Sigma \times \R^2)|_{\partial \Sigma}.$
\item The effect of $\Rot(\alpha)$ on $H_1(\Sigma \times S^1)=\Z \oplus H_1(\Sigma)$ is
\begin{equation*}
%\label{eq:RotnHomology}
\Rot(\alpha)_*=
\begin{pmatrix} \id & \alpha_* \\ 0 & \id \end{pmatrix}.
\end{equation*}
Here,~$\alpha_* \colon H_1(\Sigma) \to H_1(S^1)=\Z$ is the map induced by~$\alpha \in H^1(\Sigma) \cong [\Sigma,S^1]$ on homology.
The splitting~$H_1(\Sigma \times S^1)=\Z \oplus H_1(\Sigma)$ is obtained using the fixed section~$r \colon \Sigma \to \Sigma \times S^1$.
\item For every bundle automorphism~$H \colon \Sigma \times \R^2 \to \Sigma \times \R^2$ that satisfies~$H|_{\partial \Sigma \times \R^2}=\id$, there exists a unique~$\alpha \in H^1(\Sigma)$ such that~$H \simeq \Rot(\alpha)$ rel.\ boundary.
\end{enumerate}
\end{customprop}

Next, we formulate the nonorientable analogue of this result, where $w:=w_1(\Sigma) \colon H_1(\Sigma) \to \Z_2$.

\begin{customprop}{\ref{prop:RotAlphaAxiomsNonori}}
\label{prop:RotAlphaAxiomsSection2Nonori}
Let $\Sigma$ be a nonorientable surface with nonempty connected boundary.
For every cohomology class~$\alpha\in H^1(\Sigma,\partial \Sigma;\Z^w)$, there exists a bundle isomorphism
$$
\Rot(\alpha) \colon \Sigma \mathbin{\wt{\times}} \R^2 \to \Sigma \mathbin{\wt{\times}} \R^2
$$
that satisfies the following properties:
\begin{enumerate}
\item The homeomorphism $\Rot(\alpha)$ is orientation-preserving and restricts to the identity on the~$0$-section and on $(\Sigma \mathbin{\wt{\times}}\R^2)|_{\partial \Sigma}.$
\item The effect of $\Rot(\alpha)$ on $H_1(\Sigma \mathbin{\wt{\times}} S^1)=\Z_2 \oplus H_1(\Sigma)$ is
\begin{equation*}
\Rot(\alpha)_*=
\begin{pmatrix} \id & \ev([\alpha]_2) \\ 0 & \id \end{pmatrix}.
\end{equation*}
Here,~$\ev([\alpha]_2)  \colon H_1(\Sigma) \to \Z_2$ is the map induced by~$[\alpha]_2 \in H^1(\Sigma,\partial \Sigma;\Z_2) \cong H^1(\Sigma;\Z_2)$.
The splitting~$H_1(\Sigma \mathbin{\wt{\times}} S^1)=\Z_2 \oplus H_1(\Sigma)$ is obtained using the fixed section~$r \colon \Sigma \to \Sigma \mathbin{\wt{\times}} S^1$.
\item For every bundle automorphism~$H \colon \Sigma \mathbin{\wt{\times}} \R^2 \to \Sigma \mathbin{\wt{\times}} \R^2$ 
%that preserves local orientations and
%%Don't delete
%%Automatic because rel.\ boundary ->homeo is o-p <-> bundle aut preserves local orientations.
%%For the equivalence: true because both conditions are local.
 such that~$H|_{\partial \Sigma \mathbin{\wt{\times}} \R^2}=\id$, there exists a unique~$\alpha \in H^1(\Sigma,\partial \Sigma;\Z^w)$ such that~$H \simeq \Rot(\alpha)$ rel.\ boundary.
 \end{enumerate}
\end{customprop}

	The next lemma ascertains the existence of~$e$-nice identifications.
	
	\begin{lemma}
		\label{lem:ExistseNice-or}
		A surface~$S \subset X$ whose relative Euler number is~$e$ admits an~$e$-nice identification.
	\end{lemma}
	\begin{proof}
The proof of Lemma~\ref{lem:WhatIsTheBoundary} shows that there is a bundle isomorphism~$\fr'\colon \overline{\nu}(S) \cong \Sigma \mathbin{\wt{\times}}D^2$ that extends the $e$-nice framing on $K$.
		Since~$H_1(X_S)$ is generated by a meridian, this homeomorphism can then be arranged to be an~$e$-nice identification by adding meridians to the push-offs as e.g. in~\cite[beginning of Proof of Theorem 2.2]{ConwayOrsonPencovitch}.
		%See COP25, CM24 and COP23.
We provide more details.
Given a bundle isomorphism~$\fr \colon \nu(S) \to \Sigma \mathbin{\wt{\times}} \R^2$, assume that a push off of a curve~$\gamma \subset \Sigma$ (using $\fr^{-1} \circ r$) consists of~$n_\gamma$ meridional components in the homology of~$X_S$. 
The map $\gamma\mapsto n_\gamma$ defines an element~$\alpha\in H^1(\Sigma;\Z_d) \cong H^1(\Sigma,\partial \Sigma;\Z_d)$.
Lift this element to $H^1(\Sigma,\partial \Sigma;\Z^w)$ and
consider the bundle automorphism~$\Rot(-\alpha) \colon \Sigma \mathbin{\wt{\times}}D^2 \to \Sigma \mathbin{\wt{\times}}D^2$ that is given by Propositions~\ref{prop:RotAlphaAxiomsSection2} and~\ref{prop:RotAlphaAxiomsSection2Nonori}.
The second items of these propositions ensure that postcomposing~$\fr$ with~$\Rot(-\alpha)$ gives a new bundle isomorphism for which the push-off of each curve $\gamma$ is nullhomologous.
In order to check that this still extends the~$e$-framing on the boundary, apply (the proof of) Lemma~\ref{lem:WhatIsTheBoundary}: notice that its proof neither depends on the choice of the $e$-nice framing nor on the choice of the section $r$.
%%{XX: In the oriented case isn't the fixed section $x \mapsto (x,1)$ already nice? NO.}
	\end{proof}
	
	\begin{remark}
	\label{rem:HowManyIdentifications}
Given a divisibility $d$ surface~$S \subset X$ whose relative Euler number is~$e$, the third items of Propositions~\ref{prop:RotAlphaAxiomsSection2} and~\ref{prop:RotAlphaAxiomsSection2Nonori} imply that any two $e$-nice identifications of $S$ are related (up to homotopy rel.\ boundary) by postcomposition with $\Rot(\alpha)$ for some~$\alpha\in H^1(\Sigma,\partial \Sigma;\Z^w)$ that is divisible by $d$; if $d=0$,  any two nice framings are bundle homotopic rel.\ boundary.
%Don't delete.
%%A push off of a nullhomologous curve is nullhomologous in Z_d so one needs multiplication by $d$.
%%fr'(\gamma)=\gamma
%%Rot(\alpha) fr'(\gamma)=\gamma+\alpha(\gamma) 
%%For these to give equal push offs \alpha needs to be \equiv 0 mod d.
%%For d=0,  only \alpha=0 will do. So $fr'=Rot(0) \circ \fr$ i.e.\ $\fr'=\fr$.
In the nonorientable case, it is understood that $d=1$ if~$[S] \neq 0$ and~$d=2$ if~$[S]=0$.
These statements remain valid if $\Sigma$ has disconnected boundary because the proofs of the aforementioned propositions also apply in this more general setting.
%%(3) of A.3 applies Because B' in B.20 can be taken to be disconnected
\end{remark}

	We can now make the homeomorphism from Lemma~\ref{lem:WhatIsTheBoundary} more explicit and, simultaneously, define the coefficient systems that we will be using throughout this article.

	\begin{construction}
		\label{cons:CoeffSys}
An~$e$-nice identification~$\fr \colon \nu(S) \to \Sigma \mathbin{\wt{\times}} \R^2$ induces a homeomorphism 
		$$\fr| := \id \cup \fr| \colon \partial X_S \to  E_K \cup_e (\Sigma \mathbin{\wt{\times}}S^1).$$
		
		We first assume that~$\Sigma$ is orientable.		
		Lemma~\ref{lem:MV} implies~$H_1(E_K \cup_e (\Sigma \times S^1)) \cong H_1(\Sigma) \oplus \Z_e$, where the second summand is generated by the meridian of~$K$. 
		Since~$d$ divides~$e$, we can consider the following map that is defined by sending the meridian to~$1$ and the curves on $\Sigma$ to zero:
		$$\varphi \colon H_1(E_K \cup_e (\Sigma \times S^1)) \to \Z_d.$$ 
		If~$S \subset X$ has divisibility~$d$, then~$d$ divides~$e$ 
		%%Don't delete
		%%(consider again~$(X,\widehat{S})$).
		and~$\varphi$ agrees with the composition
		$$H_1(E_K \cup_e (\Sigma \times S^1)) \xrightarrow{\fr|^{-1},\cong} H_1(\partial X_S) \to H_1(X_S) \cong \Z_d.$$
		Indeed,  both maps send the meridian to~$1$ and the curves on $\Sigma$ to zero.
		
We now consider the case when~$\Sigma$ is nonorientable. 
Lemma~\ref{lem:H1NonOrientable} shows that since~$S$ is a simple surface,~$X_S$ is simply connected if $[S] \neq 0$ and~$\pi_1(X_S)\cong \Z_2$ if $[S]=0$.
If~$[S]=0$, then~$e$ is even by Lemma~\ref{lem:H1NonOrientable}.
%Because even
Lemma~\ref{lem:MV} implies that~$H_1(E_K \cup_e (\Sigma \mathbin{\wt{\times}}S^1)) \cong \Z_2 \oplus H_1(\Sigma)/\langle [\partial \Sigma] \rangle$, where the first summand is generated by the meridian of~$K$. We consider the following map that is defined by sending the meridian to~$1$ and the curves on $\Sigma$ to zero:
		$$\varphi \colon H_1(E_K \cup_e (\Sigma \mathbin{\wt{\times}}S^1)) \to \Z_2.$$ 
		As in the orientable case, the map~$\varphi$ agrees with the following composition:
		$$H_1(E_K \cup_e (\Sigma \mathbin{\wt{\times}}S^1)) \xrightarrow{\fr^{-1},\cong} H_1(\partial X_S) \to H_1(X_S) \cong \Z_2.$$
	\end{construction}

	Next, we use this construction to discuss identifications.
	
	\begin{remark}
		\label{rem:IdentifCovers-or}
		In the remainder of the section, for simplicity, we pick an~$e$-nice identification for every surface and identify~$E_K \cup_e (\Sigma \mathbin{\wt{\times}}S^1)$ with~$\partial X_S$.
		In particular, we use Construction~\ref{cons:CoeffSys} to identify the boundary of the~$d$-fold cover of~$X_S$ with the~$\varphi$-induced cover of~$E_K \cup_e (\Sigma \mathbin{\wt{\times}}S^1)$ when~$\Sigma$ is orientable. 
Similarly, when~$\Sigma$ is nonorientable and~$H_1(X_S)\cong \Z_2$, we use Construction~\ref{cons:CoeffSys} to identify the boundary of the~$2$-fold cover of~$X_S$ with the~$\varphi$-induced cover of~$E_K \cup_e (\Sigma \mathbin{\wt{\times}}S^1)$.
	\end{remark}
		
	\subsection{The Alexander module}
	\label{sub:AlexanderModule}
	
This section establishes the main result of this section, namely the calculation of the Alexander module of the boundary of a $\Z_d$-surface exterior.
	
	\begin{remark}
	\label{rem:HomologyBranchedUnbranched}
		For~$d>0$, a Mayer--Vietoris argument gives the short exact sequence
		$$ 
		0 \to \Z \to H_1(E_K;\Z[\Z_d]) \to H_1(\Sigma_d(K)) \to 0.
		$$
		The~$\Z$-factor is generated by a lift of the~$d$-th power of the oriented meridian of~$K$.
Thus, the $\Z$-factor is endowed with the trivial $\Z_d$-action.
		In particular,  this short exact sequence is $\Z[\Z_d]$-split by considering the projection induced map postcomposed with division by~$d$:
		%%Don't delete because p(deck)=p.
		$$
		H_1(E_K;\Z[\Z_d]) \xrightarrow{p_*} d \cdot H_1(E_K) \xrightarrow{\div d,\cong} \Z.
		$$
Thus, in what follows, we use this $\Z[\Z_d]$-splitting without any further mention:
		$$
		H_1(E_K;\Z[\Z_d])  \cong H_1(\Sigma_d(K)) \oplus \Z\langle \widetilde{\mu}_K\rangle.
		$$
	\end{remark}

	We can now prove the main results of this section. 
	We start with the orientable case.

	\begin{proposition}
		\label{prop:AlexanderModule}
		Let~$S\subset X$ be an orientable divisibility~$d$ surface of genus~$g$ with boundary a knot~$K$.
		If the Euler number of~$S$ relative to the Seifert framing is denoted by~$e$, then the choice of an~$e$-nice identification gives rise to a $\Z[\Z_{d}]$-isomorphism
		$$
		H_1(\partial X_S;\Z[\Z_{d}])
		\cong  
		\begin{cases} 
H_1(\Sigma_{d}(K)) \oplus \Z_{e/d} \oplus \Z^{2g} & \quad \text{ if } d \neq 0, \\
H_1(E_K;\Z[\Z]) \oplus 	\Z^{2g} & \quad \text{ if } d = 0.
		\end{cases} 
		$$
Both the $\Z^{2g}$-factor and the $\Z_{e/d}$-factor are endowed the trivial $\Z_d$-action.
		
		In particular, in the closed case,  we obtain
		$$H_1(\partial X_S;\Z[\Z_{d}]) \cong
		\begin{cases} 
\Z_{e/d} \oplus \Z^{2g}  & \quad \text{ if } d \neq  0, \\
\Z^{2g} & \quad \text{ if } d = 0.
		\end{cases}
		$$
		%The display is a little ugly but not having \Z^{2g} first is for consistency later.
	\end{proposition}
	\begin{proof}
		The case~$d=0$ is treated in~\cite[Lemma 5.5]{ConwayPowell}, so we assume that~$d > 0$.
%		Also, without loss of generality, we can assume~$d>0$. 
%%Removed because divsibility is \geq 0.
		Lemma~\ref{lem:WhatIsTheBoundary} and Remark~\ref{rem:IdentifCovers-or} show that the choice of an~$e$-framing $\nu(S) \cong \Sigma \times \R^2$ determines a homeomorphism~$\partial X_S \cong E_K \cup (\Sigma \times S^1)$, where the gluing identifies the meridian of~$K$ with the~$S^1$-fibre and the~$e$-framed longitude of~$K$ with the boundary of~$\Sigma$.

		We now lift this gluing to the~$d$-fold covers; implicitly using Remark~\ref{rem:IdentifCovers-or}.
		The lift of~$d$-times~$\mu_K$ gets identified to the lift of~$d$-times the~$S^1$-fibre,  and the boundary of the lifted surface is identified with the lift of~$d$-times the~$e$-framed longitude.
		We elaborate.
		By definition of an $e$-nice framing and the definition of the map~$\varphi$ from Construction~\ref{cons:CoeffSys}, the restriction of the cover to~$\Sigma \times S^1$ is~$\Sigma \times S^1 \to \Sigma \times S^1$ where the covering map is the identity on the surface factor and multiplication by~$d$ on the circle factor.
		Thus, while the meridian~$\mu_K$ and the~$S^1$-fibre~$\mu$ do not lift to loops in the cover for~$d>1$, their~$d$-fold multiples do; we write~$\widetilde{\mu}_K$ for the resulting lift of the~$d$-fold meridian, and similarly for $\widetilde{\mu}$.
		Note also that since~$S \subset X$ is embedded, we know that the relative Euler number~$e$ is a multiple of~$d^2$.
		%%By capping off with (D^4,G).
		In particular the~$e$-framed longitude lifts to a loop in the cover.
		Homologically,  this lift represents~$\frac{e}{d}\widetilde{\mu}_K+\widetilde{\ell}_K$.

		We conclude the proof of the proposition.
		First,  recall that~$H_1(E_K;\Z[\Z_d]) \cong H_1(\Sigma_d(K)) \oplus~\Z$, where~$\Z$ is generated by~$\widetilde{\mu}_K$; see Remark~\ref{rem:HomologyBranchedUnbranched}.
		%the lift of~$d$-times~$\mu_K$.
		Our discussion of the restriction of the covering map to~$\Sigma \times S^1$ implies that~$H_1(\Sigma \times S^1;\Z[\Z_d]) \cong H_1(\Sigma) \oplus \Z\langle \widetilde{\mu}\rangle$ where both summands have the trivial~$\Z_d$-action.
		The discussion from the previous paragraph combined with a Mayer--Vietoris argument with~$\Z[\Z_d]$ coefficients now yields the announced homology calculation.
		More precisely, the~$e$-nice framing and inclusions induce $\Z[\Z_d]$-isomorphisms
		\begin{align*}
			H_1(\partial X_S;\Z[\Z_d])
			&  \xrightarrow{(\fr|)_*,\cong}
			H_1(E_K \cup (\Sigma \times S^1);\Z[\Z_d]) 
			\xleftarrow{\incl_*,\cong}
			\frac{
				H_1(E_K;\Z[\Z_d]) \oplus H_1(\Sigma \times S^1;\Z[\Z_d])
			}
			{\langle e/d \cdot [\widetilde{\mu}_K]=0,  [\widetilde{\mu}_K]=[\widetilde{\mu}] \rangle}  \\
			%%%%
			&  \xrightarrow{\incl_*,\cong}
			H_1(\Sigma_d(K))  \oplus \Z_{e/d}\langle [\widetilde{\mu}] \rangle \oplus H_1(\Sigma).
		\end{align*}
		This concludes the proof of the proposition,
	\end{proof}
	
	We now consider the case that~$\Sigma$ is nonorientable.
	\begin{proposition}
		\label{prop:AlexanderModuleNonOri}
		Let~$S \subset X$ be a surface of nonorientable genus~$h$ with boundary a knot~$K$.
		Write~$e:=e(S)$ for the Euler number of~$S$ relative to the Seifert framing.
		\begin{itemize}
			\item When~$[S] \in H_2(X,\partial X;\Z_2)$ is nontrivial,  the only cyclic cover of~$X_S$ is the trivial cover, and the choice of an~$e$-nice identification gives rise to an isomorphism
			$$ H_1(\partial X_S)
			\cong
			\begin{cases}
				\Z_2 \oplus \Z_2 \oplus \Z^{h-1} &\quad \text{ if~$e$ is even}  \\
				\Z_4 \oplus \Z^{h-1} &\quad \text{ if~$e$ is odd}.
			\end{cases}
			$$
The relative Euler number $e$ is even if and only if the signature~$\sigma(X)$ is even.
			\item When~$[S] \in H_2(X,\partial X;\Z_2)$ is trivial,  then~$e$ is necessarily even,  the only nontrivial cyclic cover of~$X_S$ is the~$2$-fold cover,  and the choice of an~$e$-nice identification gives rise to a~$\Z[\Z_2]$-isomorphism
\begin{align*} H_1(\partial X_S;\Z[\Z_2])
&\cong
H_1(\Sigma_2(K)) \oplus \frac{\Z_2 \langle \widetilde{\mu}\rangle \oplus H_1(\Sigma)}{\langle (e/2,[\partial \Sigma])\rangle} 	 \\
&\cong
\begin{cases}
H_1(\Sigma_2(K))  \oplus  \Z_2 \oplus \Z_2 \oplus \Z^{h-1} &\quad \text{ if~$e/2$ is even}  \\
H_1(\Sigma_2(K))  \oplus \Z_4 \oplus \Z^{h-1} &\quad \text{ if~$e/2$ is odd.} 
\end{cases}
\end{align*}
Here $\widetilde{\mu}$ denotes a lift of twice the meridian of $K$, and the $\Z[\Z_2]$-module structure is determined by $T\eta=\eta+w_1(\eta) \widetilde{\mu}$ and $T \widetilde{\mu}= \widetilde{\mu}$ for $\eta \in H_1(\Sigma)$.
		\end{itemize}
	\end{proposition}
	\begin{proof}
Besides the assertion about the Euler number, the first item follows from Lemma~\ref{lem:MV}.
Using~$\sigma(K)$ to denote the signature of the knot $K$, as mentioned in~\cite[above Corollary 1.1]{Yasuhara}, the relative Euler number~$e(S)$  satisfies
$$ e(S) \equiv 2h+2\sigma(K)+\sigma(X) \quad \mod 4.$$ 
The required conclusion follows: $e(S)$ is even if and only if~$\sigma(X)$ is even.

		We now assume that~$[S] \in H_2(X,\partial X;\Z_2)$ is trivial; we already proved in Lemma~\ref{lem:H1NonOrientable} that this implies~$e$ is even.
The choice of coefficient systems from Construction~\ref{cons:CoeffSys} ensures that the total space of the cover $\Sigma \mathbin{\wt{\times}} S^1 \to \Sigma \mathbin{\wt{\times}} S^1$ is again homeomorphic to~$\Sigma \mathbin{\wt{\times}} S^1$.
The reasoning is now identical as in~\cite[Proposition~4.9]{ConwayOrsonPowell}.
In a nutshell, the argument is the same as in the orientable case with the difference that~$H_1(\Sigma\mathbin{\wt{\times}}S^1)\cong H_1(\Sigma)\oplus \Z_2\langle \mu \rangle$ and~$[\partial \Sigma]\in H_1(\Sigma)$ has order two (instead of vanishing) so that, after choosing an $e$-nice identification~$\fr$,  the Mayer--Vietoris sequence yields the following sequence of~$\Z[\Z_2]$-isomorphisms:
\begin{align*}
H_1(\partial X_S;\Z[\Z_2])
&  \xrightarrow{(\fr|)_*,\cong}
H_1(E_K \cup (\Sigma \mathbin{\wt{\times}} S^1);\Z[\Z_2]) 
\xleftarrow{\incl_*,\cong}
			\frac{
H_1(E_K;\Z[\Z_2]) \oplus H_1(\Sigma \mathbin{\wt{\times}} S^1;\Z[\Z_2])
			}
			{\langle e/2 \cdot [\widetilde{\mu}_K]=r_*([\partial \Sigma]),  [\widetilde{\mu}_K]=[\widetilde{\mu}]\rangle}  \\
			%%%%
			&  \xrightarrow{\incl_*,\cong}
	H_1(\Sigma_2(K)) \oplus \frac{\Z_2 \langle \widetilde{\mu}\rangle \oplus H_1(\Sigma)}{\langle (e/2,[\partial \Sigma])\rangle} 
\cong
H_1(\Sigma_2(K))  \oplus G \oplus \Z^{h-1}.
		\end{align*}
Here,  as an abelian group $G$ is either $\Z_2 \oplus \Z_2$ or $\Z_4$ according to whether $e/2$ is even or odd.
%(again thanks to Lemma~\ref{lem:MV}).
%%Thanks to COP really.
We delay a discussion of the $\Z[\Z_2]$-module structure to Appendix~\ref{sec:BundleIsosNonOri} and specifically to Corollary~\ref{cor:RotOnYe(K)Covers}.
	\end{proof}
	
	\begin{remark}
	\label{rem:HomologyBranched}
		The abelian group~$H_1(\Sigma_d(K))$ is finite if and only if~$\Delta_K(t)$ does not have a root at a~$d$-th root of unity, see e.g.~\cite[Corollary 9.8]{LickorishIntroduction}.
		This occurs for example when~$d$ is a prime power. 
		In general,  in the decomposition~$H_1(\Sigma_d(K)) \cong\Z^k \oplus TH_1(\Sigma_d(K))$,  one can argue that the~$\Z^k$ factor can never have the trivial~$\Z_d$-action, unless $k=0$.
To see this, write~$\Z=\langle t\rangle$ and~$\Z_d=\langle T \mid T^d=1\rangle$, and pick a representative of the Alexander polynomial~$\Delta_K$ with~$\Delta_K(1)=1$,  so that for any~$(x \otimes 1) \in H_1(E_K;\Z[\Z]) \otimes_{\Z[\Z]} \Z[\Z_d] \cong H_1(\Sigma_d(S))$ that is fixed by the $\Z_d$-action,
$$(x \otimes 1)=\Delta_K(1)(x \otimes 1)=\Delta_K(T)(x \otimes 1)=(\Delta_K(t)x \otimes 1)=0.$$
%The second equality is what uses being fixed.
Here, in the second equality, $\Delta_K(T)$ denotes the element of $\Z[\Z_d]$ obtained by substituting~$t \in \Z$ for~$T\in \Z_d$ in the Alexander polynomial of~$K$.
	\end{remark}

\section{Compatible pairs and equivalence of surfaces}
\label{sec:CompatiblePairTheorem}

The aim of this section is to show that the results of~\cite{ConwayKasprowski4Manifolds} apply to exteriors of knotted surfaces.
% prove Theorem~\ref{thm:CompatiblePair}.
In Section~\ref{sub:4ManifoldsWithBoundary} we prove a result about $4$-manifolds with boundary and finite cyclic fundamental group whose proof is an adaptation of~\cite[Theorem~1.8]{ConwayKasprowski4Manifolds}.
Section~\ref{sub:ProofMeta} applies this result to surface exterior, leading to the main theorem of this section (Theorem~\ref{thm:CKForSurfaceExteriors}) which provides necessary and sufficient conditions for two simple surfaces to be equivalent. 
Finally,  Section~\ref{sub:CompatiblePairEquiv} is concerned with an alternative (but equivalent) definition of compatible pairs between $4$-manifolds with finite fundamental group.

\subsection{The classification of certain~$4$-manifolds with boundary}
\label{sub:4ManifoldsWithBoundary}

We begin by proving a theorem on $4$-manifolds with boundary that can be obtained by slightly modifying a previous result of ours~\cite[Theorem 1.8]{ConwayKasprowski4Manifolds}.
This latter theorem was stated for manifolds $X_0,X_1$ with finite cyclic fundamental group under the assumption that $H_1(\partial X_i;\Z[\Z_d])^*=0$, and our goal is to explain why it holds more generally when the Alexander module is of the form~$H_1(\partial X_i;\Z[\Z_d]) \cong L\oplus T\oplus \Z^n$ with $L$ free,  $T^*=0$ and where $\Z^n$ is endowed with the trivial $\Z_d$-action.

\begin{convention}
\label{conv:LTZ}
Here and in what follows, when~$d=0,1$ the hypothesis that ``$H_1(\partial X_i;\Z[\Z_d])$ decomposes as~$L\oplus T\oplus \Z^n$ with~$L$ free,~$T^*=0$ and where~$\Z^n$ is endowed with the trivial~$\Z_d$-action" is understood to mean that~$H_1(\partial X_i;\Z[\Z_d]) \cong L\oplus T$ with~$L$ free and~$T$ torsion.
Indeed when~$d=0$ the~$\Z[\Z]=\Z[t^{\pm 1}]$-module~$\Z^n=\Z[t^{\pm 1}]/(t-1)$ is torsion, whereas when~$d=1$, the abelian group~$\Z^n$ is free.
When~$d=1$, every abelian group decomposes in this way.
\end{convention}

%\medbreak

The main technical challenge in adapting~\cite[Theorem 1.8]{ConwayKasprowski4Manifolds} to surface exteriors is related to the potential presence of order~$2$ elements in $\pi$. 
In order to overcome this issue,  it will be useful to recall some concepts from~\cite[Section 2]{ConwayKasprowski4Manifolds}.
First, in what follows,  we write~$y_g$ for the coefficient of~$y \in \Z[\pi]$ at~$g \in \pi$ and~$y_1$ for the coefficient at the neutral element.

\begin{lemma}
\label{lem:EvenElemOfOrder2}
If $X$ is a $4$-dimensional Poincar\'e complex, then $\lambda_X(x,x)_g$ is even for every nontrivial~$g \in \pi$ of order $2$ and every~$x \in H_2(X;\Z[\pi])$.
\end{lemma}
\begin{proof}
This follows because~$\lambda_X(x,x)=\mu_X(x)+\overline{\mu_X(x)}+e(x)$ with $e(x) \in \{0,1\}$; the existence of~$\mu_X \colon H_2(X;\Z[\pi]) \to Q_{+1}(\Z[\pi])$ can be established e.g. as in~\cite[Lemma~2.21]{ConwayKasprowski4Manifolds}.
%%Don't delete
%That result assumes spin and gets \lambda=\mu+\mu^*.
Here, as is common in surgery theory, we write~$Q_{+1}(\Z[\pi]):=\Z[\pi]/\langle x-\overline{x} \mid x \in \Z[\pi]\rangle.$
\end{proof}

In order to adapt the proof of~\cite[Theorem 1.8]{ConwayKasprowski4Manifolds} to our setting, it is necessary to recall some notation from that article.
First, given a closed $3$-manifold $Y$,  we set 
$$I^1(Y;\Z[\pi]):=
\im
\left( H^1(Y;\Z[\pi])
 \xrightarrow{\ev}
\overline{\Hom(H_1(Y;\Z[\pi]),\Z[\pi])}
\to
  \Hom(H_1(Y;\Z[\pi]),\Z_2)
   \right),$$
where the second map is induced by augmentation modulo $2$.
It will also be helpful to recall the following group homomorphism from~\cite[Section 2.6]{ConwayKasprowski4Manifolds}:
\begin{align}
\label{eq:hatFunction}
\widehat{} \ \colon \Herm(H^2(Y;\Z[\pi])) &\to I^1(Y;\Z[\pi]) \\
b &\mapsto (x \mapsto b(\PD_{Y}^{-1}(x),\PD_{Y}^{-1}(x))_1 \mod 2) \nonumber.
\end{align}

We are now ready to adapt~\cite[Theorem 1.8]{ConwayKasprowski4Manifolds} to allow for more general Alexander modules.
Indeed~\cite[Theorem 1.8]{ConwayKasprowski4Manifolds} assumes that~$H_1(\partial X_i;\Z[\Z_d])^*=0$ without allowing (when $d \neq 0,1$) for the presence a~$\Z^k$-factor.

\begin{theorem}
\label{thm:CK4Manifold}
Let~$X_0$ and~$X_1$ be~$4$-manifolds with~$\pi_1(X_j) \cong \Z_d$ and~$\iota_j \colon \pi_1(\partial X_j) \to \pi_1(X_j)$ surjective for~$j=0,1$.
Assume that the Alexander module is~$H_1(\partial X_i;\Z[\Z_d]) \cong L \oplus T \oplus \Z^k$ with~$L$ free, ~$T^*=0,$ and where~$\Z$ is endowed with the trivial~$\Z_d$-action.
Fix an orientation-preserving homeomorphism~$h \colon \partial X_0 \to \partial X_1$.
\begin{itemize}
\item When~$H_1(\partial X_i;\Z[\Z_d])^*=0$ or when $\pi_1=\Z_{2\ell}$ with $\ell>0$ and $H_1(\partial X_i;\Z[\Z_d]) \cong T \oplus \Z^k$,  the homeomorphism~$h$ extends to a homeomorphism~$X_0 \to X_1$ if and only if there is an isomorphism~$u \colon \pi_1(X_0) \cong \pi_1(X_1)$ with~$u \circ \iota_0=\iota_1 \circ h_*$,  there is a $k$-invariant preserving compatible pair~$(F,h)$, and~$\ks(X_0)=\ks(X_1)$.
\item When the universal covers of~$X_0,X_1$ are not spin,~$h$ extends to a homeomorphism~$X_0 \to X_1$ if and only if there is an isomorphism~$u \colon \pi_1(X_0) \cong \pi_1(X_1)$ with~$u \circ \iota_0=\iota_1 \circ h_*$,  there is a $k$-invariant preserving compatible pair~$(F,h)$, and~$\ks(X_0)=\ks(X_1)$.
\item When the universal covers of~$X_0,X_1$ are spin, ~$h$ extends to a homeomorphism~$X_0 \to X_1$ if and only if there is an isomorphism~$u \colon \pi_1(X_0) \cong \pi_1(X_1)$ with~$u \circ \iota_0=\iota_1 \circ h_*$,   there is a $k$-invariant preserving~$(F,h)$,  and the universal cover of~$X_0 \cup_h X_1$ is spin.
\end{itemize}
Additionally,  when~$H_1(\partial X_i;\Z[\Z_d])^*=0$ or when the~$X_i$ are spin, given a~$k$-invariant preserving compatible pair~$(F,h)$ as above,  the homeomorphism~$X_0 \to X_1$ extending~$h$ can be chosen to induce~$F$ on~$\Z[\Z_d]$-homology.
\end{theorem}
\begin{proof}
When $d=0,1$ (and when $k=0$),  the result is already proved in~\cite[Theorem 1.6]{ConwayKasprowski4Manifolds}.
In what follows, we therefore assume that $d \neq 0,1$.
We first note that since~$\pi$ is finite cyclic and~$\pi_1(\partial X_i) \to \pi_1(X_i)$ is surjective, the second homotopy group of~$X_i$ is stably isomorphic to the augmentation ideal:~$\pi_2(X_i) \cong_s I$.
The argument appears in~\cite[Lemma 4.6]{ConwayOrsonPencovitch} but we also outline a direct argument.
Applying \cite[Lemma 3.3]{ConwayKasprowski4Manifolds} shows that~$\pi_2(X)$ fits in a $\Z[\Z_d]$-exact sequence~$0\to \pi_2(X)\to C_2\to C_1\to C_0\to \Z\to 0$ where the~$C_i$ are~$\Z[\Z_d]$-free. 
Modules fitting in such a sequence are all stably isomorphic by Schanuel’s lemma. 
Taking the standard free~$\Z[\Z_d]$-resolution of~$\Z$,  it follows that~$\pi_2(X)$ is stably isomorphic to~$I$.

Thanks to known results in surgery theory (see e.g.~\cite[Proposition 1.5]{ConwayKasprowski4Manifolds}),  it suffices to argue that the homotopy classification of~\cite[Theorem~1.3]{ConwayKasprowski4Manifolds} (or more generally~\cite[Theorem~2.29]{ConwayKasprowski4Manifolds}) goes through under the given assumptions.
Indeed this will prove the last two items of the theorem.
We will discuss the first item at the end of the proof.

Before going through the proof of this result and explaining the required changes,  we note that three commonly made assumptions in~\cite{ConwayKasprowski4Manifolds} hold in our setting.
%; we refer to~\cite{ConwayKasprowski4Manifolds} for the relevant definitions.
To begin with, we note that since~$\pi$ is finite,~$H^i(\pi;\Z[\Z_d])=0$ for~$i>0$,  so~$\ev \colon H^2(X_i;\Z[\Z_d]) \to \overline{\Hom(H_2(X;\Z[\Z_d]),\Z[\Z_d])}$ is an isomorphism and therefore so is its dual,~$\ev^*$.
The condition~$H^i(\pi;\Z[\Z_d])=0$ for~$i>0$ also ensures that~$0=H^2(\pi;\Z[\Z_d]) \to H^2(\partial X_1;\Z[\Z_d])$ is necessarily injective and that compatible triples can be substituted for compatible pairs (recall~\cite[Proposition 2.11]{ConwayKasprowski4Manifolds}).
%the~$\ev^*$ assumption necessarily holds as does the requirement that~$H^2(\pi;\Z[\Z_d]) \to H^2(\partial X_1;\Z[\Z_d])$ be injective (indeed~$H^2(\pi;\Z[\Z_d])=0$ for~$\pi$ finite).
%Also, since~$H^i(\pi;\Z[\Z_d])=0$ for~$i>0$, compatible triples can be substituted for compatible pairs (recall~\cite[Proposition 2.11]{ConwayKasprowski4Manifolds}).
We also note that a hermitian form on~$L \oplus T \oplus \Z^k$ is weakly even if and only if it is even: this is noted in~\cite[Lemma~2.19]{ConwayKasprowski4Manifolds} when~$k=0$,  but adding~$\Z^k$ makes no difference because a form on~$\Z^k$ takes values in~$\Z$
 (to see this,  note~$Tb(x,y)=b(Tx,y)=b(x,y)$ and~$\ker(1-T)=\im (\cdot \mathcal{N}) \cong \Z$) 
 and a symmetric form on~$\Z$ is weakly even if and only if it is even.
 %%Don't delete.
 %%\mu(x)=a_x/2.
Here,  recall that a hermitian form~$\lambda \colon H \times H \to \Z[\pi]$ is \emph{weakly even} if for every~$x \in H$,  it satisfies~$\lambda(x,x) = a+\overline{a}$ for some~$a \in \Z[\pi]$. 

Next, we outline the ingredients that go into the proof of~\cite[Theorem 2.29]{ConwayKasprowski4Manifolds},  namely~\cite[Theorem 2.8]{ConwayKasprowski4Manifolds} (to obtain a homotopy equivalence, it suffices to arrange for the first and secondary obstructions to vanish), ~\cite[Theorem 2.9]{ConwayKasprowski4Manifolds} (it is possible to work with compatible pairs instead of with the~$2$-type), ~\cite[Theorem 2.15]{ConwayKasprowski4Manifolds} (the secondary obstruction can be assumed to vanish in the case where the intersection form is not weakly even),  and~\cite[Theorems~2.17-18]{ConwayKasprowski4Manifolds} (the reformulation of the secondary invariant when the intersection form is weakly even).
We argue that the hypotheses in each of these theorems also hold under the present assumptions.
\begin{itemize}
\item First,~\cite[Theorem 2.8]{ConwayKasprowski4Manifolds} is already stated for finite cyclic groups without any assumptions on the Alexander module.
\item  By contrast,~\cite[Theorem 2.9]{ConwayKasprowski4Manifolds} assumes that~$H_1(\partial X_i;\Z[\Z_d])\cong L\oplus T$. 
This is not an issue however since its proof (which is essentially~\cite[Proposition~7.7]{ConwayKasprowski4Manifolds}) in fact also works when $\pi$ is finite and~$H_1(\partial X_i;\Z[\Z_d]) \cong L\oplus T\oplus \Z^n$; this is mentioned in the statement of~\cite[Proposition~7.7]{ConwayKasprowski4Manifolds}.
\item Next,~\cite[Theorems 2.17 and 2.18]{ConwayKasprowski4Manifolds} (both of which are proved in~\cite[Section~10]{ConwayKasprowski4Manifolds}) rely on~\cite[Theorems 2.22]{ConwayKasprowski4Manifolds} which itself relies on~\cite[Theorems 9.4]{ConwayKasprowski4Manifolds}.
These theorems require that~$\pi_2$ be projective,  but in fact they also hold when~$\pi$ is finite and~$\pi_2$ is stably the augmentation ideal as noted in~\cite[Theorems 9.4]{ConwayKasprowski4Manifolds}.
%{DK: There we assume isomorphic to $I\oplus \Z[\pi]^k$. Should we have assumed stably $I$ there?
%Resolution: we actually meant stably I.
%}
The remainder of the arguments go through without change.
%%Don't delete: Theorem~2.18 was proved in Section 10, and the argument goes through without change. 
\item 
Next, we discuss~\cite[Theorem 2.15]{ConwayKasprowski4Manifolds}, which is proved in \cite[Section 11]{ConwayKasprowski4Manifolds} and, as stated,  requires~$\pi_2(X_i)$ to be projective and the Alexander module to be of the form~$L \oplus T$
The key inputs in the proof are~\cite[Theorem 11.2]{ConwayKasprowski4Manifolds} and~\cite[Proposition 2.24]{ConwayKasprowski4Manifolds}; the remainder of the proof makes no further uses of the hypotheses.
As is already noted in~\cite[Theorem 11.2]{ConwayKasprowski4Manifolds}, the hypotheses of the former can be modified to our setting, and we focus on~\cite[Proposition 2.24]{ConwayKasprowski4Manifolds}.
%The application of Proposition 2.24 requires~$\pi_2$ projective and~$H_1(\partial X_i;\Z[\Z_d]) \cong L \oplus T$ but both hypotheses can be replaced by the present ones: the former because, as we've already noted,  Theorem~9.4 holds when~$\pi_2$ is stably the augmentation, and the latter because even and weakly even are equivalent in this setting as well.
%It remains to explain why the presence of elements of order~$2$ does not cause any problems.
\item We conclude by discussing~\cite[Proposition 2.24]{ConwayKasprowski4Manifolds} since, as explained above, it plays an important role in~\cite[Theorem 2.15]{ConwayKasprowski4Manifolds}.
The hypotheses that $\pi_2(X_i)$ be projective and that~$H_1(\partial X_i;\Z[\Z_d]) \cong L \oplus T$ can be adapted to our setting since they are only included to apply~\cite[Theorem 9.4]{ConwayKasprowski4Manifolds} which, as already noted,  works in our setting.
It remains to explain why the presence of elements of order~$2$ does not cause any problems.
The second item of~\cite[Proposition 2.24]{ConwayKasprowski4Manifolds}  immediately goes through if one replaces ``weakly even" by ``hermitian forms~$b$ such that $b(x,x)_1$ is even for every $x$".
We explain why the third item of~\cite[Proposition 2.24]{ConwayKasprowski4Manifolds} goes through.
We explain how to adapt the proof of injectivity.
Assume that~$\widehat{b}=\widehat{b(\varphi)}$ for some~$\varphi \in \mathcal{G}$; see~\eqref{eq:hatFunction} for the notation.
%, then~$b$ and $b(\varphi)$.
The key observation is that both these pairings satisfy the property that their coefficient at every element (nontrivial) $g \in \pi$ of order $2$ is even;
this follows from Lemma~\ref{lem:EvenElemOfOrder2} because both pairings are obtained by pulling back the intersection form of a $4$-dimensional Poincar\'e complex.
%To see this,  observe that both pairings are obtained by pulling back the intersection form of a $4$-dimensional Poincar\'e complex~$X$.
%Since~$\lambda_X(x,x)=\mu(x)+\overline{\mu(x)}+e(x)$ with $e(x) \in \{0,1\}$ for every $x \in H_2(X;\Z[\pi])$ (the existence of $\mu_X$ can be argued as in~\cite[Lemma~2.20]{ConwayKasprowski4Manifolds}), it follows that $b-b(\varphi)$ has vanishing coefficient at every (nontrivial) group element of order $2$.
%%Don't delete
%forced by \mu
Since~$\widehat{b}=\widehat{b(\varphi)}$,  it follows that $b(x,x)_1-b(\varphi)(x,x)_1$ is even for every $x $ and thus~$b-b(\varphi)$ is weakly even.
%%see e.g. the proof of 22.0.
Since these forms are defined on~$H_1(\partial X_i;\Z[\Z_d]) \cong L \oplus T \oplus \Z^k$, as discussed above, his implies that the difference is actually an even form.
The remainder of the proof now proceeds without change.
%Apply Theorem~\ref{thm:GoalActionpdf} to realise this form as $b(\varphi')$ for some $\varphi' \in \mathcal{G}$ so that $b=b(\varphi)+b(\varphi')=b(\varphi \circ \varphi')$, as required.
\end{itemize}
We have now discussed in detail how the hypotheses of~\cite[Theorem 2.29]{ConwayKasprowski4Manifolds} are used.
The remainder of the proof of this theorem goes through without change and leads to the second and third items of the present theorem.

It remains to prove the first item.
In general,  
%%Writing pi instead of Z_d
the secondary obstruction is represented by a~$\Z[\pi]$-valued hermitian form $b:=b(c_0,c_1)$ on~$H_1(\partial X_1;\Z[\pi]) \cong H^2(\partial X_1;\Z[\pi])$.
The adjoint of this form is a map~$H^2(\partial X_i;\Z[\pi]) \to \overline{\Hom_{\Z[\pi]}(H^2(\partial X_i;\Z[\pi]),\Z[\pi])}$.
Thus if~$H_1(\partial X_i;\Z[\pi])^*=0$, this form is trivial,  so the first part of the statement follows from~\cite{ConwayKasprowski4Manifolds}.
%%%Don't cite again...Hard to cite; explained above.

It remains to consider the case where~$\pi \cong \Z_{2 \ell}$ with~$\ell >0$ and~$H_1(\partial X_i;\Z[\pi])  \cong T \oplus \Z^k$.
Since there is a $\Z[\pi]$-isomorphism~$(\Z^k)^*:=\Hom_{\Z[\pi]}(\Z^k,\Z[\pi]) \cong \Z^k$, in this case,  the adjoint of the hermitian form $b$ determines an element in
\begin{align*}
\Hom_{\Z[\pi]}(H^2(\partial X_i;\Z[\pi]),H^2(\partial X_i;\Z[\pi])^*)
&\cong \Hom_{\Z[\pi]}(\Z^k \oplus T,(\Z^k \oplus T)^*)
=\Hom_{\Z[\pi]}(\Z^k \oplus T,(\Z^k)^*) \\
&\cong \Hom_{\Z[\pi]}(\Z^k,(\Z^k)^*).
\end{align*}
Put differently, $b$ is $\Z$-valued and is determined by $T$ and by the integer-valued symmetric form on~$b|_{\Z^k \times \Z^k}$.
To see that the form is $\Z$-valued,  note that since~$\Z^k$ is fixed by the $\pi=\Z_{2 \ell}$-action,  it takes values in~$\mathcal{N} \cdot \Z[\Z_{2\ell}] \cong \Z$.
If the universal covers are nonspin,  the second item already ensures that~$h$ extends and so there is nothing to prove.
We therefore assume that the universal covers are spin (i.e.\ that their equivariant intersection forms are weakly even) 

We claim that~$b$ is weakly even. 
Since the form~$b$ takes values in~$\mathcal{N} \cdot \Z[\Z_{2\ell}] \cong \Z$,  it satisfies~$b(x,x)_h=b(x,x)_{h'}$ for every~$h,h' \in \Z_{2 \ell}$ and in particular~$b(x,x)_1=b(x,x)_g$ for every~$g$ of order~$2$.
Thus in order to prove that~$b$ is weakly even,  it suffices to prove that~$b(x,x)_g$ is even for every nontrivial~$g$ of order~$2$;  such a~$g$ exists because we assumed that~$\pi \cong \Z_{2\ell}$ with~$\ell >0$.
Setting~$X:=X_0 \cup_h X_1$, the arguments in~\cite[Section 10]{ConwayKasprowski4Manifolds} (which apply since $\lambda_{X_i}$ is weakly even for $i=0,1$) show that this is equivalent to proving that~$\lambda_X(x,x)_g$ is even for every nontrivial~$g$ of order~$2$.
The latter assertion is true thanks to Lemma~\ref{lem:EvenElemOfOrder2}.
Thus~$b$ is weakly even, as claimed.

Since $b$ is a weakly even integral form,  it is necessarily even.
The secondary obstruction now vanishes thanks to~\cite[Theorem 2.22]{ConwayKasprowski4Manifolds}.
As explained above,  it now follows from~\cite{ConwayKasprowski4Manifolds} that~$h$ extends to a homeomorphism, as required.
%%%Don't cite again...Hard to cite; explained above.
\end{proof}

\subsection{Surface exteriors and compatible pairs}
\label{sub:ProofMeta}

The next result, which consists of applying Theorem~\ref{thm:CK4Manifold} to surface exteriors,  provides a first step towards the proof of Theorem~\ref{thm:CompatiblePairIntro}.

\begin{theorem}
\label{thm:CKForSurfaceExteriors}
Let~$X$ be a simply-connected~$4$-manifold with~$\partial X=S^3$, and let $K \subset S^3$ be a knot.
When $d>0$ is not a prime,  assume that~$\Sigma_d(K)$ is a rational homology sphere.
Let~$S_0,S_1 \subset X$ be~$\Z_d$-surfaces with boundary~$K$, with the same Euler number, and
whose exteriors are either both spin or both nonspin.
The following assertions are equivalent:
\begin{itemize}
\item The surfaces~$S_0$ and~$S_1$ are equivalent rel.\ boundary.
\item There is a bundle isomorphism~$H\colon (\overline{\nu}(S_0),S_0) \to (\overline{\nu}(S_1),S_1)$ such that the homeomorphism
$$h:=\id_{E_K}  \cup H| \colon \partial X_{S_0} \to \partial X_{S_1}$$
fits in a~$k$-invariant-preserving compatible pair~$(F,h)$,  and is such that when the universal covers of the~$X_{S_i}$ are spin, the universal cover of the union~$X_{S_0} \cup_h X_{S_1}$ is also spin.
\end{itemize}
%%Only in main theorem:
%When $H_1(\partial X_{S_0};\Z[\pi])^*=0$,  or when $\pi_1=\Z_{2\ell}$ with $\ell>0$ and $H_1(\partial X_i;\Z[\Z_d]) \cong T \oplus \Z^k$, the~$k$-invariant and spin conditions can be omitted.
\end{theorem}
\begin{proof}
If the surfaces are equivalent rel.\ boundary, then an equivalence $(X,S_0) \to (X,S_1)$ gives the required compatible pair~$(F,h)$.
We therefore focus on the converse.
Assume that we are given~$H$ and~$F$ as in the theorem statement.
As mentioned in Remark~\ref{rem:Necessary}, this implicitly means that~$(H|_\partial)_* \colon \pi_1(\partial X_{S_0}) \to \pi_1(\partial X_{S_1})$ must commute with the coefficient systems to~$\Z_d$.
In other words,  there is an isomorphism~$u \colon \pi_1(X_{S_0}) \to \pi_1(X_{S_1})$ with~$u \circ \iota_0 =\iota_1 \circ h_*$, where we
use~$\iota_j \colon \pi_1(\partial X_{S_j}) \to \pi_1(X_{S_j})$ to denote the inclusion induced map for $j=0,1$.
The additivity of the Kirby--Siebenmann invariant (see e.g.~\cite[Theorem 9.2]{FriedlNagelOrsonPowell}) ensures that~$\ks(X_{S_0})=\ks(X)=\ks(X_{S_1})$.
Next, using the assumption on~$\Sigma_d(K)$, the Alexander module calculations from Propositions~\ref{prop:AlexanderModule} and~\ref{prop:AlexanderModuleNonOri} ensure that Theorem~\ref{thm:CK4Manifold} applies to the surface exteriors.
(If~$b_1(\Sigma_d(K))>0$,  then Remark~\ref{rem:HomologyBranched} states that the~$\Z$-free part of~$H_1(\Sigma_d(K))$ is never endowed with trivial~$\Z_d$-action, so Theorem~\ref{thm:CK4Manifold} cannot be applied).
Under the hypotheses of the present theorem, we deduce that the homeomorphism~$h$ extends over the surface exteriors,  therefore leading to the required equivalence.
\end{proof}

\subsection{An equivalent characterisation of compatible pairs}
\label{sub:CompatiblePairEquiv}

This short section returns to the setting of general~$4$-manifolds and provides an alternative (but equivalent) description of compatible pairs.
In brief, we show that compatible pairs, and therefore Theorem~\ref{thm:CK4Manifold},  can be formulated using the homology of  universal covers instead of twisted homology.
In what follows, given a left~$\Z[\pi]$-module~$H$, we endow~$\Hom_\Z(H,\Z)$ with the right~$\pi$-action given by~$(\gamma \cdot \varphi)(x):=\varphi(\gamma x).$ 
Taking the coefficient at $g=e$ induces a right~$\Z[\pi]$-linear map~$\ev_e \colon \Hom_{\Z[\pi]}(H,\Z[\pi]) \to \Hom_\Z(H,\Z)$ which is an isomorphism when~$\pi$ is finite.
%f(x)=\sum_g f_g(x)g
%ev_e((f \cdot \gamma)(x))=ev_e(f(\gamma x))=f_e(\gamma x)
%ev_e(f)\gamma (x)=\gamma f_e(x)
%Linearity f(\gamma x)=\gamma f(x). Take coefficient at neutral element f_e(\gamma x)=\gamma f_{\gamma^{-1}}(x)
The following proposition provides an equivalent characterisation of compatible pairs for~$4$-manifolds with finite fundamental group.

\begin{proposition}
\label{prop:EquivCompatiblePair}
Let~$X_0$ and~$X_1$ be~$4$-manifolds with finite fundamental group~$\pi$,  let~$h \colon \partial X_0 \to~\partial X_1$ be a homotopy equivalence,
and assume there is an isomorphism~$u \colon \pi_1(X_0) \to \pi_1(X_1)$ that satisfies~$u \circ \iota_0=\iota_1 \circ h$.
\begin{itemize}
\item The set of~$\Z[\pi]$-isomorphisms~$H_2(\widetilde{X}_0)  \to H_2(\widetilde{X}_1)$ that preserve the~$\Z$-intersection forms is in canonical bijection with the set of isometries~$(H_2(X_0;\Z[\pi]),\lambda_{X_0})  \to (H_2(X_1;\Z[\pi]),\lambda_{X_1})$ of the~$\Z[\pi]$-intersection forms.
\item The set of compatible pairs~$(F,h)$ is in canonical bijection with the set of pairs~$(F',h)$ where~$F'$ is a~$\Z[\pi]$-isomorphism~$H_2(\widetilde{X}_0)  \to H_2(\widetilde{X}_1)$ that preserves the~$\Z$-intersection forms and fits into the following commutative diagram of~$\Z[\pi]$-modules and~$\Z[\pi]$-linear maps:
$$
\xymatrix@C0.4cm@R0.4cm{
H_2(\partial \widetilde{X}_0)
\ar[r]\ar[d]_{h_*}&
H_2(\widetilde{X}_0)
\ar[r]\ar[d]_{F'}&
{\overbrace{H_2(\widetilde{X}_0,\partial \widetilde{X}_0)}^{\cong \Hom_\Z(H_2(\widetilde{X}_0),\Z)}}
\ar[r]\ar[d]_{((F')^*)^{-1}}&
H_1(\partial \widetilde{X}_0)
\ar[d]_{h_*}\
\ar[r]&0
\\
%%%%%%%%%%%%
H_2(\partial \widetilde{X}_1)
\ar[r]&
H_2(\widetilde{X}_1)
\ar[r]&
{\underbrace{H_2(\widetilde{X}_1,\partial \widetilde{X}_1)}_{\cong \Hom_\Z(H_2(\widetilde{X}_1),\Z)}}
\ar[r]&
H_1(\partial \widetilde{X}_1)
\ar[r]&0.
}
$$
\end{itemize}
\end{proposition}
\begin{proof}
In what follows, for $i=0,1$, we write $Q_{\widetilde{X}_i}$ for the $\Z$-intersection form on $H_2(\widetilde{X}_i)$ and $\lambda_{\widetilde{X}_i}$ for the~$\Z[\pi]$-valued hermitian form on~$H_2(\widetilde{X}_i)$ given by
$$\lambda_{\widetilde{X}_i}(x,y)=\sum_{g \in \pi}Q_{\widetilde{X}_i}(x, gy)g^{-1}.$$
In what follows, the adjoint of a hermitian form $\lambda \colon H \times H \to R$ is denoted $\Ad \lambda \colon H \to \overline{\Hom(H,R)}.$
For $i=0,1,$ consider the following diagram of~$\Z[\pi]$-modules and~$\Z[\pi]$-linear maps:
$$
\xymatrix@C0.7cm@R0.4cm{
H_2(\partial X_i;\Z[\pi])
\ar[r]\ar[d]_=&
H_2(X_i;\Z[\pi])
\ar[r]\ar[d]_=&
H_2(X_i,\partial X_i;\Z[\pi])
\ar[r]\ar[d]^{\ev \circ \PD^{-1}}_\cong&
H_1(\partial X_i;\Z[\pi])
\ar[d]_\cong \\
%%%%%%
H_2(\partial X_i;\Z[\pi])
\ar[r]\ar[d]^{\theta_i}_\cong&
H_2(X_i;\Z[\pi])
\ar[r]^-{\Ad \lambda_{X_i}}\ar[d]^{\theta_i}_\cong&
\overline{\Hom_{\Z[\pi]}(H_2(X_i;\Z[\pi]),\Z[\pi])}
\ar[r]\ar[d]^{\theta_i}_\cong&
\coker( \Ad \lambda_{\widetilde{X}_i} )
\ar[d]_\cong
\\
%%%%%%
H_2(\partial \widetilde{X}_i)
\ar[r]\ar[d]_=&
H_2(\widetilde{X}_i)
\ar[r]^-{\Ad \lambda_{X_i}}\ar[d]_=&
\overline{\Hom_{\Z[\pi]}(H_2(\widetilde{X}_i),\Z[\pi])}
\ar[r]\ar[d]^{(\ev_e)_*}_\cong&
\coker( \Ad \lambda_{\widetilde{X}_i})
\ar[d]_\cong\\
%%%%%%
H_2(\partial \widetilde{X}_i)
\ar[r]&
H_2(\widetilde{X}_i)
\ar[r]^-{\Ad Q_{\widetilde{X}_i}}&
%The overline is just to make it into a left module but it doesn't really matter because the action is trivial.
\overline{\Hom_\Z(H_2(\widetilde{X}_i),\Z)}
%The overline is just to make it into a left module but it doesn't really matter because the action is trivial.
\ar[r]&
\coker(\Ad Q_{\widetilde{X}_i} )\\
%%%%%%
H_2(\partial \widetilde{X}_i)
\ar[r]\ar[u]^=&
H_2(\widetilde{X}_i)
\ar[r]\ar[u]^=&
H_2(\widetilde{X}_i,\partial \widetilde{X}_i)
\ar[r]\ar[u]_{\ev \circ \PD^{-1}}^\cong&
H_1(\partial \widetilde{X}_i).
\ar[u]^\cong
\\
%%%%%%%%%%%%
}
$$
The map labeled $\theta_i$ is induced by the isomorphism $\Z[\pi] \otimes_{\Z[\pi]} C_*(\widetilde{X}_i)=C_*(X_i;\Z[\pi]) \to C_*(\widetilde{X}_i)$ and similarly for $\partial X_i$.
We argue that this diagram commutes.
The leftmost and rightmost squares are readily seen to commute, so we focus on the central squares.
The top-most and bottom-most central squares commute by definition of the (equivariant) intersection form.
The second central square from the top commutes e.g. by~\cite[Proposition 4.1]{KasprowskiPowellTeichner}.
The remaining square commutes by definition of~$\lambda_{\widetilde{X}_i}$ and of the augmentation map.

The commutativity of the central squares implies the first assertion by setting~$F':=\theta_1 \circ F \circ \theta_0^{-1}$.
%%Don't delete
%Q_{X_1}(\theta_1F\theta_0^{-1}(x),\theta_1F\theta_0^{-1}(y))
%=\varepsilon \lamba_1 (F\theta_0^{-1}(x),F\theta_0^{-1}(y))
%=\varepsilon \lamba_0 (\theta_0^{-1}(x),\theta_0^{-1}(y))
%Q_{X_0}(x,y).
The commutativity of the entire diagram implies the second assertion.
\end{proof}

We conclude by showing that when $\pi$ is cyclic,  the compatible pair condition reduces to verifying an equality on Alexander modules.
\begin{proposition}
\label{prop:CompatibleSimplification}
Let~$X_0$ and~$X_1$ be~$4$-manifolds with finite fundamental group~$\Z_d$,  let~$h \colon \partial X_0 \to~\partial X_1$ be a degree one homotopy equivalence,
and assume there is an isomorphism~$u \colon \pi_1(X_0) \to \pi_1(X_1)$ that satisfies~$u \circ \iota_0=\iota_1 \circ h$.
Given an isometry $F \colon \lambda_{X_0} \cong \lambda_{X_1}$, the following are equivalent:
\begin{enumerate}
\item the pair $(F,h)$ is compatible,
\item the pair $(F,h)$ satisfies 
$$\partial F=h_* \colon H_1(\partial X_0;\Z[\Z_d]) \to H_1(\partial X_1;\Z[\Z_d]).$$
\end{enumerate}
\end{proposition}
\begin{proof}
Since the first statement clearly implies the second,  we focus on showing that the equality~$\partial F=h_*$ on $H_1(\partial X_0;\Z[\Z_d])$ suffices to get a compatible pair, i.e.\ implies the corresponding equality on~$H_2(\partial X_0;\Z[\Z_d])$.
When~$d=0$,  since~$\Z \cong H_2(\partial X_i;\Z[\Z]) \to H_2(X_i;\Z[\Z])$ is the zero map (see e.g.~\cite[Lemma~3.2]{ConwayPowell}), there is nothing to prove; we therefore assume that~$d>0$.
We work with the characterisation of compatible pairs from Proposition~\ref{prop:EquivCompatiblePair}, i.e.\ with the $\Z$-homology of universal covers.
Since we assumed that~$F$ and~$h$ induce the same isomorphism on~$H_1$,  it remains to deduce that~$F$ and~$h$ restrict to the same map on~$H_2(\partial \widetilde{X}_0)$; the use of the word ``restrict" is justified because~$H_3(\widetilde{X}_i,\partial \widetilde{X}_i) \cong H^1(\widetilde{X}_i)=0$ for~$i=0,1$.
Consider the intersection form
$$Q_{\partial \widetilde{X}_i} \colon H_1(\partial \widetilde{X}_i) \times H_2(\partial \widetilde{X}_i) \to \Z.$$
Since the~$\widetilde{X}_i$ are simply-connected,  it is known that the isometry~$F$ induces an isometry of~$Q_{\partial \widetilde{X}_i}$ for $i=0,1$; this is mentioned e.g. in~\cite[page 336]{BoyerUniqueness} and follows by considering the connecting homomorphism~$\partial \colon H_2(\widetilde{X}_i,\partial  \widetilde{X}_i) \to H_1(\partial  \widetilde{X}_i)$ and observing that~$Q_{\partial \widetilde{X}_i}(\partial(x),y)=Q_{\widetilde{X}_i}^\partial(x,\operatorname{incl}(y))$ for every~$x \in H_2(\widetilde{X}_i,\partial  \widetilde{X}_i)$ and every~$y \in H_2(\partial  \widetilde{X}_i)$.
%%%
%%Draw the diagram with long exact sequence, PD and ev.
%Q_\partial(\partial F\partial x,F(\incl(y))
%=Q^\partial(\partial F^{-*} x,F(\incl(y))
%=Q^\partial(x,\incl(y))
%=Q_\partial(\partial x,y).
Note that since the homotopy equivalence~$h$ has degree one,  it also induces such an isometry of this intersection form.
Thus, since~$\partial F=h_*$ on~$H_1(\partial \widetilde{X}_0)$,
for every~$x_1 \in H_1(\partial \widetilde{X}_1)$ and every~$x_2 \in H_2(\partial \widetilde{X}_0)$ we have
$$
0
=Q_{\partial \widetilde{X}_0}((\partial F^{-1}-h_*^{-1})(x_1),x_2)
=Q_{\partial \widetilde{X}_1}(x_1,(F-h_*)(x_2)).
$$
The form~$Q_{\partial \widetilde{X}_i}$ descends to a nonsingular form~$H_1(\partial \widetilde{X}_i)/TH_1(\partial \widetilde{X}_i) \times H_2(\widetilde{X}_i) \to \Z.$
Since~$\partial F$ and~$h$ both descend to this quotient, 
%%Don't delete
%%Isos so take torsion to torsion.
 we deduce that~$F$ and~$h$ agree on~$H_2(\partial \widetilde{X}_0)$.
Thus~$F$ and~$h$ form a compatible pair.
\end{proof}

\section{The relative~$k$-invariant condition}
\label{sec:kInvariant}

We collect situations where the $k$-invariant condition from Theorem~\ref{thm:CK4Manifold} can be omitted.

\begin{proposition}
\label{prop:NokInvariant}
Let~$X$ be a simply-connected~$4$-manifold that is either closed or has~$\partial X=S^3$, and let~$S_0,S_1 \subset X$  be~$\Z_d$-surfaces with the same (relative) Euler number.
%Not sure homologous is needed here.
%%I know that relative should be said after announcing K but it's makes it very clunky.
When~$\partial X=S^3$, fix a knot $K \subset S^3$ and assume that~$S_0,S_1$ have boundary $K$.
A compatible pair $(F,h) \colon X_{S_0} \to X_{S_1}$ preserves relative~$k$-invariants in the following circumstances:
\begin{itemize}
\item if the group $\pi$ is trivial or infinite cyclic,
\item if $\pi$ is finite and the $S_i$ are spheres, Moebius bands or projective planes.
\item if $\pi$ is finite and the $S_i$ are discs with nonzero relative Euler number.
%\item if $\pi$ is finite and $H_1(\partial X_i;\Z[\Z_d])^*=0$,
%\item if the $S_i$ are Euler number zero spheres.
\end{itemize}
\end{proposition}
\begin{proof}
When the group~$\pi:=\Z_d$ is trivial or infinite cyclic,~$\cd(\pi) \leq 2$ and the result therefore follows from~\cite[Proposition 7.13]{ConwayKasprowski4Manifolds}.
If~$\pi$ is finite cyclic and the $S_i$ are spheres (resp.\ discs) with nonzero Euler number (resp.\ relative Euler number), then Proposition~\ref{prop:AlexanderModule} implies that~$H_1(\partial X_{S_i};\Z[\Z_d])^*=0$.
The same is true for Moebius bands or projective planes (without any Euler number assumption) thanks to Proposition~\ref{prop:AlexanderModuleNonOri}.
%%Because Alexander module just has Z_2+Z_2 or Z_4; no Z_e
Since~$H_4(\pi)=0$, the result then follows from~\cite[Proposition~7.15]{ConwayKasprowski4Manifolds}.
The case of Euler number zero spheres follows from Proposition~\ref{prop:EulerNumberZeroSpheres} below.
%is the last clause of Theorem~\ref{thm:CompatiblePair}.
\end{proof}

\subsection{Euler number zero spheres}
\label{sub:EulerNumberZeroSpheres}

This section completes the proof of Proposition~\ref{prop:NokInvariant} by showing that the $k$-invariant condition can be omitted in the case of Euler number zero spheres.

\begin{lemma}
\label{lem:KPTTheoremA}
Let $\pi$ be a finite group,  let $X_0,X_1$ be $4$-manifolds with $\pi_1(\partial X_i) \to \pi_1(X_i) \cong \pi$ surjective for $i=0,1$,  and let $h \colon \partial X_0 \to \partial X_1$ be a homeomorphism.
If~$\ks(X_0)=\ks(X_1)$ and~$H_4(\pi)=0$,  then there are integers $a,a',b,b' \geq 0$ so that~$h$ extends to a homeomorphism 
$$X_0 \#_a \C P^2 \#_b \overline{\C P}^2 \xrightarrow{\cong} X_1 \#_{a'} \C P^2 \#_{b'} \overline{\C P}^2.$$
\end{lemma}
\begin{proof}
The proof is analogous to the proof of~\cite[Theorem A]{KasprowskiPowellTeichnerCP2}; we give an outline for the reader's convenience, assuming some familiarity with modified surgery~\cite{KreckSurgeryAndDuality}.
Adding copies of~$\C P^2$ to the $X_i$ ensures that their normal $1$-type becomes~$\xi=(\operatorname{BSTOP} \times B\pi \to \operatorname{BSTOP})$. 
The condition~$H_4(\pi)=0$ then implies that $\Omega_4(\xi)=\Z \oplus \Z_2$ detected by the signature and the Kirby--Siebenmann invariant. 
Add more $\C P^2$ and $\overline{\C P}^2$ to the $X_i$ if necessary so that the resulting manifolds $X_0'$ and $X_1'$ have the same signature.
It follows that the union $X_0'\cup_h -X_1'$ is nullbordant over the $1$-type. 
Hence the $X_i'$ are stably homeomorphic. 
A relative version of~\cite[Theorem~2.2]{KasprowskiPowellTeichner} then shows that the $X_i'$ (and therefore the $X_i$) are $\C P^2$-stably homeomorphic.
%For this the boundary condition is not needed. But we use it later for the doubles.
%AC: To do: Extract as a lemma before the proposition
\end{proof}

This next result implies completes the proof of Proposition~\ref{prop:NokInvariant}.

	\begin{proposition}
	\label{prop:EulerNumberZeroSpheres}
Let $\pi$ be a group with $H_4(\pi)=0$ and $H^k(\pi;\Z[\pi])=0$ for $k=2,3$, and let~$X_0, X_1$ be $4$-manifolds with fundamental group $\pi$,  and boundary homeomorphic to~$\#_{i=1}^n(S^1\times~S^2)$ such that the inclusion induces a surjection $\iota_i\colon F_n\cong \pi_1(\partial X_i)\to \pi$ for $i=0,1$.
Assume that~$\ks(X_0)=\ks(X_1)$.
If~$h\colon \partial X_0\to \partial X_1$ is a homotopy equivalence such that $\iota_1\circ h_*=\iota_0\colon F_n\to \pi$, and~$F\colon \pi_2(X_0)\to \pi_2(X_1)$ is a $\Z[\pi]$-isomorphism such that $(F,h)$ is a compatible pair, then 
$$F_*(k_{X_0,\partial X_0})=h^*(k_{X_1,\partial X_1}) \in H^3(\pi,Y_0;\pi_2(X_0)).$$
	\end{proposition}
	\begin{proof}
Throughout this proof,  for $i=0,1$, we use the Hurewicz isomorphism to identify $\pi_2(X_i)$ with $H_2(\widetilde{X}_i)$, and we write~$Y_i:=\partial X_i$ for brevity.	
Observe that we have 
$$k_{X_i\#\C P^2,Y_i}=(k_{X_i,Y_i},0)\in H^3(\pi,Y_i;\pi_2(X_i))\oplus H^3(\pi,Y_i;\Z[\pi])\cong H^3(\pi,Y_i;\pi_2(X_i\#\C P^2)).$$
Hence if the statement is true for~$X_i\#\C P^2$ and~$F\oplus\id_{\Z[\pi]}$, then it is true for~$X_i$ and~$F_i$: indeed
$$\left((F\oplus \id_{\Z[\pi]})_*-h^*\right)(k_{X_1\#\C P^2,Y_1})=(F_*(k_{X_0,Y_0})-h^*(k_{X_1,Y_1}),0).$$		
%Since~$H_4(\pi;\Z)=0$, all manifolds with fundamental group~$\pi$,  boundary~$\#_{i=1}^n(S^1\times S^2)$ and given surjection~$\iota_0 \colon F_n\to \pi$ are~$\C P^2$-stably homeomorphic, i.e.\ homeomorphic up to connected sum with~$\C P^2$ and~$\overline{\C P}^2$. 		
Let~$\langle g_1,\ldots g_n\mid R\rangle$ be a presentation of~$\pi$ such that~$\iota_0\colon F_n \to \pi$ is given by sending the~$k$-th generator of~$F_n$ to~$g_k$.
 Let~$K$ be a presentation~$2$-complex for this presentation.		
 Let~$D(K)$ be  the boundary of a thickening of~$K$ in~$\R^5$,  and let~$M:=D(K)^{(2)}$ be the handle~$2$-skeleton for the corresponding handle structure. 
% {AC: Is it always the case that~$D(K)$ is a proper double? Reading Hambleton-Nicholson to refresh my memory, they seem to say that~$D(K)$ is~$h$-cobordant to the manifold obtained by doubling a~$4$d thickening of~$K$. DK: I believe for given $K$ all the boundaries of thickening are $s$-cobordant. For non-good $\pi$ we probably want to start with an actual double (we can for example get this by taking a 4d thickening times $I$).
% 
%AC: It seems to me that when you pick the handle structure below you are using that it's a double. 
%Why not work with the double from the start? DK: Because the 5d version is easier to explain. But I think we want to start with a specific double from Kreck-Schäfer. See next note.
% }
 %%Remove 3,4-handles those are \natural_{i=1}^n(S^1\times D^3).

 That is,~$M=D(K)\setminus \natural_{i=1}^n(S^1\times D^3)$, where the map~$\natural_{i=1}^n(S^1\times D^3)\to D(K)$ is determined by~$\iota_0$.
 Since it suffices to prove the statement~$\C P^2$-stably,  by Lemma~\ref{lem:KPTTheoremA}, we can assume without loss of generality that~$X:=X_0=X_1=M\#l\C P^2\#l'\overline{\C P}^2$ for some $l,l' \geq 0$. 
 Set~$Y:=Y_0=Y_1$
 
 Our plan is now to decompose~$H^3(\pi,Y;\pi_2(X))$ as a direct sum of three summands, and to prove that the equality~$F_*(k_{X,Y})=h^*(k_{X,Y})$ holds for each of the three components.
 Such a decomposition would follow from a decomposition of~$\pi_2(X) \cong H_2(X;\Z[\pi])$ but in fact, we will need a little more. 
 In the following claim, it will be helpful to recall that the universal coefficient spectral sequence shows that~$H^1(\pi;\Z[\pi])$ injects into $H^2(Y;\Z[\pi])\cong H_2(Y;\Z[\pi])$.
 
 \begin{claim}
 	For some integers~$l,l' \geq 0$,  the exact sequence
 	\begin{equation}
 		\label{eq:LESDecompo}
 		0 \to H_2(Y;\Z[\pi])/H^1(\pi;\Z[\pi])\to H_2(X;\Z[\pi])\to H_2(X,Y;\Z[\pi])\to H_1(Y;\Z[\pi])\to 0
 	\end{equation}
 	decomposes as a direct sum of the three following exact sequences:
 	\begin{align}
 		&0\to \im(d^2_K)\to C^2(\wt K)\xrightarrow{\proj} \coker(d^2_K)\to 0\to0, \label{eq:1} \\
 		&0\to 0\to \ker(d_2^K)\to C_2(\wt K)\xrightarrow{d_2^K} \ker d_1^K\to 0,\label{eq:2} \\
 		&0\to0\to \Z[\pi]^{l+l'}\xrightarrow{\id}\Z[\pi]^{l+l'}\to 0\to 0.\label{eq:3}
 	\end{align}
 	In particular,  there is an isomorphism
 	$$\pi_2(X) \cong C^2(\wt K)\oplus \ker(d_2^K) \oplus  \Z[\pi]^{l+l'}.$$
 \end{claim}
 \begin{proof}
 	Since~$X=M\#l\C P^2\#l'\overline{\C P}^2$, it suffices to prove that the sequence in~\eqref{eq:LESDecompo} with~$(X,Y)$ substituted for~$(M,\partial M)$ decomposes as the direct sum of the sequences in~\eqref{eq:1} and~\eqref{eq:2}.
 	The inclusion~$(M,\partial M)\to (D(K),\natural_{i=1}^n(S^1\times D^3))$ induces the following commutative diagram, where the middle map is an isomorphism thanks to excision:
 	%%%Don't delete.
 	%{AC: Why is the middle map an iso? Are you thinking of it as the inclusion inducing a homeo~$M/\partial M \cong D(K)/\natural S^1 \times D^3$. DK: This is excision. $\natural_{i=1}^n(S^1\times D^3)$ is the complement of $M$ (it's given by the 3 and 4-handles with in the dual decomposition are a 0-handle with some 1-handles.)}
 	$$
 	\xymatrix@C0.3cm{
 		0\ar[r]&\frac{H_2(\partial M;\Z[\pi])}{H^1(\pi;\Z[\pi])}\ar[r]&H_2(M;\Z[\pi])\ar[r]\ar[d]&H_2(M,\partial M;\Z[\pi])\ar[r]\ar[d]^{\cong}&H_1(\partial M;\Z[\pi])\ar[r]\ar[d]^{\cong}&0 \\
 		%%%%
 		&0\ar[r]&H_2(D(K);\Z[\pi])\ar[r]&H_2(D(K),\natural_{i=1}^n(S^1\times D^3);\Z[\pi])\ar[r]&H_1(\natural_{i=1}^n(S^1\times D^3);\Z[\pi])\ar[r]&0.
 	}
 	$$
 	%		\[\begin{tikzcd}
 		%			H_2(M;\Z[\pi])\ar[d]\ar[r]&H_2(M,\partial M;\Z[\pi])\ar[d,"\cong"]\ar[r]&H_1(\partial M;\Z[\pi])\ar[d,"\cong"]\ar[r]&0\\
 		%			H_2(D(K);\Z[\pi])\ar[r]&H_2(D(K),\natural_{i=1}^n(S^1\times D^3);\Z[\pi])\ar[r]&H_1(\natural_{i=1}^n(S^1\times D^3);\Z[\pi])\ar[r]&0.
 		%		\end{tikzcd}\]
 	We now analyze the bottom row and deduce the isomorphism type of the upper row.
 	As explained in e.g.~\cite[Proposition II.3]{KreckSchafer}, the handle complex of~$D(K)$ is given by 
 	%{
 		%DK: The reference is Kreck-Schäfer Proposition II.3. They say they use a specific model, but I don't think it is relevant. Their proof should work for all doubles. But we can also just use their model to make the citation easier.
 		%}
 	
 	\[0 \to C^0(\wt K)\xrightarrow{d^1_K} C^1(\wt K)\xrightarrow{\bsm d_K^2 \\ 0 \esm}C^2(\wt K)\oplus C_2(\wt K)\xrightarrow{\bsm 0 \ d_2^K \esm}C_1(\wt K)\xrightarrow{d_1^K}C_0(\wt K) \to 0 .\]	
 	Using that the inclusion~$\natural_{i=1}^n(S^1\times D^3) \hookrightarrow D(K)$ is homotopic to the inclusion of the~$1$-skeleton, we see that bottom line of the diagram is isomorphic to
 	$$
 	\xymatrix@C0.4cm{
 		&0\ar[r]&H_2(D(K);\Z[\pi])\ar[r]\ar[d]^{\cong}&H_2(D(K),\natural_{i=1}^n(S^1\times D^3);\Z[\pi])\ar[r]\ar[d]^{\cong}&H_1(\natural_{i=1}^n(S^1\times D^3);\Z[\pi])\ar[r]\ar[d]^{\cong}&0 \\
 		%%%%%
 		&0\ar[r]&
 		\coker d_K^2\oplus \ker(d_2^K) \ar[r]&
 		\coker d_K^2\oplus C_2(\wt K) \ar[r]^{\bsm 0 & d_2^K \esm}&
 		\ker d_1^K
 		\ar[r]&0.
 	}
 	$$
 	%		$$
 	%0\to \overbrace{\cong H_2(D(K);\Z[\pi])}^{\coker d_K^2\oplus \ker(d_2^K)} \to \overbrace{H_2(D(K),\natural_{i=1}^n(S^1\times D^3);\Z[\pi])}^{\cong  \coker d_K^2\oplus C_2(\wt K)} \to \overbrace{H_1(\natural_{i=1}^n(S^1\times D^3);\Z[\pi])}^{\cong \ker d_1^K} \to 0.
 	%		$$		
 	%		\[0 \to \coker d_K^2\oplus \ker(d_2^K)\to \coker d_K^2\oplus C_2(\wt K)\to \ker d_1^K \to 0.\]		
 	%Note in particular that $H_2(D(K);\Z[\pi]) \cong \coker d_K^2\oplus \ker(d_2^K)$.
 	Since $C_*(\widetilde{M})$ is obtained by truncating $C_*(\widetilde{D}(K))$,  we deduce that $H_2(M;\Z[\pi]) \cong  C^2(\wt K)\oplus \ker(d_2^K)$ and that the inclusion induced map~$H_2(M;\Z[\pi])\to H_2(D(K);\Z[\pi])$ is given by the projection
 	$$
 	\overbrace{H_2(M;\Z[\pi])}^{\cong C^2(\wt K)\oplus \ker(d_2^K)} 
 	\to \overbrace{H_2(D(K);\Z[\pi])}^{\cong \coker d_K^2\oplus \ker(d_2^K)}.
 	$$
 	In conclusion, we obtain the following commutative diagram of chain complexes:
 	$$
 	\xymatrix@C0.8cm{
 		0\ar[r]&
 		\frac{H_2(\partial M;\Z[\pi])}{H^1(\pi;\Z[\pi])}\ar[r]\ar@{-->}[d]^{\cong}&
 		H_2(M;\Z[\pi])\ar[r]\ar[d]^{\cong}&
 		H_2(M,\partial M;\Z[\pi])\ar[r]\ar[d]^{\cong}
 		&H_1(\partial M;\Z[\pi])\ar[r]\ar[d]^{\cong}&
 		0\\
 		%%%%%%
 		0\ar[r]&
 		\im(d^2_K) \ar[r]^-{\bsm \incl \\ 0 \esm}&
 		C^2(\widetilde{K}) \oplus \ker(d_2^K) \ar[r]^-{\proj\oplus \incl}&
 		%\bsm \proj &0 \\ 0 & \incl \esm}&
 	\coker(d^2_K) \oplus C_2(\widetilde{K}) \ar[r]^-{\bsm 0&d_2^K\esm}&
 	\ker(d_1^K)\ar[r]&0.
 	%%%%%%
 }
 $$
 A direct verification now shows that the bottom sequence decomposes as claimed.
 \end{proof}

 We now show that~$F_*(k_{X,Y})=h^*(k_{X,Y})$.
 Let~$B$ be a model for the Postnikov~$2$-type of~$X$ that contains~$Y$, and let~$c \colon X  \to B$ and~$c_1 \colon X \to B$ be~$3$-connected maps with~$c=c_1 \circ~h$. 
 To see that these maps exist,  use~$H_4(\pi)=0$ to argue as in the proof of~\cite[Proposition~7.15]{ConwayKasprowski4Manifolds}.
 Since~$F|_{H_2(Y;\Z[\pi])/H^1(\pi;\Z[\pi])}=h_*=(c_1)_*^{-1}c_*$, 
 %%Don't delete
 %% F_*=h_* on H_2(Y). Can you remind me why they agree on H^1(\pi) (and therefore on H_2(Y)/H^1(\pi))?
% By assumption, h_* is compatible with the identifications of pi_1(X_i) with pi, hence H^1(pi)->H^1(Y)-h^*->H^1(Y) is the same as H^1(pi)->H^1(Y). Thus h^* or h_*:H_2(Y)->H_2(Y) by PD induces a map H_2(Y)/H^1(pi)->H_2(Y)/H^1(pi). Since F| is the same map on H_2(Y), this is also true for F| and the maps still agree.
the map~$F_*-(c_1)_*^{-1}c_*\colon H_2(X;\Z[\pi]) \to H_2(X;\Z[\pi])$ factors through $H_2(X;\Z[\pi])/H_2(Y;\Z[\pi])$. 
Since~$(F,h)$ is a compatible pair and~$H^k(\pi;\Z[\pi])=0$ for~$k=2,3$, the isomorphism~$F$ defines an isomorphism~$(F^*)^{-1}$ on relative homology.
%%%Don't delete
%I rewrote the proof of 4.3 using additionally that H^k(\pi;Z\pi)=0 for k=2,3 and that (F,h) is a compatible pair. The latter is only needed to get a map on H_2(X,Y) it does not really matter whether it is F^*^{-1} or some other map. So if we really want, we could have F,G,h being an iso of sequences but I don't think this extra generality would ever be useful. 
Since both~$(F^*)^{-1}$ and~$(c_1)_*^{-1}c_*$ induce~$h_*$ on~$H_1(Y;\Z[\pi])$,
 the map~$(F^*)^{-1}-(c_1)_*^{-1}c_* \colon H_2(X,Y;\Z[\pi]) \to H_2(X,Y;\Z[\pi])$ has image in~$\ker(\partial\colon H_2(X,Y;\Z[\pi])\to H_1(Y;\Z[\pi]))$. 
 Thus we obtain the commutative diagram
 \[\xymatrix{
 H_2(X;\Z[\pi])/H_2(Y;\Z[\pi])\ar[r]\ar[d]^{F_*-(c_1)_*^{-1}c_*}&H_2(X,Y;\Z[\pi])\ar[d]^{(F^*)^{-1}-(c_1)_*^{-1}c_*}\\
 H_2(X;\Z[\pi])\ar[r]&\ker(\partial).
}\]
We assert that the map~$F_*-(c_1)_*^{-1}c_*$ on the left of this diagram factors through $H_2(X,Y;\Z[\pi])$.
Using the decomposition of \eqref{eq:LESDecompo}, the bottom row is the direct sum
$$
C^2(\wt K) \oplus \ker(d_2^K) \oplus \Z[\pi]^{l+l'}\xrightarrow{ \proj \oplus \id \oplus \id} \coker(d^2_K) \oplus \ker(d_2^K) \oplus \Z[\pi]^{l+l'}.
$$
Thus, if we write the left vertical map as $(F_1,F_2,F_3)$, then the presence of the two identity summands in the bottom horizontal maps implies that $F_2$ and $F_3$ factor through~$H_2(X,Y;\Z[\pi])$.
%%Don't delete
%%Just go up-right, then down, then down-left (because down left is id so an iso).
It remains to show that $F_1 \colon C^2(\wt K) \to  H_2(X;\Z[\pi])$ factors through $ H_2(X,Y;\Z[\pi])$.
Using the decomposition of \eqref{eq:LESDecompo} for the top row, we obtain the exact sequence
$$0\to H_2(X;\Z[\pi])/H_2(Y;\Z[\pi])\to H_2(X,Y;\Z[\pi])\to \ker d_1^{K}\to 0.$$
Applying the Ext functor~$\Ext^*_{\Z[\pi]}(-;C^2(\wt K))$, we get the exact sequence
\[\Hom_{\Z[\pi]}(H_2(X,Y;\Z[\pi]),C^2(\wt K))\to \Hom_{\Z[\pi]}\left(\frac{H_2(X;\Z[\pi])}{H_2(Y;\Z[\pi])}, C^2(\wt K)\right)\to \Ext^1_{\Z[\pi]}(\ker d_1^{K};C^2(\wt K)).\]
%%Don't delete.
%0->ker(d_1)->C_1->Im(d_1)->0
%0->Im(d_1)->C_0->H_0=Z->0
Note that~$\Ext^1_{\Z[\pi]}(\ker d_1^{K};C^2(\wt K))\cong \Ext^3(\Z;C^2(\wt K))\cong H^3(\pi;C^2(\wt K))=0$, where the last equality uses that~$C^2(\wt K)$ is~$\Z[\pi]$-free together with our assumption that~$H^3(\pi;\Z[\pi])=0.$
It therefore follows that the map~$F_1 \colon C^2(\wt K) \to  H_2(X;\Z[\pi])$ factors through~$ H_2(X,Y;\Z[\pi])$, and as a consequence that~$F_*-(c_1)_*^{-1}c_*=(F_1,F_2,F_3)$ factors through~$H_2(X,Y;\Z[\pi])$,  as asserted.

Using the assertion,
%We deduce that the map~$F_*-(c_1)_*^{-1}c_*$ factors through~$H_2(X,Y;\Z[\pi])$. 
since~$k_{X,Y}$ maps trivially into~$H^3(\pi,Y;H_2(X,Y;\Z[\pi]))$ (this is \cite[Lemma~6.1]{ConwayKasprowskiKinvariant}), we conclude that the image of~$k_{X,Y}$ under~$F_*-(c_1)_*^{-1}c_*$ is trivial. 
This proves that~$F_*(k_{X,Y})=(c_1)_*^{-1}c_*(k_{X,Y})=h^*(k_{X,Y})$, as needed.
	\end{proof}

\section{Spin unions}
\label{sec:SpinUnion}

In this section,  we assume that the universal covers~$\widetilde{X}_{S_0}$ and~$\widetilde{X}_{S_1}$ are spin and study when it is possible to extend a homeomorphism  $S_0 \to S_1$ to a bundle isomorphism~$H \colon \overline{\nu}(S_0) \to \overline{\nu}(S_1)$ such that the universal cover of~$X_{S_0} \cup_h X_{S_1}$ is spin; here recall that~$h:=\id_{E_K} \cup H$.
Recall that this condition is necessary in order to apply Theorem~\ref{thm:CKForSurfaceExteriors}.
\medbreak

\begin{remark}
\label{rem:Intertwine}
Recall from Definition~\ref{def:NiceFraming-or} that given an~$e$-nice identification,  there is an associated~$e$-nice section~$s \colon S \to \partial X_S\setminus E_K$ such that~$s([\gamma])=0 \in H_1(X_S)$ for every~$\gamma \subset S$.
For brevity, we also say that a bundle isomorphism~$H \colon \nu(S_0) \to \nu(S_1)$
that is the identity on a closed tubular neighborhood of $K$
 \emph{intertwines the coefficient systems to~$\Z_d$} 
 %if it is the identity on a tubular neighborhood of $K$ and
 if the homeomorphism~$h:=\id_{E_K} \cup H|$ is such that~$h_* \colon \pi_1(\partial X_{S_0}) \to \pi_1(\partial X_{S_1})$ commutes with the coefficient system described in Construction~\ref{cons:CoeffSys}.
%This occurs if $H$ sends one $e$-nice section to the other.
Observe that intertwines the coefficient systems to~$\Z_d$ if and only if~$H$ preserves a choice of $e$-nice sections,  meaning that~$H \circ s_0=s_1 \circ \theta$ for some $e$-nice sections $s_0,s_1$ and some homeomorphism $\theta \colon S_0 \to S_1$.
%%Don't delete.
%commute: genus curves map to zero under push off.
%%but then since H preserves section and thus push offs, true for the other section.
%%Conversely if one has $H$ intertwines the coefficient systems and one picks $\fr_0$ then $\fr_1 :=\fr_0 \circ H$ will also be $e$-nice because $H$ preserves the coefficient systems.
\end{remark}

\begin{convention}
Thoughout this section, given a bundle isomorphism $H \colon (\overline{\nu}(S_0),S_0) \to  (\overline{\nu}(S_1),S_1)$ with $H|_{\overline{\nu}(K)}=\id$, we write
$h:=\id_{E_K} \cup H| \colon \partial X_{S_0} \to \partial X_{S_1}$.
\end{convention}

While Propositions~\ref{prop:BoyerAdaptSpin},~\ref{prop:Audrick} and ~\ref{prop:SpindEvenAutomatic} are key inputs for Theorem~\ref{thm:CompatiblePairIntro},  the results of this section also combine to yield the following theorem that might be of independent interest.

\begin{theorem}
\label{thm:SpinUnion}
Let~$X$ be a simply-connected~$4$-manifold with boundary~$\partial X=S^3$,  let $K$ be a knot such that when~$d>0$ is not a prime, $\Sigma_d(K)$ is a rational homology sphere,
and let~$S_0,S_1 \subset~X$ be~$\Z_d$-surfaces with boundary $K \subset S^3$ whose exteriors have spin universal covers.
%%Don't delete.
%%Don't need homologous this time.
Fix a rel.\ boundary homeomorphism 
$$ \theta \colon S_0 \to S_1.$$
Assume that $S_0$ and $S_1$ have the same relative Euler number and, when $S_0,S_1$ are characteristic,  additionally assume that $\theta$ preserves the Freedman--Kirby or Guillou--Marin quadratic forms depending on whether the $S_i$ are orientable or not.
Then $\theta$ extends to a bundle isomorphism
$$H\colon (\overline{\nu}(S_0),S_0) \to (\overline{\nu}(S_1),S_1)$$
that intertwines the coefficient systems to $\Z_d$.
%%Don't delete
%%Understood to mean H|_{\nu(K)}=\id, see def.
Furthermore:
\begin{itemize}
\item  Given such a bundle isomorphism,  there exists another bundle isomorphism $H'$ with the same properties and such that, in addition, the universal cover of~$X_{S_0} \cup_{h'} X_{S_1}$ is spin.
When the $S_i$ are characteristic or $d>0$ is even,  $H'=H$.
\item When the $S_i$ are characteristic or $X$ is spin, the condition on $\Sigma_d(K)$ can be omitted.
In those cases,~$X_{S_0} \cup_{h} X_{S_1}$ is itself spin.
\item 
When~$H_1(\partial X_{S_i};\Z[\Z_d])^*=0$, for $i=0,1$  no condition on $H$ is required for the universal cover of~$X_{S_0} \cup_{h} X_{S_1}$ to be spin.
\end{itemize}
\end{theorem}

We record the following consequence of Theorem~\ref{thm:SpinUnion} as it might be of independent interest.
Note that this is a generalisation of~\cite[Proposition 6.2]{ConwayOrsonPowell} which concerned $\Z_2$-surfaces in $D^4$ 
and~\cite[Section 3]{Pyronneau} which treats orientable surfaces whose complements are simply-connected.
\begin{corollary}
\label{cor:SpinAutomatic}
Let~$X$ be a simply-connected~$4$-manifold with boundary~$\partial X=S^3$,  let $K$ be a knot, and let~$S_0,S_1 \subset X$ be~$\Z_d$-surfaces with boundary $K \subset S^3$ whose exteriors have spin universal covers.
Assume that $S_0$ and $S_1$ have the same relative Euler number.
There is a bundle isomorphism~$H\colon (\overline{\nu}(S_0),S_0) \to (\overline{\nu}(S_1),S_1)$
with $H|_{\overline{\nu}(K)}=\id$
 such that 
\begin{itemize}
\item if $\Sigma_d(K)$ is a rational homology sphere,  then the universal cover of~$X_{S_0} \cup_h X_{S_1}$ is spin. 
\item if the~$[S_i]$ are characteristic or $X$ is spin,  then the union~$X_{S_0} \cup_{h} X_{S_1}$ is spin.
\end{itemize}
\end{corollary}
\begin{proof}
In all but the characteristic case,  the corollary is a consequence of Theorem~\ref{thm:SpinUnion}.
It therefore remains to prove that in the characteristic case,  there exists a quadratic-form-preserving rel.\ boundary homeomorphism $\theta$ as in Theorem~\ref{thm:SpinUnion}.

Since the $S_i$ have a single boundary component, their intersection forms are nonsingular.
The isometry type of a~$\Z_2$-valued quadratic refinement of a~$\Z_2$-valued nonsingular symmetric form on a~$\Z_2$-vector-space is determined by its Arf invariant~\cite[Theorem~III.1.12]{BrowderSimply} and by its Brown invariant in the case of a~$\Z_4$-valued quadratic refinement~\cite[final remark of Section~II]{GuillouMarin}.
Since the surfaces~$S_0$ and~$S_1$ have the same relative Euler number and boundary,  it follows that~$\operatorname{Arf}(q_{S_0})=\operatorname{Arf}(q_{S_1})$ or~$\operatorname{Brown}(q_{S_0})=\operatorname{Brown}(q_{S_1})$ depending on whether or not the~$S_i$ are orientable (see e.g.~\cite{FreedmanKirby,Klug,GilmerLivingston}) and thus,  since the intersection forms of the~$S_i$ are nonsingular, we deduce that~$(H_1(S_0;\Z_2),Q_{S_0},q_{S_0}) \cong (H_1(S_1;\Z_2),Q_{S_1},q_{S_1})$.
Realise this isometry by a rel.\ boundary homeomorphism~$\theta \colon S_0 \to S_1$; see e.g.~\cite[Section 2.1 and the discussion following Theorem 6.4]{FarbMargalit} in the orientable case and~\cite{GadgilPancholi} in the nonorientable case.
\end{proof}

We describe the proof structure of Theorem~\ref{thm:SpinUnion}.
Propositions~\ref{prop:BoyerAdaptSpin} and~\ref{prop:Audrick} prove the result when~$X$ is spin and the~$[S_i]$ are characteristic respectively.
Proposition~\ref{prop:SpindEvenAutomatic} is concerned with the remaining cases.

% then show that in each of this situations, the conclusion of Theorem~\ref{thm:SpinUnion} can be achieved.
%as well as the structure of this section

\begin{remark}
Lee and Wilczynski state the special case of Theorem~\ref{thm:SpinUnion} that the union of sphere exteriors can always be taken to have the same~$w_2$-type as the individual 
%pieces
exteriors~\cite[Bottom of page 340]{LeeWilczy}.
%  we do not find their argument to be convincing.
\end{remark}

Before carrying out the plan described above, we record a fact about vector bundles.

\begin{lemma}
\label{lem:ExtendHomeoOverDiscBundle}
let~$X$ be a simply-connected~$4$-manifold with boundary~$\partial X=S^3$,  let~$K$ be a knot, and let~$S_0,S_1 \subset~X$ be surfaces with boundary $K$,  and fix a rel.\ boundary homeomorphism~$\theta \colon S_0 \to S_1.$
If $S_0$ and $S_1$ have the same relative Euler number, then $\theta$ extends to a bundle isomorphism
$$H\colon (\overline{\nu}(S_0),S_0) \to (\overline{\nu}(S_1),S_1).$$
This bundle isomorphism can be assumed to be the identity on $\overline{\nu}(K)$ and, additionally,
given a choice of $e$-nice identifications and associated $e$-nice sections $s_0,s_1$,  it can also be assumed to send one $e$-nice section to the other, i.e.\ to satisfy~$H \circ s_0=s_1 \circ \theta$.
In particular,~$H$ intertwines the coefficient systems to~$\Z_d$.
\end{lemma} 
\begin{proof}
A proof can be obtained using the argument from~\cite[Lemma 6.20]{ConwayOrsonPowell}.
\end{proof}

\subsection{The case when $X$ is spin.}
\label{sub:CaseWorkSpin}

We prove a strengthening of Theorem~\ref{thm:SpinUnion} when $X$ is spin.
This involves an argument that appeared in work of Boyer when~$\pi=1$ (see~\cite[pages 45-46]{BoyerRealization}) but as we now argue, applies more generally.
Here recall that when $S$ is a nonorientable $\Z_d$-surface, ``$d$ odd" is understood to mean $d=1$.

\begin{proposition}
\label{prop:BoyerAdaptSpin}
Assume~$X$ is spin.
Let $K \subset S^3$ be a knot,  let~$S_0,S_1 \subset X$ be $\Z_d$-surfaces with boundary $K$, and let~$H \colon (\overline{\nu}(S_0),S_0) \to (\overline{\nu}(S_1),S_1)$ be a bundle isomorphism 
with $H|_{\overline{\nu}(K)}=\id$
%%Don't delete.
%%The existence of H implie same relative Euler numbers. 
%%I don't think homologous is needed.
\begin{itemize}
\item  There is another bundle isomorphism~$H' \colon (\overline{\nu}(S_0),S_0) \to (\overline{\nu}(S_1),S_1)$ with $H'|_{S_0}=H|_{S_0}$
%{AC: The homeo~$\Rot(\alpha)$ is the identity on $\overline{\nu}(K)$ right?}
%%We say that Rot(\alpha) is the identity over K but that means being the identity on \nu(K)
and $H'|_{\overline{\nu}(K)}=\id$
such that~$X_{S_0} \cup_{h'} X_{S_1}$ is spin.
\item If~$d$ is odd and $H$ intertwines the coefficient systems to~$\Z_d$, then so does $H'$.
\item If $d$ is odd, $S_0,S_1$ are orientable, and $(F,H)$ is a $k$-invariant preserving extendable compatible pair,  then there is another~$k$-invariant preserving extendable compatible pair $(F',H')$ with $H'|_{S_0}=H_{S_0}$ and such that $X_0 \cup_{h'} X_1$ is spin.
\end{itemize}
\end{proposition}
\begin{proof}
We build on the argument from~\cite[Page 45]{BoyerRealization}.
Restrict the spin structure on $X$ to~$X_{S_i}$ and then further to a spin structure~$\mathfrak{s}_i$ on~$\partial X_{S_i}=Y_i \cup E_K$.
Note that since~$\mathfrak{s}_i$ is obtained from the spin structure on $X$, it extends to a spin structure~$\widetilde{\mathfrak{s}}_i$ on~$\overline{\nu}(S_i)$.
%maybe uniquely.
%%%Don't delete
%%Obstruction theory?
%H^1(\nu,\partial \nu)\cong H_3(\nu)=0 many spin structures
%%%Or:
%%if the spin structure extends over nu S, then it does so uniquely because H^1(nu S;Z/2)->H^1(Y;Z/2) is injective (and spin structure are 1:1 with H^1(-;Z/2)).
If~$H$ were to take~$\widetilde{\mathfrak{s}}_1$ to~$\widetilde{\mathfrak{s}}_0$, then~$X_{S_0} \cup_h X_{S_1}$ would be spin.
We will precompose $H$ with a bundle isomorphism in order to arrange for this to hold.

The spin structures~$H^*(\widetilde{\mathfrak{s}}_1)$ and~$\widetilde{\mathfrak{s}}_0$ differ by an element of~$H^1(\overline{\nu}(S_0);\Z_2)$.
Write $p \colon \overline{\nu}(S_0) \to~S_0$ for the projection and fix an~$x \in H^1(S_0;\Z_2) \cong H^1(S_0,\partial S_0;\Z_2)$ so that~$p^*(x) \cdot H^*(\widetilde{\mathfrak{s}}_1)=\widetilde{\mathfrak{s}}_0$.
Pull the class~$x \in H^1(S_0,\partial S_0;\Z_2)$  back to an $\alpha \in H^1(S_0,\partial S_0;\Z^w)$ and consider~$\Rot(\alpha) \colon \overline{\nu}(S_0) \to \overline{\nu}(S_0)$.
%,z \mapsto \gamma(p(z)) \cdot z$, where $\cdot$ denotes the action $S^1$ on $\overline{\nu}(S_0)$ by rotation of the fibres.
Recall that~$\Rot(\alpha)|_{S_0}=\id_{S_0}$.
Propositions~\ref{prop:RotalphaSpinOrientable} and~\ref{prop:RotalphaSpinNonOrientable} imply that~$H':=H \circ \Rot(\alpha)$ now satisfies
%~$(H')^*(\widetilde{\mathfrak{s}}_1)=\widetilde{\mathfrak{s}}_0$:
$$
(H')^*(\widetilde{\mathfrak{s}}_1)
=\Rot(\alpha)^*H^*(\mathfrak{s}_1)
=p|^*([\alpha]_2) H^*(\mathfrak{s}_1)
=p^*(x) \cdot H^*(\mathfrak{s}_1)
=\mathfrak{s}_0.
$$
%that~$H':=H \circ \Rot(\alpha)$ now satisfies~$(H')^*(\widetilde{\mathfrak{s}}_1)=\widetilde{\mathfrak{s}}_0$.
Proposition~\ref{prop:RotAlphaAxiomsSection2} implies that~$\Rot(\alpha) \colon \overline{\nu}(S_0) \to \overline{\nu}(S_0)$ induces~$\bsm \id &\alpha \\ 0 & \id \esm$ on $H_1(\partial \overline{\nu}(S_0)) \cong \Z \oplus H_1(S_0)$; recall Proposition~\ref{prop:RotAlphaAxiomsSection2Nonori} in the nonorientable case.
Thus when~$H$ intertwines the coefficient systems to~$\Z_d$,  the composition~$H \circ \Rot(\alpha)$ intertwines the coefficient systems if and only if~$\alpha$ is divisible by~$d$.
%%Don't 
%%Rot(\alpha) has that component \alpha \colon H_1(\Sigma) \to \Z \to \Z_d that is given by \alpha you need \alpha to be divisible by $d$.
Thus, in general,  replacing~$H$ by~$H':=H \circ \Rot(\alpha)$ as we did above will not preserve this property.
When~$d$ is odd, we can instead define~$H':=H \circ \Rot(d\alpha)$ because the previous explanation shows that when~$d$ is odd,~$\Rot(d\alpha)$ and~$\Rot(\alpha)$ induce the same map on~$\Z_2$-homology.
%%And Rot^d will be ok for coefficient systems.
%%%In the nonorientable case d odd means d=1.

At this point, we have proved the first two items and it remains to prove the third.
The new compatible pair $(F',H')$ will be constructed by applying Proposition~\ref{prop:BuildCompatiblePair} below to
$$f_\partial:=\Rot(d\alpha)-\id \colon H_1(\partial X_{S_0};\Z[\Z_d]) \to H_1(\partial X_{S_0};\Z[\Z_d]).$$
Note that since Proposition~\ref{prop:RotAlphaAxiomsSection2} ensures that~$\Rot(d\alpha)=\bsm \id & \alpha \\ 0&\id \esm$, it follows that~$\Rot(d\alpha)-\id$ is of the form~$\bsm 0&* \\ 0&0 \esm$.
This verifies the assumption of Proposition~\ref{prop:BuildCompatiblePair} whose conclusion concerns the 
%$\Z[\Z_d]$-linear 
map~$Nf_\partial(x):=\sum_{g\in\Z_d}gf(g^{-1}x).$
Here,  since the $\Z_d$-action on $H_1(\Sigma)$ is trivial, 
%%Don't delete.
%%The map * is defined on H_1(\Sigma); that's the only nontrivial component of Nf_\partial
 this map simplifies to~$Nf_\partial=d f_\partial$ and thus~$d \cdot Nf_\partial=d^2 f_\partial =\Rot(d^3 \alpha)-\id$
%%Don't delete.
%d^2 f_\partial=d(\Rot(d\alpha)-\id)=d^2triangular matrix with d\alpha=triangular matrix with d^3\alpha=\Rot(d^3 \alpha)-\id.$ 
and~$d \cdot N f_\partial+ \id=\Rot(d^3\alpha)$.
Proposition~\ref{prop:BuildCompatiblePair} provides an isometry~$G \colon \lambda_{X_{S_0}} \cong \lambda_{X_{S_0}}$ so that~$(G,\Rot(d^3\alpha))$
 is~$k$-invariant preserving and satisfies
$$ \partial G=\Rot(d^3\alpha) \colon H_1(\partial X_{S_0};\Z[\Z_d]) \to H_1(\partial X_{S_0};\Z[\Z_d]).$$
Proposition~\ref{prop:CompatibleSimplification} then implies that~$(G,\Rot(d^3\alpha))$ forms a $k$-invariant preserving compatible pair.
We then set~$(F',h'):=(F,H) \circ (G,\Rot(d^3\alpha)) $ which is again a $k$-invariant preserving compatible pair.
Since~$d$ is odd,  so is $d^3$ and thus,  as proved in the second item $X_{S_0} \cup_{h'} X_{S_1}$ is spin.
\end{proof}

It remains to prove the result (namely Proposition~\ref{prop:BuildCompatiblePair}) that was used during this argument.
%We begin with a preliminary lemma.
%	\begin{lemma}
%	\label{lem:kinvariantd}
%		Let~$(X_0,Y_0)$ and~$(X_1,Y_1)$ be~$4$-dimensional Poincar\'e pairs with~$\pi_1(X_i) \cong \Z_d$ and~$\iota_j \colon \pi_1(Y_i) \to \pi_1(X_i)$ surjective  for $i=0,1$.
%		Fix a degree one homotopy equivalence $h \colon Y_0 \to Y_1$ and assume there is a group isomorphism $u \colon \pi_1(X_0) \to \pi_1(X_1)$ that satisfies~$u \circ \iota_0=\iota_1 \circ h_*$.
%		If there is a compatible triple~$(F,G,h)$, then 
%		$$F_*(d \cdot k_{X_0,Y_0})=h^*(d \cdot k_{X_1,Y_1}) \in H^3(B\pi,Y_0;\pi_2(X_1)).$$
%	\end{lemma}
%	\begin{proof}
%		The proof is essentially the same as the proof of \cite[Proposition~7.13]{ConwayKasprowski4Manifolds}, except that in the first line of the claim, the map~$\ev\colon H^2(X_1;H_2(X_1;\Z[\pi]))\to \Hom_{\Z[\pi]}(H_2(X_1;\Z[\Z_d]),H_2(X_1;\Z[\Z_d]))$ is no longer surjective but potentially has cokernel \[\Ext^3_{\Z[\Z_d]}(H_0(X_1;\Z[\Z_d]),H_2(X_1;\Z[\Z_d]))\cong \Ext^3_{\Z[\Z_d]}(\Z,I)\cong \Z_d.\]
%		Hence the claim only shows that the right-hand square in the diagram commutes for elements of divisibility~$d$. Hence the paragraph after the claim only applies to~$dk_{X_1,Y_1}$ and~$dk_{X_0,Y_0}$ instead of~$k_{X_1,Y_1}$ and~$k_{X_0,Y_0}$.
%	\end{proof}
For this, given $\Z[\Z_d]$-modules~$A,B$ and a $\Z$-linear map $f \colon A \to B$,   we will consider the associated~$\Z[\Z_d]$-linear map~$N f\colon A \to B$ defined by 
$$Nf(x)=\sum_{g\in\Z_d}gf(g^{-1}x).$$
We are now ready to establish the result that was used during the proof of Proposition~\ref{prop:BoyerAdaptSpin}.
		\begin{proposition}
		\label{prop:BuildCompatiblePair}
		Let $d>0$ and let~$S \subset X$ be an orientable~$\Z_d$-surface with Euler number $e$, and recall the decomposition
		$$H_1(\partial X_S;\Z[\Z_d])\cong  H_1(\Sigma_d(K)) \oplus \Z_{e/d} \oplus H_1(\Sigma).$$
		Given a $\Z$-linear map~$f\colon H_1(\Sigma)\to \Z_{e/d}\oplus H_1(\Sigma)$ and
		$$
		f_\partial:=\bsm 0&f \\0&0 \esm \colon H_1(\partial X_S;\Z[\Z_d]) \to H_1(\partial X_S;\Z[\Z_d]),
		$$
		there is an isometry $G \colon \lambda_{X_S} \cong \lambda_{X_S}$ such that
		\begin{enumerate}
			\item the maps $\partial G$ and $\id_\partial+Nf_\partial$ agree on $H_1(\partial X_S;\Z[\Z_d])$,
			% the pair~$(G,\id_\partial+Nf_\partial)$ is compatible,
			\item %the pair~$(G,\id_\partial+d \cdot Nf_\partial)$ is compatible 
			the maps $\partial G$ and $\id_\partial+d \cdot Nf_\partial$ agree on $H_1(\partial X_S;\Z[\Z_d])$ and, additionally satisfy
			$$ G_*(k_{X_S,\partial X_S})=(\id_\partial+d \cdot Nf_\partial)^*(k_{X_S,\partial X_S}).$$
			%and $k$-invariant preserving in the sense that 
		\end{enumerate}
	\end{proposition}
	\begin{proof}
We begin with the construction of the map~$G$.
As a first step, we construct a~$\Z$-linear map~$\phi' \colon H_2(\partial X_S;\Z[\Z_d]) \to H_2(X_S;\Z[\Z_d])^*$ such that $f_\partial$ factors as 
		\[H_1(\partial X_S;\Z[\Z_d])\xrightarrow{\ev\circ\PD^{-1}}H_2(\partial X_S;\Z[\Z_d])^*\xrightarrow{~\phi'~}H_2(X_S;\Z[\Z_d])^*\xrightarrow{\partial\circ\PD\circ\ev^{-1}}H_1(\partial X_S;\Z[\Z_d]).\]
Use Proposition~\ref{prop:AlexanderModule} to view~$H_1(\Sigma)$ as a summand of~$H_1(\partial X_S;\Z[\Z_d])$ and consider the composition~$H_1(\Sigma) \subset H_1(\partial X_S;\Z[\Z_d]) \xrightarrow{f} H_1(\partial X_S;\Z[\Z_d])$.
Since the surface~$\Sigma$ is orientable,~$H_1(\Sigma)$ is free abelian and so this map lifts along~$\partial \circ \PD \circ \ev^{-1} \colon H_2(X_S;\Z[\Z_d]) \twoheadrightarrow H_1(\partial X_S;\Z[\Z_d])$ yielding 
$$
(0 \ \widehat{f}) \colon \overbrace{H_1(\partial X_S;\Z[\Z_d])}^{\cong H_1(\Sigma_d(K)) \oplus \Z_{e/d} \oplus H_1(\Sigma)} \to H_2(X_S;\Z[\Z_d])^*.
$$
Consider the evaluation map 
$$\ev_H:=\begin{pmatrix}
\ev_T& 0\\
0&\ev_\Sigma
\end{pmatrix}
 \colon \overbrace{H_1(\partial X_S;\Z[\Z_d])}^{\cong (H_1(\Sigma_d(K)) \oplus \Z_{e/d}) \oplus H_1(\Sigma)} \to \overbrace{H_1(\partial X_S;\Z[\Z_d])^{**}}^{\cong H_1(\Sigma_d(K))^{**}\oplus H_1(\Sigma)^{**}} .$$
 %%Direct sum because it is evaluation
Since $H_1(\Sigma)$ is a free abelian group endowed with the trivial $\Z_d$-action, the evaluation map~$\ev_\Sigma$ is an isomorphism. 
Precomposing $(0 \ \widehat{f})$ with $(0 \ \ev_\Sigma^{-1})$ and $(\PD \circ \ev^{-1})^*$, we obtain the required~$\phi'$ as well as the following commutative diagram:
$$
\xymatrix{
& H_2(\partial X_S;\Z[\Z_d])^{*} \ar[d]_{( \PD \circ \ev^{-1})^*,\cong}
\ar@/^6pc/[dd]^{:=\phi'}
 \\
& H_1(\partial X_S;\Z[\Z_d])^{**} \ar[d]^{(0 \ \widehat{f} \circ \ev_\Sigma^{-1})} \\
& H_2(X_S;\Z[\Z_d]) \ar[d]^{\partial \circ \PD \circ \ev^{-1}} \\
%%%%
H_1(\partial X_S;\Z[\Z_d]) \ar[r]^{f_\partial}
\ar[ru]^{(0 \ \wh f)}
\ar@/^1pc/[ruu]^{\ev_H}
\ar@/^4pc/[ruuu]^{\ev \circ \PD^{-1}}
& H_1(\partial X_S;\Z[\Z_d]).
}
$$
We outline why the leftmost triangle commutes, i.e.\ why $\PD^* \circ \ev \circ \PD^{-1}=\ev^* \circ \ev_H$.
This reduces to showing that $\langle \PD(\alpha),\PD^{-1}(x)\rangle=\langle \alpha,x \rangle \in \Z[\Z_d]$ for every~$\alpha \in H^1(\partial X_S;\Z[\Z_d])$ and every $x \in H_1(\partial X_S;\Z[\Z_d])$.
Thanks to~\cite[Lemmas A.5 and A.6]{ConwayKasprowski4Manifolds}, this equality can then be verified using properties of cap and cup products. 
%%Don't delete: for pi_1=Z, can replace evaluations by caps and then use the usual tricks which include using the commutativity of cup. Over ZZ_d,  graded symmetry is annoying as is going from cap to evaluation which is why we reference those lemmas from our appendix.
%%Perhaps up to adding an \overline{•} on one of the maps and/or adding a sign.

Taking the norm construction of $\phi'$ we obtain a $\Z[\Z_d]$-linear map 
		$$\phi:=N \phi' \colon H_2(\partial X_S;\Z[\Z_d])^*\to H_2(X_S;\Z[\Z_d])^*.$$
		
		Since~$H_2(\partial X_S;\Z[\Z_d])$ is a free abelian group endowed with the trivial~$\Z_d$-action, it follows that the evaluation map~$H_2(\partial X_S;\Z[\Z_d])\to H_2(\partial X_S;\Z[\Z_d])^{**}$ is an isomorphism. Applying~\cite[Lemma~9.26]{ConwayKasprowski4Manifolds} ensures the existence of a map~$\psi\colon H_2(X_S;\Z[\Z_d])\to H_2(\partial X_S;\Z[\Z_d])$ with $\psi^*=\phi$.
Using $i\colon H_2(\partial X_S;\Z[\Z_d])\to H_2(X_S;\Z[\Z_d])$ to denote the inclusion induced map, set
		$$G:=\id-i\circ\psi.$$
We now verify that $G$ is an isometry.
The naturality of Poincar\'e duality and of the evaluation map yield~$i^*\circ \ev\circ \PD^{-1}=\ev\circ \PD^{-1}\circ \partial=0$.
It follows that
		\begin{equation}
			\label{eq:G*1}
			\PD\circ \ev^{-1}\circ G^*\circ \ev\circ \PD^{-1}=\id-\PD\circ\ev^{-1}\circ\phi\circ\ev\circ\PD^{-1}\circ\partial.
		\end{equation}
Applying the connecting homomorphism $\partial$ as well as the definition of $\phi'$ and $\phi$, we obtain
	%%Don't delete
	%%One gets the thing with \phi instead of \phi' but then uses \phi=N\phi'.
	%%Get \partial-\partial \circ \PD\circ\ev^{-1}\circ\phi\circ\ev\circ\PD^{-1}\circ\partial. 
	%% \partial-\partial \circ \PD\circ\ev^{-1}\circ N\phi'\circ\ev\circ\PD^{-1}\circ\partial.
	%%Linearity of 	%%Then use the linearity of the norm map N.
	%% \partial-N(\partial \circ \PD\circ\ev^{-1}\circ\phi'\circ\ev\circ\PD^{-1}\circ\partial).
	%% =\partial-Nf_\partial \circ\partial.
		\begin{equation}
			\label{eq:G*}
			\partial\circ \PD\circ \ev^{-1}\circ G^*\circ \ev\circ \PD^{-1}=\partial-Nf_\partial\circ \partial.
		\end{equation}
		Using the definition of~$G:=\id-i\circ\psi$ and \eqref{eq:G*1} one verifies that $G$ is an isometry.
		%%Don't delete.\\
		%%One has to show that G^* \circ ev \circ \PD^{-1} \circ j \circ G= ev \circ \PD^{-1} \circ j.
		%%Using the definition of $G=\id-i \circ \psi$, the LHS reduces to G^* \circ \ev \circ \PD^{-1} \circ j.
		%%Thus we have to show that G^* \circ \ev \circ \PD^{-1} \circ j=\ev \circ \PD^{-1} \circ j
		%%Apply \PD \circ \ev^{-1} to both sides to get
		%%\PD \circ \ev^{-1} \circ G^* \circ \ev \circ \PD^{-1} \circ j=j.
		%%By 0.2 the LHS=j- (junk \circ j \circ \partial)=j, as required.
		 
		 It remains to prove that~$\partial G=\id+Nf_\partial$.
		 The argument is inspired by~\cite[Lemma~1.12]{BoyerUniqueness}. 
		 Since~$\id-Nf_\partial$ is the inverse of~$\id+Nf_\partial$
		 %%Don't delete.
		 %Upper triangular so Nf_\partial^2=0
		 and~$\partial\circ \PD\circ \ev^{-1}$ is surjective, this is equivalent to showing that
		 \begin{equation}
		 	\label{eq:WTSSpinFix}
\overbrace{(\partial G)^{-1}\circ \partial\circ \PD\circ \ev^{-1}}^{=\partial\circ \PD\circ \ev^{-1}\circ G^*}
=(\id-Nf_\partial)\circ \partial\circ \PD\circ \ev^{-1}.
		 \end{equation}
		 This follows from \eqref{eq:G*} when precomposing with $\ev\circ \PD$. 
		 This concludes the proof of (1).

		We now prove (2).
		The above argument showing that~$\id-i\circ \psi$ is an isometry with boundary~$\id_\partial+Nf_\partial$ also shows that~$G:=\id-i\circ d \psi$ is an isometry with~$\partial G=\id_\partial+dNf_\partial$.
		We will prove that this~$G$ also satisfies the condition on the~$k$-invariant.
		First,  note the following relation (where the second equality is a consequence of~\cite[Proposition 7.13]{ConwayKasprowski4Manifolds} which applies because the group~$H^3(\Z_d;H_2(X_{S_1};\Z[\Z_d])) \cong H^3(\Z_d;I)\cong H^3(\Z_d;\Z)=\Z_d$ is~$d$-torsion):
		%%Don't delete.
		%%H_2(X_S;\Z[\Z_d]) \cong_s I by the begining of the proof of Theorem~\ref{thm:CK4Manifold}.
		\begin{align*}
			dk_{X_S,\partial X_S}-d(i\circ \psi)_*(k_{X_S,\partial X_S})
			&=(\id-i\circ \psi)_*(dk_{X_S,\partial X_S})=(\id+Nf_\partial)^*(dk_{X_S,\partial X_S}) \\
			&=dk_{X_S,\partial X_S}+d(Nf_\partial)^*(k_{X_S,\partial X_S}).
		\end{align*}
		Cancelling the~$dk_{X_S,\partial X_S}$ terms, we deduce~$-d(i\circ \psi)_*(k_{X_S,\partial X_S})=d(Nf_\partial)^*(k_{X_S,\partial X_S}),$ and so
		\begin{align*}
			(\id-i\circ d\psi)_*(k_{X_S,\partial X_S})
			&=k_{X_S,\partial X_S}-d(i\circ \psi)_*(k_{X_S,\partial X_S})=k_{X_S,\partial X_S}+d(Nf_\partial)^*(k_{X_S,\partial X_S}) \\
			&=(\id+dNf_\partial)^*(k_{X_S,\partial X_S}),
		\end{align*}
		as needed.
		This concludes the proof of the proposition.
\end{proof}

\subsection{The characteristic case}
\label{sub:SpinUnionCharacteristic}

Next we prove Theorem~\ref{thm:SpinUnion} in the case when the surfaces are characteristic.

\begin{proposition}
\label{prop:Audrick}
Let~$S_0,S_1 \subset X$ be simple characteristic surfaces with the same boundary and the same relative Euler number $e$,
and let $\theta \colon S_0 \to S_1$ be a homeomorphism.
The following assertions are equivalent:
%%%Don't delete.
%%(1) could the stronger statement
%%For any H with H|_=\theta that preserves e-nice, the gluing is spin and such an $H$ exists
\begin{enumerate}
\item 
There is a bundle isomorphism~$H \colon (\overline{\nu}(S_0),S_0) \to (\overline{\nu}(S_1),S_1)$ with $H|_{S_0}=\theta$ that preserves a choice of $e$-nice section, i.e.\ $H \circ s_0=s_1 \circ \theta$,  and such that~$X_{S_0} \cup_h X_{S_1}$ is spin.
%%%Don't delete
%%Understood to mean H|_{\nu(K)}=\id, see def.
\item The induced isomorphism $\theta$ preserves the Freedman--Kirby or Guillou--Marin forms.
\end{enumerate}
\end{proposition}
\begin{proof}
The argument is a generalisation of the arguments from~\cite[Section 6]{ConwayOrsonPowell} for $\Z_2$-surfaces in $D^4$ and of Pyronneau in the case of surfaces with simply-connected complements~\cite[Section~3]{Pyronneau}.
%A detailed proof will appear in upcoming work of Pyronneau but we outline the idea of the proof for the reader's convenience.
For brevity, we write $q_S$ to refer to either the Freedman--Kirby form (in the orientable case) or the Guillou--Marin form (in the nonorientable case).
Consider the circle bundle~$Y:=\partial X_S \setminus E_K$.
As in~\cite[Section 6.3]{ConwayOrson}, since~$S$ is simple,  the choice of an~$e$-nice section~$s \colon S \to Y$ gives rise to the \emph{exterior Freedman--Kirby} form~$\widehat{q}_S \colon H_1(s(S)) \to \Z_2$ (resp.\ Guillou--Marin form~$\widehat{q}_S \colon H_1(s(S)) \to \Z_4$).
The definition of this form is analogous to that of~$q_S$, but given~$[\gamma] \in H_1(s(S);\Z_2)$, thanks to the hypothesis on the section~$s$ and that fact that~$S$ is simple,  it is possible to use a \emph{disc}~$D \subset X_S$ with boundary~$\gamma$ to calculate~$\widehat{q}_S(\gamma)$.
Without the simplicity assumption,~$\gamma$ would be nullhomologous in~$X_S$ but might not be nullhomotopic.
The terminology refers to the fact that a disc~$D$ can be found in the exterior of~$S$.
The proof of~\cite[Proposition 6.15]{ConwayOrsonPowell} shows that the section $s$ induces an isometry 
\begin{equation}
\label{eq:ExteriorQuadratic}
s_* \colon (H_1(S;\Z_2),q_S) \to (H_1(s(S);\Z_2),\widehat{q}_S).
\end{equation}
%Set~$Y:=\partial X_S \setminus E_K$ and, 
Given a spin structure~$\mathfrak{s}$ on~$Y$,  write~$q_{KT}(\mathfrak{s}) \colon H_1(s(S);\Z_2) \to \Z_4$ for the Kirby--Taylor quadratic refinement associated to $S$ and the nice section $s$~\cite{KirbyTaylor}; see also~\cite[Section 6.1]{ConwayOrsonPowell}.
In the orientable case,~$q_{KT}(\mathfrak{s})$ takes values in~$2 \cdot \Z_4 \cong \Z_2$.
The proof of~\cite[Lemma 6.16]{ConwayOrsonPowell} shows that if~$\mathfrak{s}$ extends over~$X_S$, then~$q_{KT}(\mathfrak{s})=\widehat{q}_S$.
%This is the step where~$S$ being simple is used: given~$[\gamma] \in H_1(S;\Z_2)$, the argument in~\cite{ConwayOrsonPowell} involves picking a \emph{disc}~$D \subset X_S$ with boundary~$\gamma$ to calculate the (exterior) Freedman--Kirby (or Guillou--Marin) form~$q_S(\gamma)$; in general (thanks to the choice of a nice identification as in Section~\ref{sub:AlexanderModule}),~$\gamma$ is nullhomologous in~$X_S$ but might not be nullhomotopic.
Since $X_S$ is spin, 
%%Don't delete
%so that <0
this implies that
$$
1
\leq
|\{ \mathfrak{s} \in \Spin(Y) \mid \mathfrak{s} \text{ extends over $X_S$} \}|
\leq 
|\{ \mathfrak{s} \in \Spin(Y) \mid \ q_{KT}(\mathfrak{s})=\widehat{q}_S \}|
=2.
$$
For the equality,  the proof of~\cite[Proposition 6.10]{ConwayOrsonPowell} shows that there are two spin structures~$\mathfrak{s}$ on~$Y$ such that~$q_{KT}(\mathfrak{s})=\widehat{q}_S$. 
These two spin structures are differentiated by their value on the~$S^1$-fibre $\mu$ that we will frequently also refer to as ``the meridian".
Here recall that~$H_1(Y) \cong \Z^{2g}\oplus \Z\langle \mu \rangle$ in the orientable case and~$H_1(Y) \cong \Z^h \oplus \Z_2\langle \mu \rangle$ in the nonorientable case; 
see e.g. Section~\ref{sec:Setup} and~\cite[Proposition 4.2]{ConwayOrsonPowell} for the latter isomorphism.

 When $S$ is characteristic,  we claim that~$H_1(X_S;\Z_2)=0$ if and only if $X$ is nonspin.
To see this,  recall from Lemmas~\ref{lem:MV} and~\ref{lem:H1NonOrientable} that $H_1(X_S;\Z_2)=0$ if and only if~$d$ is odd (by which we mean~$[S]\neq 0$ in the nonorientable case).
In turn,  $d$ is odd if and only if $[S] \neq 0 \in H_2(X,\partial X;\Z_2)$ which is equivalent to the existence of an $x \in H_2(X;\Z_2)$ such that $Q_X([S],x)=1$ mod $2$.
Since $S$ is characteristic, this latter condition is equivalent to $X$ being nonspin.
This establishes the claim.

Since there are $|H^1(X_S;\Z_2)|=|H_1(X_S;\Z_2)| \in \{1,2\}$ spin structures on $X_S$, the claim (together with  the fact that if a spin structure on $Y$ extends over $\overline{\nu}(S)$ or over an~$X_S$ with $H_1(X_S;\Z_2)=0$, then it is trivial on the meridian)
%%Don't delete 
%In general, if [S^1] is trivial in H_1(Z;Z/2) then the spin structure restricted to it is trivial. Probably we can just use that without a proof, but a proof would be to pick an immersed nullbordism and then pulling back the spin structure gives a spin nullbordism from the S^1.
implies the following assertions.
\begin{itemize}
\item If~$X$ is nonspin,  since $H_1(X_S;\Z_2)=0$, there is a unique spin structure~$\mathfrak{s}$ on~$Y$ that extends over~$X_S$, namely the unique spin structure~$\mathfrak{s}$ with~$q_{KT}(\mathfrak{s})=\widehat{q}_S$ and such that~$\mathfrak{s}$ does not extend over~$\overline{\nu}(S)$ (since~$X$ is spin,~$\mathfrak{s}$ cannot extend over both~$X_S$ and~$\overline{\nu}(S)$).
Since~$X$ is nonspin if and only if $H_1(X_S;\Z_2)=0$, it follows that~$\mathfrak{s}$ is the unique spin structure on~$Y$ with~$q(\mathfrak{s})=\widehat{q}_S$ and that vanishes on the meridian.
%%Don't delete
%This does not contradict~$s$ not extending over $\nu$.
%For example for CP^1 in CP^2, \nu(S) is not spin.
\item If $X$ is spin,  since $H_1(X_S;\Z_2)=0$,  there are two spin structures~$\mathfrak{s}$ on~$Y$ that extend over~$X_S$, namely the spin structures that satisfy~$q_{KT}(\mathfrak{s})=\widehat{q}_S$.
%%Don't delete.
%%Distinct because different values on meridian.
Only one of these spin structures extends over $\overline{\nu}(S)$, namely the one that is trivial on the meridian.
This is the spin structure obtained by restricting the spin structure on $X$ to $X_S$.
\end{itemize}
In particular,  if a spin structure on $Y$ satisfies~$q_{KT}(\mathfrak{s})=\widehat{q}_S$ and is trivial on the meridian, then it extends over $X_S$.
We will later use this fact to prove the $(2) \Rightarrow (1)$ direction of the proposition.

First, however we prove the $(1)\Rightarrow (2)$ direction: we assume that the bundle isomorphism $H$ exists and prove that $\theta_*=(H|_{S_0})_*$ preserves the Freedman--Kirby forms or Guillou--Marin forms.
Let~$s_0$ and~$s_1$ be~$e$-nice sections that are preserved by~$H$, i.e.~$H \circ s_0=s_1 \circ \theta,$ and let~$\widehat{q}_{S_0}$ and~$\widehat{q}_{S_1}$ be the corresponding exterior quadratic forms.
Since~$S_0$ and~$S_1$ are characteristic, their exteriors are spin;  see e.g.~\cite[Proposition~7.18]{FriedlNagelOrsonPowell}.
By hypothesis,  it is possible to pick spin structures on~$X_{S_0}$ and~$X_{S_1}$,  such that when restricted to spin structures~$\mathfrak{s}_0$ and~$\mathfrak{s}_1$ on~$Y_0$ and~$Y_1$, they satisfy~$H^*(\mathfrak{s}_1)=\mathfrak{s}_0$.
Let~$q_{KT_0}(\mathfrak{s}_0)$ and~$q_{KT_1}(\mathfrak{s}_1)$ be the Kirby--Taylor quadratic forms associated to~$s_0,\mathfrak{s}_0$ and~$s_1,\mathfrak{s}_1$ respectively.
Since the spin structures $\mathfrak{s}_0$ and $\mathfrak{s}_1$ extend over the respective surface exteriors,  it follows that
%Since these spin structures extend over the surface exteriors,  it follows that
$$
\widehat{q}_{S_0}
=q_{KT_0}(\mathfrak{s}_0)
=q_{KT_0}(H^*(\mathfrak{s}_1))
=q_{KT_1}(\mathfrak{s}_1) \circ H_*
=\widehat{q}_{S_1} \circ H_*.
$$
The second equality is valid because $H$ preserves 
%a choice of 
the~$e$-nice sections; we refer to~\cite[Equations~(17) and~(18)]{ConwayOrsonPowell} for more details.
Using the isometry $(s_i)_* \colon q_{S_i} \cong \widehat{q}_{S_i}$, this equation can now be rewritten as $q_{S_0} \circ (s_0)_*^{-1}=q_{S_1} \circ (s_1)_*^{-1} \circ H_*$, i.e.\  $q_{S_0} \circ (s_0)_*^{-1}=q_{S_1} \circ \theta_* \circ (s_0)_*^{-1}$.
Thus $\theta_*=(H|_{S_0})_*$ preserves the Freedman--Kirby or Guillou--Marin forms as required.

We now prove the converse, namely we assume that~$\theta_*$ preserve the Freedman--Kirby or Guillou--Marin forms and 
%%Don't delete
%construct the required bundle isomorphism $H$.
%By assumption the relative Euler numbers of~$S_0$ and~$S_1$ agree.
%We can therefore apply Lemma~\ref{lem:ExtendHomeoOverDiscBundle} to extend~$(\theta)_*$ to a bundle isomorphism~$H \colon (\overline{\nu}(S_0),S_0) \to (\overline{\nu}(S_1),S_1)$ that preserves a choice of nice sections.
%It remains to
 prove that
 there is a bundle isomorphism~$H$ with $H|_{S_0}=\theta$ that preserves a choice of $e$-nice section, and such that~$X_{S_0} \cup_h X_{S_1}$ is spin.
% ~$X_{S_0} \cup_h X_{S_1}$ is spin.
The proof proceeds similarly to~\cite[Proof of Proposition~6.2]{ConwayOrsonPowell}.
 Given $e$-nice sections $s_0,s_1$ (which exist by Lemma~\ref{lem:ExistseNice-or}), apply Lemma~\ref{lem:ExtendHomeoOverDiscBundle} to obtain a bundle isomorphism~$H$ with $H|_{S_0}=\theta$ and~$H \circ s_0=s_1 \circ \theta$.
Fix a meridian $\mu_0$ for $S_0$ and set $\mu_1:=H(\mu_0)$.
Pick a spin structure on~$X_{S_1}$ that is trivial on $\mu_1$, restrict it to a spin structure~$\mathfrak{s}_1$ on~$Y_1$ and use~$h \colon Y_0 \to Y_1$ to pull it back to a spin structure~$\mathfrak{s}_0$ on~$Y_0$.
This spin structure is trivial on $\mu_0$.
%%Because s_1 trivial on $\mu_1$ and because $H$ bundle iso.
%{AC: This is still $Y_i=\partial X_{S_i} \setminus E_K$.}
As in~\cite[Equation~(18)]{ConwayOrsonPowell}, it follows from the definition of $\mathfrak{s}_0$,  the equality~$q_{KT}(\mathfrak{s}_1)=\widehat{q}_{S_1}$ (which, recall, holds because~$\mathfrak{s}_1$ extends over~$X_{S_1}$),  and~$H \colon (\overline{\nu}(S_0),S_0) \to (\overline{\nu}(S_1),S_1)$ inducing an isometry of the Freedman--Kirby forms that
$$q_{KT}(\mathfrak{s}_0) \cong q_{KT}(\mathfrak{s}_1) \circ H_*= {q}_{S_1} \circ H_*= \widehat{q}_{S_0}.$$
In the first equality, we also used that $H$ preserves a choice of nice sections.
%%Don't delete.
%%~\cite[Equations (17) and (18)]{ConwayOrsonPowell}

Since the spin structure~$\mathfrak{s}_0=h^*(\mathfrak{s}_1)$ is trivial on $\mu_0$ and~$q_{KT}(\mathfrak{s}_0)=\widehat{q}_{S_0}$, it extends over~$X_{S_0}$.
It follows that~$X_{S_0} \cup_h X_{S_1}$ is spin, as required.
This concludes the proof of the proposition.
%%Don't delete yet.
%When~$X$ is spin, this already ensures that~$\mathfrak{s}_0$ extends over~$X_{S_0}$: when $e$ is odd,  this is clear, whereas for $e$ even,  the argument is the same as in~\cite[Corollary~6.17 and proof of Proposition~6.2]{ConwayOrsonPowell}.
%When~$X$ is nonspin,  since~$\mathfrak{s}_1$ does not extend over~$\overline{\nu}(S_1)$ and since~$H$ is fibre-preserving, it follows that~$\mathfrak{s}_0$ does not extend over~$\overline{\nu}(S_0)$; since we also know that~$q_{KT}(\mathfrak{s}_0)=\widehat{q}_{S_0{$, this implies that~$\mathfrak{s}_0$ extends over~$X_{S_0}$, as required.
%Thus, in both cases $h$ preserves the spin structures.
\end{proof}

%%%Don't delete yet; folded it into cor:SpinAutomatic
%The following remark is not in the sequel but might be of independent interest.
%\begin{remark}
%We argue that every isometry of the Freedman--Kirby or Guillou--Marin forms is realised by a homeomorphism.
%Since the $S_i$ have a single boundary component, their intersection forms are nonsingular.
%The isometry type of a~$\Z_2$-valued quadratic refinement of a~$\Z_2$-valued nonsingular symmetric form on a~$\Z_2$ vector-space is determined by its Arf invariant~\cite[Theorem~III.1.12]{BrowderSimply} and by its Brown invariant in the case of a~$\Z_4$-valued quadratic refinement~\cite[final remark of Section II]{GuillouMarin}.
%Since~$S_0$ and~$S_1$ have the same relative Euler number and boundary,  it follows that~$\operatorname{Arf}(q_{S_0})=\operatorname{Arf}(q_{S_1})$ or~$\operatorname{Brown}(q_{S_0})=\operatorname{Brown}(q_{S_1})$ depending on whether or not the~$S_i$ are orientable (see e.g.~\cite{FreedmanKirby,Klug,GilmerLivingston}) and thus,  since the intersection forms of the~$S_i$ are nonsingular, we deduce that~$(H_1(S_0;\Z_2),Q_{S_0},q_{S_0}) \cong (H_1(S_1;\Z_2),Q_{S_1},q_{S_1})$.
%The argument is now concluded by realising this isometry by a rel.\ boundary homeomorphism~$S_0 \to S_1$; see e.g.~\cite[Section 2.1 and the discussion following Theorem 6.4]{FarbMargalit} in the orientable case and~\cite{GadgilPancholi} in the nonorientable case.
%\end{remark}

\subsection{The almost spin case}
\label{sub:AlmostSpin}

The final proposition is specific to the case when $d \neq 0$ is even.

%%Don't delete
%{XX: We weren't sure how this meshed with Proposition~\ref{prop:BoyerAdaptSpin}. 
%This one seems to say always spin while the other says you might have to change something.
%XX says ``The answer might be that we are taking some even cover which kills the change (because spin is a mod 2 thing)."
%}

\begin{proposition}
\label{prop:SpindEvenAutomatic}
Let~$d>0$ be even and let~$K \subset S^3$ be a knot such that~$\Sigma_d(K)$ is a rational homology sphere when~$d>2$.
If~$S_0,S_1 \subset X$ are~$\Z_d$-surfaces with boundary~$K$ such that the universal covers~$\widetilde{X}_{S_0}$ and~$\widetilde{X}_{S_1}$ are spin,
%%%Don't delete.
%%The existence of H implie same relative Euler numbers. 
%%I don't think homologous is needed.
then for any bundle isomorphism~$H \colon (\overline{\nu}(S_0),S_0) \to (\overline{\nu}(S_1),S_1)$
with~$H|_{\overline{\nu}(K)}=\id$,
 the universal cover of~$X_{S_0} \cup_h X_{S_1}$ is spin.
\end{proposition}
\begin{proof}
Proposition~\ref{prop:AlexanderModule} and~\ref{prop:AlexanderModuleNonOri} imply that~$H_1(\partial X_{S_i};\Z[\Z_d]) \cong \Z^{2g} \oplus \Z_{e/d} \oplus H_1(\Sigma_d(K))$ in the orientable case and~$H_1(\partial X_{S_i};\Z[\Z_d]) \cong \Z^{h-1} \oplus G \oplus H_1(\Sigma_d(K))$ with~$|G|=4$ in the nonorientable case.
The assumption on~$H_1(\Sigma_d(K))$ ensures that~$H_1(\partial X_{S_i};\Z[\Z_d])$ decomposes as~$T \oplus \Z^k$ where~$T^*=0$ (recall Convention~\ref{conv:LTZ}) as is required to apply Theorem~\ref{thm:CK4Manifold}.
%%Don't delete
%%When e=0 the meridian still has the trivial action.
The proof of the first item of Theorem~\ref{thm:CK4Manifold} then implies that the universal cover of the union is spin.
Here are some more details.
During that proof,  we argued that the secondary obstruction~$b=b(c_0,c_1)$ is even hermitian.
Apply~\cite[Theorem 9.4]{ConwayKasprowski4Manifolds} to realise $b$ as $b(\varphi)$ for some $\varphi \in \mathcal{G}$.
It follows from the definition of the action of~$\mathcal{G}$ that~$b$ vanishes in $\operatorname{Herm}(H^2(\partial X_{S_1};\Z[\Z_d]))/\mathcal{G}.$
The fact that the universal cover of the union is spin now follows from~\cite[Theorem 2.18 and Lemma 2.21]{ConwayKasprowski4Manifolds}.
%b=0+b(\varphi)$
The fact that these results from~\cite{ConwayKasprowski4Manifolds} apply was discussed during the proof of Theorem~\ref{thm:CK4Manifold}.
\end{proof}

\begin{remark}
\label{rem:Normal1TypeUnion}
We collect a consequence of the results of this section that will be useful when considering normal $1$-types in Section~\ref{sec:DiscsSpheres}.
Here, we assume some familiarity with terminology from Kreck's modified surgery theory~\cite{KreckSurgeryAndDuality}.
Recall that for a surface~$S \subset X$, the exterior~$X_S$ is spin if and only if~$X$ is spin or~$S$ is characteristic;  see e.g.~\cite[Proposition~7.18]{FriedlNagelOrsonPowell}.
For~$\Z_d$-surfaces, Propositions~\ref{prop:BoyerAdaptSpin} and~\ref{prop:Audrick} therefore show that if~$X_{S_0}$ and~$X_{S_1}$ are spin,
% then  unless~$X$ is spin,  the~$S_i$ are ordinary and~$d$ is even, 
then there is a bundle isomorphism~$H \colon (\overline{\nu}(S_0),S_0) \to (\overline{\nu}(S_1),S_1)$
with $H'|_{\overline{\nu}(K)}=\id$
 such that~$M:=X_{S_0} \cup_h X_{S_1}$ is spin and such that the maps~$\pi_1(X_{S_i}) \to \Z_d$ extend to a map~$\pi_1(M) \to \Z_d$.
%%Don't delete
%%One might worry about the case where~$X$ is spin,  the~$S_i$ are ordinary and~$d$ is even.
%%But this can't happen: if d is even and X is spin, then S.x=0 (because d even) =x.x (because spin) so S is characteristic.
Similarly,  Proposition~\ref{prop:SpindEvenAutomatic} shows that when the~$X_{S_i}$ are not spin but the~$\widetilde{X}_{S_i}$ are spin (so that~$d$ is necessarily positive and even),  the same is true for~$M$ and Lemma~\ref{lem:ExtendHomeoOverDiscBundle} shows that, again, the coefficient systems can be assumed to extend.
Put differently, it is possible to choose~$H$ so that~$M$ maps into the normal~$1$-type of the~$X_{S_i}$ in a way that extends normal~$1$-smoothings on these latter $4$-manifolds.
%Because same pi_1 of union because pi_1-surj.
\end{remark}

\subsection{Spin unions}
\label{sub:ProofSpinUnion}

Let~$X$ be a simply-connected~$4$-manifold with~$\partial X=S^3$, and let~$K \subset \partial X$ be a knot.
Recall that Theorem~\ref{thm:SpinUnion} states that if~$S_0,S_1 \subset X$  are~$\Z_d$-surfaces with boundary~$K$ whose relative Euler numbers agree and whose exteriors have spin universal covers, then in nearly all cases a homeomorphism~$\theta \colon S_0 \to S_1$ (that preserves the appropriate quadratic forms in the characteristic case) extends to a homeomorphism~$H\colon (\overline{\nu}(S_0),S_0) \to (\overline{\nu}(S_1),S_1)$ such that the universal cover of~$X_{S_0} \cup_h X_{S_1}$ is spin.

\begin{proof}[Proof of Theorem~\ref{thm:SpinUnion}]
Since the relative Euler numbers of the~$S_i$ agree,  Lemma~\ref{lem:ExtendHomeoOverDiscBundle} ensures that the homeomorphism $\theta$ extends to a bundle isomorphism~$ (\overline{\nu}(S_0),S_0) \to (\overline{\nu}(S_1),S_1)$
that can be assumed to intertwine the coefficient systems to $\Z_d$.
If the surfaces~$S_0$ and~$S_1$ are characteristic, then the result follows from Proposition~\ref{prop:Audrick}.
If $X$ is spin and $d$ is odd, then the result follows from Proposition~\ref{prop:BoyerAdaptSpin}.
%d even doesn't preserve coefficient systems.
Recall that for a surface $S \subset X$, the exterior~$X_S$ is spin if and only if $X$ is spin or $S$ is characteristic;  see e.g.~\cite[Proposition~7.18]{FriedlNagelOrsonPowell}.
The remaining cases are therefore the case where the~$X_{S_i}$ are spin and~$d \geq 0$ is even as well as the case when the~$\widetilde{X}_{S_i}$ are spin but the~$X_{S_i}$ are not spin.
When $d=0$,  the Alexander modules are torsion and so there is nothing to check; recall the first item of Theorem~\ref{thm:CK4Manifold}.
%%Don't delete.
%; see also~\cite[Theoren 3.12]{ConwayPowell}
%%Seemed like too much to also cite this?!
In the remaining situations, $d>0$ is even and the theorem follows from Proposition~\ref{prop:SpindEvenAutomatic}. 
\end{proof}

\section{Proof of the main theorem}
\label{sec:ProofMain}

We prove a variation on 
Theorem~\ref{thm:CompatiblePairIntro} from the introduction.
The proof of the statement in the introduction is identical, the only difference is that, there, the homeomorphism~$\theta \colon S_0 \to S_1$ was not fixed: instead, a bundle isomorphism $H$ with the necessary properties was given.

%\begin{theorem}
\begin{customthm}
{\ref{thm:CompatiblePairIntro}}
\label{thm:CompatiblePair}
Let~$X$ be a simply-connected~$4$-manifold with boundary~$\partial X=S^3$, and let~$K \subset~S^3$ be a knot.
When~$d>0$ is not a prime, assume that $\Sigma_d(K)$ is a rational homology sphere.
Let~$S_0,S_1 \subset X$ be~$\Z_d$-surfaces with boundary~$K$, with the same Euler number, and that are either both characteristic or both ordinary.
Fix a rel.\ boundary homeomorphism 
$$ \theta \colon S_0 \to S_1.$$
When the $S_i$ are characteristic,  assume furthermore that~$\theta$ preserves the Freedman--Kirby or Guillou--Marin quadratic form depending on whether the~$S_i$ are orientable or not.
The following assertions are equivalent:
\begin{itemize}
\item There is a rel.\ boundary homeomorphism~$\Theta \colon (X,S_0) \to (X,S_1)$ with $\Theta|_{S_0}=\theta$.
\item There is an isometry $F\colon \lambda_{X_{S_0}} \cong \lambda_{X_{S_1}}$ of the equivariant intersection forms and a bundle isomorphism $H \colon \overline{\nu}(S_0) \to \overline{\nu}(S_1)$ that is the identity on a tubular neighborhood of $K$ with~$H|_{S_0}=\theta$ such that~$(F,\id_{E_K} \cup H|)$ preserves the relative~$k$-invariants,
$$\partial F=(\overbrace{\id_{E_K} \cup H|}^{:=h})_* \colon H_1(\partial X_{S_0};\Z[\Z_d]) \to H_1(\partial X_{S_1};\Z[\Z_d]),$$
and,  when $X$ is spin,  $d=1$, and $S_0,S_1$ are nonorientable,~$X_{S_0} \cup_h X_{S_1}$ is spin.
\end{itemize}
The~$k$-invariant condition can be omitted in the following circumstances: either~$\pi$ is trivial,  or~$\pi$ is infinite cyclic,  or the $S_i$ are discs with nonzero Euler number, or the $S_i$ are Moebius bands. 
The condition involving the quadratic forms is automatic when the $S_i$ are discs or Moebius bands. 
\end{customthm}
\begin{proof}
%%Don't delete
%%Boyer has to work harder in the case odd intersection form and e=0 because he wants isotopy. He changes F but he still needs it to fit in the diagram.
%%We don't have this issue.
The existence of a homeomorphism~$(X,S_0) \to (X,S_1)$ extending~$\theta$ is readily seen to imply the existence of~$F$ and~$H$.
We therefore focus on the converse.
Fix an isometry~$F$ and a bundle isomorphism~$H \colon (\overline{\nu}(S_0),S_0) \to (\overline{\nu}(S_1),S_1)$ with $H|_{S_0}=\theta$, such that~$(F,h):=(F,\id_{E_K} \cup H|_\partial)$ is~$k$-invariant preserving, satisfies $\partial F=h_*$,  and is such that, when $X$ is spin, $S_0,S_1$ are nonorientable and $d=1$,  the universal cover of~$X_0 \cup_h X_1$ is spin.
The goal is to prove that there is a $k$-invariant preserving extendable compatible pair~$(F',H')$ with $H'|_{S_0}=\theta$ and such that the universal cover of~$X_0 \cup_{h'} X_1$ is spin, when relevant.
The conclusion will then follow from Theorem~\ref{thm:CKForSurfaceExteriors}.

Proposition~\ref{prop:CompatibleSimplification} shows that the equality~$\partial F=h_*$ suffices to ensure that~$(F,h)$ forms a compatible pair.
The clauses concerning the~$k$-invariant condition are proved in Proposition~\ref{prop:NokInvariant}.
The final sentence of the theorem follows from the fact that in those cases~$H_1(S_i)\cong \Z$.
We therefore focus on producing a new~$k$-invariant preserving extendable compatible pair~$(F',H')$ with~$H'|_{S_0}=\theta$ that satisfies the spin condition, when relevant.
When~$X$ is spin,~$S_0,S_1$ are nonorientable and~$d=1$,  by assumption, we can take~$(F,H):=(F',H')$.
When~$X$ is spin,~$S_0,S_1$ are orientable and~$d$ is odd, the pair~$(F',H')$ is produced by the third item of Proposition~\ref{prop:BoyerAdaptSpin}.
When~$S_0,S_1$ are characteristic, since the homeomorphism~$\theta$ preserves the quadratic forms, Proposition~\ref{prop:Audrick} guarantees that we can take~$(F',H'):=(F,H).$
The remaining cases are therefore the case where the~$X_{S_i}$ are spin and~$d \geq 0$ is even 
%This now covers all the spin exterior cases
as well as the case when the~$\widetilde{X}_{S_i}$ are spin but the~$X_{S_i}$ are not spin; in both cases we can again take~$(F',H'):=(F,H)$.
Indeed, when~$d=0$,  the Alexander modules are torsion and so there is nothing to check; recall the first item of Theorem~\ref{thm:CK4Manifold}; in the remaining situations,~$d>0$ is even and the result follows from Proposition~\ref{prop:SpindEvenAutomatic}. 
As mentioned above, the conclusion of the theorem now follows from Theorem~\ref{thm:CKForSurfaceExteriors}.
\end{proof}

\section{Compatible pairs and pointed hermitian forms}
\label{sec:CompatiblePairs}

The goal of this section is to relate Theorems~\ref{thm:CompatiblePairClosed} and~\ref{thm:CompatiblePair} to the work of Lee and Wilczynski on simple spheres~\cite{LeeWilczy,LeeWilczyOdd}.
For this,  we begin by recalling some terminology.

\begin{notation}
A \emph{pointed hermitian form} is a triple~$(H,\lambda,z)$ where~$(H,\lambda)$ is a hermitian form and~$z \in H$.
Two pointed hermitian forms~$(H_0,\lambda_0,z_0)$ and~$(H_1,\lambda_1,z_1)$ are \emph{isometric} if there is an isometry~$f \colon  (H_0,\lambda_0) \cong (H_1,\lambda_1)$ with~$f(z_0)=z_1$.
Given a $\Z_d$-sphere $S \subset X$ with $d \neq 0$,  we consider the pointed hermitian form $(H_2(\Sigma_d(S)),\lambda_{\Sigma_d(S)},[\widetilde{S}])$, where 
$$\lambda_{\Sigma_d(S)}(x,y)= \sum_{k=0}^{d-1} Q_{\Sigma_d(S)}(x,T^ky) T^{-k}.$$
If $S \subset X$ is a $\Z_d$-surface with boundary a knot $K \subset \partial X = S^3$ satisfying $H_1(\Sigma_d(K))=0$,  then the inclusion induces an isomorphism ~$i_* \colon H_2(\Sigma_d(S)) \to H_2(\Sigma_d(S),\partial \Sigma_d(S))$ and we consider the pointed hermitian form $(H_2(\Sigma_d(S)),\lambda_{\Sigma_d(S)},i_*^{-1}([\widetilde{S}])).$
\end{notation}

For discs and spheres, the next result shows that a pointed isometry between the equivariant intersection forms of the branched covers determines an extendable compatible pair.
%%Don't delete
%{AC: Does this generalise to higher genus? (that pointed isometry gives a $(F,H)$ such that $\partial F=h$.}

\begin{proposition}
\label{prop:LeverageLW}
Let $d>0$, let~$K$ be a knot with~$H_1(\Sigma_d(K))=0$,  and let~$S_0,S_1 \subset X$ be 
%homologous
 $\Z_d$-discs with boundary~$K$ and the same relative Euler number.
\begin{enumerate}
\item \label{item:Pointed1} A pointed isometry~$F^\Sigma \colon \lambda_{\Sigma_d(S_0)} \cong \lambda_{\Sigma_d(S_1)}$ restricts to an isometry~$F \colon \lambda_{X_{S_0}} \cong~\lambda_{X_{S_1}}$.
\item  After possibly precomposing a fixed linear bundle isomorphism~$H \colon \overline{\nu}(S_0) \cong \overline{\nu}(S_1)$ with the 
%orientation-reversing 
%%%Used to be written with \wt{S}_i
bundle automorphism~$\overline{\nu}(S_0) \cong D^2 \times D^2$ given by~$(x,y) \mapsto (x,-y)$,  the induced homeomorphism~$h:=\id_{E_K} \cup H|$ is compatible with the isometry $F$ from~\eqref{item:Pointed1}:
$$\partial F=h_* \colon H_1(\partial X_{S_0};\Z[\Z_d]) \to H_1(\partial X_{S_1};\Z[\Z_d]).$$  
\end{enumerate}
The same assertions hold for spheres.
\end{proposition}
\begin{proof}
We focus on discs as the argument for spheres is similar.
%%Don't delete
%%Detailed argument in CompatiblePairsSpheres.pdf.
Throughout this proof we work with covers rather than with local coefficients; recall Proposition~\ref{prop:EquivCompatiblePair}.
%we showed in Proposition~\ref{prop:EquivCompatiblePair} that this does not affect the result.
Furthermore,  for brevity, we denote algebraic intersections by~$x \cdot y$ instead of e.g.~$Q_{\Sigma_d(S_i)}(x,y)$.
The plan of the proof is to define~$F$, then give a convenient description of~$\partial F$ in terms of~$F^\Sigma$ and, finally, to verify the equality~$\partial F=\widetilde{h}_* \colon H_1(\partial \widetilde{X}_{S_0}) \to H_1(\partial \widetilde{X}_{S_1}).$
Since~$H_1(\partial \widetilde{X}_{S_i}) \cong \Z_{e/d}$ is generated by~$\widetilde{\mu}_i$ (because~$H_1(\Sigma_d(K))=0$, recall Proposition~\ref{prop:AlexanderModule}) and~$\widetilde{h}_*(\widetilde{\mu}_0)=\widetilde{\mu}_1$ (because~$H$ is a fibre-preserving), this will reduce to proving that~$\partial F(\widetilde{\mu}_0)=\widetilde{\mu}_1.$
\begin{claim}
%Don't delete: automatically isometry.
The isometry~$F^\Sigma$ restricts to an isometry~$F \colon H_2(\widetilde{X}_{S_0}) \to H_2(\widetilde{X}_{S_1})$.
\end{claim}
\begin{proof}
For $i=0,1,$ consider the Mayer--Vietoris sequences for~$\Sigma_d(S_i)=\widetilde{X}_{S_i} \cup \overline{\nu}(S_i)$:
 %{One could probably look at~$(\Sigma_d(S_i),\overline{\nu}(S_i))$ instead. Useful when boundary.}
\begin{equation}
\label{eq:RelativeDiagram}
\xymatrix{
0 \ar[r] & H_2(\widetilde{X}_{S_0}) \ar[r]^-{\incl_0} & H_2(\Sigma_d(S_0))\ar[r]^-{\partial}\ar[d]^{F^\Sigma} & {\overbrace{H_1(\widetilde{S}_0 \times S^1)}^{\cong \Z\widetilde{\mu}_0}} \ar[r]&\ldots  \\
%%%%
0 \ar[r] & H_2(\widetilde{X}_{S_1}) \ar[r]^-{\incl_1} & H_2(\Sigma_d(S_1))\ar[r]^-{\partial} & {\underbrace{H_1(\widetilde{S}_1 \times S^1)}_{\cong \Z\widetilde{\mu}_1}} \ar[r]&\ldots  \\
}
\end{equation}
The definition of the connecting homomorphism implies that for~$i=0,1$
$$\im(\incl_i)=\{  x \in H_2(\Sigma_d(S_i)) \mid Q_{\Sigma_d(S_i)}(x,[\widetilde{S}_i])=0\}.$$
We deduce that~$F^\Sigma$ restricts to an isometry~$\im(\incl_0) \to \im(\incl_1)$ of the restricted forms: 
%%Don't delete
%%Recall that ``Furthermore,  for brevity, we denote algebraic intersections by~$x \cdot y$ instead of e.g.~$Q_{\Sigma_d(S_i)}(x,y)$."
indeed since~$F^\Sigma$ is a pointed isometry, we get~$F^\Sigma(x) \cdot [\widetilde{S}_1]=F^\Sigma(x) \cdot F^\Sigma([\widetilde{S}_0])=x \cdot [\widetilde{S}_0]$.
%%Don't delete
%e \neq 0 not needed
Since the inclusion induces an isomorphism~$H_2(\widetilde{X}_{S_i})\cong \im(\incl_i)$, we deduce that~$F^\Sigma$ restricts to an isometry 
$$ F :=F^\Sigma|\colon H_2(\widetilde{X}_{S_0}) \to H_2(\widetilde{X}_{S_1}).$$
This concludes the proof of the claim.
\end{proof}

This concludes the first step of the proof.
Since~$H|$ is obtained by restricting a bundle isomorphism,  we have~$h(\widetilde{\mu}_0)=\widetilde{\mu}_1$.
Thus,  since~$H_1(\partial \widetilde{X}_{S_0}) \cong \Z_{e/d}$ is generated by~$\widetilde{\mu}_0$ (because~$H_1(\Sigma_d(K))=0$, recall Proposition~\ref{prop:AlexanderModule}), it remains to prove that~$\partial F(\widetilde{\mu}_0)=\widetilde{\mu}_1.$

We move on to the second step of the proof which consists of giving a convenient reformulation of~$\partial F \colon H_1(\partial \widetilde{X}_{S_0}) \to H_1(\partial \widetilde{X}_{S_1})$ in terms of~$F^\Sigma.$ 
The existence of the map~$F$ and the fact that~$H_k(\overline{\nu}(S_i))=0$ for $k=1,2$ respectively ensure that~$F^\Sigma$ induces isomorphisms
\begin{align*}
&H_2(\Sigma_d(S_0),\widetilde{X}_{S_0}) \to H_2(\Sigma_d(S_1),\widetilde{X}_{S_1}), \\
&H_2(\Sigma_d(S_0),\overline{\nu}(\widetilde{S}_0)) \to H_2(\Sigma_d(S_1),\overline{\nu}(\widetilde{S}_1)).
\end{align*}
Using excision,  we think of these as isomorphisms
\begin{align*}
 F_{\rel}^\nu &\colon H_2(\overline{\nu}(\widetilde{S}_0),\widetilde{S}_0 \times S^1) \to H_2(\overline{\nu}(\widetilde{S}_1), \widetilde{S}_1 \times S^1), \\
 %%%%
  F_{\rel} &\colon H_2(\widetilde{X}_{S_0},\widetilde{S}_0 \times S^1) \to H_2(\widetilde{X}_{S_1},\widetilde{S}_1 \times S^1).
\end{align*}
Consider now the following diagram in which we use the relative Mayer--Vietoris sequence to obtain the inclusion induced isomorphism~$H_2(\widetilde{X}_{S_i},\widetilde{S}_i \times S^1) \oplus H_2(\overline{\nu}(\widetilde{S}_i),\widetilde{S}_i \times S^1) \cong H_2(\Sigma_d(S_i),\widetilde{S}_i \times S^1)$:
$$
\xymatrix{
 \ar[r] & H_2(\Sigma_d(S_0)) \ar[d]^{F^\Sigma}_\cong \ar[r] & {\overbrace{H_2(\Sigma_d(S_0),\widetilde{S}_0 \times S^1)}^{\cong H_2(\widetilde{X}_{S_0},\widetilde{S}_0 \times S^1) \oplus H_2(\overline{\nu}(\widetilde{S}_0),\widetilde{S}_0 \times S^1)}} \ar[r]^-{\partial_0^{\times S^1}}\ar[d]^{F_{\rel} \oplus F^\nu_{\rel}}_\cong & H_1(\widetilde{S}_0 \times S^1) \ar[r]\ar@{-->}[d]^{F_\partial}_\cong&0  \\
%%%%
\ar[r] & H_2(\Sigma_d(S_1)) \ar[r] &
{\underbrace{H_2(\Sigma_d(S_1),\widetilde{S}_1 \times S^1)}_{\cong H_2(\widetilde{X}_{S_1},\widetilde{S}_1 \times S^1) \oplus H_2(\overline{\nu}(\widetilde{S}_1),\widetilde{S}_1 \times S^1)}} 
\ar[r]^-{\partial_1^{\times S^1}} & H_1(\widetilde{S}_1 \times S^1) \ar[r]&0.  \\
}
$$
The map~$F_\partial$  exists because the left square commutes.
It is calculated by picking a surface in~$\Sigma_d(S_0)$ with boundary the given curve, applying~$F^\Sigma$ and then taking the boundary of the resulting surface.
In particular,~$F_\partial$ agrees with the map induced on this same group by~$F_{\rel}$ and~$F_{\rel}^\nu$.
%%By the corresponding LESs.

We relate~$F_\partial$ with the map $\partial F$ that appears in the definition of a compatible pair.
Here, we write~$F_! \colon H_2(\widetilde{X}_{S_0},\partial \widetilde{X}_{S_0}) \to H_2(\widetilde{X}_{S_1},\partial \widetilde{X}_{S_1})$ for the map induced by $F$ via Poincar\'e duality and the universal coefficient theorem.
\begin{claim}
\label{claim:Compatible3}
The following diagram commutes:
$$
\xymatrix{
H_2(\widetilde{X}_{S_0},\partial \widetilde{X}_{S_0}) \ar[r]^-{\partial_0} \ar@/_5pc/[ddd]^{F_!}
& H_1(\partial \widetilde{X}_{S_0}) \ar@/^3pc/[ddd]^{\partial F} 
\ar[r]&0
\\
%%%
H_2(\widetilde{X}_{S_0},\widetilde{S}_0 \times S^1)\ar[r]^-{\partial_0^{\times S^1}} \ar[d]^{F_{\rel}} \ar[u]_{\incl_0^{\times S^1}}
& H_1(\widetilde{S}_0 \times S^1) \ar[d]^{F_\partial}\ar[u]^{\incl_0^{\times S^1}} 
\ar[r]&0
\\
%%%
H_2(\widetilde{X}_{S_1},\widetilde{S}_1 \times S^1) \ar[r]^-{\partial_1^{\times S^1}} \ar[d]^{\incl_1^{\times S^1}}
& H_1(\widetilde{S}_1 \times S^1)\ar[d]_{\incl_1^{\times S^1}} 
\ar[r]&0
\\
%%%
H_2(\widetilde{X}_{S_1},\partial \widetilde{X}_{S_1}) \ar[r]^-{\partial_1}& H_1(\partial \widetilde{X}_{S_1})
\ar[r]&0. \\
}
$$
In particular, the commutativity of the right hand side of the diagram implies that
$$\partial F(\widetilde{\mu}_0)
=\incl_1^{\times S^1} \circ F_\partial(\widetilde{\mu}_0).$$
%%%Don't delete
%%View \mu as both in H_1(\widetilde{S}_0 \times S^1) and H_1(\partial \widetilde{X}_{S_0})
%%Aka omit incl_0^{\times S^1} from the notation
\end{claim}
\begin{proof}
The central square clearly commutes (by definition of $F_\partial$) as does the outermost part of the diagram (by definition of $\partial F$).
The upper and lower squares commute by naturality.
We prove that the leftmost part of the diagram commutes, i.e.\ that~$\incl_1^{\times S^1} \circ F_{\rel}=F_!  \circ \incl_0^{\times S^1}$.
%%Don't delete
%%Recall that ``Furthermore,  for brevity, we denote algebraic intersections by~$x \cdot y$ instead of e.g.~$Q_{\Sigma_d(S_i)}(x,y)$."
First, note that given~$x \in H_2(\widetilde{X}_{S_0})$ and~$y \in H_2(\widetilde{X}_{S_0},\widetilde{S}_0 \times S^1)$, we have
\begin{align*}
 F(x) \cdot_{\widetilde{X}_{S_1}} F_!(\incl_0^{\times S^1}(y))
&=x \cdot_{\widetilde{X}_{S_0}} \incl_0^{\times S^1}(y)
=x \cdot_{\Sigma_d(S_0)} \incl_0^{\times S^1}(y) \\
&=F^{\Sigma}(x) \cdot_{\Sigma_d(S_1)} F^{\Sigma}(\incl_0^{\times S^1}(y))
=F(x) \cdot_{\widetilde{X}_{S_1}} \incl_1^{\times S^1}(F_{\rel}(y)).
\end{align*}
The first equality follows from the definition of $F_!$, 
%F_!=ev \circ PD^{-1} F^{-*} \circ \ev \circ PD^{-1} and then use the algebraic definition of the intersection form.
the second follow by reasoning geometrically (and omitting inclusion maps from the notation), the third because $F^\Sigma$ is an isometry, and the fourth from the definition of $F_{\rel}$ (and again omitting inclusion maps from the notation).

Since this equation holds for every~$x \in H_2(\widetilde{X}_{S_0})$ and since the adjoint of this relative pairing is an isomorphism, we deduce that~$F_!(\incl_0^{\times S^1}(y))=\incl_1^{\times S^1}(F_{\rel}(y))$ for every~$y \in H_2(\widetilde{X}_{S_0},\widetilde{S}_0 \times S^1)$.
%%%Old.
%~$y \in H_2(\widetilde{X}_{S_0},\partial \widetilde{X}_{S_0})$, we have
%$$ F(x) \cdot F_!(y)=x \cdot y=F(x) \cdot \incl_1^{\times S^1}(F_{\rel}(y)).$$
%Since the adjoint of this relative pairing is an isomorphism, we deduce that~$\incl(F_{\rel}(y))=F_!(y)$ for every~$y \in H_2(\widetilde{X}_{S_0},\partial \widetilde{X}_{S_0})$ and in particular for~$y=\incl_0^{\times S^1}(y')$.

Thus, the leftmost part of the diagram commutes and, finally, the rightmost part of the diagram commutes because all the other parts commute and because $\partial_0^{\times S^1}$ is surjective.
% finally, for the rightmost part of the diagram, we note that for~$x \in H_2(\widetilde{X}_{S_0},\widetilde{S}_0 \times S^1)$, we have
%$$
%\incl \circ F_\partial (x)
%=\incl\circ \partial_1 \circ  F_{\rel}(x)
%=\partial_1 \circ \incl \circ  F_{\rel}(x)
%= \partial_1 \circ F_!(\incl(x))
%=\partial F(\incl(x)).
%$$
This concludes the proof of the claim.
\end{proof}
The claim implies that~$\partial F(\widetilde{\mu}_0)
=\incl_1^{\times S^1} \circ F_\partial(\widetilde{\mu}_0)$.
Our goal is now to calculate this quantity.
As explained above the claim,~$F_\partial(\widetilde{\mu}_0)=F_\partial \circ \partial_0^{\times S^1}([G_0])= \partial_1^{\times S^1} \circ F_{\rel}([G_0])$, where~$G_0 \subset \widetilde{X}_{S_0}$ is a surface with boundary~$\widetilde{\mu}_0$.
We also write~$D_i$ for a meridional disc with~$\partial D_i=\widetilde{\mu}_i$ for $i=0,1$.
%Note that~$F_\partial(\widetilde{\mu}_0)= \partial_1 \circ F_{\rel}([G_0])$.
The explanation from the paragraph above the claim together with the commutativity of the diagram from the claim gives
%Under the displayed identifications,~$\Sigma_i \cup D_i$ is identified with~$\Sigma_i - D_i$ and the connecting homomorphism is given by 
$$ \partial_1^{\times S^1} \circ  F_{\rel}([G_0]) -\partial_1^{\times S^1} \circ F_{\rel}^\nu([D_0])
=F_\partial \circ \partial_0^{\times S^1} ([G_0] -[D_0])
=F_\partial(\widetilde{\mu}_0-\widetilde{\mu}_0)
=0.$$
We deduce that 
\begin{equation}
\label{eq:MVEquality}
\partial_1^{\times S^1} \circ F_{\rel}([G_0]) =\partial_1^{\times S^1} \circ F_{\rel}^\nu([D_0]) \in H_1(\widetilde{S}_1 \times S^1).
\end{equation}
The next claim calculates this quantity.
\begin{claim}
\label{claim:Compatible2}
The following equality holds:~$\partial_1^{\times S^1} \circ F_{\rel}^\nu([D_0])=\pm \widetilde{\mu}_1 \in H_1(\widetilde{S}_1 \times S^1).$
\end{claim}
\begin{proof}
Consider the following commutative square:
$$
\xymatrix{
0 \ar[r]&H_2(\overline{\nu}(S_0),\widetilde{S}_0 \times S^1)\ar[d]^{F^\nu_{\rel}}_\cong \ar[r]^-{\partial_0^{\times S^1}} &{\overbrace{H_1(\widetilde{S}_0 \times S^1)}^{\cong \Z\langle\widetilde{\mu}_0\rangle}}\ar[d]^{F_\partial}_\cong
\ar[r] &0
 \\
%%%%
0 \ar[r]&H_2(\overline{\nu}(S_1),\widetilde{S}_1 \times S^1) 
\ar[r]^-{\partial_1^{\times S^1}}&{\underbrace{H_1(\widetilde{S}_1 \times S^1)}_{\cong \Z\langle\widetilde{\mu}_1\rangle}}
\ar[r]&0.
\\
}
$$
%We've already explained why the rightmost map is the same as the one from the previous diagram.
The rightmost zeros occur because we are working with discs.
The commutativity of the diagram implies that~$\partial_1^{\times S^1} \circ F_{\rel}^\nu([D_0])= F_\partial \circ \partial_0^{\times S^1}([D_0])
=F_\partial(\widetilde{\mu}_0)$.
%Since the rightmost groups are isomorphic to~$\Z$ and
Since~$F_\partial$ is an isomorphism,  we also have that~$F_\partial(\widetilde{\mu}_0)=\pm \widetilde{\mu}_1$.
This concludes the proof of the claim.
\end{proof}

%It remains to prove that~$F_{\rel}=F_!$ as it will then follow that~$(F,h)$ is a compatible pair.
%\begin{claim}
%\label{claim:Compatible3}
%We have~$F_{\rel}=F_! \colon H_2(\widetilde{X}_{S_0},\partial \widetilde{X}_{S_0}) \to H_2(\widetilde{X}_{S_1},\partial \widetilde{X}_{S_1})$.
%\end{claim}
%\begin{proof}
%Given~$x \in H_2(\widetilde{X}_{S_0})$ and~$y \in H_2(\widetilde{X}_{S_0},\partial \widetilde{X}_{S_0})$, we have
%$$ F(x) \cdot F_!(y)=x \cdot y=F(x) \cdot F_{\rel}(y).$$
%Since the adjoint of this relative pairing is an isomorphism, we deduce that~$F_{\rel}(y)=F_!(y)$ for every~$y \in H_2(\widetilde{X}_{S_0},\partial \widetilde{X}_{S_0})$.
%Thus~$F_{\rel}=F_!$, establishing the claim.
%\end{proof}

We can now conclude the proof that~$(F,h)$ is a compatible pair.
Indeed, combining Claim~\ref{claim:Compatible3},  with the equality in~\eqref{eq:MVEquality}, and Claim~\ref{claim:Compatible2}, we obtain
\begin{align*}
\partial F(\widetilde{\mu}_0)
&=\incl_1^{\times S^1} \circ F_\partial(\widetilde{\mu}_0)
=\incl_1^{\times S^1}  \circ \partial_1^{\times S^1} \circ F_{\rel}([G_0])
=\incl_1^{\times S^1}  \circ \partial_1^{\times S^1} \circ F_{\rel}^\nu([D_0])
=\incl_1^{\times S^1} (\pm \widetilde{\mu}_1) \\
&=\pm \widetilde{h}_*(\widetilde{\mu}_0).
\end{align*}
After possibly precomposing~$H$ with the 
%orientation-reversing 
homeomorphism of~$\overline{\nu}(S_0) \cong D^2 \times D^2,(x,y) \mapsto (x,-y)$,  we can assume that $\widetilde{h}_*$ and $\partial F$ agree on $\widetilde{\mu}_0$.
Since~$H_1(\partial \widetilde{X}_{S_0})$ is cyclic with~$\widetilde{\mu}_0$ as a generator, we deduce that~$\partial F=\widetilde{h}_*$ as required.
\end{proof}

\section{The automorphism invariant}
\label{sec:Injection}

This section provides a quantitative take on the results of Theorem~\ref{thm:CompatiblePairIntro}.
Namely, we aim to estimate the number of surfaces in a simply-connected $4$-manifold with prescribed genus, (relative) Euler number,  type, boundary and whose exteriors have a fixed equivariant intersection form.
We focus on the case where $\pi$ is finite cyclic as the corresponding analysis for $\pi$ infinite cyclic was carried out in~\cite[Section 6]{ConwayPiccirilloPowell}.
%This is also sets up the question of realisation for future work.

\subsection{Set-up}
\label{sub:SetUpInjection}
Throughout this section,  it will be necessary to navigate between the equivariant intersection form and the $\Z$-intersection form on the universal cover.
We record a remark on this topic for later use and then outline the idea underlying this section.

\begin{remark}
\label{rem:HermHSymmH}
Given a group~$\pi$ and a~$\Z[\pi]$-module~$H$,  taking the coefficient at $g=e$ is seen to induce a map between the set of~$\Z[\pi]$-hermitian forms on~$H$ and the set of~$\pi$-invariant integral symmetric bilinear forms on~$H$:
$$(\ev_e)_* \colon \operatorname{Herm}(H) \to \operatorname{Sym}_\pi(H).$$
When $\pi$ is finite, this map is an isomorphism whose inverse is given by sending a~$\pi$-invariant symmetric bilinear form~$Q$ to the hermitian form~$(x,y) \mapsto \sum_{g \in \pi} Q(x,gy)g^{-1}.$
As was already implicitly noted during the proof of Proposition~\ref{prop:EquivCompatiblePair},  taking the coefficient at $g=e$ is also seen to induce the following commutative diagram of $\Z[\pi]$-modules and $\Z[\pi]$-linear maps:
$$
\xymatrix{
0 \ar[r]& \ker(\Ad \lambda) \ar[r]\ar[d]_=& H \ar[r]^-{\Ad \lambda}\ar[d]_=& \Hom_{\Z[\pi]}(H,\Z[\pi]) \ar[r]\ar[d]_\cong^{(ev_e)_*}&\ar[d]_\cong \coker(\Ad \lambda) \ar[r]\ar[d]_\cong&0  \\
%%%
0 \ar[r]& \ker(\Ad Q) \ar[r]& H \ar[r]^-{\Ad Q}& \Hom_{\Z}(H,\Z) \ar[r]&\coker(\Ad Q) \ar[r]&0.  \\
}
$$
Additionally,  a~$\Z[\pi]$-automorphism of $H$ is an isometry of $Q$ if and only if it is an isomorphism of the associated hermitian form $\lambda$.
Consequently, we will often go back and forth between hermitian forms~$\lambda$ and $\pi$-invariant symmetric forms $Q$.
\end{remark}

Next, we outline the idea underlying this section.
Fix a hermitian form $(H,\lambda)$ and let $(H,Q)$ be the associated symmetric form as in Remark~\ref{rem:HermHSymmH}.
Given a knot $K \subset S^3$, we write 
$$\widehat{Y}_e(K) \to Y_e(K):=E_K \cup_e (\Sigma \mathbin{\wt{\times}}S^1)$$
%for brevity and $\widehat{Y}_e(K)$ 
for the $d$-fold cover associated to the coefficient system described in Construction~\ref{cons:CoeffSys}.
Given a~$\Z_d$-surface $S \subset X$ with boundary~$K$, relative Euler number $e$ and~$(H_2(\widetilde{X}_S),Q_{\widetilde{X}_S}) \cong (H,Q)$,   we will construct an isomorphism 
$$
\coker(\Ad Q) \xrightarrow{\cong} H_1(\widehat{Y}_e(K)).
$$
The construction of this isomorphism will depend on several choices and, in essence, our goal will be to extract an invariant (which we call the \emph{automorphism invariant}) from it nonetheless.
The next construction provides the essential idea on how this isomorphism will be defined.

\begin{construction}
\label{cons:bForManifolds}
Fix a closed $3$-manifold $Y$ and an epimorphism $\varphi \colon \pi_1(Y) \twoheadrightarrow \pi$. 
Denote the corresponding $\pi$-cover by $\widehat{Y}.$
Given a~$4$-manifold~$V$ with fundamental group~$\pi$,  consider the following commutative diagram of exact sequences of $\Z[\pi]$-modules and $\Z[\pi]$-linear maps: 
$$
\xymatrix{
H_2(\widetilde{V})
\ar[r]^-{\Ad Q_{\widetilde{V}}}\ar[d]^=&
\Hom_{\Z}(H_2(\widetilde{V}),\Z)
\ar[r]\ar[d]^-{\PD\circ ev^{-1}}_\cong&
\coker(\Ad Q_{\widetilde{V}})
\ar[r]\ar@{-->}[d]^{D_V}_\cong&
0 \\
%%%%%
H_2(\widetilde{V})
\ar[r]^-j&
H_2(\widetilde{V},\partial \widetilde{V})
\ar[r]&
H_1(\partial \widetilde{V})
\ar[r]&
0.
}
$$
As in Section~\ref{sub:CompatiblePairEquiv},  we endow $\Hom_{\Z}(H_2(\widetilde{V}),\Z)$ with the right-$\pi$-action given by~$g \cdot \varphi(x):=\varphi(gx).$
%%PD and ev are then equivariant.

When there is a $\Z[\pi]$-isomorphism~$F \colon  (H,Q) \cong (H_2(\widetilde{V}),Q_{\widetilde{V}})$ and an orientation-preserving homeomorphism~$\gamma \colon \partial V\to Y$ such that the composition~$\pi_1(Y)\xrightarrow{\gamma_*} \pi_1(\partial V) \to \pi_1(V) \cong \pi$ agrees with the epimorphism~$\varphi$,  we consider the following isomorphism of $\Z[\pi]$-modules:
$$
b(V):=\gamma_*  \circ D_V \circ \partial F \colon \coker(\Ad Q) \to H_1(\widehat{Y}).
$$
Here $\partial F \colon \coker(\Ad Q_{\widetilde{V}}) \to \coker(\Ad Q)$ denotes the $\Z[\pi]$-isomorphism induced by $F$.
\end{construction}

\subsection{The automorphism invariant}
\label{sub:SetUpInjection}

We begin the construction of the automorphism invariant by applying Construction~\ref{cons:bForManifolds} to the case where~$V$ is a surface exterior. 
Some notation is in order.
To begin with, fix a simply-connected~$4$-manifold~$X$ with boundary~$\partial X=S^3$ (recall our conventions), 
%a homeomorphism~$h \colon \partial X \xrightarrow{\cong} S^3$
and a knot~$K \subset \partial X$.
%The goal of this section is to define the sets of surfaces and embeddings that we will be interested in.

\begin{notation}
Consider the set~$\Surf(g)(X,K)$ of rel.\ boundary equivalence classes of genus~$g$ embedded simple surfaces in~$X$ with boundary~$K$.
Given a positive integer $d>0$, an integer~$e \in \Z$, and a hermitian form~$(H,\lambda)$ over $\Z[\Z_d]$, we also consider the subset
\[ \Surf_{d,e,\lambda}(g)(X,K) \subset \Surf(g)(X,K)\]
of $\Z_d$-surfaces with relative Euler number $e$ and whose exteriors have equivariant intersection form isometric to~$\lambda$.
For nonorientable genus~$h$ surfaces,  $d$ is understood to be either $2$ or $1$ (depending on whether or not the surface is trivial in $H_2(X,\partial X;\Z_2)$) and we write~$\Surf^{-}(h)(X,K)$ and~$\Surf_{d,e,\lambda}^{-}(h)(X,K)$ for the corresponding sets of simple surfaces.
For brevity,  we write
$$\Surf_{d,e,\lambda}^w(g)(X,K)$$
 to encompass both the orientable and nonorientable cases, where~$w \in \{\pm\}$.

Similarly, we consider the set~$\Emb^w(g)(X,K)$ of rel.\ boundary equivalence classes of embeddings~$\iota \colon \Sigma \hookrightarrow X$ such that~$\iota(\Sigma)$ is a simple surface of (nonorientable) genus $g$ and with boundary~$K$. 
The group~$\Homeo_\partial(\Sigma)$ of rel.\ boundary self homeomorphisms of~$\Sigma$ acts on~$\Emb^w(g)(X,K)$ via~$\theta \cdot \iota=\iota \circ \theta^{-1}$ and one checks that the surjection~$\Emb^w(g)(X,K) \to \Surf^w(g)(X,K)$ descends to a bijection~$\Emb^w(g)(X,K)/\Homeo_\partial(\Sigma)\cong \Surf^w(g)(X,K)$.
Homeomorphisms are assumed to preserve local orientations, which is automatic when said homeomorphisms are rel.\ boundary.
%%Don't delete
%{I think we always want to preserve the local orientation. But if $\partial\neq\emptyset$ then any homeo preserves it anyway.}
The aforementioned bijection also induces a bijection
$$\Emb_{d,e,\lambda}^w(g)(X,K)/\Homeo_\partial(\Sigma)\cong \Surf_{d,e,\lambda}^w(g)(X,K).$$
The knot group of the exteriors of these surfaces is isomorphic to~$\Z_d$.
\end{notation}

\begin{notation}
\label{not:Kinvariant}
We write~$\Aut_{\Z[\pi]}(Q)$ for the group of~$\Z[\pi]$-isometries of the $\Z[\pi]$-module~$H$ that preserve the symmetric bilinear form $Q$.
\end{notation}

We apply Construction~\ref{cons:bForManifolds} to surface exteriors.

\begin{construction}
\label{cons:biota}
Let~$\iota \colon \Sigma \hookrightarrow X$ be an embedding that represents an element of~$\Emb_{d,e,\lambda}^{w}(g)(X,K)$.  
Choose a normal bundle~$(\nu(\iota),f \colon \nu(\iota) \hookrightarrow X)$ of~$\iota$,  an isometry~$F \colon \lambda \cong \lambda_{X \setminus f(\nu(\iota))}$,  and an~$e$-nice identification~$\fr \colon \nu(\iota) \to \Sigma \mathbin{\wt{\times}} \R^2$.
Consider the resulting boundary homeomorphism
$$\overline{\fr}:=\id_{E_K} \cup \fr \colon \partial (X \setminus f(\nu(\iota))) \to Y_e(K)$$ 
as well as the following $\Z[\Z_d]$-isomorphism
\begin{equation}
\label{eq:b}
b^{\fr}(\iota):=b(X \setminus f(\nu(\iota)))=\overline{\fr}_* \circ D_{X \setminus f(\nu(\iota))} \circ \partial F \colon \coker(\Ad Q) \to H_1(\widehat{Y}_e(K)).
\end{equation}
\end{construction}

We record the fact that $b^{\fr}(\iota)$ preserves linking forms once restricted to $\Z$-torsion subgroups.
Before doing so, we recall the notion of the boundary linking form of a symmetric bilinear form.
\begin{construction}
\label{cons:LinkingForm}
Fix a closed~$3$-manifold~$Y$,  an epimorphism~$\varphi \colon \pi_1(Y) \twoheadrightarrow \pi$,  and a hermitian form~$\lambda$ over~$\Z[\pi]$.
%, an isometry~$F \colon \lambda_V \cong \lambda$.
Given a~$\Z[\pi]$-module~$H$, we write~$T_\Z H$ for the~$\Z$-torsion subgroup of~$H$, viewed as an abelian group.
%%Don't delete.
%T is also a submodule
When~$\pi$ is finite,  consider the linking form on the~$\pi$-cover~$\widehat{Y}$ of~$Y$:
$$\ell_{\widehat{Y}} \colon T_\Z H_1(\widehat{Y}) \times T_\Z H_1(\widehat{Y}) \to \Q/\Z.$$
Given a $\pi$-invariant integral symmetric bilinear form~$Q \colon H \times H \to \Z$, where $H$ is free as an abelian group, 
%Continuing with the assumption that~$\pi$ is finite,  assume from now on that~$H$ is free as an abelian group and write~$\lambda(x,y)=\sum_{g \in \pi} Q(x,gy)g^{-1}$ for some $\pi$-invariant integral symmetric bilinear form~$Q \colon H \times H \to \Z$; recall Remark~\ref{rem:HermHSymmH}.
we then consider the boundary linking form of $Q$
(see e.g.~\cite[page 243]{RanickiExact}):
\begin{align*}
\partial Q \colon T_{\Z} \coker(\Ad Q)  \times T_{\Z} \coker(\Ad Q) \to \Q/\Z \\
([x],[y]) \mapsto \frac{1}{n}y(z).
\end{align*}
Here,  since $T_{\Z} \coker(\Ad Q)$ is torsion, there exists an $n\in \Z$ such that $nx=\Ad Q(z)$.
We also write~$\Aut_{\Z[\pi]}(\partial Q)$ for the group of~$\Z[\pi]$-isometries of the $\Z[\pi]$-module~$\coker(\Ad Q)$ that preserve the linking form~$\partial Q$.
\end{construction}

The next proposition proves that~$b^{\fr}(\iota)$ preserves these linking forms.

\begin{proposition}
\label{prop:bfrPreservesLinking}
Continuing with the notation from Construction~\ref{cons:biota}, the isomorphism $b^{\fr}(\iota)$ preserves the linking forms on the $\Z$-torsion submodules:
 $$
 b^{\fr}(\iota) \colon
(T_\Z\coker(\Ad Q),\partial Q) \to (T_\Z H_1(\widehat{Y}_e(K)),\unaryminus \ell_{\widehat{Y}_e(K)}).
 $$
This isometry is a $\Z[\pi]$-isomorphism on the underlying modules.
\end{proposition} 
\begin{proof}
Set~$X_\iota:=X \setminus f(\nu(\iota))$ for brevity.
We outline why~$\partial F,D_{X_\iota}$ and~$\overline{\fr}_*$ are isometries of the linking form.
For~$\overline{\fr}_* \colon \ell_{\partial \widehat{X}_\iota} \cong \ell_{\widehat{Y}_e(K)}$, this is clear since~$\overline{\fr}$ is an orientation-preserving homeomorphism that lifts to the covers.
For~$\partial F \colon  \partial Q \cong \partial Q_{\widetilde{X}_\iota}$,  we refer to the explanation in~\cite[page~336]{BoyerUniqueness} applied to the universal cover of~$X_\iota$ viewed as a simply-connected~$4$-manifold.
Finally, for~$D_{X_\iota}\colon  \partial Q_{\widetilde{X}_\iota} \cong \unaryminus \ell_{\partial \widehat{X}_\iota}$ this can be seen by adapting the argument in~\cite[Proposition 3.5]{ConwayPowell} to the simply-connected~$4$-manifold~$\widetilde{X}_\iota$; we do not outline the details since the result is also implicit in~\cite[page~336]{BoyerUniqueness}.
%The fact that~$b^{\fr}(\iota)$ is a~$\Z[\pi]$-isomorphism holds because it holds for~$\partial F,D_V$ and~$\overline{\fr}_*$.
Since~$\partial F,D_V$ and~$\overline{\fr}_*$ are~$\Z[\pi]$-isomorphism, so is~$b^{\fr}(\iota)$.
\end{proof}

The isomorphism~$b^{\fr}(\iota)$ is not yet an invariant of~$\iota$ as the following proposition describes.

\begin{proposition}
\label{prop:ChangeIdentification}
Continuing with the notation from Construction~\ref{cons:biota},
the following holds.
\begin{itemize}
\item For any two choices of isometries $F,F'$ and normal bundles $f,f'$, the isomorphisms constructed in~\eqref{eq:b} differ by precomposition with $\partial G$ for some isometry $G \in \Aut_{\Z[\pi]}(Q)$.
% the resulting $b$ changes $b$ by precomposition with $\partial G$ for some isometry $G \in \Aut_{\Z[\pi]}(Q)$.
\item For any two $e$-nice identifications $\fr,\fr'$ of $\iota(\Sigma) \subset X$, there is a class~$\alpha \in d \cdot H^1(\Sigma,\partial \Sigma;\Z^w)$ such that $\fr=\Rot(\alpha) \circ \fr'$ and $b^{\fr}(\iota)=\Rot(\alpha)_* \circ b^{\fr'}(\iota)$.
%\item When $X$ is spin and $d$ is odd,
%%{DK: and $d$ odd? Otherwise shouldn't it be $lcm(2,d)$ instead of $2d$?},  
%if we choose a spin structure~$\mathfrak{s}$ on~$\Sigma\mathbin{\wt{\times}} \R^2$ and only consider those~$e$-nice identifications~$\fr$ that pull~$\mathfrak{s}$ back to the spin structure of~$X$ restricted to~$\nu(S)$, then for any two such~$e$-nice framings $\fr,\fr'$, there is an~$\alpha \in 2d \cdot H^1(\Sigma,\partial \Sigma;\Z^w)$ such that $\fr=\Rot(\alpha) \circ \fr'$ and $b^{\fr}(\iota)=\Rot(\alpha)_*^d \circ b^{\fr'}(\iota)$.
\end{itemize}
\end{proposition}
\begin{proof}
The first assertion follows from the definitions, so we focus on the second.
% and third.
Given two choices of~$e$-nice identifications~$\fr \colon  \nu(\iota) \to  \Sigma \mathbin{\wt{\times}} \R^2$ and~$\fr' \colon \nu(\iota) \to \Sigma \mathbin{\wt{\times}} \R^2$,  Remark~\ref{rem:HowManyIdentifications} ensures that~$\fr=\Rot(\alpha) \circ \fr'$ for some cohomology class~$\alpha \in d \cdot H^1(\Sigma,\partial \Sigma;\Z^w)$. % and some~$k \geq 0$.
We refer to Appendix~\ref{sec:BundleIsos} for the construction of the bundle isomorphism~$\Rot(\alpha) \colon \Sigma \mathbin{\wt{\times}} \R^2 \to  \Sigma \mathbin{\wt{\times}} \R^2$ but we recall from Propositions~\ref{prop:RotAlphaAxiomsSection2} and~\ref{prop:RotAlphaAxiomsSection2Nonori}
 that~$\Rot(\alpha)|_{\partial \Sigma \times \R^2}=\id$.
In particular~$\Rot(\alpha)$ 
%or rather its restriction to the disc bundles
extends over~$E_K$ by the identity leading to an orientation-preserving homeomorphism~$\Rot(\alpha)\colon Y_e(K) \to Y_e(K)$.
Temporarily setting~$X_\iota:=X \setminus f(\nu(\iota))$ as well as~$b^{\fr}(\iota):= (\id_{E_K} \cup \fr) \circ D_{X_\iota} \circ \partial F$ and~$b^{\fr'}(\iota):=(\id_{E_K} \cup \fr') \circ D_{X_\iota} \circ \partial F$, it follows that
\begin{align*}
b^{\fr}(\iota)&=
(\id_{E_K} \cup \fr)_* \circ D_{X_\iota} \circ \partial F \\
&=(\id_{E_K} \cup (\Rot(\alpha) \circ \fr'))_* \circ D_{X_\iota} \circ \partial F \\
&=\Rot(\alpha)_* \circ (\id_{E_K} \cup  \fr')_* \circ D_{X_\iota} \circ \partial F \\
&=\Rot(\alpha)_* \circ  b^{\fr'}(\iota).
\end{align*}
%For the last assertion, if the~$e$-nice identifications~$\fr,\fr'$ pull the fixed spin structure~$\mathfrak{s}$ back to the spin structure of~$X$ restricted to~$\nu(S)$, then~$\fr=\Rot(\alpha)\circ \fr'$ for some~$\alpha\in 2d \cdot H^1(\Sigma,\partial \Sigma;\Z^w)$. 
%Indeed it must be~$\Rot(\beta)^{kd}$ since it must preserve the coefficient systems and~$k$ must be even otherwise it won't induce trivial map on spin structures.
This concludes the proof of the proposition.
\end{proof}

In order to obtain a map from $\Emb_{d,e,\lambda}^w(g)(X,K)$, we introduce some notation.
\begin{notation}
The group~$\Aut(\lambda)=\Aut_{\Z[\Z_d]}(Q)$ acts on~$\Aut_{\Z[\Z_d]}(\partial Q)$ by~$F \cdot h:=h \circ \partial  F^{-1}|_{T_\Z}$.
Similarly~$\Aut_{\Z[\Z_d]}(Q)$ acts on~$\Iso_{\Z[\Z_d]}(\coker(\Ad Q),H_1(\widehat{Y}_e(K)))$ by precomposition.
Our notation for the subgroup of linking form preserving $\Z[\Z_d]$-isometries is 
$$
\Iso_{\Z[\Z_d]}^{\ell k}(\coker(\Ad Q),H_1(\widehat{Y}_e(K))).
$$
\end{notation}

The next result is concerned with the assignment $\iota \mapsto b^{\fr}(\iota)$.

\begin{proposition}
\label{prop:bWellDef}
Continuing with the notation from Construction~\ref{cons:biota},
%if~$X$ is nonspin or $d>0$ is even, then 
the assignment $\iota \mapsto b^{\fr}(\iota)$ determines a map
$$
\Emb_{d,e\lambda}^{w}(g)(X,K) \to
\frac{\Iso_{\Z[\Z_d]}^{\ell k}(\coker(\Ad Q),H_1(\widehat{Y}_e(K)))}{\Aut_{\Z[\Z_d]}(Q) \times d \cdot H^1(\Sigma,\partial \Sigma;\Z^w)}.
$$
\end{proposition}
\begin{proof}
We verify the independence of $b(\iota):=b^{\fr}(\iota)$ on the choice of the normal bundle~$(\nu(\iota),f)$ by adapting the verifications from~\cite[page 40]{ConwayPiccirilloPowell} and~\cite[Proposition 4.8]{ConwayMiller}.
The fact that the maps are well defined then follows from Proposition~\ref{prop:ChangeIdentification}.
Assume that~$f \colon \nu(\iota) \hookrightarrow X$ and~$f' \colon \nu(\iota) \hookrightarrow X$ are two embeddings of the normal bundle~$\nu(\iota)$.
We obtain two normal bundles of~$\iota$ inside~$X$ namely~$f \colon  \nu(\iota) \hookrightarrow X$ and~$f'  \colon \nu(\iota) \hookrightarrow X.$
The uniqueness of normal bundles from~\cite[Theorem 9.3]{FreedmanQuinn} ensures the existence of a rel.\ boundary homeomorphism 
$G \colon X \to X$ with~$
G \circ f  = f'$.
This ensures  that~$G$ preserves tubular neighborhoods, i.e.~$G(f(\overline{\nu}(\iota))=f'(\overline{\nu}(\iota)).$
%%Don't delete
%Indeed~\eqref{eq:IsotopyPunctured} implies
%$$
%G(f(\overline{\nu}(e|_{\Sigma}))
%=G \circ f \circ \fr^{-1}(\Sigma \mathbin{\wt{\times}} \R^2)
%=f' \circ (\fr')^{-1}(\Sigma \mathbin{\wt{\times}} \R^2)
%=f'(\overline{\nu}(e|_{\Sigma^\circ})).
%~$$
It follows that~$G \colon X \to X$ restricts to a homeomorphism~$X \setminus f(\nu(\iota)) \to X \setminus f'(\nu(\iota))$.
After choosing an~$e$-nice identification~$\fr \colon  \nu(\iota) \to  \Sigma \mathbin{\wt{\times}} \R^2$, the resulting situation can be summarized in the following diagram:
$$
\xymatrix@C1.5cm{
%& \partial (X \setminus f(\nu(\iota))) \ar[r]^-{\subset}\ar[d]^-{=} \ar[l]_-{\id_{E_K} \cup (\fr| \circ f^{-1})} 
%& X \setminus f(\nu(\iota)) \ar[r]^-{\subset}\ar[d]^-{=}  
%& X \ar[d]^-{=} \\
%%%%%
Y_e(K) \ar[d]^= 
& \partial (X \setminus f(\nu(\iota))) \ar[r]^-{\subset}\ar[d]^-{G} \ar[l]_-{\id_{E_K} \cup (\fr| \circ f^{-1})} 
& X \setminus f(\nu(\iota)) \ar[r]^-{\subset}\ar[d]^-{G}  
& X \ar[d]^-{G} \\
%%%%%
Y_e(K)
& \partial (X \setminus f'(\nu(\iota))) \ar[r]^-{\subset}  \ar[l]_-{\id_{E_K} \cup (\fr| \circ {f'}^{-1})}
& X \setminus f'(\nu(\iota)) \ar[r]^-{\subset}
 & X. \\
}
$$
The central and right squares clearly commute and, for the left square,
since~$G|_{\partial X}=\id_{\partial X}$,
%%Don't delete.
%Recall that \partial X=S^3 from the conventions.
the equation~$G \circ f  = f'$  ensures that~$G \circ (\fr| \cup f^{-1})=(\fr| \cup {f'}^{-1}).$
Thus $G \colon X \setminus f(\nu(\iota)) \to X \setminus f'(\nu(\iota))$ is a homeomorphism rel.\ boundary,  which is readily seen to imply that the induced isomorphisms on the Alexander modules coincide in the orbit set.
%%Don't delete.
%See "AC map b well def (note from CPP).pdf" in dropbox.
\end{proof}

In order for these maps to yield injections, we need to take the relative $k$-invariant and the possible spin structures on the boundary into account.

\begin{construction}
\label{cons:kInvariantMap}
Let~$\iota \colon \Sigma \hookrightarrow X$ be an embedding that represents an element of~$\Emb_{d,e,\lambda}^{w}(g)(X,K)$.  
Choose a normal bundle~$(\nu(\iota),f \colon \nu(\iota) \hookrightarrow X)$ of~$\iota$,  an isometry~$F \colon \lambda \cong \lambda_{X \setminus f(\nu(\iota))}$,  and an~$e$-nice identification~$\fr \colon \nu(\iota) \to \Sigma \mathbin{\wt{\times}} \R^2$.
Set $X_\iota:=X\setminus f(\nu(\iota))$ and consider the cohomology class
$$
k^{\fr}(\iota):=F_*^{-1} \circ (\overline{\fr}^*)^{-1}(k_{X_\iota,\partial X_\iota})
=(\overline{\fr}^*)^{-1} \circ F_* (k_{X_\iota,\partial X_\iota})
\in H^3(\Z_d,Y_e(K);H).
$$
The group $\Aut_{\Z[\Z_d]}(Q) \times d \cdot H^1(\Sigma,\partial \Sigma;\Z^w)$ acts on $H^3(\Z_d,Y_e(K);H)$ by 
%%Don't delete
%{AC: It could potentially be $(\Rot(\alpha)^*)^{-1}$. I can't quite decide. DK: Since $(\Rot(\alpha)^*)^{-1}=\Rot(-\alpha)^*$ it doesn't matter for the quotient.}
$$(F,\alpha) \cdot k
:= F_*^{-1} \circ \Rot(\alpha)^*(k)
:=\Rot(\alpha)^*(k) \circ F_*^{-1} (k).
$$
If $d$ is odd and $X\setminus f(\nu(\iota))$ is spin, then 
%%Unique because $H_1(X\setminus f(\nu(\iota));\Z_2)=0$ because d odd
the unique spin structure on $X\setminus f(\nu(\iota))$ induces a unique spin structure $\mathfrak{s}_\iota$ on $\partial(X\setminus f(\nu(\iota)))$. 
Consider
	$$\mathfrak{s}^{\fr}(\iota):=(\overline{\fr}^*)^{-1}(\mathfrak{s}_\iota)\in\Spin(Y_e(K)).$$
The group $\Aut_{\Z[\Z_d]}(Q) \times d \cdot H^1(\Sigma,\partial \Sigma;\Z^w)$ acts on $\Spin(Y_e(K))$ by
$$(F,\alpha) \cdot\mathfrak{s}:=\Rot(\alpha)^*(\mathfrak{s}).$$
%If~$X$ is nonspin or $d>0$ is even, then 
If $d$ is even or $X\setminus f(\nu(\iota))$ is nonspin, then
the assignment $\iota \mapsto k^{\fr}(\iota)$ determines a map
$$
\Emb_{d,e,\lambda}^{w}(g)(X,K) \to
\frac{H^3(\Z_d,Y_e(K);H)}{\Aut_{\Z[\Z_d]}(Q) \times d \cdot H^1(\Sigma,\partial \Sigma;\Z^w)}.
$$
%If~$X$ is spin and $d$ is odd, then 
If $d$ is odd and $X\setminus f(\nu(\iota))$ is spin, then
the assignment $\iota \mapsto (k^{\fr}(\iota),\mathfrak{s}^{\fr}(\iota))$ determines a map
$$
\Emb_{d,e,\lambda}^{w}(g)(X,K) \to
\frac{H^3(\Z_d,Y_e(K);H) \times \Spin(Y_e(K))}{\Aut_{\Z[\Z_d]}(Q) \times d \cdot H^1(\Sigma,\partial \Sigma;\Z^w)}.
$$
The proofs of these assertions follow along the same lines as the proofs of Propositions~\ref{prop:ChangeIdentification} and~\ref{prop:bWellDef}.
\end{construction}

\begin{remark}
Note that the case ``$d$ is odd and $X\setminus f(\nu(\iota))$ is spin" occurs either when $\iota(\Sigma)$ is characteristic and $X$ is nonspin or when $\iota(\Sigma)$ is ordinary and $X$ is spin.
In the former case,  the spin structure~$\mathfrak{s}^{\fr}(\iota)$ does not extend over the normal bundle, whereas in the latter it does.
For this reason,  we decompose~$\Spin(Y_e(K))$ into the set of spin structures whose restriction to $Y_e(K)$ extend over the disc bundle and those that do not:
$$
\Spin(Y_e(K))=
\Spin^{\operatorname{ext}}(Y_e(K))
\sqcup
\Spin^{\operatorname{not-ext}}(Y_e(K)).
$$
We also decompose~$\Emb_{d,e,\lambda}^w(g)(X)$ according to whether or not the surface is characteristic in~$X$:
$$\Emb_{d,e,\lambda}^w(g)(X)=\Emb_{d,e,\lambda}^{w,\charac}(g)(X) \sqcup \Emb_{d,e,\lambda}^{w,\ord}(g)(X).$$
The assignment $\iota \mapsto (k^{\fr}(\iota),\mathfrak{s}^{\fr}(\iota))$ therefore determines maps
\begin{align*}
\Emb_{d,e,\lambda}^{w,\charac}(g)(X,K) \to
\frac{H^3(\Z_d,Y_e(K);H) \times \Spin^{\operatorname{not-ext}}(Y_e(K))}{\Aut_{\Z[\Z_d]}(Q) \times d \cdot H^1(\Sigma,\partial \Sigma;\Z^w)} &\quad \text{if $X$ is not spin,} \\
%%%
\Emb_{d,e,\lambda}^{w,\ord}(g)(X,K) \to
\frac{H^3(\Z_d,Y_e(K);H) \times \Spin^{\operatorname{ext}}(Y_e(K))}{\Aut_{\Z[\Z_d]}(Q) \times d \cdot H^1(\Sigma,\partial \Sigma;\Z^w)} &\quad \text{if $X$ is spin}.
\end{align*}
\end{remark}

The first main result of this section is the following.

\begin{theorem}
\label{thm:Injectionv1}
Let~$X$ be a simply-connected~$4$-manifold with~$\partial X=S^3$,  let~$d>0$ be an integer,  let~$\lambda \colon H \times H \to \Z[\Z_d]$ be a hermitian form over $\Z[\Z_d]$, and let $e \in \Z$ be another integer.
Let~$K \subset~S^3$ be a knot such that when~$d>0$ is not a prime~$\Sigma_d(K)$ is a rational homology sphere.
\begin{itemize}
\item In the characteristic case,  the automorphism invariant, the relative $k$-invariant and the spin structure on the boundary give rise to an injective map
$$
\Xi_{\Emb} \colon \Emb_{d,e,\lambda}^{w,\charac}(g)(X) 
\hookrightarrow
\begin{cases}
\frac{\Iso_{\Z[\Z_d]}^{\ell k}(\coker(\Ad Q),H_1(\widehat{Y}_e(K)))
\times H^3(\Z_d,Y_e(K);H) \times\Spin^{\operatorname{not-ext}}(Y_e(K))}
{\Aut_{\Z[\Z_d]}(Q) \times d \cdot H^1(\Sigma,\partial \Sigma;\Z^w)} 
\quad& \text{if $X$ is not spin,} \\
 %%%
 \frac{\Iso_{\Z[\Z_d]}^{\ell k}(\coker(\Ad Q),H_1(\widehat{Y}_e(K)))
\times H^3(\Z_d,Y_e(K);H) }
{\Aut_{\Z[\Z_d]}(Q) \times d \cdot H^1(\Sigma,\partial \Sigma;\Z^w)}   \quad & \text{if $X$ is spin.}
 \end{cases}
$$
\item In the ordinary case,  the automorphism invariant, the relative $k$-invariant and the spin structure on the boundary give rise to an injective map
$$
\Xi_{\Emb} \colon \Emb_{d,e,\lambda}^{w,\ord}(g)(X) 
\hookrightarrow
\begin{cases}
\frac{\Iso_{\Z[\Z_d]}^{\ell k}(\coker(\Ad Q),H_1(\widehat{Y}_e(K)))
\times H^3(\Z_d,Y_e(K);H) }
{\Aut_{\Z[\Z_d]}(Q) \times d \cdot H^1(\Sigma,\partial \Sigma;\Z^w)} 
\quad& \text{if $X$ is not spin,} \\
 %%%
 \frac{\Iso_{\Z[\Z_d]}^{\ell k}(\coker(\Ad Q),H_1(\widehat{Y}_e(K)))
\times H^3(\Z_d,Y_e(K);H)  \times\Spin^{\operatorname{ext}}(Y_e(K))}
{\Aut_{\Z[\Z_d]}(Q) \times d \cdot H^1(\Sigma,\partial \Sigma;\Z^w)} 
 \quad & \text{if $X$ is spin.}
 \end{cases}
$$
\end{itemize}
When the $k$-invariant condition is automatic, the $H^3(\Z_d,Y_e(K);H)$ factors can be omitted.
The action of $d \cdot H^1(\Sigma,\partial \Sigma;\Z^w)$ on spin structures is trivial in the characteristic case and nontrivial in the ordinary case.
\end{theorem}
\begin{proof}
The arguments from Proposition~\ref{prop:bWellDef} and Construction~\ref{cons:kInvariantMap} ensure that the maps are well defined, so we focus on showing that they are injective. 
%{Technically this is not true because we do not mod out 2 copies acting on the factors separately. But if follows from the proofs of those statements. }
Assume that $\iota_0,\iota_1$ are embeddings with~$(b^{\fr_0}(\iota_0),k^{\fr_0}(\iota_0))=(b^{\fr_1}(\iota_1),k^{\fr_1}(\iota_1))$ and, when relevant, $\mathfrak{s}^{\fr_0}(\iota_0)=\mathfrak{s}^{\fr_1}(\iota_1)$.
We aim to show that $\iota_0$ and $\iota_1$ are equivalent rel.\ boundary.

If~$(b^{\fr_0}(\iota_0),k^{\fr_0}(\iota_0))=(b^{\fr_1}(\iota_1),k^{\fr_1}(\iota_1))$ in the orbit set,  one verifies that that there exists an isometry~$F$ such that~$(F,h):=\id_{E_K} \cup (\fr_1^{-1} \circ \fr_0'))$ is a~$k$-invariant preserving compatible pair where~$\fr_0'=\Rot(\alpha)\circ \fr_0$; here we are also using Proposition~\ref{prop:EquivCompatiblePair} to phrase compatible pairs in terms of universal covers. 
When the~$k$-invariant condition is automatic, the~$H^3$ factor is not needed.

We explain why~$b^{\fr_0}(\iota_0)=b^{\fr_1}(\iota_1)$ ensures that there is an isometry~$G$ such that~$(G,\id \cup \fr_1^{-1} \circ \fr_0')$ is a compatible pair; the additional verification involving the relative~$k$-invariant is left to the reader.
The definition of the automorphism invariant and of the orbit sets implies that there exists isometries~$G_0,G_1,G'$ and a class~$\alpha \in d \cdot H^1(\Sigma,\partial \Sigma;\Z^w)$ that make the following diagram commute:
$$
\xymatrix{
H_2(X_{S_0})\ar[r]&
H_2(X_{S_0},\partial X_{S_0}) \ar[rr]&&
H_1(\partial X_{S_0}) \ar[ddd]_{\id \cup \fr_0}^\cong \\
%%%
H_2(X_{S_0})\ar[r]^-{\Ad Q_{X_{S_0}}}\ar[u]_=&
H_2(X_{S_0})^* \ar[r]\ar[u]^-{\PD \circ \ev^{-1}}_\cong&
\coker(\Ad Q_{X_{S_0}}) \ar[ru]^-{D_{X_{S_0}}}_\cong & \\
%%%%
&&\coker(\Ad Q)\ar@{-->}[ddd]_-{\partial G'}^\cong \ar[rd]_-{b^{\fr_0}(\iota_0)}^\cong \ar[u]^-{\partial G_0}_\cong& \\
%%%%
H\ar[r]^-{\Ad Q}\ar[uu]^{G_0}_\cong\ar@{-->}[d]_{G'}^\cong&
H^* \ar[rr]\ar[uu]^{(G_0^*)^{-1}}_\cong\ar@{-->}[d]_-{((G')^*)^{-1}}^\cong \ar[ru]^-{\proj}&&
H_1(\widehat{Y}_e(K)) \ar@{-->}[d]_-{\Rot(\alpha)}^\cong \\
%%%%
%%%%
H\ar[r]^-{\Ad Q} \ar[dd]_-{G_1}^\cong&
H^*\ar[rr] \ar[dd]_-{(G_1^*)^{-1}}^\cong\ar[rd]^-{\proj}&&
H_1(\widehat{Y}_e(K)) \\
&&\coker(\Ad Q) \ar[ru]^-{b^{\fr_1}(\iota_1)}_\cong  \ar[d]_-{\partial G_1}^\cong& \\
%%%%
H_2(X_{S_1})\ar[r]^-{\Ad Q_{X_{S_1}}}\ar[d]^=&
H_2(X_{S_1})^* \ar[r]\ar[d]_{\PD \circ \ev^{-1}}^\cong&
\coker(\Ad Q_{X_{S_1}}) \ar[rd]_{D_{X_{S_1}}}^\cong& \\
%%%%
H_2(X_{S_1})\ar[r]&
H_2(X_{S_1},\partial X_{S_1}) \ar[rr]&&
H_1(\partial X_{S_1}) \ar[uuu]^{\id \cup \fr_1}_\cong.
}
$$
If $d$ is even or if~$X\setminus f(\nu(\iota))$ is nonspin (e.g.  if $X$ is spin and the~$S_i$ are characteristic or if $X$ is nonspin and the~$S_i$ are ordinary), then (the proof of) Theorem~\ref{thm:CompatiblePairIntro} ensures that the embeddings are equivalent.
%%Don't delete.
%:{Using 5.8 for the case were $d$ even and there is a spin structure.} 
If $d$ is odd and $X\setminus f(\nu(\iota))$ is spin (e.g. if $X$ is spin and $S$ is ordinary or if $X$ is nonspin and $S$ is characteristic),
%%Don't delete.
%%For S charac, we proved in the spin section that d is odd iff X is nonspin, so if X nonspin and S charac, we have d odd.
%%For For S ordinary, if X is spin, then d is odd.
%%Proof
%S ordinary means there exists a y such that S.y \neq y.y
%Assume X is spin, so ordinay means that there exists a y \in H_2(X) such that S.y=1 mod 2.
%%As explained in the spin section, this means that d is odd.
 then $\mathfrak{s}^{\fr'_0}(\iota_0)=\mathfrak{s}^{\fr_1}(\iota_1)$ ensures that $X_{S_0}\cup_h X_{S_1}$ is spin and again (the proof of) Theorem~\ref{thm:CompatiblePairIntro} ensures that the embeddings are equivalent rel.\ boundary.
Here,  the fact that~$d$ is odd was used to rewrite $\mathfrak{s}^{\fr_0}(\iota_0)=\mathfrak{s}^{\fr_1}(\iota_1)$ as $\mathfrak{s}^{\fr_0'}(\iota_0)=\mathfrak{s}^{\fr_1}(\iota_1)$.
%%Don't delete.
%Rot(d\alpha) doesn't change spin structures for d odd
The statement about the action on spin structures is proved in Proposition~\ref{prop:RotalphaSpinNonOrientable}.
\end{proof}

Next we use Theorem~\ref{thm:Injectionv1} to obtain an injection on
$$\Emb_{d,e,\lambda}^{w}(g)(X,K)/\Homeo_\partial(\Sigma)\cong \Surf_{d,e,\lambda}^{w}(g)(X,K).$$
\begin{construction}
We define an action of~$\Homeo_\partial(\Sigma)$ on~$H_1(\widehat{Y}_e(K))$.
We first describe how a homeomorphism~$\theta \in \Homeo_\partial(\Sigma)$ induces a bundle
isomorphism
$$\theta_{bundle}\colon \Sigma\mathbin{\wt{\times}} \R^2\to  \Sigma\mathbin{\wt{\times}} \R^2.$$
In the orientable case, ~$\Sigma\mathbin{\wt{\times}} \R^2=\Sigma \times \R^2$, the homeomorphism~$\theta$ is orientation-preserving (because~$\theta|_{\partial}=\id$),  and we set~$\theta_{bundle}:=\theta\times \id_{\R^2}$.
%%Don't delete.
%In general not clear that f(\Sigma) and g(\Sigma) being o-p equivalent implies that f,g are o-p homeo?!
In general, we use the model for~$\Sigma\mathbin{\wt{\times}} \R^2$ given by~$\wh \Sigma \times_{\Z_2} \R^2$, where~$p \colon \wh \Sigma \to \Sigma$ denotes the orientable double cover and~$\Z_2$ acts on~$\R^2$ by~$(x,y) \mapsto (x,-y)$; see Appendix~\ref{sub:AutomorphismNonOrientable} for more details.
For each~$x\in \wh\Sigma$ choose~$y_x\in p^{-1}(\theta(p(x)))$ such that the map sending~$[x,z]\in \wh\Sigma\times_{\Z_2}\R^2$ to~$[y_x,z]$ preserves the orientation of~$\Sigma\mathbin{\wt{\times}} \R^2$. 
%%Don't delete
%%If we pick the other preimage y_x' instead then Ty_x=y_{x'} so the map [y_x,(z_1,z_2)] -> [y_x',(z_1,z_2)]=[y_x,(z_1,-z_2)] is an orientation reversing homeo of R^2, so only one choice can make theta_bundle be o-p (since we know what happens in the surface direction). 
Define
$$\theta_{bundle}([x,z]):=[y_x,z].$$
If~$\Sigma$ is orientable and~$\theta$ is orientation preserving, then~$\theta_{bundle}=\theta\times \id_{\R^2}$.
%%Don't delete Section 2 says pick a preferred section.
Note also for that for the preferred section~$r\colon\Sigma\to \wh \Sigma\times_{\Z_2}\R^2$, $x\mapsto [x',(1,0)]$ with~$x'\in p^{-1}(x)$,  we have~$\theta_{bundle}\circ s=s\circ \theta$.
%%Don't delete
%s(\theta(x))=[x',(1,0)] where $p(x')=\theta(x).
%\theta_{bundle}(s(x))=[y_{s(x)},(1,0)] where $p(y_{s(x)})=\theta(p(s(x))=\theta(x) so can take $y_{s(x)}=x'.
%%%%%%
%%%Don't delete (until later)
%For each~$x\in \wh\Sigma$ choose $y_x\in p^{-1}(\theta(p(x)))$, where $p\colon \wh\Sigma\to \Sigma$ is the orientation double cover (which is the trivial double if $\Sigma$ is orientable), such that the map sending $[x,z]\in \wt\Sigma\times_{\Z_2}\R^2\cong \Sigma\mathbin{\wt{\times}} \R^2$ to $[y_x,z]$ preserves the orientation of $\Sigma\mathbin{\wt{\times}} \R^2$. Then $\theta_{bundle}([x,z])=[y_x,z]$. Note that if $\Sigma$ is orientable and $\theta$ is orientation preserving, then $\theta_{bundle}=\theta\times \id_{\R^2}$.
% %isomorphism~$\theta_{bundle}:=f \times \id_{S^1} \colon \Sigma \mathbin{\wt{\times}} \R^2 \to  \Sigma \mathbin{\wt{\times}} \R^2$; here we are using that~$\Sigma$ is orientable to identify~$\Sigma \mathbin{\wt{\times}} S^1$ with~$\Sigma \times S^1$. 
%{DK: Done. Note that it does not agree with your previous definition if $\theta$ is orientation reversing. I added a flip on $\R^2$ in this case to make it orientation preserving on the bundle. But if $\partial\Sigma\neq\emptyset$ and we have the identity on the boundary, then every homeo is o.p.}
%Note that for the section $r\colon\Sigma\to \wh \Sigma\times_{\Z_2}\R^2$, $x\mapsto [x',(1,0)]$ with $x'\in p^{-1}(x)$ we have $\theta_{bundle}\circ s=s\circ \theta$.
Since the homeomorphism $\theta_{bundle}$ is the identity on~$\partial \Sigma \times S^1$,  it extends over~$E_K$ by the identity therefore leading to an orientation-preserving homeomorphism
%%Don't delete. o-p: this is by construction of \theta_bundle, i.e.\ how we pick y_x.
$$\theta_{Y_e(K)}\colon Y_e(K) \to Y_e(K).$$
This homeomorphism preserves the coefficient systems
%Only muddles genus curve
and therefore induces a homeomorphism on the~$\Z_d$-covers. 
This leads to an action of~$\Homeo_\partial(\Sigma)$ on~$\Iso_{\Z[\Z_d]}^{\ell k}(\coker(\Ad Q),H_1(\widehat{Y}_e(K)))$ by postcomposition.
For the same reasons,  the group~$\Homeo_\partial(\Sigma)$ also acts on $H^3(\Z_d,Y_e(K);H)$ by~$\theta \cdot k=(\theta_{Y_e(K)}^*)^{-1}(k)$ and on $\Spin(Y_e(K))$ by~$\theta\cdot \mathfrak{s}=(\theta_{Y_e(K)}^*)^{-1}(\mathfrak{s})$.
\end{construction}

Decompose~$\Surf_{d,e,\lambda}^w(g)(X)$ according to whether or not the surface is characteristic in~$X$:
$$\Surf_{d,e,\lambda}^w(g)(X)=\Surf_{d,e,\lambda}^{w,\charac}(g)(X) \sqcup \Surf_{d,e,\lambda}^{w,\ord}(g)(X).$$
As we will see in Section~\ref{sub:AutomorphismClosed},
the next theorem implies Theorem~\ref{thm:InjectionSurfIntro} from the introduction.
Here,  the semidirect product~$H^1(\Sigma,\partial \Sigma;\Z^w) \rtimes \Homeo_\partial(\Sigma)$ involves the natural action by $\theta \cdot \alpha=\theta^*(\alpha)$.

\begin{theorem}
\label{thm:InjectionSurf}
Let~$X$ be a simply-connected~$4$-manifold with~$\partial X=S^3$,  let~$d>0$ be an integer,  let~$\lambda \colon H \times H \to \Z[\Z_d]$ be a hermitian form over $\Z[\Z_d]$, and let $e \in \Z$ be another integer.
Let~$K \subset~S^3$ be a knot such that when~$d>0$ is not a prime~$\Sigma_d(K)$ is a rational homology sphere.

In the characteristic case,   the map~$\Xi_{\Emb}$ is $\Homeo_\partial(\Sigma)$-equivariant and thus induces an injection
$$
\Xi_{\Surf} \colon \Surf_{d,e,\lambda}^{w,\charac}(g)(X) 
\hookrightarrow
\begin{cases}
\frac{\Iso_{\Z[\Z_d]}^{\ell k}(\coker(\Ad Q),H_1(\widehat{Y}_e(K)))
\times H^3(\Z_d,Y_e(K);H) \times\Spin^{\operatorname{not-ext}}(Y_e(K))
 }{\Aut_{\Z[\Z_d]}(Q) \times (d \cdot H^1(\Sigma,\partial \Sigma;\Z^w) \rtimes \Homeo_\partial(\Sigma))} \quad& \text{if $X$ is not spin,} \\
 %%%
 \frac{\Iso_{\Z[\Z_d]}^{\ell k}(\coker(\Ad Q),H_1(\widehat{Y}_e(K)))
\times H^3(\Z_d,Y_e(K);H) }{\Aut_{\Z[\Z_d]}(Q) \times (d \cdot H^1(\Sigma,\partial \Sigma;\Z^w) \rtimes \Homeo_\partial(\Sigma))}  \quad & \text{if $X$ is spin.}
 \end{cases}
$$
In the ordinary case,   the map~$\Xi_{\Emb}$ is $\Homeo_\partial(\Sigma)$-equivariant and thus induces an injection
$$
\Xi_{\Surf} \colon \Surf_{d,e,\lambda}^{w,\ord}(g)(X) 
\hookrightarrow
\begin{cases}
\frac{\Iso_{\Z[\Z_d]}^{\ell k}(\coker(\Ad Q),H_1(\widehat{Y}_e(K)))
\times H^3(\Z_d,Y_e(K);H) }
{\Aut_{\Z[\Z_d]}(Q) \times (d \cdot H^1(\Sigma,\partial \Sigma;\Z^w) \rtimes \Homeo_\partial(\Sigma))} \quad& \text{if $X$ is not spin,} \\
 %%%
 \frac{\Iso_{\Z[\Z_d]}^{\ell k}(\coker(\Ad Q),H_1(\widehat{Y}_e(K)))
\times H^3(\Z_d,Y_e(K);H)  \times\Spin^{\operatorname{ext}}(Y_e(K))}{\Aut_{\Z[\Z_d]}(Q) \times (d \cdot H^1(\Sigma,\partial \Sigma;\Z^w) \rtimes \Homeo_\partial(\Sigma))}  \quad & \text{if $X$ is spin.}
 \end{cases}
$$
When the $k$-invariant condition is automatic, the $H^3(\Z_d,Y_e(K);H)$ factors can be omitted.
The action of $d \cdot H^1(\Sigma,\partial \Sigma;\Z^w)$ on spin structures is trivial in the characteristic case and nontrivial in the ordinary case.
\end{theorem}
\begin{proof}
We first prove
that~$\Xi_{\Emb}$ and~$\Xi_{\Emb}'$ preserve the~$\Homeo_\partial(\Sigma)$ actions. 
The injectivity statement of the theorem then follows from Theorem~\ref{thm:Injectionv1}.
Fix a homeomorphism~$\theta \in \Homeo_\partial(\Sigma)$.
Write~$f_\iota \colon \nu(\iota) \hookrightarrow X$ for a choice of normal bundle and similarly for~$\theta \cdot \iota$.
A verification using the definition of normal bundles shows that~$X_\iota:=X \setminus f_\iota(\nu(\iota))$ equals~$X_{\theta \cdot \iota}:=X \setminus f_{\iota \circ \theta^{-1}}(\nu(\iota))$.
%see Conway-Miller 4.18.
Next, given an~$e$-nice identification~$\fr_\iota \colon \nu(\iota) \to \Sigma \mathbin{\wt{\times}} \R^2$ for~$\iota(\Sigma)$,
%We therefore have to compare~$\id_{E_K} \cup \fr_e\circ f^{-1}$ with~$\id_{E_K} \cup \fr_{e \circ \theta^{-1}} \circ f^{-1}$ where~$\fr_{e \circ \theta^{-1}}~$ is some~$e$-nice identification of~$e\circ \theta^{-1}$.
the composition~$\theta_{bundle} \circ \fr_\iota$ is an~$e$-nice identification for~$(\theta \cdot \iota)(\Sigma)=\iota\circ \theta^{-1}(\Sigma)$
since $(\theta_{bundle}\circ \fr_\iota)^{-1}\circ s=\fr_\iota^{-1}\circ s\circ \theta^{-1}$.
%%Don't delete. The commuting of the \thetas and s was described before wehen we defined \theta_bundle.
Here recall that $s$ denotes our preferred section.
%; here we used that~$\theta_{bundle}=\theta \times \id$.
%%Don't delete DK: (1) follows since $\theta$ is rel boundary and (2) is obvious since it is about all genus curves so not affected by $\theta^{-1}$.
%{Here it is matters that~$\theta_{bundle}$ be of a nice form e.g.~$\theta \times \id$.}
Thus as manifolds with (parametrised) boundary we have that~$(X_{\theta \cdot \iota},\id_{E_K} \cup (\fr_{\iota \circ \theta^{-1}} \circ f^{-1}))$ equals~$(X_{\theta \cdot \iota},\theta_{Y_e(K)}  \circ (\id_{E_K} \cup (\fr_\iota \circ f^{-1}))).$
Taking the maps on Alexander modules,  we obtain~$\Xi_{\Emb}(\theta \cdot \iota)
=\Xi_{\Emb}(\iota \circ \theta^{-1})
=\theta \cdot \Xi_{\Emb}(\iota)$
and similarly for~$\Xi_{\Emb}'$.

To avoid writing the target as a double quotient, one needs to understand how the actions of~$\Aut_{\Z[\Z_d]}(Q) \times d \cdot H^1(\Sigma,\partial \Sigma;\Z^w)$ and~$\Homeo_\partial(\Sigma)$ interact. 
First note that the action of~$\Aut_{\Z[\Z_d]}(Q)$ commutes with the~$\Homeo_\partial(\Sigma)$-action. 
We have that~$\Rot(\alpha)\circ \theta_{Y_e(K)}=\theta_{Y_e(K)}\circ \Rot(\theta^*\alpha)$ for all~$\alpha\in H^1(\Sigma,\partial \Sigma;\Z^w)$ and~$\theta\in\Homeo_\partial(\Sigma)$ since both are fibrewise homeomorphisms over~$\theta$ that act on the fibre over~$\gamma$ by adding~$\alpha(\theta(\gamma))=\theta^*(\alpha)(\gamma)$ copies of the~$S^1$-fibre. 
%%Don't delete
%In other words, If you think of Rot(\alpha) as rotation by \alpha\colon Sigma\to S^1 in the orientable case, then if you first apply theta (on the base) then alpha (in the fibres) that is the same as first applying theta\circ alpha in the fibers and then theta on the base.
%%Alternatively think of \theta as \theta \Sigma \to \Sigma' so that it becomes clear that \alpha needs to change to give a cohomology class on the other surface; need to pullback.
It follows that we have a semidirect product~$H^1(\Sigma,\partial \Sigma;\Z^w) \rtimes  \Homeo_\partial(\Sigma)$ with the natural action of~$\Homeo_\partial(\Sigma)$ on~$H^1(\Sigma,\partial \Sigma;\Z^w)$.
\end{proof}

In the case of spheres and discs with nonzero Euler number,  the statement of Theorem~\ref{thm:InjectionSurf} simplifies drastically. 
In order to state this result, we write~$\Iso_{\Z[\Z_d]}(\partial Q,-\ell_{\widehat{Y}_e(K)})$ for the set of isometries $\partial Q \cong -\ell_{\widehat{Y}_e(K)}$ and, in addition,  when the surfaces are closed, the knot $K$ is assumed to be empty.
Here, we note that although, technically speaking, this section has only treated surfaces with nonempty boundary,  everything works analogously in the closed case, as we will describe in more detail in Section~\ref{sub:AutomorphismClosed}.
% the next section.

\begin{corollary}
\label{cor:Section8DiscsSpheres}
Continuing with the notation from Theorem~\ref{thm:InjectionSurf},  for spheres or discs with nonzero Euler number $e$,  the map~$\Xi_{\Surf}$ induces an injection
$$
\Xi_{\Surf} \colon \Surf_{d,e,\lambda}^{\charac}(0)(X,K) ,\Surf_{d,e,\lambda}^{\ord}(0)(X,K) 
\hookrightarrow
 \frac{\Iso_{\Z[\Z_d]}(\partial Q,-\ell_{\widehat{Y}_e(K)})}{\Aut_{\Z[\Z_d]}(Q)}.
$$
In the case of Moebius bands or projective planes, $\Xi_{\Surf}$ induces an injection
%%Don't delete
%DK: Also, the dH^1(Sigma) action is non-trivial I think: on the Z/4 (or Z/2+Z/2) it adds the meridian to the base curve (by definition of Rot(\alpha))
%That would show that in 1.16 we only have 1 band if det(K) is a prime power (because Aut(Q) acts by -1 on both Z/det(K) and Z/4 while 2H^1 only acts by -1 on the Z/4). 
$$
\Xi_{\Surf} \colon \Surf_{d,e,\lambda}^{-,\charac}(1)(X,K),\Surf_{d,e,\lambda}^{-,\ord}(1)(X,K) 
\hookrightarrow
 \frac{\Iso_{\Z[\Z_d]}(\partial Q,-\ell_{\widehat{Y}_e(K)})}
 {\Aut_{\Z[\Z_d]}(Q) \times d \cdot H^1(\Sigma,\partial \Sigma;\Z^w)}.
$$
%The action of $d \cdot H^1(\Sigma,\partial \Sigma;\Z^w)$ on spin structures is trivial in the characteristic case and nontrivial in the ordinary case.
%%{XXX:: Should this be deleted? There are no spin structures here.  XX:Done.}
%%H^1(X_S;\Z_2)=H_2(\partial X_S;\Z_2)=0.
In this latter statement,  recall that $d \in \{1,2\}$.
\end{corollary}
\begin{proof}
In each of these cases,  the mapping class group is trivial (see e.g.~\cite[Theorem 3.4]{EpsteinCurves} for the case of Moebius bands) and so it follows that the $\Homeo_\partial(\Sigma)$ action is trivial.
For discs and spheres, it is additionally the case that~$H^1(\Sigma,\partial \Sigma)=0$.
In all cases,  the Alexander module is torsion (by Propositions~\ref{prop:AlexanderModule} and~\ref{prop:AlexanderModuleNonOri}) and as a consequence
$$\Iso_{\Z[\Z_d]}^{\ell k}(\coker(\Ad Q),H_1(\widehat{Y}_e(K)))=\Iso_{\Z[\Z_d]}(\partial Q,-\ell_{\widehat{Y}_e(K)}).$$
The~$k$-invariant condition is automatic (by Proposition~\ref{prop:NokInvariant}) and  the unions of surface exteriors are automatically spin (recall Theorem~\ref{thm:CK4Manifold} or Section~\ref{sec:SpinUnion}).
%Torsion Alexander module aka genus 0 and 1.
%(by Corollary~\ref{cor:SpinAutomatic}).
Combining these facts, the proofs of Theorems~\ref{thm:Injectionv1} and~\ref{thm:InjectionSurf} (and of Theorem~\ref{thm:InjectionSurfClosed} below in the case of closed surfaces) yield the~result.
\end{proof}

\section{Closed surfaces}
\label{sec:Closed}

In this section, we adapt the results from Sections~\ref{sec:ProofMain} and~\ref{sec:Injection} to the setting of closed surfaces.
\subsection{Extending surface homeomorphisms ambiently}

We prove a variation on 
Theorem~\ref{thm:CompatiblePairClosedIntro} from the introduction.
The proof of the statement in the introduction is identical, the only difference is that, there, the homeomorphism~$\theta \colon S_0 \to S_1$ was not fixed: instead, a bundle isomorphism$H$ with the necessary properties was given.

\begin{customthm}{\ref{thm:CompatiblePairClosedIntro}}
\label{thm:CompatiblePairClosed}
Let~$X$ be a closed simply-connected~$4$-manifold and let~$S_0,S_1 \subset X$ be~$\Z_d$-surfaces with the same Euler number that are either both characteristic or both ordinary.
Fix a homeomorphism 
$$ \theta \colon S_0 \to S_1.$$
When the $S_i$ are characteristic,  assume furthermore that $\theta$ preserves the Freedman--Kirby or Guillou--Marin quadratic forms depending on whether the $S_i$ are orientable or not.
The following assertions are equivalent:
\begin{itemize}
\item There is a homeomorphism~$\Theta \colon (X,S_0) \to (X,S_1)$ with $\Theta|_{S_0}=\theta$.
\item There is an isometry $F\colon \lambda_{X_{S_0}} \cong \lambda_{X_{S_1}}$ of the equivariant intersection forms and a bundle isomorphism $H \colon \overline{\nu}(S_0) \to \overline{\nu}(S_1)$ with $H|_{S_0}=\theta$ such that~$(F,H|)$ preserves the relative~$k$-invariants,
$$\partial F=(H|)_* \colon H_1(\partial X_{S_0};\Z[\Z_d]) \to H_1(\partial X_{S_1};\Z[\Z_d]), $$
and,  when $X$ is spin,  $d=1$, and $S_0,S_1$ are nonorientable, $X_{S_0} \cup_{H|} X_{S_1}$ is spin.
\end{itemize}
The~$k$-invariant condition can be omitted in the following circumstances: either~$\pi$ is trivial,  or~$\pi$ is infinite cyclic,  or the $S_i$ are spheres, or the $S_i$ are projective planes. 
The condition involving the quadratic forms is automatic when the $S_i$ are spheres or projective planes. 
\end{customthm}
\begin{proof}
The theorem can either be obtained by repeating all the steps as 
that we carried up to this point or as a formal consequence of Theorem~\ref{thm:CompatiblePair}.
We outline this later approach following~\cite[Proof of Theorem 1.4]{ConwayPowell}.
We assume that $d \neq 0$ since the $d=0$ case is~\cite[Theorem 1.4]{ConwayPowell}.
%%%%
The necessity of the conditions follows as in the case with nonempty boundary: restrict the homeomorphism $(X,S_0) \to (X,S_1)$ to the exteriors and observe that this yields a compatible pair.

We therefore focus on the sufficiency of the conditions.
Since Euler numbers of the $S_i$ agree, we can extend $\theta$ to a bundle isomorphism~$H\colon (\overline{\nu}(S_0),S_0) \to (\overline{\nu}(S_1),S_1)$.
Choose a disc~$D^2 \subset S_0$ and use an ambient isotopy to take~$\theta(D^2)$ to~$D^2$.
After this 
%In particular,
%after an
ambient isotopy, assume that~$S_0$ and $S_1$ coincide on a disc~$D^2 \subseteq S_0 \cap S_1$.
Assume that the normal bundles also coincide over this~$D^2$.
Consider the preimage~$\mathring{D}^2 \times \R^2 \subseteq \nu(S_i)$.
This is homeomorphic to an open~$4$-ball~$\mathring{D}^4$.
Remove this~$(\mathring{D}^4,\mathring{D}^2)$ from~$(X,S_i)$ to obtain~$(N,S_i^\circ)$, with~$\partial S_i^\circ=S_i \cap \partial N$ the unknot~$K$ in~$\partial N=S^3$.
Note that $\theta$ restricts to a homeomorphism rel.\ boundary~$S_0^\circ \to S_1^\circ$.
%Restricting the homeomorphism $\theta$ to $S_0^\circ$ yields a rel.\ boundary homeomorphism $\theta| \colon S_0^\circ \to S_1^\circ$. 
The exterior of~$S_i \subset X$ equals the exterior of~$S_i^\circ \subset N$.
Since~$X_{S_i}=N_{S_i^\circ}$, the equivariant intersection forms and boundaries are unchanged,  and the compatible pair~$(F,H|)$ therefore also induces a compatible pair~$(F,H|) \colon N_{S_0^\circ} \to N_{S_1^\circ}$.
Theorem~\ref{thm:CompatiblePair} provides a rel.\ boundary homeomorphism of pairs~$\Phi' \colon (N,S_0^\circ) \to (N,S_1^\circ)$ that extends~$\theta|_{S^\circ}$.
We recover the required homeomorphism of pairs~$\Phi \colon (X,S_0) \to (X,S_1)$ by gluing~$\Phi'$ with the identity homeomorphism~$(D^4,D^2) \to (D^4,D^2)$.
This homeomorphism extends $\theta$ by construction.
The fact that the~$k$-invariants can be omitted for Euler number $0$ spheres follows from Proposition~\ref{prop:NokInvariant}.
The final sentence of the theorem follows from the fact that in those cases $H_1(S_i)$ is either $\Z$ or $\Z_2$.
%So automatically quad form preserving.
\end{proof}

\subsection{Nice identifications}
\label{sub:NiceIdentificationsClosed}
In order to extend the results from Section~\ref{sec:Injection} (and specifically Theorems~\ref{thm:Injectionv1} and~\ref{thm:InjectionSurf}) to closed surfaces, we define~$e$-nice identifications in the closed setting.
Given a closed surface~$\Sigma$, fix a model~$\xi=\left( \Sigma \mathbin{\wt{\times}}_e \R^2 \to \Sigma\right)$ for the Euler number~$e$ rank~$2$ vector bundle with~$w_1(\xi)=w_1(\Sigma)$, and fix a section~$r \colon \Sigma^{(1)} \to \Sigma \mathbin{\wt{\times}}_e \R^2$ over its~$1$-skeleton.

	\begin{definition}
		\label{def:NiceFraming-orClosed}
Given a closed surface~$S \subset X$ with normal Euler number~$e$,  a bundle isomorphism~$\fr\colon \nu(S) \xrightarrow{\cong} \Sigma \mathbin{\wt{\times}}_e \R^2$ is a~\emph{nice identification} if pushing off any loop~$\gamma \subset S$ using~$\fr^{-1}\circ r$ results in a curve that is nullhomologous in~$X_S$.
We refer to~$\fr^{-1}|\circ r \colon \Sigma^{(1)} \to \nu(S)$ as a \emph{nice section}.
	\end{definition}
	
\begin{lemma}
\label{lem:ExistseNiceClosed}
A surface~$S \subset X$ whose relative Euler number is~$e$ admits an~$e$-nice identification.
\end{lemma}
\begin{proof}
Choose a bundle isomorphism~$\fr \colon \nu(S) \cong \Sigma \mathbin{\wt{\times}}_e \R^2$.
Delete a disc from the surface $S$ and its image in the model in order to obtain a bundle isomorphism~$\fr \colon \overline{\nu}(S^\circ) \cong \Sigma^\circ \mathbin{\wt{\times}}\R^2$, where $\Sigma^\circ \mathbin{\wt{\times}}\R^2$ denotes our model rank $2$ vector bundle over the surface with one boundary component.
%%%Don't delete.
%%{AC: Technically it's only isomorphic to it. Is that problematic?}
%%%No: could build the model so that these things agree on the nose.
As in the proof of Lemma~\ref{lem:ExistseNice-or}, postcompose this bundle isomorphism with an appropriate~$\Rot(\alpha)$ in order to obtain an~$e$-nice framing $\overline{\nu}(S^\circ) \cong \Sigma^\circ \mathbin{\wt{\times}}\R^2$ in the sense of Definition~\ref{def:NiceFraming-or}.
Since $\Rot(\alpha)$ restricts to the identity on the boundary,  we can then cap off this new isomorphism in order to obtain the required $e$-nice identification.
\end{proof}

\subsection{The automorphism invariants for closed surfaces}
\label{sub:AutomorphismClosed}

Throughout this section, ~$\Sigma^\circ$ will denote a surface with one boundary component and~$\Sigma:=\Sigma^\circ \cup D^2$ will refer to the closed surface obtained by capping off~$\Sigma^\circ$ with a disc.
\begin{construction}
\label{cons:Capoff}
Let~$X$ be a closed~$4$-manifold,  fix a~$4$-ball in~$B \subset X$,  and set~$X^\circ:=X \setminus B^\circ$.
Fix a embedding~$\iota_D \colon D^2 \hookrightarrow B$ of a small standardly properly embedded disc and consider the unknot~$U:=\partial \iota_D(D^2) \subset \partial X=S^3$.
Continuing with the notation from Section~\ref{sec:Injection}, consider the map that assigns to an embedding~$\iota \colon \Sigma^\circ \to X^\circ$ with~$\partial \iota(\Sigma^\circ)=U$ the embedding~$\Sigma \hookrightarrow X$ obtained by capping off~$\iota$ with~$\iota_D$:
$$
\operatorname{capoff} \colon \Emb(X^\circ,U) \to \Emb(X).
$$
%%Don't delete
%%Well def because RHS is rel.\ boundary
This map is surjective: isotope a given embedding~$\widehat{\Sigma} \hookrightarrow X$ until a small ball matches~$B$ and then puncture the outcome to obtain an embedding~$\Sigma^\circ \hookrightarrow X^\circ$ with~$\partial \iota(\Sigma^\circ)=U$.

Capping off an embedding does not affect the genus, orientation character, divisibility, Euler number,  and equivariant intersection form,  leading to a surjective map
$$
\operatorname{capoff} \colon \Emb_{d,e,\lambda}^w(g)(X^\circ,U) \to  \Emb_{d,e,\lambda}^w(g)(X).
$$
\end{construction}

Section~\ref{sec:Injection} constructs an injective map~$\Xi_{\Emb}$ from~$ \Emb_{d,e,\lambda}^w(g)(X^\circ,U)$ into an algebraically defined set; recall Theorem~\ref{thm:Injectionv1}.
For brevity, we also use the following notation for the ``numerator" of the target of this map:
% the targets of the maps in Theorem~\ref{thm:Injectionv1}:
$$
\mathcal{I}:=\begin{cases}
\Iso_{\Z[\Z_d]}^{\ell k}(\coker(\Ad Q),H_1(Y_e;\Z[\Z_d]))
\times H^3(\Z_d,Y_e;H) \times\Spin^*(Y_e)
  \quad& \text{or,} \\
 %%%
\Iso_{\Z[\Z_d]}^{\ell k}(\coker(\Ad Q),H_1(Y_e;\Z[\Z_d]))
\times H^3(\Z_d,Y_e;H).  \quad & \text{}
 \end{cases}
$$
 Here,  note that~$Y_e(U)=E_U \cup (\Sigma^\circ \mathbin{\wt{\times}} S^1)$ agrees with the Euler number~$e$ circle bundle~$Y_e$ over~$\widehat{\Sigma}$, thus explaining we write~$Y_e$ in place of~$Y_e(U)$.

\begin{construction}
We argue that repeating all the steps from Section~\ref{sec:Injection} produces a map 
$$
\Xi_{\Emb}^{\closed} \colon \Emb_{d,e,\lambda}^w(g)(X)
\to
\frac{\mathcal{I}}{{\Aut_{\Z[\Z_d]}(Q) \times d \cdot H^1(\Sigma;\Z^w)}}.
$$
Indeed all the steps from Construction~\ref{cons:biota}, Proposition~\ref{prop:ChangeIdentification}, and Proposition~\ref{prop:bWellDef} carry through without change using the definition of an~$e$-nice framing from Section~\ref{sub:NiceIdentificationsClosed} and noting that, in the closed case, every bundle isomorphism is still of the form~$\Rot(\alpha)$ thanks to the combination of Propositions~\ref{prop:cohomology-version} and~\ref{prop:BundleIsoHE} (these results apply both in the orientable and nonorientable settings).
In more detail, we apply the combination of these propositions with $(B,B')=(\Sigma,\emptyset),E=\Sigma \mathbin{\wt{\times}}_e \R^2$ (so that~$P/SO(2)=\wh B$) to obtain
\begin{equation}
\label{eq:AutBundleClosed}
 \pi_0(\Aut^+(\Sigma \mathbin{\wt{\times}}_e \R^2))
\cong 
\pi_0(\Gamma(\Sigma,\wh \Sigma\times_{\Z_2}SO(2)))
\cong  H^1(\Sigma;\Z^w).
\end{equation}
\end{construction}

\begin{construction}
\label{cons:DiagramClosedEmbedding}
Combining the map $\Xi_{\Emb}^{\closed} $ with the capping off map from Construction~\ref{cons:Capoff}, we obtain the following diagram:
\begin{equation}
\label{eq:CapOffDiagram}
\xymatrix{
\Emb_{d,e,\lambda}^w(g)(X^\circ,U) \ar@{->>}[r]^{\operatorname{capoff}}\ar@{^{(}->}[d]^{\Xi_{\Emb}}& \Emb_{d,e,\lambda}^w(g)(X)\ar[d] \\
\frac{\mathcal{I}}{{\Aut_{\Z[\Z_d]}(Q) \times d \cdot H^1(\Sigma^\circ,\partial \Sigma^\circ;\Z^w)}} \ar@{-->}[r]&
\frac{\mathcal{I}}{{\Aut_{\Z[\Z_d]}(Q) \times d \cdot H^1(\Sigma;\Z^w)}}.
}
\end{equation}
We argue that the dashed map exists, that it is a bijection and that it makes the diagram commute.
Since the action of~$H^1(\Sigma^\circ,\partial \Sigma^\circ;\Z^w) \cong H^1(\Sigma,D^2;\Z^w)$ on $\mathcal{I}$ factors through an action of~$H^1(\Sigma;\Z^w)$,  the identity $\id_{\mathcal{I}}$ induces the dashed map; the latter is necessarily surjective and makes the diagram commute.
Since~$H^1(\Sigma^\circ,\partial \Sigma^\circ;\Z^w) \cong H^1(\Sigma,D^2;\Z^w)$ surjects onto $H^1(\Sigma;\Z^w)$, this map is also injective and therefore an bijection.
\end{construction}

\begin{proposition}
\label{prop:InjectionClosedEmbedding}
The capoff map defines a bijection
$$
\operatorname{capoff} \colon \Emb_{d,e,\lambda}^w(g)(X^\circ,U) \xrightarrow{\cong}  \Emb_{d,e,\lambda}^w(g)(X)
$$
\end{proposition}
\begin{proof}
Consider the commutative diagram displayed in~\eqref{eq:CapOffDiagram}.
Since the cap off map is a surjection (recall Construction~\ref{cons:Capoff}),  the left map is injective (by Theorem~\ref{thm:Injectionv1}) and the bottom map is a bijection, it follows that the capoff map is also injective and thus a bijection.
\end{proof}

As in Section~\ref{sec:Injection}, it remains to pass from embeddings to embedded surfaces.
For this, we need to consider the action of surface homeomorphisms on the sets that feature in~\eqref{eq:CapOffDiagram}.
\begin{construction}
We define an action of $\Homeo(\Sigma)$ on~$\frac{\mathcal{I}}{{\Aut_{\Z[\Z_d]}(Q) \times d \cdot H^1(\Sigma;\Z^w)}}$.
Given a surface homeomorphism $\theta \colon \Sigma \to \Sigma$,  since $\theta^*(\Sigma \mathbin{\wt{\times}}_e \R^2)$ has the same $w_1$ and Euler number as $\Sigma \mathbin{\wt{\times}}_e \R^2$,  there is a bundle isomorphism $\theta^*(\Sigma \mathbin{\wt{\times}}_e \R^2) \cong \Sigma \mathbin{\wt{\times}}_e \R^2$ covering the identity.
It follows that $\theta$ is covered by a bundle isomorphism 
$$\theta_{bundle} \colon \Sigma \mathbin{\wt{\times}}_e \R^2 \to \Sigma \mathbin{\wt{\times}}_e \R^2.$$
This bundle isomorphism is nonunique but Propositions~\ref{prop:cohomology-version} and~\ref{prop:BundleIsoHE} from the appendix (recall~\eqref{eq:AutBundleClosed}) ensure that any two isomorphism differ by an element of~$\Aut(\Sigma \mathbin{\wt{\times}}_e \R^2)\cong H^1(\Sigma;\Z^w)$. 
Thus the group~$\Homeo(\Sigma)$ nevertheless induces an action on~$\frac{\mathcal{I}}{{\Aut_{\Z[\Z_d]}(Q) \times d \cdot H^1(\Sigma;\Z^w)}}.$
\end{construction}

We now prove Theorem~\ref{thm:InjectionSurfIntro} from the introduction.
In order to ease notation, we use the abbreviations~$\Iso:=\Iso_{\Z[\Z_d]}^{\ell k}(\coker(\Ad Q),H_1(Y_e;\Z[\Z_d])), H^3:=H^3(\Z_d,Y_e;H)$ and~$\Spin^{*}:=\Spin^{*}(Y_e)$.

\begin{customthm}{\ref{thm:InjectionSurfIntro}}
\label{thm:InjectionSurfClosed}
Let~$X$ be a closed simply-connected~$4$-manifold,  let~$d>0$ be an integer,  let $e \in \Z$ be another integer, and let $\lambda \colon H \times H \to \Z[\Z_d]$ be a hermitian form over $\Z[\Z_d]$.

The cap off map induces a bijection
$$
\operatorname{capoff} \colon \Surf_{d,e,\lambda}^w(g)(X^\circ,U)
\xrightarrow{\cong}
\Surf_{d,e,\lambda}^w(g)(X).
$$
In the characteristic case,  the map $\Xi_{\Emb}^{\closed}$ induces a injection
$$
\Xi_{\Surf}^{\closed} \colon \Surf_{d,e,\lambda}^{w,\charac}(g)(X) 
\hookrightarrow
\begin{cases}
\frac{\Iso
\times H^3 \times\Spin^{\operatorname{not-ext}}
 }{\Aut(Q) \times d \cdot H^1(\Sigma;\Z^w)}\Big/\Homeo(\Sigma) \quad& \text{if $X$ is not spin,} \\
 %%%
 \frac{\Iso
\times H^3 }{\Aut(Q) \times d \cdot H^1(\Sigma;\Z^w)}\Big/\Homeo(\Sigma)  \quad & \text{if $X$ is spin.}
 \end{cases}
$$
In the ordinary case,  the map $\Xi_{\Emb}^{\closed}$ induces a injection
$$
\Xi_{\Surf}^{\closed}\colon \Surf_{d,e,\lambda}^{w,\ord}(g)(X) 
\hookrightarrow
\begin{cases}
\frac{\Iso
\times H^3 }
{\Aut(Q) \times d \cdot H^1(\Sigma;\Z^w)}\Big/\Homeo(\Sigma) \quad& \text{if $X$ is not spin,} \\
 %%%
 \frac{\Iso
\times H^3  \times\Spin^{\operatorname{ext}}}{\Aut(Q) \times d \cdot H^1(\Sigma;\Z^w)}\Big/\Homeo(\Sigma)  \quad & \text{if $X$ is spin.}
 \end{cases}
$$
When the $k$-invariant condition is automatic, the $H^3$ factors can be omitted.
The action of the group~$d \cdot H^1(\Sigma;\Z^w)$ on spin structures is trivial in the characteristic case and nontrivial in the ordinary case.
%$$
%\Xi_{\Surf}^{\closed} \Surf_{d,e,\lambda}^w(g)(X)
% \hookrightarrow
% \left(\frac{\mathcal{I}}{{\Aut_{\Z[\Z_d]}(Q) \times d \cdot H^1(\Sigma;\Z^w)}}\right)/\Homeo(\Sigma).
%$$
\end{customthm}
\begin{proof}
Consider the following diagram:
$$
\xymatrix{
\Surf_{d,e,\lambda}^w(g)(X^\circ,U) \ar@{->>}[r]^-{\operatorname{capoff},\cong}\ar@{^{(}->}[d]& \frac{\Emb_{d,e,\lambda}^w(g)(X)}{\Homeo(\Sigma,D)} \ar@{^{(}->}[d] \ar@{-->}[r]^{\cong}&
\frac{\Emb_{d,e,\lambda}^w(g)(X)}{\Homeo(\wh  \Sigma)}\ar[d]^{\Xi_{\Surf}^{\closed}}
\\
%%%%%%%
\frac{\mathcal{I}}{Aut_{\Z[\Z_d]}(Q) \times (d \cdot H^1(\Sigma^\circ,\partial \Sigma^\circ;\Z^w) \rtimes \Homeo_\partial(\Sigma))} \ar[r]^-\cong&
\frac{\frac{\mathcal{I}}{{\Aut_{\Z[\Z_d]}(Q) \times d \cdot H^1(\Sigma;\Z^w)} }}{ \Homeo(\Sigma,D)}\ar@{-->}[r]^{\cong}&
\frac{\frac{\mathcal{I}}{{\Aut_{\Z[\Z_d]}(Q) \times d \cdot H^1(\Sigma;\Z^w)} }}{ \Homeo(\Sigma)}.
}
$$
The left hand square commutes because it is obtained by taking the quotient of the diagram in~\eqref{eq:CapOffDiagram} by the~$\Homeo(\Sigma^\circ,\partial \Sigma^\circ)\cong \Homeo(\Sigma,D^2)$ actions.
For the existence of the dashed horizontal bijections,  the argument is similar to the one from Construction~\ref{cons:DiagramClosedEmbedding}: 
Since the action of~$\Homeo(\Sigma^\circ,\partial \Sigma^\circ)\cong \Homeo(\Sigma,D^2)$ factors through an action of~$\Homeo(\Sigma)$,  the identity $\id_{\mathcal{I}}$ induces the (necessarily surjective) dashed maps that then necessarily make the diagram commute.
Since~$\Homeo(\Sigma^\circ,\partial \Sigma^\circ)\cong \Homeo(\Sigma,D^2)$ surjects onto $\Homeo(\Sigma)$,  these maps are also injective and therefore are bijections.
The fact that $\Xi_{\Surf}^{\closed}$ is injective now follows immediately.
\end{proof}

\section{Simple discs and spheres}
\label{sec:DiscsSpheres}

This section applies our results to discs and spheres.
Section~\ref{sub:Cancellation} combines Theorem~\ref{thm:CompatiblePair} with cancellation theorems of Lee-Wilczynski~\cite{LeeWilczy,LeeWilczyOdd} yielding a criterion for certain simple discs to be equivalent rel.\ boundary.
Section~\ref{sub:DiscsSpheresNoCancellation} applies the injection from Theorem~\ref{thm:InjectionSurf} to settings where cancellation is not available.

\subsection{Cancellation methods}
\label{sub:Cancellation}

We first adapt a result of Lee-Wilczynski's from spheres to discs.
\begin{theorem}
\label{prop:PointedFormsDiscs}
Let $d>0$,  let $K$ be a knot with $H_1(\Sigma_d(K))=0$ and let~$D_0,D_1$ be simple discs with boundary $K$ that represent a divisibility $d$ homology class $x \in H_2(X,\partial X)$.
If~$d=1$ or if~$b_2(X)>6$ and 
\begin{equation}
\label{eq:LWCancellationLTSignature}
b_2(X)> \max_{0\leq j<d}\Big|\sigma (X)-\frac{2j(d-j)}{d^2}\, x\cdot x +\sigma_K (e^{\frac{2\pi i j}{d}})\Big|,
\tag{$\star$}
\end{equation}
then there is a pointed isometry $\lambda_{\Sigma_d(S_0)} \cong \lambda_{\Sigma_d(S_1)}$. 
\end{theorem}
\begin{proof}
We assert that the discs become equivalent rel.\ boundary after taking the connected sum of $X$ with sufficiently many copies of $S^2\times S^2$.
This follows by adapting~\cite[Uniqueness portion of Proposition 2.1]{LeeWilczy} (see also~\cite[Theorem 2.5]{LeeWilczyOdd}) from spheres to discs as we now explain.
We assume some familiarity with Kreck's modified surgery theory~\cite{KreckSurgeryAndDuality}.
Let~$(B,\xi)$ denote the normal~$1$-type of the disc exteriors.
Remark~\ref{rem:Normal1TypeUnion} shows that there is a bundle isomorphism~$H \colon \overline{\nu}(S_0) \to \overline{\nu}(S_1)$ and normal $1$-smoothings $X_{S_i} \to B$ for $i=0,1$ that extend to~$X_{S_0} \cup_h X_{S_1} \to B$.
Here recall that $h:=\id_{E_K} \cup H|$.
Since the disc exteriors have the same signature and Kirby--Siebenmann invariant,  an Atiyah-Hirzebruch spectral sequence argument shows that~$X_{S_0} \cup_h X_{S_1}$ is zero bordant in~$\Omega_4(B,\xi)$ where~$(B,\xi)$ denotes the normal~$1$-type of the~$X_{S_i}$; this follows e.g. by combining~\cite[Proposition 1.1]{HambletonHambletonSmoothStructures} with~\cite[Theorem 2]{KreckSurgeryAndDuality}.
%%Perhaps Teichner's thesis has the statement written more explicit dunno.
Kreck's modified surgery theory then implies that~$h$ extends stably to a homeomorphism~\cite[Corollary~3]{KreckSurgeryAndDuality}.
Extending this stable homeomorphism over the disc bundles then proves the assertion.

The pointed isometry is then obtained by following the steps explained in~\cite[Section~3]{LeeWilczy}.
In a nutshell, the verifications from~\cite{LeeWilczy} go through because, since $H_1(\Sigma_d(K))=0$,  the (equivariant) intersection form of the branched cover is nonsingular, just as for spheres; here are some details.
The idea is to verify that the pointed hermitian forms satisfy~$(1)-(4)$ of~\cite[Theorem 6.1]{LeeWilczy}, as well as~\cite[(6.2) and~(6.3)]{LeeWilczy} and conditions~$(7),(8)$ and~$(9)$ of~\cite[Theorem 7.1]{LeeWilczy}.
As explained in~\cite[argument leading up to~(3.5)]{LeeWilczy}, the existence of the required pointed isometry is then known to follow.
The first four conditions and~(6.2) were verified in~\cite[Sections 3-6]{ConwayOrsonPencovitch}.
Condition~(6.3) is a consequence of the stable equivalence and~$(9)$ is equivalent to the inequality involving the Levine--Tristram signature from~\eqref{eq:LWCancellationLTSignature}; see e.g.~\cite[Section 3.4]{ConwayOrsonPencovitch}.
It remains to verify that~(7) and~(8) hold; Lee-Wilczynski's verifications for spheres in~\cite[Proof of Theorem~1.1, page~365]{LeeWilczy} are purely algebraic, apart from their invocation of  of~\cite[Lemma~4.2]{LeeWilczy} which can be replaced by~\cite[Lemma 6.2]{ConwayOrsonPencovitch}.
\end{proof}

We now prove Theorem~\ref{thm:PointedFormsDiscsIntro} from the introduction.

\begin{customthm}{\ref{thm:PointedFormsDiscsIntro}}
\label{thm:Discs}
Let $d>0$,
%%See the max.
let $K$ be a knot with $H_1(\Sigma_d(K))=0$ and let~$D_0,D_1$ be simple discs with boundary $K$,
nonzero Euler number,  and the same type.
%that represent a divisibility~$d$ homology class~$x \in H_2(X,\partial X)$.
The following assertions are equivalent:
\begin{enumerate}
\item the discs~$D_0$ and~$D_1$ are equivalent rel.\ boundary,
\item there is a pointed isometry~$\lambda_{\Sigma_d(D_0)} \cong \lambda_{\Sigma_d(D_1)}$.
\end{enumerate}
When the discs represent a divisibility~$d$ homology class~$x \in H_2(X,\partial X)$,
a pointed isometry exists if~$d=1$ or if~$b_2(X)>6$ and 
$$b_2(X)> \max_{0\leq j<d}\Big|\sigma (X)-\frac{2j(d-j)}{d^2}\, x\cdot x +\sigma_K (e^{\frac{2\pi i j}{d}})\Big|.$$
\end{customthm}
\begin{proof}
If the discs are equivalent rel.\ boundary, then the pointed hermitian forms of the branched covers are certainly isometric.
We therefore focus on the converse.
Proposition~\ref{prop:LeverageLW} shows that the existence of a pointed isometry~$\lambda_{\Sigma_d(S_0)} \cong \lambda_{\Sigma_d(S_1)}$ ensures the existence of an extendable compatible pair between the disc exteriors.
Theorem~\ref{thm:CompatiblePair}, now ensures that the discs are equivalent rel.\ boundary.
In the case of discs,  the $k$-invariant condition is not needed thanks to the Euler number assumption.
The final sentence follows from Theorem~\ref{prop:PointedFormsDiscs}.
\end{proof}

We refer to~\cite[Addendum 1]{LeeWilczy} for a discussion of situations in which cancellation applies to manifolds with~$3 \leq b_2(X) \leq 6$.
%Theorem~\ref{thm:Discs} also applies in those settings.
%%Technically one would have to check the verifications for spheres go through as well. 

We conclude this section with a remark concerning surfaces of higher genus.
\begin{remark}
\label{rem:Sunukjian}
When~$b_2(X)>|\sigma(X)|+2$,  Sunukjian~\cite[Theorem 7.4]{Sunukjian} proves that homologous~$\Z_d$-surfaces of positive genus $g$ are equivalent provided 
$$
b_2(X)+2g > \operatorname{max}_{0 \leq j <d}
\Big|\sigma(X)-\frac{2j(d-j)}{d^2}Q_X(x,x)\Big|.
$$
Thanks to work of Lee and Wilczynski~\cite{LeeWilczyGenus}, this condition can be reformulated as requiring that the surfaces have nonminimal genus.

We relate Sunukjian's work to our own.
The argument below closely follows~\cite[Proof of Theorem 7.4]{Sunukjian}.
Since $X$ is simply-connected and~$b_2(X)>|\sigma(X)|+2$, the work of Freedman~\cite{Freedman} produces a homeomorphism~$X=X' \# S^2 \times S^2$, Sunukjian's assumptions together with work of Wall~\cite{WallDiffeomorphisms} and the result of Lee and Wilczynski are seen to imply that, without loss  of generality, one of the surfaces can be assumed to lie in $X'$, so that the equivariant intersection form of its exterior splits off a hyperbolic.
Since the surfaces are externally stably equivalent (by the same normal~$1$-type argument as in Theorem~\ref{thm:PointedFormsDiscsIntro}),  the equivariant intersection forms are stably isometric.
Hambleton and Kreck's results on cancellation~\cite[Theorem A]{HambletonKreckCancellation} therefore imply that these forms are isometric.
Thus, if one could control the relative~$k$-invariants, then Sunukjian's result
% (and a generalisation to surfaces with boundary)
%%No because there is no LW with boundary.
would follow from Theorem~\ref{thm:CompatiblePairClosed}.

Sunukjian asserts that the surfaces can be taken to be isotopic but we believe that the application of Wall's~\cite{WallDiffeomorphisms} only leads to an equivalence.
It is conceivable that Sunukjian's techniques could be extended further using recent results on concordance of surfaces with boundary~\cite{Hellsten}.
\end{remark}

\subsection{Discs and spheres, without cancellation}
\label{sub:DiscsSpheresNoCancellation}

The work of Lee-Wilczynski~\cite{LeeWilczy,LeeWilczyOdd} and Theorem~\ref{thm:Discs} provide settings in which cancellation results yield uniqueness results for spheres and discs.
The next theorem (which is Theorem~\ref{thm:SpheresDiscsNoCancellation} from the introduction) considers settings in which these cancellation results are not available; the outcome is that in many cases the equivariant intersection form of the surface exteriors is a complete invariant.
The case~$d=0$ is excluded since it was already treated in~\cite{ConwayPowellDiscs,ConwayPowell}.

\begin{customthm}{\ref{thm:SpheresDiscsNoCancellationIntro}}
\label{thm:SpheresDiscsNoCancellation}
%%Phrasing it with homology classes to compare with LW.
Let~$d\neq 0$ and~$e \neq 0$ be integers such that~$n=e/d$ is an odd prime power or~$1,2,4,2p^k$ with~$p$ an odd prime power.
\begin{itemize}
\item Given a closed simply-connected~$4$-manifold~$X$,  two~$\Z_d$-spheres in~$X$ with Euler number $e$ and the same type are equivalent if and only if their exteriors have isometric equivariant intersection forms.
\item Given a simply-connected~$4$-manifold~$X$ with~$\partial X = S^3$ and a knot~$K \subset \partial X = S^3$ with~$H_1(\Sigma_d(K))=0$,  two~$\Z_d$-discs in~$X$ with boundary~$K$,  Euler number $e$, and the same type are equivalent rel.\ boundary if and only if their exteriors have isometric equivariant intersection forms.
\end{itemize}
\end{customthm}
\begin{proof}
Fix a hermitian form~$(H,\lambda)$ over~$\Z[\Z_d]$.
Abbreviate the abstract manifold homeomorphic to the boundary of the relevant surface exterior by~$Y$.
As explained in Corollary~\ref{cor:Section8DiscsSpheres},
in genus~$g=0$,  allowing for the knot $K$ to be empty in the closed case, the injective map from Theorem~\ref{thm:InjectionSurf} therefore takes the form
$$
\Xi \colon \Surf_{x,\lambda}^e(g)(X,K) \hookrightarrow \frac{\Iso_{\Z[\Z_d]}(\partial Q,-\ell_{\widehat{Y}})}{\Aut_{\Z[\Z_d]}(Q)}.
%No Homeo(S^2) action because H_1(S^2)=0 or whatever
$$
%Here, we used that when $d$ odd and the surface exteriors are spin, $\Spin(Y_e(K))$ is a singleton.
%To see this, note that $H^1(Y_e(K);\Z_2)=0$ thanks for Proposition~\ref{prop:AlexanderModule} together with the fact that~$n=e/d$ is assumed to be odd.
For spheres,  $Y=L(n,1)$, so the linking form is isometric to~$(x,y) \mapsto -\frac{1}{n}xy.$
When~$n$ is an odd prime power or~$1,2,4,2p^k$ with $p$ an odd prime power,  the only isometries of this linking form are multiplication by~$\pm 1$.
Since multiplication by~$\pm 1$ is an isometry of any hermitian form,
%Since~$\Aut_{\Z[\Z_d]}(Q)$ always contains~$\{ \pm 1\}$, 
the orbit set is trivial whenever it is nonempty.

Next, we consider discs whose boundary~$K$ satisfies~$H_1(\Sigma_d(K))=0$.
Proposition~\ref{prop:AlexanderModule} ensures that~$H_1(Y;\Z[\Z_d]) \cong \Z_n$ is generated by a lifted meridian~$\widetilde{\mu}$.
We describe how to calculate~$\ell_{\widehat{Y}}(\widetilde{\mu},\widetilde{\mu}) \in \Q/\Z$.
First, observe that~$n\widetilde{\mu}$ is homologous to $\widetilde{\lambda}_K+n\widetilde{\mu}$, where $\widetilde{\lambda}_K$ denotes a lifted Seifert longitude of $K$.
These curves are nullhomologous and, explicitly, the latter bounds the disc~$D$.
It follows that, possibly up to a sign,~$\ell_{\widehat{Y}}(\widetilde{\mu},\widetilde{\mu})=\frac{1}{n}(\widetilde{\mu} \cdot D)=\frac{1}{n}.$
The remainder of the argument follows as in the closed case.
\end{proof}

\begin{remark}
We collect some remarks surrounding Theorem~\ref{thm:SpheresDiscsNoCancellation}.
\begin{itemize}
\item When cancellation applies, Theorem~\ref{thm:SpheresDiscsNoCancellation} is not a new result since, in those cases,  Lee-Wilczynski already proved that homologous simple spheres are isotopic~\cite{LeeWilczyOdd,LeeWilczy}.
Out of the cancellation range,  Theorem~\ref{thm:SpheresDiscsNoCancellation} can be viewed as an improvement on Lee-Wilczynski's theorem that homologous simple
%Don't  delete:I left in homologous because, really, they don't have the refined statement with only e and type.
 spheres are determined by the pointed isometry type of the pointed equivariant intersection form of their branched cover.
For example,  when~$X=\C P^2$ and~$d=2$, 
Theorem~\ref{thm:SpheresDiscsNoCancellation} implies that~$\Z_2$-spheres in $\C P^2$ are equivalent (in fact isotopic~\cite{ConwayOrson}) because the equivariant intersection forms of such sphere exteriors are necessarily isometric to $\Z_- \times \Z_- \to \Z[\Z_2],(x,y) \mapsto -2(1-T)x\overline{y}$; see~\cite[Corollary~2.4]{ConwayOrson} for the argument.
This result is not immediate from the methods of Lee-Wilczynski; a lengthier discussion can be found in~\cite[Example 2.3]{ConwaySimpleSpheres}.
\item We conjecture that the second item of Theorem~\ref{thm:SpheresDiscsNoCancellation} can be generalised as follows:
there are at most~$|\Aut(\ell_{\Sigma_d(K)})/\Aut(\lambda)|$ many
%%Intentionally wrote \lambda (no \ZZ_d subscript) instead of Q because easier to write below.
 equivalence rel.\ boundary classes of~$\Z_d$-discs~$D \subset~X$ with the same boundary~$K$,  
 the same type,  relative Euler number~$e \neq 0$, and 
equivariant intersection form $\lambda$.
A helpful first step towards proving this conjecture would be to show that the linking form on the $d$-fold cover of~$Y_e(K)$ splits as the linking form of a lens space with the linking form on $\Sigma_d(K)$.
%%Don't delete.
%%Careful ZZ_d splitting on the module level.
Note the contrast between this conjecture and the fact that there are knots which bound infinitely many $\Z$-discs in $\C P^2$ with the same equivariant intersection form~\cite{ConwayDaiMiller}.
 \end{itemize}
 \end{remark}

\section{Simple Moebius and projective planes}
\label{sec:Nonorientable}

This section is concerned with simple projective planes and Moebius bands in simply-connected~$4$-manifolds. 
Section~\ref{sub:RP2} shows how our methods promptly recover Lawson's result that~$\Z_2$-projective planes in the~$4$-sphere are unknotted, whereas Section~\ref{sub:Moeb} is concerned with the enumeration of properly embedded Moebius bands in the~$4$-ball.

\subsection{Simple projective planes}
\label{sub:RP2}

We prove Theorem~\ref{thm:ProjectivePlanesIntro} from the introduction; here recall that we say that two surfaces have the \emph{same type} if they are either both characteristic or both ordinary.
As mentioned then, this result can be viewed as a generalisation of Lawson's result that~$\Z_2$-projective planes in the~$4$-sphere are unknotted.

\begin{customthm}{\ref{thm:ProjectivePlanesIntro}}
\label{thm:ProjectivePlanes}
Let $X$ be a closed simply-connected~$4$-manifold.
% and let $e \in \Z$ be an integer.
\begin{itemize}
\item Two projective planes in $X$ with simply-connected complements and the same odd Euler number and type are equivalent if and only if their exteriors have isometric intersection forms.
The intersection forms are isometric if the projective planes are homologous.
\item Two $\Z_2$-projective planes in $X$ with the same Euler number are equivalent if and only if their exteriors have isometric equivariant intersection forms .
\end{itemize}
The same assertions hold for Moebius bands with boundary a knot $K \subset S^3$ satisfying~$\det(K)=1$.
\end{customthm}
\begin{proof}
The arguments in Section~\ref{sub:AutomorphismClosed} (and specifically Theorem~\ref{thm:InjectionSurfClosed}) show that the result for projective planes reduces to the case of Moebius bands with unknotted boundary.
We therefore focus on the case of Moebius bands with boundary a knot~$K \subset S^3$ satisfying~$\det(K)=1$.
%We focus on projective planes as the arguments are identical in the case of Moebius bands with boundary a knot whose trivial determinant is trivial.

We first claim that once the Euler number, knot group and type are fixed, the Moebius bands are determined by the equivariant intersection form of their exteriors, i.e.\ that an extendable compatible pair necessarily exists.
We do this by arguing that the group of isometries~$\Iso_{\Z[\Z_d]}(\partial Q;-\ell_{\widehat{Y}_e(K)})$ contains at most two elements, namely multiplication by~$\pm 1$,
and then considering the actions by~$\Aut_{\Z[\Z_d]}(Q)$ and~$H^1(\Sigma,\partial \Sigma;\Z^w)$.
Once we show that the orbit set is trivial, the conclusion will then follow from Corollary~\ref{cor:Section8DiscsSpheres}.

%It then follows that for any bundle isomorphism~$H \colon (\overline{\nu}(S_0),S_0) \to (\overline{\nu}(S_1),S_1)$, the homeomorphism~$H|$ will fit in a compatible~pair.
%%Can then negate the isometry of the form.

When the knot group is trivial and the Euler number is odd
% (i.e.\ when~$\sigma(X)$ is odd) 
or when the knot group is order~$2$ and the Euler number is of the form~$e=2n$ with~$n$ odd,  Proposition~\ref{prop:AlexanderModuleNonOri} implies that the Alexander module is isomorphic to~$\Z_4$.
A rapid verification shows that the only linking forms on~$\Z_4$ are~$(x,y) \mapsto xy/4$ and~$(x,y)\mapsto -xy/4$.
In both cases, the group of isometries consists of two elements, namely multiplication by~$\pm 1$.
It follows that~$\Iso(\partial Q;-\ell_{\widehat{Y}_e(K)})/\Aut(Q)$ is trivial and the conclusion now follows from Corollary~\ref{cor:Section8DiscsSpheres}.

When the knot group is order $2$ and $e=2n$ with $n$ even,  Proposition~\ref{prop:AlexanderModuleNonOri} implies that the Alexander module is isomorphic to~$\Z_2\langle \widetilde{\mu} \rangle \oplus \Z_2\langle \gamma \rangle$ where~$w_1(\gamma)=-1$,  and that the module structure is determined by~$T\wt {\mu}=\wt{\mu}$ and $T\gamma=\gamma+\widetilde{\mu}$.
It follows that $|\Aut_{\Z[\Z_2]}(\Z_2 \oplus \Z_2)|=2$. %, as required.
Corollary~\ref{cor:RotOnYe(K)Covers} shows that~$\Rot(2\alpha)$ acts by $\gamma \mapsto \gamma +\alpha(\gamma)\widetilde{\mu}=\gamma +\widetilde{\mu}$.
It follows that~$\Iso(\partial Q;-\ell_{\widehat{Y}_e(K)})/H^1(\Sigma,\partial \Sigma;\Z^w)$ is trivial and the conclusion now follows from Corollary~\ref{cor:Section8DiscsSpheres}.
%(0,1)
This concludes the proof of the claim.
%T(1,0)=(1,0), T(0,1)=(1,1), T(1,1)=(1,0)+(1,1)=(0,1)
%%A priori, there are 6 automorphisms.
%%f(1,0) could be (0,1),(1,0),(0,1). Can be (1,0). 
%%T f(1,0)=f(1,0) so f(1,0) has to be (1,0) because only (1,0) is fixed.
%%f(0,1) can be (0,1) or (1,1).

Since nonorientable $\Z_2$-surfaces are nullhomologous (recall Lemma~\ref{lem:H1NonOrientable}), whether or not they are characteristic is determined by whether or not $X$ is spin.
The second item now follows from the claim and Theorem~\ref{thm:CompatiblePairClosedIntro}.
For the first item,  it suffices to show that if~$S \subset X$ is a projective plane with simply-connected complement,  then~$Q_{X_S}$ only depends on~$[S]$:
the result will then again follow  from the claim and Theorem~\ref{thm:CompatiblePairClosedIntro} 
(the condition on spin unions is automatic in nonorientable genus~$1$ because the Alexander modules are torsion; recall Theorem~\ref{thm:CK4Manifold}).
To see this,  use~$H_1(X_S)=0$ and the long exact sequence of the pair to deduce that the inclusion~$X_S \subset X$ induces an isometry of~$Q_{X_S}$ with the restriction of $Q_X$ to the subgroup~$\{ y \in H_2(X) \mid Q_X^{\Z_2}([S],y)=0 \mod 2\}. $
%%0->H_2(X_S)->H_2(X)->Z_2->0
%%%%%
%No Guillou--Marin in genus 1.
\end{proof}

\subsection{$\Z_2$-Moebius bands in~$D^4$}
\label{sub:Moeb}

We generalise Lawson's result in another direction to obtain an upper bound on the number of~$\Z_2$-Moebius in~$D^4$ that a knot~$K \subset S^3$ can bound.
For this, we set $d_K:=\operatorname{max}(n_K,1)$, where $n_K$ denotes the number of distinct prime factors of $\det(K)$.
We also note that Moebius bands in $D^4$ with abelian knot group are necessarily~$\Z_2$-Moebius bands.

\begin{customthm}{\ref{thm:MoebiusIntro}}
\label{thm:Moebius}
Up to isotopy rel.\ boundary, a knot~$K \subset S^3$ bounds at most $2^{d_K-1}$ Moebius bands in~$D^4$ with abelian knot group and a given Euler number.
\end{customthm}
\begin{proof}
%%Can't do outside of D^4 because not cyclic.
We aim to apply Corollary~\ref{cor:Section8DiscsSpheres}.
We assert that the exterior of a~$\Z_2$-Moebius band~$S \subset~D^4$ has cyclic~$H_1(\partial X_S;\Z[\Z_2])$ and that the intersection form of its universal cover is~$4\det(K)$.
%We assert that the exterior of any~$\Z_2$-Moebius band in~$D^4$ has cyclic Alexander module and that the intersection form of its universal cover is~$4\det(K)$.
Proposition~\ref{prop:AlexanderModuleNonOri} shows that~$H_1(\partial X_S;\Z[\Z_2]) 
\cong H_1(Y_e(K);\Z[\Z_2])
\cong H_1(\Sigma_2(K)) \oplus \Z_4.
$
We explain the last isomorphism since,  a priori,  the~$\Z_4$ summand could have been a~$\Z_2 \oplus \Z_2$ summand.
Consider the long exact sequence of the pair~$(X_S,\partial X_S)$ with~$\Z[\Z_2]$ coefficients.
Since~$H_2(X_S;\Z[\Z_2]) \cong \Z_-$ (e.g. by~\cite[Proposition 4.12]{ConwayOrsonPowell}),  it follows that~$\mathcal{A}_{\partial}:=H_1(\partial X_S;\Z[\Z_2])$ is cyclic and is presented by~$Q_{\widetilde{X}_S}$.
In particular,  the order of this finite group is~$4\det(K)$.
Since~$\det(K)$ is odd,~$\mathcal{A}_{\partial} \cong \Z_{\det(K)} \oplus G$ where~$G$ is~$\Z_2 \oplus \Z_2$ or~$\Z_4$.
Since~$\mathcal{A}_{\partial}$ is cyclic,~$G$ must be~$\Z_4$, as asserted.

Since the intersection form is isometric to~$Q:= (\Z,4\det(K))$ and since Moebius bands in $D^4$ are necessarily characteristic, Corollary~\ref{cor:Section8DiscsSpheres} states that the number of $\Z_2$-Moebius bands in $D^4$ with boundary $K$ and Euler number $e$ is at most 
$$
\Big|
 \frac{\Iso_{\Z[\Z_2]}(\partial Q,-\ell_{\widehat{Y}_e(K)})}
 {\Aut_{\Z[\Z_2]}(Q) \times d \cdot H^1(\Sigma,\partial \Sigma;\Z^w)}
 \Big|
 =
 \Big|
 \frac{\Iso_{\Z[\Z_2]}(\partial Q,-\ell_{\widehat{Y}_e(K)})}
 { \lbrace \pm 1 \rbrace \times 2 \Z\langle \alpha\rangle}
 \Big|.
$$
%The equivariant intersection form in~$D^4$ is~$(\Z_-,(1-T)4\det(K))$ but perhaps we don't care.
Here~$\alpha$ denotes the cohomology class that evaluates to~$1$ on the core of the Moebius band~$\Sigma$.

It remains to understand the numerator and the action.
Fix an isometry~$\partial Q \cong -\ell_{\widehat{Y}_e(K)}$ and identify~$\Iso_{\Z[\Z_2]}(\partial Q,-\ell_{\widehat{Y}_e(K)})$ with~$\Aut_{\Z[\Z_2]}(\partial Q)$.
%%%Don't delete.
%%The Rot(2\alpha)-action is trivial on the odd-torsion 
%%It's -1 on the Z_4 part.
%%These facts are invariant under any isomorphism \partial Q \cong -\ell_Y
%%Those it's fine to identify both
Since~$Q_{\widetilde{X}_S}$ is represented by~$(4\det(K))$, the boundary linking form is~$(x,y) \mapsto -\frac{xy}{4\det(K)}.$
The isometries of this linking form correspond to the~$x \in \Z_{4\det(K)}$ with~$x^2=1$.
There are $2^{d_K+1}$ such elements,  unless $\det(K)=1$ in which case there are $2$ such elements.
%When~$\det(K)$ is a (necessarily odd) prime power (resp.\ when~$\det(K)=1$), there are~$4$ (resp.~$2$) such~$x$.
After quotienting by the group action of~$\Aut_{\Z[\Z_2]}(Q)=\Aut(\lambda)=\{ \pm 1\}$, only half of these remain.

In order to obtain the upper bound~$2^{d_K-1}$ in the case where~$\det(K)\neq 1$, we need to better understand the actions.
The four automorphisms we are considering are~$(a,b) \mapsto (\varepsilon_1a,\varepsilon_2b)$ with~$\varepsilon_1,\varepsilon_2 \in \{ \pm 1 \}$.
The action of~$\Aut_{\Z[\Z_2]}(Q)=\Aut(\lambda)=\{ \pm 1\}$ on~$\Z_{\det(K)} \oplus \Z_4\langle \gamma\rangle$ is by~$(a,b) \mapsto (-a,-b)$.
On the other hand,  Corollary~\ref{cor:RotOnYe(K)Covers} shows that~$\Rot(2\alpha)$ acts by the identity on~$\Z_{\det(K)}$ and by~$\gamma \mapsto \gamma +\alpha(\gamma)\widetilde{\mu}=-\gamma$ 
%%Don't delete
%recall w_1(\gamma)=-1.
on the generator of the second factor, i.e.\ by~$(a,b) \mapsto (a,-b)$.
Here, recall that~$\widetilde{\mu}=2\gamma$ denotes a lifted meridian of $K$.
%%Don't delete
%%Recall that \mu=2 so
%1+2=3=-1
It follows that the orbit set has order~$2^{d_K-1}$ and, since Moebius bands in $D^4$ are necessarily characteristic,  the theorem now follows from Corollary~\ref{cor:Section8DiscsSpheres}.
\end{proof}

When~$\det(K)=1$,  Theorem~\ref{thm:Moebius} follows from~\cite[Theorem B]{ConwayOrsonPowell}; the other cases are new.

\section{Homotopy ribbon discs with knot group~$BS(1,2)$}
\label{sec:BS(12)}

%We can probably recover~\cite{ConwayPowellDiscs} because the Alexander module is torsion.
%This would mean writing Theorem~\ref{thm:CKForSurfaceExteriors} for more groups.
%TBD if one could do higher genus.
As a final application,  we venture outside the realm of simple surfaces and show how the methods of this paper and~\cite{ConwayKasprowski4Manifolds} can be used to study discs with knot group the Baumslag-Solitar group~$BS(1,2)$.
For this, recall that a disc $S \subset X$ with boundary $K \subset S^3$ is \emph{homotopy ribbon} if the inclusion induced map $\pi_1(E_K) \to \pi_1(X_S)$ is surjective.
When the Euler number of $S$ relative to the Seifert framing is zero, $\partial X_S \cong S^3_0(K)$ is the $0$-surgery on $K$,  so homotopy ribbon discs are seen to have~$\pi_1(\partial X_S) \to \pi_1(X_S)$ surjective.
%As explained in the introduction, 
The following result generalises the classification of homotopy ribbon discs in the $4$-ball from~\cite{ConwayPowellDiscs} to more general~$4$-manifolds with boundary $S^3$; we refer to~\cite[page 2]{ConwayKasprowski4Manifolds} for the definition of a compatible triple.

%We begin with the case of discs and recover a result from~\cite{ConwayPowellDiscs}.

\begin{theorem}
\label{thm:BS12}
Let $X$ a simply-connected $4$-manifold with boundary $S^3$ and let~$K \subset S^3$ be a knot.
%Let $K \subset S^3$ be a knot.
Two~$BS(1,2)$-homotopy ribbon discs $S_0,S_1 \subset X$ with boundary~$K$ and vanishing relative Euler number are equivalent rel.\ boundary if and only if there exists a bundle isomorphism~$H \colon (\overline{\nu}(S_0),S_0) \to (\overline{\nu}(S_1),S_1)$ such that $h:=(\id_{E_K} \cup H|)$ satisfies $\iota_1 \circ h_*=u \circ \iota_0$ for some isomorphism $u \colon \pi_1(X_{S_0}) \to \pi_1(X_{S_1})$, and $h$ fits into a compatible triple $(F,G,h)$.
\end{theorem}
\begin{proof}
The conditions are necessary, so we focus on their sufficiency.
Since the discs are homotopy ribbon, and we assumed the existence of a compatible triple,~\cite[Theorem~1.8]{ConwayKasprowski4Manifolds} implies that the homeomorphism~$h \colon \partial X_{S_0} \to \partial X_{S_1}$ extends over the disc exteriors leading to the required equivalence rel.\ boundary.
This theorem applies because the~$\Z[BS(1,2)]$-Alexander module of~$\partial X_{S_i}$  is torsion for $i=0,1$; this factfollows from~\cite[Proposition~2.10]{CochranOrrTeichner},  and is also explained in~\cite[Proof of Lemma~2.1]{ConwayPowellDiscs}.
Since~$BS(1,2)$ is $2$-dimensional, the relative $k$-invariant can be omitted~\cite[Proposition 2.12]{ConwayKasprowski4Manifolds}.
%%%Don't delete yet.
%As explained in~\cite[Lemma 2.5]{ConwayPowellDiscs},  the Alexander trick shows that it suffices to prove the disc exteriors are homeomorphic rel.\ boundary.
%The disc exteriors are aspherical by~\cite[Lemma~2.1]{ConwayPowellDiscs}.
%Since the Alexander modules are torsion (this follows from~\cite[Proposition~2.10]{CochranOrrTeichner}, as also explained in~\cite[Lemma~2.1]{ConwayPowellDiscs}),  we deduce from the hypothesis and~\cite[Theorem~1.8]{ConwayKasprowski4Manifolds} that the disc exteriors are indeed homeomorphic rel.\ boundary, thus leading to the result.
\end{proof}

In order to recover~\cite[Theorem 1.4]{ConwayPowellDiscs}, one notes that when $X=D^4$, the disc exteriors are aspherical~\cite[Lemma 2.1]{ConwayPowellDiscs} (so that the condition involving compatible triples holds vacuously) and applies the Alexander trick to upgrade the equivalence to an isotopy.

\appendix

\section{Bundle isomorphisms of~$\Sigma \times \R^2$.}
\label{sec:BundleIsos}

Given an orientable surface $\Sigma$ with empty or connected boundary, the first goal of this section is to prove Proposition~\ref{prop:RotAlphaAxioms} which asserts that,  for every $\alpha \in H^1(\Sigma),$ there exists a bundle isomorphism
$$\Rot(\alpha) \colon  \Sigma \times \R^2 \to  \Sigma \times \R^2$$
satisfying a certain number of properties.
Taking circle bundles leads to a bundle isomorphism
$$\Rot(\alpha) \colon \Sigma \times S^1 \to  \Sigma \times S^1.$$
Recall the coefficient system~$\pi_1(\Sigma \times S^1) \to \Z_d$ from Construction~\ref{cons:CoeffSys} that sends the curves in $\Sigma$ to zero and the meridian to~$1$.
The second goal of this section is to calculate the effect of~$\Rot(\alpha)^d|_{\Sigma \times S^1}$ on the Alexander module with respect to this coefficient system.
This was needed in Section~\ref{sec:Injection} to prove that the map $\Xi_{\Emb}$ is well defined.
Finally we will study the effect of $\Rot(\alpha)$ on spin structures since this was used both in Section~\ref{sec:SpinUnion} and in Section~\ref{sec:Injection}.

\subsection{The construction of $\Rot(\alpha)$.}

This section constructs the bundle isomorphism~$\Rot(\alpha)$ and establishes Proposition~\ref{prop:RotAlphaAxioms}.
In what follows,~$p \colon \Sigma \times \R^2 \to \Sigma$ denotes the bundle projection.

\begin{construction}
\label{cons:Rotn}
Given $\alpha \in H^1(\Sigma) \cong H^1(\Sigma,\partial \Sigma),$ choose a representative~$f_\alpha\colon \Sigma\to SO(2)$ of the class~$\alpha$ in~$H^1(\Sigma,\partial \Sigma)\cong [(\Sigma,\partial \Sigma),(S^1,*)] \cong [(\Sigma,\partial \Sigma),(SO(2),*)]$.
Given $z \in \Sigma \times \R^2$, the required bundle isomorphism is 
%then 
$$
\Rot(\alpha)(z)=(f_{\alpha}(p(z))\cdot z.
$$
%%%Don't delete
%%every automorphism assigns to a point x in the base an element of O(2) by restricting the autmorphism to the fibre over x. Since this has to be done continuously, this gives a map X->O(2
%Returning to the case $X=\Sigma$, picking a basis $\{ \alpha_i\}$ of $H^1(\Sigma)$, the~$\{ \Rot(\alpha_i)\}$ generate $[\Sigma,SO(2)] \cong H^1(\Sigma)$.
\end{construction}

In order to calculate the effect of~$\Rot(\alpha)$ on integral homology, we require a bundle theoretic lemma.
		Let~$E\xrightarrow{p}B$ be a rank~$2$ vector bundle with structure group~$SO(2)$ and let~$Y \to B$ be its circle bundle.
		%So automatically orientable
A map~$\phi \colon B\to SO(2)\cong S^1$ induces the bundle automorphism
$$f_\phi \colon E \to E,e \mapsto \phi(p(e))\cdot e.$$
The next lemma describes the effect of~$f_\phi$ on the first homology of $Y$.

	\begin{lemma}
		\label{lem:h1-action}
Let~$E\xrightarrow{p}B$ be a rank $2$ vector bundle with structure group~$SO(2)$, let~$Y \to B$ be its circle bundle, and write~$i\colon S^1\to Y$ for the inclusion of the fibre.
		Given~$\phi \colon B\to SO(2)\cong S^1$,  the bundle automorphism~$f_\phi \colon E \to E$ satisfies: 
$$(f_\phi)_*=\id+i_*\phi_*(p|_Y)_*\colon H_1(Y)\to H_1(Y).$$ 
	\end{lemma}
\begin{proof}
Given~$[\gamma] \in H_1(Y)$, we need to show that~$f_\phi([\gamma])=[\gamma]+i_*(\phi_*((p|_Y)_*([\gamma])))\in H_1(Y)$.
Restricting the bundle~$E \to B$ to the loop~$p(\gamma)\subset B$ yields a trivial circle bundle~$Y|_{\gamma} \cong T^2\xrightarrow{p} S^1$. 
%%Don't delete: SO(2) so orientable bundle.
The bundle automorphism~$f_\phi$ restricts to a bundle automorphism~$f'_\phi\colon Y|_\gamma \to Y|_\gamma$.
This bundle automorphism is induced by~$\phi'\colon S^1\xrightarrow{\gamma}Y\xrightarrow{p}B\xrightarrow{\phi}SO(2)\cong S^1$, so that~$f_\phi'(e)=\phi'(p(e))e$.
Base~$H_1(Y|_\gamma)\cong \Z^2$ by the~$S^1$-fibre~$\mu$ in the second coordinate and by some longitude~$\ell$ in the first coordinate; note~$p_*(\ell)=1\in \Z\cong H_1(S^1)$.
%the calculation above is not in homology but on the space level.
Since~$f_\phi'(\mu)=\phi'(p_*(\mu))\mu=\phi'(0)\mu=\mu$. and~$f_\phi'(\ell)=\phi'(p_*(\ell))\ell=\phi'(1)\ell$, we deduce that
%%%Don't delete
%Choose a basis of~$H_1(Y|_\gamma)\cong \Z^2$ such that~$(1,0)$ maps to~$1\in \Z\cong H_1(S^1)$ under~$p'_*$, and such that~$(0,1)$ is represented by the~$S^1$-fibre. Then~$(f_\phi)_*$ is represented by
%%%%%%
\begin{align*}
(f_\phi)_*=(f_\phi')_* \colon H_1(Y|_\gamma)&\to H_1(Y|_\gamma) \\
\bsm a \\ b \esm &\mapsto \bsm a \\ \phi'(a)b \esm.
\end{align*}
We now calculate~$(f_\phi)_*([\gamma])=(f_\phi')_*([\gamma])$.
By construction,~$[\gamma]=(1,z)\in H_1(T^2)\cong\Z^2$ for some~$z\in\Z$. 
%%Don't delete.
%%(1,z) because fibre above gamma, so \gamma \times S^1
Using the previous calculation and the identification $S^1 \cong SO(2)$, we conclude that\
%%Using a \phi'(1) \in S^1 a bit sneakily.
$$
f_\phi([\gamma])=
(f'_\phi)_*(1,z)
=(1,z+\phi'(1))
=(1,z)+(0,\phi'(1))
=[\gamma]+i_*(\phi_*((p|_Y)_*([\gamma]))).
$$
This concludes the proof of the lemma.
	\end{proof}

We now prove the main result of this section.

\begin{proposition}
\label{prop:RotAlphaAxioms}
For every class~$\alpha\in H^1(\Sigma)$, there exists a bundle isomorphism
$$
\Rot(\alpha) \colon \Sigma \times \R^2 \to \Sigma \times \R^2
$$
that satisfies the following properties:
\begin{enumerate}
\item The homeomorphism $\Rot(\alpha)$ is orientation-preserving and restricts to the identity on the~$0$-section and on $(\Sigma \times \R^2)|_{\partial \Sigma}.$
\item The effect of $\Rot(\alpha)$ on $H_1(\Sigma \times S^1)=\Z \oplus H_1(\Sigma)$ is
\begin{equation*}
%\label{eq:RotnHomology}
\Rot(\alpha)_*=
\begin{pmatrix} \id & \alpha_* \\ 0 & \id \end{pmatrix}.
\end{equation*}
Here,~$\alpha_* \colon H_1(\Sigma) \to H_1(S^1)=\Z$ is the map induced by~$\alpha \in H^1(\Sigma) \cong [\Sigma,S^1]$ on homology.
The splitting~$H_1(\Sigma \times S^1)=\Z \oplus H_1(\Sigma)$ is obtained using the fixed section~$r \colon \Sigma \to \Sigma \times S^1$.
\item For every bundle automorphism~$H \colon \Sigma \times \R^2 \to \Sigma \times \R^2$ that satisfies~$H|_{\partial \Sigma \times \R^2}=\id$, there exists a unique~$\alpha \in H^1(\Sigma)$ such that~$H \simeq \Rot(\alpha)$ rel.\ boundary.
%%Don't delete
%Automatically op as a homeo because rel boundary.
%The bundle iso is the op as a bundle iso because over id.
\end{enumerate}
\end{proposition}
\begin{proof}
The bundle isomorphism~$\Rot(\alpha)$ is constructed in Construction~\ref{cons:Rotn}.
The first property holds by construction,
%%Don't delete
%Rot(\alpha) is orientation preserving because we are using SO(2) in the construction and not O(2).
%%The bundle iso is op (iff the underlying homeo is op because over id).
whereas the second follows immediately from Lemma~\ref{lem:h1-action} applied to~$B=\Sigma,E=\Sigma \times \R^2$ and~$Y=\Sigma \times S^1$.
%%Don't delete
%%\Rot(\alpha)(\mu)=\mu+0 (because p(\mu)=0
%%\Rot(\alpha)(s(x))=s(x)+i_*f_\alpha(x)=s(x)+\alpha_*(x). because p \circ s=id.
We prove the third property.
Given an orientable rank~$2$ vector bundle~$\zeta$ over~$X=\Sigma$,  the assignment~$f \mapsto (z \mapsto (f(p(z))\cdot z)$ gives rise to a bijection~$[(\Sigma,\partial \Sigma),(SO(2),*)]\cong \pi_0(\Aut_{\partial \Sigma}(\zeta))$, where~$\pi_0(\Aut_{\partial \Sigma}(\zeta))$ denotes the set of homotopy classes of orientation-preserving bundle automorphisms of~$\zeta$ that restrict to the identity over~$\partial \Sigma$.
%%Don't delete
%%Homotopy means through homotopies that are rel.\ boundary at each time.
This fact is probably folklore, but a proof can also be obtained  by applying Proposition~\ref{prop:BundleIsoHE} below in the setting where $B=\Sigma,B'=\partial \Sigma$ and the bundle~$\zeta=(E \to B)$ is orientable, so that, in the notation of that proposition,~$P/SO(2)=B \times \Z_2, P/SO(2)\times_{\Z_2} SO(2)=B \times SO(2)$ and~$\Gamma_{B'}(B,P/SO(2)\times_{\Z_2} SO(2))=\operatorname{Map}((B,B'), (SO(2),*))$.
%Picking a basis~$\{ \alpha_i\}$ of~$H^1(\Sigma) \cong  H^1(\Sigma,\partial \Sigma)~$, the~$\{ \Rot(\alpha_i)\}$ generate~$[(\Sigma,\partial \Sigma),(SO(2),*)] \cong H^1(\Sigma,\partial \Sigma)$.
\end{proof}

\subsection{The effect of~$\Rot(\alpha)$ on the Alexander module.}

Next we describe the effect of~$\Rot(\alpha)$ on the Alexander module of $Y:=\Sigma \times S^1$ with respect to the coefficient system~$\varphi \colon \pi_1(Y)\to \Z_d$ from~\ref{cons:CoeffSys}.
Since the preferred section~$s \colon \Sigma \to \Sigma \times S^1$ satisfies~$\varphi \circ s_*=0$ (as does, in fact, any section),
%%Don't delete
%%By def of \varphi but clunky to write that in the sentence.
%s(x)=(x,b) so it's basically a curve in \Sigma, no S^1 factor to it so is sent to zero.
%no meridional component.
it lifts to a section~$\widetilde{s} \colon \Sigma \to  (\Sigma \times  S^1)^\varphi=:Y^\varphi$ on the $\varphi$-induced cover and thus gives rise to a splitting~$H_1(Y^\varphi)=\Z \oplus H_1(\Sigma)$.
To see this,  it is helpful to remember that, 
%And s(x)=(x,e_1).
the covering~$Y^\varphi=\Sigma \times S^1 \to \Sigma \times S^1=Y$ is given by the identity on the first factor and by~$z \mapsto z^d$ on the second.
Next, a class~$\alpha \in H^1(\Sigma) \cong H^1(\Sigma,\partial \Sigma)\cong [(\Sigma,\partial \Sigma),(S^1,*)] \cong [(\Sigma,\partial \Sigma),(SO(2),*)]$ determines an element~$\alpha_* \colon H_1(\Sigma) \to H_1(S^1)=\Z$.
Proposition~\ref{prop:RotAlphaAxioms} implies that
\begin{equation}
\label{eq:RotnHomology}
\Rot(\alpha)_*^n=
\begin{pmatrix} \id & n \alpha_* \\ 0 & \id \end{pmatrix}.
\end{equation}
Thanks to our choice of coefficient system~$\varphi \colon \pi_1(\Sigma \times S^1) \to \Z_d$,  we see that~$\Rot(\alpha)^n$ lifts to the cover if and only if~$n$ is a multiple of~$d$.
In order to deduce the effect of~$\Rot(\alpha)^{dk}$ on the Alexander module, it only remains to calculate the effect on homology of the covering projection.

\begin{lemma}
\label{lem:JustificationLift}
For the coefficient system~$\varphi \colon \pi_1(\Sigma \times S^1) \to \Z_d$ from Construction~\ref{cons:CoeffSys}, the~$d$-fold covering projection~$Y^\varphi:=(\Sigma \times S^1)^\varphi \to \Sigma \times S^1=:Y$ induces the following map on homology:
	\[H_1(Y^\varphi)\cong \Z \oplus H_1(\Sigma)\xrightarrow{\bsm \cdot d&0 \\ 0&\id \esm }\Z \oplus H_1(\Sigma)\cong H_1(Y).\]
\end{lemma}
\begin{proof}
Write~$N:=\Sigma \times D^2$ (resp.~$N^\varphi$) for the disc bundle with circle bundle $Y$ (resp.\ $Y^\varphi$).
The covering map~$Y^\varphi \to Y$ extends to a branched covering~$N^\varphi=\Sigma \times D^2 \to \Sigma \times D^2=N$ that is given by the identity on the first factor and by~$z \mapsto z^d$ on the second.
A calculation
%LES
 shows that~$H_i(N,Y)=0$ for~$i\neq 2$ and is generated by a meridional disc for~$i=2$.
The same holds in the branched cover so that,  when we consider the long exact sequences of the pairs~$(N^\varphi,Y^\varphi)$ and~$(N,Y)$, we obtain the following diagram in which the rows are exact and the squares commute:
$$
\xymatrix{
&&&&H_1(\Sigma)\ar[ld]_-{\widetilde{s}_*}\\
%%%%%%%
&0 \ar[r]
%^-{j_*}
&H_2(N^\varphi,Y^\varphi) \ar[r]\ar[d]^-{p_*}
&H_1(Y^\varphi) \ar[r]^-{i_*}\ar[d]^-{p_*}
&H_1(N^\varphi)\ar[r]\ar[d]^-{p_*}_-{\cong}
\ar[u]_-{\proj_*}^\cong
&0 \\
%%%%%%%
&0\ar[r]
%^-{j_*}
&H_2(N,Y) \ar[r]
&H_1(Y) \ar[r]^-{i_*}
&H_1(N)\ar[r]\ar[d]^-{\proj_*}_\cong
&0 \\
%%%%%%
&&&&H_1(\Sigma)\ar[lu]^-{s_*}.\\
}
$$
%Observe that~$H_2(N,Y) \cong H^2(N) \cong H_2(N)^* \cong \Z$.
%We deduce that~$H_2(N,Y) \cong \Z$ is generated by a meridional disc~$D$ and 
%Note that~$j_*([\Sigma])=Q_X([\Sigma],[\Sigma])[D]=e \cdot [D]$.
%The choice of a nice identification~$N \cong N_e(h)$ gives rise to a section~$\Sigma \to Y$ that splits the composition~$Y \to N \xrightarrow{\proj} \Sigma$.
%Since nice identifications lift to the covers,  a similar splitting exists in the cover; the splittings commute with the covering projections.
%(double check).
The bundle projection induces identifications~$H_k(N) \cong H_k(\Sigma)$ and~$H_k(N^\varphi) \cong H_k(\Sigma)$.
%Base~$H_2(N) \cong H_2(\Sigma)$ and~$H_2(N^\varphi) \cong H_2(\Sigma)$ with the homology class of~$\Sigma$ and~$\Sigma$ respectively.
Base~$H_2(N,Y)$ and~$H_2(N^\varphi,Y^\varphi)$ with the~$D^2$-fibres.
%meridional discs.
The previous commutative diagram then becomes
%%= because bases where chosen
$$
\xymatrix{
&0 \ar[r] 
&{\overbrace{H_2(N^\varphi,Y^\varphi)}^{= \Z}} \ar[r]^-{\bsm \id \\ 0\esm}
\ar[d]^-{\cdot d}
&{\overbrace{H_1(Y^\varphi)}^{=\Z \oplus \Z^{2g}}} \ar[r]_-{\bsm 0& \id \esm}\ar[d]^-{p_*}
&{\overbrace{H_1(\Sigma)}^{=\Z^{2g}}}\ar[r]\ar[d]^-{=}
\ar@/_1pc/[l]_-{\bsm 0 \\ \id \esm} 
&0 \\
%%%%%%%
&0 \ar[r]
%^-{\cdot e}
&{\underbrace{H_2(N,Y)}_{=\Z}} \ar[r]^-{\bsm \id \\ 0\esm}
&{\underbrace{H_1(Y)}_{=\Z \oplus \Z^{2g}}} \ar[r]^-{\bsm 0& \id \esm}
&{\underbrace{H_1(\Sigma)}_{=\Z^{2g}}}\ar[r]\ar@/^1pc/[l]^-{\bsm 0 \\ \id \esm} 
&0. \\
}
$$
The conclusion now promptly follows.
\end{proof}

Finally, we can prove the main result of this section.

\begin{proposition}
\label{prop:RotAlexanderModule}
Given~$\alpha \in H^1(\Sigma)$,  the bundle isomorphism~$\Rot(\alpha)^d$ induces the following automorphism on the Alexander module~$H_1(Y^\varphi)\cong \Z \oplus H_1(\Sigma)$ of its circle bundle:
$$(\widetilde{\Rot(\alpha)^d})_*=\begin{pmatrix} \id & \alpha_* \\ 0 & \id \end{pmatrix}.$$
\end{proposition}
\begin{proof}
We have argued that on the homology of~$Y$, the bundle isomorphism~$\Rot(\alpha)^d$ acts as~$\bsm \id &d \alpha \\ 0 & \id \esm$.
Combined with Lemma~\ref{lem:JustificationLift}, this implies that
$$
\begin{pmatrix}
\cdot d&0 \\ 0&\id
\end{pmatrix}
\circ 
(\widetilde{\Rot(\alpha)^d})_*
=
(\Rot(\alpha)^d)_*
\circ
\begin{pmatrix}
\cdot d&0 \\ 0&\id 
\end{pmatrix}.
$$
Since $\Z$ is torsion-free, a short calculation then implies the result.
\end{proof}

\subsection{The effect of $\Rot(\alpha)$ on spin structures}
\label{sub:RotalphaSpinOrientable}

Let $p \colon \Sigma \times D^2 \to \Sigma$ be the projection,  and let~$p| \colon \Sigma \times S^1 \to \Sigma$ be the induced projection.
Given a spin structure $\mathfrak{s} \in \Spin(\Sigma \times D^2)$,  we consider  the restricted spin structure $\mathfrak{s}| \in \Spin(\Sigma \times S^1)$.
Furthermore, given a class $\alpha \in H^1(\Sigma)$, we write $[\alpha]_2 \in H^1(\Sigma;\Z_2)$ for its reduction mod $2$.
\begin{proposition}
\label{prop:RotalphaSpinOrientable}
Given a class~$\alpha \in H^1(\Sigma)$ and a spin structure $\mathfrak{s} \in \Spin(\Sigma \times D^2)$,  the effect of the bundle automorphism~$\Rot(\alpha)| \colon \Sigma \times S^1 \to \Sigma \times S^1$ on~$\mathfrak{s}| \in \Spin(\Sigma \times S^1)$ is by
$$
\Rot(\alpha)^*(\mathfrak{s}|)
=
p|^*([\alpha]_2) \cdot 
\mathfrak{s}| 
\in \Spin(\Sigma \times S^1).
$$
\end{proposition}
\begin{proof}
We give an argument here but note that the result also follows from the more general Proposition~\ref{prop:RotalphaSpinNonOrientable} below.
The group~$H^1(\Sigma \times S^1;\Z_2)$ acts freely and transitively on~$\Spin(\Sigma \times S^1)$, and we consider the difference class~$\beta \in H^1(\Sigma \times S^1;\Z_2)$ between~$\mathfrak{s}|$ and~$\Rot(\alpha)^*(\mathfrak{s}|)$.
We will evaluate~$\beta$ on curves in~$H_1(\Sigma \times S^1)$ and aim to show that~$\beta(\gamma')=\alpha(p(\gamma'))$ mod~$2$ for every~$\gamma'\in H_1(\Sigma \times S^1)$.

%Note that~$\Rot(\alpha)|$ restricted to~$\{*\}\times S^1$ is homotopic to the identity and hence~$\beta$ is in the image of~$p^*\colon H^1(\Sigma;\Z_2)\to H^1(\Sigma\times S^1;\Z_2)$.
For a simple closed curve~$\gamma \subset \Sigma$, the homeomorphism~$\Rot(\alpha)$ restricts to a homeomorphism~$\Rot(\alpha)|_\gamma$ of~$\gamma\times S^1\subseteq \Sigma\times S^1$ and it suffices to show that the action of~$\Rot(\alpha)|_\gamma$ on~$\mathfrak{s}|_{\gamma\times S^1}$ is trivial if and only if~$\alpha(\gamma)\equiv 0 \mod 2$. 
Note that~$\Rot(\alpha)|_\gamma$ restricted to~$\{*\}\times S^1$ is homotopic to the identity and hence~$\mathfrak{s}|_{\Rot(\alpha)(\gamma\times S^1)}
-
\mathfrak{s}|_{\gamma\times S^1}$ is in the image of~$p^*\colon H^1(\gamma;\Z_2)\to H^1(\gamma\times S^1;\Z_2)$. Thus we have to show that $\mathfrak{s}|_{\Rot(\alpha)(\gamma\times\{*\})}
-
\mathfrak{s}|_{\gamma\times \{*\}} $ is trivial if and only if $\alpha(\gamma)=0$.
Proposition~\ref{prop:RotAlphaAxioms} implies that
\begin{equation}
\label{eq:ForSpin}
[\Rot(\alpha)(\gamma\times\{*\})]=[\gamma\times\{*\}]+\alpha(\gamma)[\{*\}\times S^1]\in H_1(\gamma\times S^1;\Z_2).
\end{equation}
Given a spin structure $\mathfrak{t}$ on an orientable surface $\Sigma$,  the assignment $H_1(\Sigma;\Z_2) \to \Z_2$ that maps~$[\gamma]$ to~$\mathfrak{t}|_{\gamma} \in \Omega_1^{\Spin}(S^1) \cong \Z_2$ determines a quadratic refinement of the intersection form.
Therefore, restricting the spin structure $\mathfrak{s}$ to the left hand side of~\eqref{eq:ForSpin}, we obtain the following equation in~$\Omega_1^{\Spin}(S^1) \cong \Z_2$:
\begin{align*}
\mathfrak{s}|_{\Rot(\alpha)(\gamma\times\{*\})}
-
\mathfrak{s}|_{\gamma\times \{*\}} 
&\equiv
\alpha(\gamma)(\{*\}\times S^1)^*\overbrace{(\mathfrak{s}|_{\{*\}\times S^1}) }^{=0 \in \Omega_1^{\Spin}(S^1)}
+\alpha(\gamma)\overbrace{Q_{T^2}^{\Z_2}([\gamma\times\{*\}],[\{*\}\times S^1])}^{=1} \\
&=\alpha(\gamma).
\end{align*}
Here in the last equality, we used that, by assumption, the spin structure restricted to $\{*\}\times S^1$ bounds by construction.
This concludes the proof of the proposition.
\end{proof}

\section{Bundle isomorphisms of~$\Sigma \mathbin{\wt{\times}} \R^2$. }
\label{sec:BundleIsosNonOri}

This section is concerned with the nonorientable analogues of the results from Appendix~\ref{sec:BundleIsos}.
Namely,  given a nonorientable surface $\Sigma$ with nonempty boundary,  it constructs the bundle isomorphism~$\Rot(\alpha) \colon \Sigma \mathbin{\wt{\times}}\R^2 \to \Sigma \mathbin{\wt{\times}} \R^2$ from Proposition~\ref{prop:RotAlphaAxiomsSection2Nonori}, establishes its properties,  calculates the effect of~$\Rot(\alpha)^2|_{\Sigma \mathbin{\wt{\times}} S^1}$ on the Alexander module with respect to the coefficient system from Construction~\ref{cons:CoeffSys}, and studies the resulting action on the set of spin structures.

\begin{convention}
We recall from Convention~\ref{conv:SectionIsFraming} our preferred choice of model for the rank $2$ vector bundle~$\Sigma\mathbin{\wt{\times}} \R^2$ with nontrivial $w_1$.
Namely, we use~$p \colon \widehat{\Sigma} \to \Sigma$ to denote the orientation double cover of $\Sigma$,  consider the~$\Z_2$-action on~$\R^2$ given by~$(x,y)\mapsto (x,-y)$ and set 
$$\Sigma\mathbin{\wt{\times}} \R^2:=\widehat{\Sigma} \times_{\Z_2} \R^2$$
We also note that~$\Z_2$ acts $SO(2)$ by conjugation with~$\bsm
1&0\\0&-1
\esm$ so that, on the level of circle bundles
%%Note that this is compatible with how Z_2 acts on R^2 above.
$$
\Sigma\mathbin{\wt{\times}} S^1 \cong  \widehat{\Sigma} \times_{\Z_2} SO(2).
$$
\end{convention}

\medbreak

In the orientable setting, bundle automorphisms of~$\Sigma \mathbin{\wt{\times}} \R^2$ are determined by classes in~$H^1(\Sigma,\partial \Sigma)$ relatively directly, essentially using the isomorphisms
$$
\pi_0(\Aut_{\partial \Sigma}(\Sigma \times \R^2))
\cong  [(\Sigma,\partial \Sigma),(SO(2),*)]
\cong  [(\Sigma,\partial \Sigma),(S^1,*)]
\cong H^1(\Sigma,\partial \Sigma).
$$
In the nonorientable setting,  using a cohomology class to define a bundle automorphism is more challenging as $\Sigma$ need not come with a map to $S^1$. 
Setting $w:=w_1(\Sigma) \colon \pi_1(\Sigma) \to \Z_2$, as we will explain,  the strategy will instead involve isomorphisms
\begin{equation}
\label{eq:IsosForNonOri}
\pi_0(\Aut_{\partial \Sigma}(\Sigma \mathbin{\wt{\times}} \R^2)) 
\cong
 \pi_0( \Gamma_{\partial \Sigma}(\Sigma,\widehat{\Sigma} \times_{\Z_2}SO(2)))
\cong
 \pi_0( \Gamma_{\partial \Sigma}(\Sigma,\widehat{\Sigma} \times_{\Z_2} S^1))
\cong 
H^1(\Sigma,\partial \Sigma;\Z^w).
\end{equation}
Here~$\pi_0(\Aut_{\partial \Sigma}(\Sigma \mathbin{\wt{\times}} \R^2))$ denotes the set of homotopy classes of rel.\ boundary bundle automorphisms $\Sigma \mathbin{\wt{\times}} \R^2 \to \Sigma \mathbin{\wt{\times}} \R^2$.
%%Don't delete
% that preserve local orientations.
%%Automatic because rel.\ boundary ->homeo is o-p <-> bundle aut preserves local orientations.
%%For the equivalence: true because both conditions are local.
Also, $\pi_0(\Gamma_{\partial \Sigma}(\Sigma,\widehat{\Sigma} \times_{\Z_2} S^1))$ denotes the  set of homotopy classes of sections~$s \colon \Sigma \to \widehat{\Sigma} \times_{\Z_2} S^1$ whose restriction to $\partial \Sigma$ satisfy~$s(x)=[\widehat{x},1]$ for some $\widehat{x} \in \partial \widehat{\Sigma}$ that lifts~$x\in \partial \Sigma$.
We say that such sections are \emph{boring} on $\partial \Sigma$.
%%In pi_0 the homotopies have to be boring on the boundary throughout.

%The bundles that appear in~\eqref{eq:IsosForNonOri} are closely related to the bundles~$\Sigma \mathbin{\wt{\times}} \R^2$ and~$\Sigma \mathbin{\wt{\times}} S^1$.
%Indeed, since rank~$2$ vector bundles over surfaces with nonempty boundary are determined by their first Stiefel-Whitney class,
%% and Euler class,
%%%Don't delete
%%Surfaces with boundary have vanishing~$H^2$ and therefore vanishing Euler class.
%%Both bundles have the same~$w_1$ thanks to the presence of the orientation double cover leading to 
%there are bundle isomorphisms
%%%w_1=w_1(\Sigma).
%\begin{align*}
%&\widehat{\Sigma} \times_{\Z_2} \R^2 \cong \Sigma\mathbin{\wt{\times}} \R^2, \\
%&\widehat{\Sigma} \times_{\Z_2} SO(2)  \cong \Sigma\mathbin{\wt{\times}} S^1.
%\end{align*}
%Here,  observe that~$\widehat{\Sigma} \times_{\Z_2} SO(2) \cong \widehat{\Sigma} \times_{\Z_2} S^1$ is the circle bundle of~$\widehat{\Sigma} \times_{\Z_2} \R^2$.

The construction of the bundle isomorphism $\Rot(\alpha)$ will involve the isomorphisms of~\eqref{eq:IsosForNonOri}.
%Namely 
Section~\ref{sub:Rot(s)} will first describe how a section $s \in \Gamma_{\partial \Sigma}(\Sigma,\Sigma\mathbin{\wt{\times}} S^1)$ gives rise to an automorphism
$$\Rot(s) \colon \Sigma \mathbin{\wt{\times}} \R^2 \xrightarrow{\cong} \Sigma \mathbin{\wt{\times}} \R^2.$$
Section~\ref{sub:CohomologyReformulation} will then reformulate this construction in terms of~$H^1(\Sigma,\partial \Sigma;\Z^w)$, leading to the definition of $\Rot(\alpha)$.
Section~\ref{sub:RotAlphaAlexanderNonOri} focuses on the effect of~$\Rot(\alpha)$ on the Alexander module.
Section~\ref{sub:RotalphaSpinNonOrientable} studies the effect of~$\Rot(\alpha)$ on spin structures.
Section~\ref{sub:AutomorphismNonOrientable} proves that every bundle automorphism~$\Sigma \mathbin{\wt{\times}} \R^2 \to  \Sigma \mathbin{\wt{\times}} \R^2$ can be written as~$\Rot(\alpha)$ for some~$\alpha \in H^1(\Sigma,\partial \Sigma;\Z^w).$

\subsection{The construction of $\Rot(s)$.}
\label{sub:Rot(s)}

We describe how a section $s \colon \Sigma \to \Sigma \mathbin{\wt{\times}} S^1$ gives rise to a bundle automorphism~$\Rot(s) \colon \Sigma \mathbin{\wt{\times}} \R^2 \to \Sigma \mathbin{\wt{\times}} \R^2$.
We then calculate the effect of $\Rot(s)$ on homology.
Throughout this section, we use the models $\widehat{\Sigma} \times_{\Z_2} \R^2$ and $\widehat{\Sigma} \times_{\Z_2} S^1 \cong \widehat{\Sigma} \times_{\Z_2} SO(2)$ for~$\Sigma \mathbin{\wt{\times}} \R^2$ and~$\Sigma \mathbin{\wt{\times}} S^1$ respectively.

\begin{construction}(The automorphism $\Rot(s)$).
\label{cons:Rots}
Given a section~$s\colon \Sigma\to \wh \Sigma\times_{\Z_2}SO(2) \cong \Sigma \mathbin{\wt{\times}} S^1$
%%Used to be written
%, with~$s(b)=[b',a]$
 that is boring on $\partial \Sigma$, consider the bundle automorphism
\begin{align*}
\Rot(s) \colon &\Sigma \mathbin{\wt{\times}} \R^2 \to \Sigma \mathbin{\wt{\times}} \R^2 \\
&e=[b',v]\mapsto [b',a^{-1}v].
\end{align*}
where $a \in SO(2)$ is such that~$s(p(b'))=[b',a]$.
%{AC: It used to be written differently, but in a way that didn't really make sense to me. }
A short verification shows that this definition does not depend on the representative of $s(p(b'))$ and $e$.
%Used to be written s(b).
%%%%%%%
%%The key is that a depends on b'. 
%%Details Write p=diag(1,-1). Note that that p^{-1}=p. Write T for the generator of Z_2.
%% If s(b)=[b',a]=[Tb',pap], then [b',v]=[Tb',pv] gets mapped to [Tb',pa^{-1}ppv]=[Tb',pa^{-1}pv]=[b',a^{-1}v].
\end{construction}

Our goal is to describe $\Rot(s)_* \colon H_1(\Sigma \mathbin{\wt{\times}} S^1) \to H_1(\Sigma \mathbin{\wt{\times}} S^1)$.
We fix a decomposition of~$H_1(\Sigma \mathbin{\wt{\times}} S^1)$ that will be useful to describe this action matricially.
This decomposition was alluded to in Section~\ref{sec:Setup} (where we referenced~\cite[Proposition 4.2]{ConwayOrsonPowell} for the details) but we take the time to recall it with the present set-up and notation.
%For this we use the homeomorphism $S^1\cong SO(2)$ to identify $S(E)$ with $\wh \Sigma\times_{\Z_2}SO(2)$.

\begin{construction}(A decomposition of $H_1(\Sigma \mathbin{\wt{\times}} S^1)$.)
\label{cons:SplittingS(E)}
Consider the $2$-fold covering map $p \colon \widehat{\Sigma} \to \Sigma$ and the section~$r\colon \Sigma \to \Sigma \mathbin{\wt{\times}} S^1 \cong \wh\Sigma \times_{\Z_2} SO(2),b \mapsto [b',1]$ for any~$b'\in p^{-1}(b)$; since~$1$ is a fixed point under the~$\Z_2$ action, this does not depend on the choice of~$b'$. 
%%%Write p=diag(1,-1). Note that that p^{-1}=p. Write T for the generator of Z_2.
%%If b'=Tb'', then (P.I=PIP^{-1}=I) Flip on 1 \in S^1 doesn't do anything.
%%If [b',1]=[Tb',p \cdot 1]=[b'',1] =[b'',1]
As discussed in Construction~\ref{cons:CoeffSys}, this leads to a splitting
\begin{equation}
\label{eq:SESH1(SE)}
\xymatrix{0 \ar[r]& \Z_2 \ar[r]& H_1(\Sigma \mathbin{\wt{\times}} S^1 ) \ar[r]^-{q_*}& H_1(\Sigma) \ar[r]\ar@/_1pc/[l]_{r_*}& 0.}
\end{equation}
Here $q \colon \Sigma \mathbin{\wt{\times}} S^1=\widehat{\Sigma} \times_{\Z_2} S^1 \to \Sigma,[b',v] \mapsto p(b')$ denotes the bundle projection map.
We therefore obtain an isomorphism 
$$H_1(\Sigma \mathbin{\wt{\times}} S^1 )\cong \Z_2 \oplus H_1(\Sigma).$$
In what follows, we use this identification implicitly.

Next, consider the following map that sends the~$S^1$-fibre to~$1$ and the image of~$r_*$ to zero:
$$\varphi \colon H_1(\Sigma \mathbin{\wt{\times}} S^1 ) \to \Z_2.$$
This map provides a second way of obtaining the aforementioned splitting of~\eqref{eq:SESH1(SE)} as well as the coefficient system needed to define the Alexander module.
Indeed,  the Alexander module is given by $H_1((\Sigma \mathbin{\wt{\times}} S^1)^\varphi)$ where $(\Sigma \mathbin{\wt{\times}} S^1)^\varphi$ denotes the $\varphi$-induced $2$-fold cover of $\Sigma \mathbin{\wt{\times}} S^1$.
%This map also determines an element in
%$$
%H^1(S(E);\Z_2)
%=H^1(\Sigma
%$$
\end{construction}

Given a section~$s \colon \Sigma \to \Sigma \mathbin{\wt{\times}} S^1$,  in order to describe $\Rot(s)_* \colon H_1(\Sigma \mathbin{\wt{\times}} S^1) \to H_1(\Sigma \mathbin{\wt{\times}} S^1)$, it will be helpful to consider the composition~$\varphi \circ s_* \colon H_1(\Sigma) \to \Z_2$.

\begin{proposition}
\label{prop:Rot(s)Homology}
Given a section~$s \colon \Sigma \to \Sigma \mathbin{\wt{\times}} S^1$ that is boring on $\partial \Sigma$,  the automorphism~$\Rot(s)$ induces the following map
%A section~$s \in \Gamma_{\partial \Sigma}(\Sigma,\widehat{\Sigma} \times_{\Z_2} SO(2))$ acts 
on~$H_1(\Sigma \mathbin{\wt{\times}} S^1) \cong \Z_2 \oplus H_1(\Sigma)$
$$
\Rot(s)_*=
\begin{pmatrix}
\id& \varphi \circ s_*\\
0& \id
\end{pmatrix}
\colon 
H_1(\Sigma \mathbin{\wt{\times}} S^1)
\to 
H_1(\Sigma \mathbin{\wt{\times}} S^1).
$$
%where~$\overline{s}\colon \Sigma\to \RP^\infty$ is the composition~$x\circ s$. 
\end{proposition}
\begin{proof}
Recall that $p \colon \widehat{\Sigma} \to \Sigma$ denotes the orientation double cover, let $b_0\in\Sigma$ be a basepoint and fix a basepoint~$b'_0\in p^{-1}(b_0) \subset \widehat{\Sigma}$.
Consider the map~$q\colon \Sigma \mathbin{\wt{\times}} S^1\to \Sigma,[b',v]\mapsto p(b')$ so that, by construction, $q\circ s=q$. 
Let~$j\colon S^1\to \Sigma \mathbin{\wt{\times}} S^1,v\mapsto [b_0',v]$ be the inclusion of the fibre.
Then~$\Rot(s)\circ j$ and $j$ are homotopic since the maps differ by a rotation of $S^1$. 
%%Don't delete.
%%Because  On the fibre Rot(s) is just a rotation.
This leads to the following commutative diagram:
$$
\xymatrix{
0\ar[r]&
\Z_2
\ar[r]^-{j_*}\ar[d]^=&
H_1(\Sigma \mathbin{\wt{\times}} S^1)
\ar[r]^-{q_*}\ar[d]^-{\Rot(s)_*}&
H_1(\Sigma)
\ar[r]\ar[d]^=\ar@/_1pc/[l]_{r_*}&
0 \\
%%%%%%%%%%%%%%%
0\ar[r]&
\Z_2
\ar[r]^-{j_*}&
H_1(\Sigma \mathbin{\wt{\times}} S^1)
\ar[r]^-{q_*}&
H_1(\Sigma)
\ar[r]\ar@/^1pc/[l]_{r_*}&
0.
}
$$
Let $r_*\colon H_1(\Sigma)\to H_1(\Sigma \mathbin{\wt{\times}} S^1)$ be the splitting from Construction~\ref{cons:SplittingS(E)}. 
It follows that, using the induced decomposition $H_1(\Sigma \mathbin{\wt{\times}} S^1)\xrightarrow{(\varphi,q_*),\cong}\Z_2 \oplus H_1(\Sigma)$, we have 
$$\Rot(s)_*=\begin{pmatrix}
	\id&\varphi \circ \Rot(s)_*\circ r_*\\ 0&\id
\end{pmatrix}.$$
Define the section $\overline{s}\colon \Sigma \to \Sigma\mathbin{\wt{\times}} S^1$ by $\overline{s}(b)=[b',a^{-1}]$ when $s(b)=[b',a]$.
We observe that the equality~$\Rot(s)_*\circ r_*=\overline{s}_*$ holds on $H_1(\Sigma)$.
Indeed,  representing a class~$x \in H_1(\Sigma)$ by a loop~$t \mapsto b_t$,  
%Given $b \in H_1(\Sigma)$,  
observe that by construction, for each $t \in[0,1]$, we have
$$\Rot(s)(r(b_t))=\Rot(s)([b_t',1])=[b_t',a^{-1}]=\overline{s}(b_t).$$
%Hence $\Rot(s)_*\circ r_*=\overline{s}_*$,  as asserted
Note that $\varphi \circ \overline{s}_*=\varphi \circ s_*$.
%%\overline{s} just switches the a component.
%% \varphi maps the S^1-fibre to zero so no change when composing.
This concludes the proof of the proposition.
\end{proof}

We now work towards determining the effect of $\Rot(s)$ on the Alexander module.

\begin{proposition}
If a section~$s \in \Gamma_{\partial \Sigma}(\Sigma,\Sigma \mathbin{\wt{\times}} S^1)$ satisfies~$0=\varphi\circ s_*\colon H_1(\Sigma)\to \Z_2$,  then the automorphism $\Rot(s) \colon \Sigma \mathbin{\wt{\times}} S^1 \to \Sigma \mathbin{\wt{\times}} S^1$ lifts to the $\varphi$-induced $2$-fold covers.
\end{proposition}
\begin{proof}
It suffices to prove that $\varphi \circ \Rot(s)_*=\varphi$.
Under the identification $H_1(\Sigma \mathbin{\wt{\times}} S^1)=\Z_2 \oplus~H_1(\Sigma)$ from Construction~\ref{cons:SplittingS(E)},  the map~$\varphi$ is given by $\bsm 1 & 0\esm$.
The proposition then follows from Proposition~\ref{prop:Rot(s)Homology} and the calculation 
$\bsm 1 & 0\esm \circ \bsm 
\id&\varphi \circ s_*\\ 0&\id \esm
=\bsm 1 &\varphi \circ s_* \esm
=\bsm 1 &0 \esm
$.
\end{proof}

Let $f\colon S^1\to S^1$ be the canonical $2$-fold cover.
The $2$-fold cover $g \colon (\Sigma \mathbin{\wt{\times}} S^1)^\varphi \to \Sigma \mathbin{\wt{\times}} S^1 $ is given by $g=\id \times f \colon \wh\Sigma\times_{\Z_2}S^1\to \wh\Sigma\times_{\Z_2}S^1$.
%%well def.
It follows that the map $r^\varphi\colon \Sigma \to (\Sigma \mathbin{\wt{\times}} S^1)^\varphi, b\mapsto [b',1]$ 
with $b' \in p^{-1}(b)$ is well defined (i.e.\ independent of the choice of~$b'$) 
%%Same check as for $r$.
and constitutes
a lift of~$r \colon  \Sigma \to \Sigma \mathbin{\wt{\times}} S^1$ and again induces a splitting $H_1((\Sigma \mathbin{\wt{\times}} S^1)^\varphi) \cong \Z_2 \oplus H_1(\Sigma)$.

\begin{lemma}
\label{lem:LiftRots}
If a section~$s \in \Gamma_{\partial \Sigma}(\Sigma,\Sigma \mathbin{\wt{\times}} S^1)$ satisfies $0=\varphi\circ s_*\colon H_1(\Sigma)\to \Z_2$, then it lifts to a section $\wt s\colon \Sigma\to (\Sigma \mathbin{\wt{\times}} S^1)^\varphi$ and
$$\Rot(\wt s)=\wt{\Rot(s)} \colon (\Sigma \mathbin{\wt{\times}} S^1)^\varphi \xrightarrow{\cong} (\Sigma \mathbin{\wt{\times}} S^1)^\varphi.$$
\end{lemma}
\begin{proof}
Since the cover~$g\colon (\Sigma \mathbin{\wt{\times}} S^1)^\varphi\to \Sigma \mathbin{\wt{\times}} S^1$ is given by~$\id\times f\colon \wh \Sigma\times_{\Z_2} S^1\to \wh \Sigma\times_{\Z_2} S^1$, the section~$\wt s$ is defined by setting~$\wt s(b):=[b',a]$ with $a$ such that $s(b)=[b',f(a)]$. 
	 Consider the diagram
	 \[\begin{tikzcd}
(\Sigma \mathbin{\wt{\times}} S^1)^\varphi\ar[r,"\Rot(\wt s)"]\ar[d,"g"]&(\Sigma \mathbin{\wt{\times}} S^1)^\varphi\ar[d,"g"]\\
\Sigma \mathbin{\wt{\times}} S^1\ar[r,"\Rot(s)"]&\Sigma \mathbin{\wt{\times}} S^1.
	 \end{tikzcd}\]
For~$[b',v]\in (\Sigma \mathbin{\wt{\times}} S^1)^\varphi$,  using consecutively $g:=\id \times f$, the definition of $\Rot(s)$ (together with the fact that~$s(p(b'))=s(b)=[b',f(a)]$),  the property~$f(a^{-1}v)=f(a^{-1})f(v)=f(a)^{-1}f(v)$
(viewing~$S^1$ in~$\C$, the map $f$ is given by squaring), the definition of $g$ again, and the definition of~$\Rot(\widetilde{s})$, we obtain the equality
\begin{align*}
\Rot(s)(g([b',v]))
&=\Rot(s)([b',f(v)])=[b',f(a)^{-1}f(v)]=[b',f(a^{-1}v)]=g([b',a^{-1}v]) \\
&=g(\Rot(\wt s)([b',v])).
\end{align*}
 	 Hence the diagram commutes and $\Rot(\wt s)$ is a lift of $\Rot(s)$ as claimed.
\end{proof}

The splitting $r^\varphi \colon \Sigma \to (\Sigma \mathbin{\wt{\times}} S^1)^\varphi$ gives rise to a decomposition $H_1((\Sigma \mathbin{\wt{\times}} S^1)^\varphi) \cong \Z_2 \oplus H_1(\Sigma)$.
Mapping the $S^1$-fibre to $1$ and the image of $r_*^\varphi$ to zero gives rise to a homomorphism 
$$\psi\colon H_1((\Sigma \mathbin{\wt{\times}} S^1)^\varphi)\to \Z_2.$$
We conclude by describing the effect of $\widetilde{\Rot(s)}$ on the Alexander module.

\begin{corollary}
\label{cor:AlexanderModuleRotSectionVersion}
If $s\in \Gamma_{\partial \Sigma}(\Sigma,\Sigma \mathbin{\wt{\times}} S^1)$ is a section that satisfies~$\varphi \circ s_*=0$ and~$\wt s$ is a lift of~$s$,  then $\Rot(\wt s)_*$ acts on $H_1((\Sigma \mathbin{\wt{\times}} S^1)^\varphi)\cong \Z_2 \oplus H_1(\Sigma)$ by 
	$$
	\Rot(\wt s)_*=
	\begin{pmatrix}
		\id&\psi \circ \wt s_*\\ 0&\id
	\end{pmatrix}
	\colon 
	H_1((\Sigma \mathbin{\wt{\times}} S^1)^\varphi)
	\to 
	H_1((\Sigma \mathbin{\wt{\times}} S^1)^\varphi).
	$$
\end{corollary}
\begin{proof}
Apply Proposition \ref{prop:Rot(s)Homology} with $(\Sigma \mathbin{\wt{\times}} S^1)^\varphi$ and $\Rot(\wt s)$ instead of $\Sigma \mathbin{\wt{\times}} S^1$ and~$\Rot(s)$.
%%Legit because (\Sigma \mathbin{\wt{\times}} S^1)^\varphi \cong (\Sigma \mathbin{\wt{\times}} S^1)
\end{proof}

\subsection{Reformulation in terms of cohomology classes} 
\label{sub:CohomologyReformulation}

We state the first proposition in this section for general pairs of base spaces $(B,B')$.
Later, we will take~$(B,B')$ to either be~$(\Sigma,\emptyset)$ for~$\Sigma$ a closed surface or~$(\Sigma,\partial \Sigma)$ for~$\Sigma$ a surface with one boundary component.
 Let~$w\colon \pi_1(B)\to \Z_2$ be some homomorphism and let~$\wh B$ be the associated double cover of~$B$. 
As in Construction~\ref{cons:SplittingS(E)}, there is a map $\varphi\colon H_1(\wh B\times_{\Z_2} SO(2))\to \Z_2$
that maps the generator of the $H_1$ of the fibre to $1$ and the image of $r_*(H_1(B))$ to zero, where $r \colon B \to B\times_{\Z_2} SO(2)$ is the preferred section.
%%Z or Z_2->H_1()->H_1(B)
%%In the oriented case map the Z=H_1(fibre) to Z_2
The first goal of this section is to establish the isomorphism~$\pi_0(\Gamma_{B'}(B,\wh B\times_{\Z_2}SO(2))) \cong H^1(B,B';\Z^w)$.

\begin{proposition}
	\label{prop:cohomology-version}
	There is an isomorphism
	\[\Theta\colon \pi_0(\Gamma_{B'}(B,\wh B\times_{\Z_2}SO(2)))\xrightarrow{\cong} H^1(B,B';\Z^w)\]
	such that the following composition maps a section $s$ to~$\varphi\circ s_*$:
	%%%Don't delete.
	%%%This part is technically only needed for surfaces in the nonorientable case.
	$$
	\pi_0(\Gamma_{B'}(B,\wh B\times_{\Z_2}SO(2)))
	\xrightarrow{\Theta} H^1(B,B';\Z^w)
	\xrightarrow{\operatorname{Red}_2} H^1(B,B';\Z_2)
	\xrightarrow{\ev} \Hom(H_1(B,B'),\Z_2).
	$$	
\end{proposition}
\begin{proof}
	Since~$SO(2)\simeq K(\Z,1)$ and the~$\Z_2$-action corresponds to the action~$1\mapsto -1$ on~$\Z$,
	%%Don't delete
	%%The action by diag(1,-1) on SO(2) is the same as the ``flip" action on S^1 that fixes~$1$ and~$-1$ and induces~$-1$ on homology.
	we have an isomorphism
	$$\Theta\colon \pi_0(\Gamma_{B'}(B,\wh B\times_{\Z_2}SO(2)))\cong \pi_0(\Gamma_{B'}(B,\wh B\times_{\Z_2}K(\Z,1)))\cong H^1(B,B';\Z^w).$$ 
	Here, the last isomorphism is given by \cite[Section 3 and Theorem~3.6]{Siegel2}; note that~\cite[Theorem~3.6]{Siegel2} is stated for absolute cohomology but that the statement can also be seen to hold for relative cohomology.
	%%Don't delete.
	%The map will be the same and then use naturality of the long exact sequence and the 5-lemma.
	%exact sequence because (co)homology theory.
	It remains to verify the claim about the mod 2 reduction of~$\Theta([s])$.

	Recall that the \emph{homotopy orbit} of a~$G$-space~$X$ is given by~$X_{hG}:=X\times_{G}EG$ and the projection~$EG\to *$ induces a map~$X_{hG}\to X_G:=X/G$. 
	A verification shows that when~$X$ is a~$G$-CW complex with a free~$G$-action, then the map~$X_{hG}\to X_G$ is a homotopy equivalence.
	%%Don't delete.
	%%Uploaded a note in the dropbox.
	In particular, since~$S^\infty$ with the antipodal~$\Z_2$-action is a model for~$E\Z_2$ and the~$\Z_2$-actions on~$\wh B$ and~$\wh B\times SO(2)$ are free, the projections induce homotopy equivalences
	\begin{align*}
		&\pr \colon \wh B\times_{\Z_2}S^\infty\xrightarrow{\simeq}  B, \\
		&\pr' \colon \wh B\times_{\Z_2}(SO(2) \times S^\infty) =(\wh  B \times SO(2))\times_{\Z_2}S^\infty
		\xrightarrow{\simeq} \wh B\times_{\Z_2}SO(2).
		%%%%
		%Alternatively, the projection S^1xS^infty to S^1 is Z/2-equivariant and hence induces a map of fibre bundles. This induces a map of the LES of homotopy groups. Since the identity on the base and the projection are (weak) homotopy equivalences, the map on total spaces is a (weak) homotopy equivalence as claimed.
	\end{align*}
	We now construct a~$\Z_2$-equivariant map~$g\colon SO(2)\times S^\infty\to \RP^\infty$ that induces reduction mod~$2$ on fundamental groups. 
	Here, we endow~$\RP^\infty$ with the trivial~$\Z_2$-action. 
	For this, we start with the canonical map~$[0,1]\times S^\infty\to \Sigma S^\infty\cong S^\infty$. This is~$\Z_2$-equivariant when we endow~$[0,1]\times S^\infty$ and~$\Sigma S^\infty$ with the action~$(t,x)\mapsto (1-t,-x)$ and~$S^\infty$ with the antipodal map~$x\mapsto -x$. 
	When composed with the projection~$S^\infty\to \RP^\infty$, this map factors through~$S^1\times S^\infty\cong SO(2)\times S^\infty$. 
	%%Become antipodes when going from \Sigma S^\infty to S^\infty.
	This completes the construction of~$g$. 
	Note that by construction~$g(1,x)=g(1,y)$ for all~$x,y\in S^\infty$.
	%All that gets crushed.

	Consider the following composition, which involves a choice of a homotopy inverse of $\pr'$:
	\[\varphi'\colon \wh B\times_{\Z_2}SO(2)
	\xleftarrow{\pr',\simeq}\wh B\times_{\Z_2}(SO(2)\times S^\infty)
	\xrightarrow{\id\times g} \wh B\times_{\Z_2}\RP^\infty
	\cong B\times \RP^\infty
	\xrightarrow{\pr_2}\RP^\infty.\]
	
	We claim that~$\varphi'$ represents~$\varphi\in \Hom(H_1(\wh B \times_{\Z_2} SO(2)),\Z_2)$.
	We first observe that the map~$\varphi' \colon H_1(\wh B \times_{\Z_2} SO(2)) \to \Z_2$ takes the $SO(2)$-fibre to the generator of $H_1(\RP^\infty) \cong \Z_2$: since~$\pr'$ takes the~$SO(2)$-fibre to itself, so does any of its homotopy inverses on homology and,  
	%pr'(\mu)=\mu'=pr'(g(\mu')) so g(\mu')=\mu.
	by construction of~$g$, the~$SO(2)$-fibre is sent to a generator of~$\pi_1(\RP^\infty)\cong \Z_2$. 
	It remains to show that the composition of~$\varphi'$ with the preferred
	%%Don't delete
	%{XXX: Fine for $B=\Sigma$ with boundary. Is there a section for general $(B,B')$? XXX: Yes, still $b\mapsto[b,1]$ as in Construction B.4. Note that this is a section for $\wh B\times_{\Z_2}SO(2)$ which has trivial Euler number not for $B\mathbin{\wt{\times}} S^1$ which can have non-trivial Euler number.} 
	section~$r\colon B\to \wh B\times_{\Z_2}SO(2), b\mapsto [b',1]$ is nullhomotopic
	 or equivalently that~$(\varphi' \circ r)_* \colon H_1(B) \to \Z_2$ is the zero map.
%%%
	%b' \in p^{-1}(b).
	Considering the map~$\overline{r} \colon \widehat{B} \times_{\Z_2} S^\infty \to \wh B\times_{\Z_2}(SO(2)\times S^\infty), [x,y] \mapsto [x,1,y]$, we will prove that~$\varphi'\circ r$ is homotopic to
the following map (which is constant):	
%%Don't delete Because $\overline r$ puts a $1$ in the $SO(2)$-component and $g(1,-)$ is constant.
	\[B\xleftarrow{\pr,\simeq}\wh B\times_{\Z_2}S^\infty\xrightarrow{\overline{r}}\wh B\times_{\Z_2}(SO(2)\times S^\infty)
	\xrightarrow{\id\times g} \wh B\times_{\Z_2}\RP^\infty
	\cong B\times \RP^\infty
	\xrightarrow{\pr_2}\RP^\infty.\]
	% where~$\overline{r}([x,y])=[x,1,y]$. 
By definition of these maps,
%	In turn
	 this reduces to proving that the portion of the following diagram involving the dashed maps commutes (i.e.\ $\overline{r}\circ \pr^{-1} \simeq (\pr')^{-1}\circ r$):
	$$
	\xymatrix{
		B \ar@{-->}@/^1pc/[r]^{r}& \wh B\times_{\Z_2}SO(2) \ar[l]^p \\
		%%%
		\wh B\times_{\Z_2}S^\infty \ar[u]^{\pr,\simeq}
		\ar@{-->}@/^1pc/[r]^{\overline{r}}&\wh B\times_{\Z_2}(SO(2)\times S^\infty)  \ar[l]\ar[u]_{\pr',\simeq}.
	}
	$$
	We have~$\overline{r}\circ \pr^{-1}\simeq (\pr')^{-1}\circ \pr'\circ\overline{r}\circ \pr^{-1}. $
	Since $\overline{r}$ and thus also $\pr'\circ\overline{r}$ has second entry constant to 1, this equals~$(\pr')^{-1}\circ (r\circ p)\circ \pr'\circ\overline{r}\circ \pr^{-1}, $ where $p$ is the projection $\wh B\times SO(2) \to B$. 
	%%Don't delete.
	%%A quick check shows that \pr'\circ\overline{r}=(r\circ p)\circ \pr'\circ\overline{r}
	Now~$p\circ \pr'=\pr\circ p'$ (where $p'\colon \wh B\times SO(2)\times S^\infty \to \wh B\times S^\infty$) and so
	$$
	\overline{r}\circ \pr^{-1}\simeq
	(\pr')^{-1}\circ r\circ \pr\circ p'\circ\overline{r}\circ \pr^{-1}
	%%Used p' \overline{r}(x,y)=p'(x,1,y)=(x,y).%%%%
	=(\pr')^{-1}\circ r\circ \pr\circ \pr^{-1}
	\simeq (\pr')^{-1}\circ r.$$
	%%%For the second p' \circ \ol{r} =\id by looking at both definitions. 
	Thus,  $\varphi' \circ r$ is nullhomotopic
	%$(\varphi' \circ r)_* \colon H_1(B) \to \Z_2$ is the zero map 
	and hence~$\varphi'$ represents~$\varphi$,
	 as claimed. 
%%Don't delete:
%% because \varphi really goes to \Z_2, even in the oriented case.
%%%This is not the coefficient system for oriented surfaces, this is an intermediate adhoc construction that is local.
	
	Consider the following diagram where the first three vertical maps are isomorphisms by \cite[Theorem~3.6]{Siegel2}, the commutativity of the first two squares follows from the naturality of this theorem, and the commutativity of the final square is mentioned in~\cite[proof of Theorem~3.6]{Siegel2}:
	$$
	\xymatrix@C0.25cm{
		\pi_0(\Gamma_{B'}(B,\wh B\times_{\Z_2}SO(2)))
		\ar[r]^-\cong\ar[d]^{\cong}_\Theta&
		\pi_0(\Gamma_{B'}(B,\wh B\times_{\Z_2}(SO(2)\times S^\infty)))
		\ar[r]\ar[d]^-\cong&
		\pi_0(\Gamma_{B'}(B,\wh B\times_{\Z_2}\RP^\infty))
		\ar[r]^-\cong\ar[d]^-\cong&
		\pi_0(\Gamma_{B'}(B,B\times\RP^\infty)) \ar[d]^-\cong
		\\
		%%%%%%%
		H^1(B,B';\Z^w)
		\ar[r]^=&
		H^1(B,B';\Z^w)
		\ar[r]^{\operatorname{Red}_2}&
		H^1(B,B';\Z_2)
		&\ar[l]
		[(B,B'),(\RP^\infty,*)].
	}
	$$
	By definition of the map $\varphi'$, starting from a section $s$ in the upper left corner and mapping it all the way across the top row and then down yields $s \circ \varphi'$.
	%{XXX: Where is the commutativity of the diagram being used? DK: This describes the top then down composition. Down then bottom is what we want to compute. I added the $\Theta$ into the diagram so that it relates to the prop statement more clearly (I would prefer it on the other side of the arrow but don't know how to do it in xymatrix).}
	Upon mapping this element to~$H^1(B,B';\Z_2) \cong \Hom(H_1(B,B'),\Z_2)$ yields~$(\varphi')_* \circ s_*=\varphi
	\circ s_*$, as required.
\end{proof}

We now restrict again to the case of surfaces.
As before, let~$\Sigma$ be a surface with nonempty boundary,  let~$\wh\Sigma$ be its orientation double cover,  and set~$w:=w_1(\Sigma)\colon \pi_1(\Sigma)\to \Z_2$. 
Recall the coefficient system~$\varphi\colon H_1(\Sigma \mathbin{\wt{\times}} S^1)\to \Z_2$ from Construction~\ref{cons:SplittingS(E)}.

We now prove the first main result of this section.

\begin{proposition}
\label{prop:RotAlphaAxiomsNonori}
Let $\Sigma$ be a nonorientable surface with nonempty boundary.
For every cohomology class~$\alpha\in H^1(\Sigma,\partial \Sigma;\Z^w)$, there exists a bundle isomorphism
$$
\Rot(\alpha) \colon \Sigma \mathbin{\wt{\times}} \R^2 \to \Sigma \mathbin{\wt{\times}} \R^2
$$
that satisfies the following properties:
\begin{enumerate}
\item The homeomorphism $\Rot(\alpha)$ is orientation-preserving and restricts to the identity on the~$0$-section and on $(\Sigma \mathbin{\wt{\times}}\R^2)|_{\partial \Sigma}.$
\item The effect of $\Rot(\alpha)$ on $H_1(\Sigma \mathbin{\wt{\times}} S^1)=\Z_2 \oplus H_1(\Sigma)$ is
\begin{equation*}
\Rot(\alpha)_*=
\begin{pmatrix} \id & \ev([\alpha]_2) \\ 0 & \id \end{pmatrix}.
\end{equation*}
Here,~$\ev([\alpha]_2)  \colon H_1(\Sigma) \to \Z_2$ is the map induced by~$[\alpha]_2 \in H^1(\Sigma,\partial \Sigma;\Z_2) \cong H^1(\Sigma;\Z_2)$.
The splitting~$H_1(\Sigma \mathbin{\wt{\times}} S^1)=\Z_2 \oplus H_1(\Sigma)$ is obtained using the fixed section~$r \colon \Sigma \to \Sigma \mathbin{\wt{\times}} S^1$.
\item For every bundle automorphism~$H \colon \Sigma \mathbin{\wt{\times}} \R^2 \to \Sigma \mathbin{\wt{\times}} \R^2$ 
%that preserves local orientations and
%%Don't delete
%%Automatic because rel.\ boundary ->homeo is o-p <-> bundle aut preserves local orientations.
%%For the equivalence: true because both conditions are local.
 such that~$H|_{\partial \Sigma \mathbin{\wt{\times}} \R^2}=\id$, there exists a unique~$\alpha \in H^1(\Sigma,\partial \Sigma;\Z^w)$ such that~$H \simeq \Rot(\alpha)$ rel.\ boundary.
\end{enumerate}
\end{proposition}
\begin{proof}
The bundle isomorphism $\Rot(\alpha)$ is constructed by combining Construction~\ref{cons:Rots} (which constructs $\Rot(s)$) with Proposition~\ref{prop:cohomology-version} (which relates sections to cohomology classes).
The first item follows from the construction of $\Rot(s)$ and the translation to cohomology classes from Proposition~\ref{prop:cohomology-version}.
%Rot(\alpha) is orientation preserving because we are using SO(2) in the construction and not O(2).
The second item follows from Proposition~\ref{prop:Rot(s)Homology} (which describes the induced map~$\Rot(s)_* \colon H_1(\Sigma \mathbin{\wt{\times}} S^1) \to H_1(\Sigma \mathbin{\wt{\times}} S^1)$) and Proposition~\ref{prop:cohomology-version}.
For the third item,  Proposition~\ref{prop:RealiseRotNonOri} below proves that every bundle automorphism $\Sigma \mathbin{\wt{\times}} \R^2 \to \Sigma \mathbin{\wt{\times}} \R^2$ arises uniquely (up to homotopy) as~$\Rot(s)$ for some section~$s \in  \pi_0( \Gamma_{\partial \Sigma}(\Sigma,\Sigma \mathbin{\wt{\times}} S^1))$, and the result then follows from Proposition~\ref{prop:cohomology-version}.
This concludes the proof of the proposition.
\end{proof}

%%Don't delete.
%Here,  we note that a bundle automorphism~$H \colon \Sigma \mathbin{\wt{\times}} \R^2 \to \Sigma \mathbin{\wt{\times}} \R^2$ preserves local orientations if and only if the underlying homeomorphism is orientation-preserving.
%Since rel.\ boundary homeomorphisms are necessarily orientation-preserving, this explains why there is no discrepancy between the first and third statements of Proposition~\ref{prop:RotAlphaAxiomsNonori}.
%%covering the identitys

\subsection{The effect on Alexander modules}
\label{sub:RotAlphaAlexanderNonOri}

Given~$\alpha\in H^1(\Sigma,\partial \Sigma;\Z^w)$, this section describes the effect of $\Rot(\alpha)^2$ on the Alexander modules of $\Sigma \mathbin{\wt{\times}} S^1$ and $Y_e(K)=E_K \cup \Sigma \mathbin{\wt{\times}} S^1$.
These results were used in Section~\ref{sec:Injection} to prove that the map $\Xi_{\Emb}$ is well defined.

First, note that if~$\ev([\alpha]_2) \colon H_1(\Sigma) \to \Z_2$ is the zero map, then $\Rot(\alpha)$ lifts to the $\varphi$-covers (combine Lemma~\ref{lem:LiftRots} with Proposition~\ref{prop:cohomology-version}).
Indeed,  for a section $s$ with associated cohomology class $\alpha$, the homomorphism~$ \ev([\alpha]_2)$ corresponds to~$\varphi \circ s_*$. 
In particular,  for every~$\alpha\in H^1(\Sigma,\partial \Sigma;\Z^w)$,  the bundle automorphism $\Rot(2\alpha) \simeq \Rot(\alpha)^2$ lifts in this manner.
%%Don't delete
%%Essentially by naturality of the Siegel construction.
%%If you multiply a class by 2, then the action on S^1 \simeq SO(2) \simeq K(Z,1) is multiplied  by 2 etc.

\begin{proposition}
\label{prop:RotAlexanderModuleNonOri}
Given~$\alpha\in H^1(\Sigma,\partial \Sigma;\Z^w)$,  the bundle isomorphism $\Rot(\alpha)^2$ induces the following automorphism on the Alexander module~$H_1((\Sigma \mathbin{\wt{\times}} S^1)^\varphi)= \Z_2 \oplus H_1(\Sigma)$ of its circle bundle:
	$$
	\widetilde{\Rot(\alpha)^2}_*=
	\begin{pmatrix}
		\id& \ev([\alpha]_2) \\ 0&\id
	\end{pmatrix}
	\colon 
	H_1((\Sigma \mathbin{\wt{\times}} S^1)^\varphi)
	\to 
	H_1((\Sigma \mathbin{\wt{\times}} S^1)^\varphi).
	$$
Here,~$\ev([\alpha]_2)  \colon H_1(\Sigma) \to \Z_2$ is the map induced by~$[\alpha]_2 \in H^1(\Sigma,\partial \Sigma;\Z_2) \cong H^1(\Sigma;\Z_2)$.
The splitting~$H_1((\Sigma \mathbin{\wt{\times}} S^1)^\varphi) = \Z_2 \oplus H_1(\Sigma)$ is obtained using the fixed section~$r \colon \Sigma \to \Sigma \mathbin{\wt{\times}} S^1$.
\end{proposition}
\begin{proof}
%Corollary~\ref{cor:AlexanderModuleRotSectionVersion} shows that~$
%	\Rot(\wt s)_*= \bsm \id&\psi \circ \wt s_*\\ 0&\id \esm$.
This follows from Corollary~\ref{cor:AlexanderModuleRotSectionVersion} and Proposition~\ref{prop:cohomology-version}.
\end{proof}

Next,  we apply this result to calculate the effect of~$\id_{E_K} \cup \Rot(\alpha)^2$ on the Alexander module of~$Y_e(K)=E_K \cup (\Sigma \mathbin{\wt{\times}} S^1)$, i.e.\ on the boundary of surface exteriors.
Fix the cell structure on $\Sigma$ that consists of a~$0$-cell $e^0$ and a~$1$-cell~$e^1_\partial$ to construct~$\partial \Sigma$, then $h$ 1-cells $\gamma_1,\dots,\gamma_h$ attached to the~$0$-cell, followed by a 2-cell attached via~$e^1_{\partial}\cdot\gamma_1\cdot \gamma_1\cdots\gamma_h \cdot \gamma_h$.
We then endow
$$H_1(\Sigma) \cong \Z \oplus \Z^{h-1}$$ 
with the basis that consists of the curve~$\gamma:=\gamma_1+\cdots+\gamma_h$ in the first summand and~$\gamma_1,\ldots,\gamma_{h-1}$ in the remaining summands.
With respect to this basis, we have~$[\partial \Sigma]=2\gamma=(2,0,\ldots,0)$.

Given an integer~$e$ and a knot~$K \subset S^3$, recall from (the proof of) Proposition~\ref{prop:AlexanderModuleNonOri} (and Lemma~\ref{lem:MV}) that there is an inclusion-induced isomorphism
$$H_1(Y_e(K))\cong  \frac{H_1(\Sigma) \oplus \Z_2\langle \mu\rangle}{([\partial \Sigma],e)}.$$
Here $\mu$ denotes the $S^1$-fibre of $\Sigma \mathbin{\wt{\times}} S^1$, which is identified with the meridian of $K$.
Using the previously described generating set of $H_1(\Sigma)$, it follows that 
$$ H_1(Y_e(K))
\cong
\begin{cases}
\Z_2\langle \mu\rangle \oplus \Z_2\langle \gamma\rangle \oplus \Z^{h-1} &\quad \text{ if~$e$ is even,}  \\
\Z_4\langle \gamma\rangle \oplus \Z^{h-1} &\quad \text{ if~$e$ is odd.} 
\end{cases}
$$
Next,  given~$\alpha\in H^1(\Sigma,\partial \Sigma;\Z^w)$,  recall that~$\ev([\alpha]_2)  \colon H_1(\Sigma) \to \Z_2$ denotes the map induced by~$[\alpha]_2 \in H^1(\Sigma,\partial \Sigma;\Z_2) \cong H^1(\Sigma;\Z_2)$ on homology.
\begin{itemize}
\item  If $e$ is even,  $\ev([\alpha]_2)  \colon H_1(\Sigma) \to \Z_2$ descends to a map~$\Z_2 \oplus \Z^{h-1} \cong  H_1(\Sigma)/\langle \partial \Sigma\rangle \to \Z_2$ that we again denote by $\ev([\alpha]_2) $.
\item If $e$ is odd,  $\ev([\alpha]_2)  \colon H_1(\Sigma) \to \Z_2$ induces a map  $\Z_4 \oplus \Z^{h-1} \cong \frac{H_1(\Sigma) \oplus \Z_2}{([\partial \Sigma],1)} \to \Z_2 \xrightarrow{\cdot 2} \Z_4$ that we also denote by $\ev([\alpha]_2)$.
\end{itemize}
The following result describes the effect of $\id_{E_K} \cup \Rot(\alpha)$ on $H_1(Y_e(K))$.

\begin{corollary}
\label{cor:RotOnYe(K)}
Given an integer $e$, a knot $K \subset S^3$ and~$\alpha\in H^1(\Sigma,\partial \Sigma;\Z^w)$, the effect of the bundle isomorphism $\id_{E_K} \cup \Rot(\alpha)$ on~$H_1(Y_e(K))$ is as follows.
%no : on purpose.
\begin{itemize}
\item If $e$ is even with $H_1(Y_e(K)) \cong  \Z_2 \oplus \Z_2 \oplus \Z^{h-1}$, then 
$$
(\id_{E_K} \cup \Rot(\alpha))_*=
\begin{pmatrix}
\id&\ev([\alpha]_2)&\ev([\alpha]_2)\\
0&\id&0\\
0&0&\id
\end{pmatrix}.
$$
\item If $e$ is odd with $H_1(Y_e(K)) \cong  \Z_4 \oplus \Z^{h-1}$,  then 
$$
(\id_{E_K} \cup \Rot(\alpha))_*=
\begin{pmatrix}
\id&\ev([\alpha]_2) \\
0&\id
\end{pmatrix}.
$$
\end{itemize}
\end{corollary}
\begin{proof}
This is a direct consequence of Proposition~\ref{prop:RotAlphaAxiomsNonori} and the above description of $\ev([\alpha]_2)$.
\end{proof}

Next we describe the situation on the $2$-fold covers.
Given an even integer $e$ and a knot $K \subset S^3$, recall from (the proof of) Proposition~\ref{prop:AlexanderModuleNonOri} that there is an inclusion-induced isomorphism
$$H_1(Y_e(K);\Z[\Z_2])\cong H_1(\Sigma_2(K)) \oplus \frac{H_1(\Sigma) \oplus \Z_2\langle \widetilde{\mu}\rangle}{([\partial \Sigma],e/2)}.$$
Here $\widetilde{\mu}$ denotes the $S^1$-fibre of $(\Sigma \mathbin{\wt{\times}} S^1)^\varphi$, which is identified with the lifted meridian of $K$.
Using the previously described generating set of $H_1(\Sigma)$, it follows that 
$$ H_1(Y_e(K);\Z[\Z_2])
\cong
\begin{cases}
H_1(\Sigma_2(K)) \oplus \Z_2\langle \widetilde{\mu}\rangle \oplus \Z_2\langle \gamma\rangle \oplus \Z^{h-1} &\quad \text{ if~$e/2$ is even,}  \\
H_1(\Sigma_2(K)) \oplus \Z_4\langle \gamma\rangle \oplus \Z^{h-1} &\quad \text{ if~$e/2$ is odd.} 
\end{cases}
$$
Given~$\alpha\in H^1(\Sigma,\partial \Sigma;\Z^w)$,  we use the same notation as above, namely 
\begin{itemize}
\item  if $e/2$ is even,  $\ev([\alpha]_2)  \colon H_1(\Sigma) \to \Z_2$ descends to a map~$\Z_2 \oplus \Z^{h-1} \cong  H_1(\Sigma)/\langle \partial \Sigma\rangle \to \Z_2$ that we again denote by $\ev([\alpha]_2)$,
\item if $e/2$ is odd,  $\ev([\alpha]_2)  \colon H_1(\Sigma) \to \Z_2$ induces a map  $\Z_4 \oplus \Z^{h-1} \cong \frac{H_1(\Sigma) \oplus \Z_2}{([\partial \Sigma],1)} \to \Z_2 \xrightarrow{\cdot 2} \Z_4$ that we also denote by $\ev([\alpha]_2)$.
\end{itemize}
The following result describes the effect of $\id_{E_K} \cup \Rot(\alpha)^2$ on the Alexander module of $Y_e(K)$.

\begin{corollary}
\label{cor:RotOnYe(K)Covers}
Given an even integer $e$, a knot $K \subset S^3$ and~$\alpha\in H^1(\Sigma,\partial \Sigma;\Z^w)$, the bundle isomorphism $\id_{E_K} \cup \widetilde{\Rot(\alpha)^2}$ induces the following automorphism on the Alexander module~$H_1(\widehat{Y}_e(K))$:
%\cong  H_1(\Sigma_2(K)) \oplus \G \oplus \Z^{h-1}$:
\begin{itemize}
\item If $e/2$ is even with $H_1(Y_e(K)) \cong  H_1(\Sigma_2(K)) \oplus \Z_2 \oplus \Z_2 \oplus \Z^{h-1}$, then 
$$
(\id_{E_K} \cup \widetilde{\Rot(\alpha)^2})_*=
\begin{pmatrix}
\id&0&0&0\\
0&\id&\ev([\alpha]_2)&\ev([\alpha]_2)\\
0&0&\id&0\\
0&0&0&\id
\end{pmatrix}.
$$
\item If $e/2$ is odd with $H_1(Y_e(K)) \cong  H_1(\Sigma_2(K)) \oplus \Z_4 \oplus \Z^{h-1}$,  then 
$$
(\id_{E_K} \cup \widetilde{\Rot(\alpha)^2})_*=
\begin{pmatrix}
\id&0&0 \\
0&\id&\ev([\alpha]_2) \\
0&0&\id
\end{pmatrix}.
$$
\end{itemize}
The $\Z[\Z_2]$-module structure is determined by $T[\eta]=[\eta]+w_1(\eta) \widetilde{\mu}$ and $T \widetilde{\mu}= \widetilde{\mu}$ for $[\eta] \in H_1(\Sigma)$.
Here~$\widetilde{\mu}$ denotes the homology class of a lift of twice the meridian of $K$.
\end{corollary}
\begin{proof}
The first two assertions are a direct consequence of Proposition~\ref{prop:RotAlexanderModuleNonOri} and the above description of $\ev([\alpha]_2)$.
We therefore focus on the description of the module structure.
The deck transformation $T$ certainly fixes $\widetilde{\mu}$ and,
identifying $(\Sigma \mathbin{\wt{\times}} S^1)^\varphi$ with $\Sigma \mathbin{\wt{\times}} S^1$, it satisfies~$T[x,1]=[x,-1]$.
%(Tx,1)~(x,-1)
If the bundle over $\eta$ is orientable, then $T\eta$ and $\eta$ are homotopic using $([x,1],t) \mapsto [(x,e^{\pi it})]$.
If the bundle over $\eta$ is nonorientable, then we can use the same homotopy everywhere except at one point: at that point, from one side, the homotopy will rotate clockwise and from the other counter-clockwise,  thus creating an additional meridian.
\end{proof}

\subsection{The effect of $\Rot(\alpha)$ on spin structures}
\label{sub:RotalphaSpinNonOrientable}

Let $p \colon \Sigma \mathbin{\wt{\times}} D^2 \to \Sigma$ be the projection and let~$p| \colon \Sigma \mathbin{\wt{\times}} S^1 \to \Sigma$ be the induced
projection.
Given a spin structure $\mathfrak{s} \in \Spin(\Sigma\mathbin{\wt{\times}}
D^2)$,  we consider  the restricted spin structure $\mathfrak{s}| \in
\Spin(\Sigma \mathbin{\wt{\times}} S^1)$.

\begin{proposition}
     \label{prop:RotalphaSpinNonOrientable}
     Given $\alpha \in H^1(\Sigma,\partial \Sigma;\Z^w)$ and a spin structure
$\mathfrak{s} \in \Spin(\Sigma \mathbin{\wt{\times}} D^2)$,  the effect of the
bundle automorphism~$\Rot(\alpha)| \colon \Sigma \mathbin{\wt{\times}} S^1 \to
\Sigma\mathbin{\wt{\times}} S^1$ on~$\mathfrak{s}| \in \Spin(\Sigma \mathbin{\wt{\times}}
S^1)$ is by
     $$
     \Rot(\alpha)^*(\mathfrak{s}|)
     =
     p|^*([\alpha]_2) \cdot
     \mathfrak{s}|
     \in \Spin(\Sigma \mathbin{\wt{\times}} S^1).
     %%%
     %(p^*)^{-1} \left(
     %\operatorname{Red}_2(\alpha) \cdot p^*(\mathfrak{s}) \right) \in \Spin(\Sigma \times D^2).
     $$
If $\mathfrak{s} \in \Spin(\Sigma \mathbin{\wt{\times}} S^1)$ does not extend over $\Sigma \mathbin{\wt{\times}} D^2$, then $\Rot(\alpha)$ acts trivially on $\mathfrak{s}$.
%%Don't delete.
%%Because \eps_x=0 in that case.
\end{proposition}
\begin{proof}
Since~$T(\Sigma\mathbin{\wt{\times}} S^1)\colon \Sigma\mathbin{\wt{\times}}
S^1\to BO$ is trivial, a spin structure on~$T(\Sigma\mathbin{\wt{\times}} S^1)$
is the same as a lift to the fibre~$K(\Z_2,1)$ of~$B\Spin\to BO$ or
equivalently 
%(canonically) 
an element of~$H^1(\Sigma\mathbin{\wt{\times}} S^1;\Z_2)$. 
%%Don't delete
%%Here the point is that since the tangent bundle is trivial a lift to Bspin is the same as a map to K(Z,2), and that's really a cohomology class.
%%So usually we'd get a free and transitive action but here it's legit to say we get it on the nose.
Under this identification, the action of~$H^1(\Sigma\mathbin{\wt{\times}} S^1;\Z_2)$ on~$\Spin(\Sigma
\mathbin{\wt{\times}} S^1) \cong H^1(\Sigma\mathbin{\wt{\times}} S^1;\Z_2)$ is then given by left multiplication. 
Hence we want to show that
 \[\mathfrak{s}|-\Rot(\alpha)^*(\mathfrak{s}|)
 =p^*([\alpha]_2)\in
H^1(\Sigma\mathbin{\wt{\times}} S^1;\Z_2).\]
Given~$x\in
H^1(\Sigma\mathbin{\wt{\times}} S^1;\Z_2)$,  using that~$H^1(\Sigma\mathbin{\wt{\times}} S^1;\Z_2)$ is dual to
$H_1(\Sigma\mathbin{\wt{\times}} S^1;\Z_2)$, it follows from
the second item of Proposition~\ref{prop:RotAlphaAxiomsNonori} that
$$\Rot(\alpha)^*(x)=x+\epsilon_x p^*([\alpha_2]),$$
where~$\epsilon_x\in\Z_2$ is given by
the restriction of~$x$ to the fibre~$\{*\}\times S^1$ of
$\Sigma\mathbin{\wt{\times}} S^1$. 
Hence it remains to show that the restriction
of~$\mathfrak{s}|$ to~$\{*\}\times S^1$ is non-trivial, as a cohomology class.
This is true since~$\mathfrak{s}|$ extends over~$\Sigma\mathbin{\wt{\times}} D^2$ and thus
restricts to the bounding spin structure on~$\{*\}\times S^1$. 
As observed by Milnor \cite[p.~201]{Milnor-spin}, the bounding spin
structure on~$TS^1$ corresponds to the non-trivial element of
$H^1(S^1;\Z_2)\cong \Z_2$.
\end{proof}

\subsection{Sections and bundle automorphisms}
\label{sub:AutomorphismNonOrientable}

At this point, it only remains to prove the last item of Proposition~\ref{prop:RotAlphaAxiomsNonori}, namely the fact that every bundle automorphism~$\Sigma \mathbin{\wt{\times}} \R^2 \to \Sigma \mathbin{\wt{\times}} \R^2$ can be written uniquely as~$\Rot(\alpha)$ for some~$\alpha \in H^1(\Sigma,\partial \Sigma;\Z^w).$
Thanks to Proposition~\ref{prop:cohomology-version}, this reduces to proving that every bundle automorphism~$\Sigma \mathbin{\wt{\times}} \R^2 \to \Sigma \mathbin{\wt{\times}} \R^2$ can be written uniquely as~$\Rot(s)$ for some section~$s \colon \Sigma \to \widehat{\Sigma }\mathbin{\wt{\times}} S^1$ that is boring on~$\partial \Sigma$.
This will involve working with the  principal~$O(2)$-bundle~$P \to \Sigma$ associated to~$\Sigma \mathbin{\wt{\times}} \R^2$.
As a consequence, this section is mainly concerned with principal bundles.

In this section, we will implicitly assume that all vector bundles have a metric, i.e.\ structure group~$O(n)$, and that all bundle maps preserve this metric.

\medbreak

\begin{construction}
For a principal $G$-bundle $p\colon P\to B$,  consider the fibre bundle~$P\times_G G\to B$ with fibre $G$, where the $G$-action is given by the principal $G$-action on $P$ and the conjugation action on $G$, i.e.\ $[x,h]=[gx,ghg^{-1}]$. 
In other words, the group~$G$ acts on $P \times G$ by~$g \cdot (x,h)=(gx,ghg^{-1})$ and~$P\times_G G\to B$ is the orbit set of this action.
Note that for $G$ abelian, one gets~$P\times_G G=B \times G$.
\end{construction}

Let~$B'\subseteq B$ be a subset.
We now relate the set~$\Aut_{B'}(P)$ of bundle isomorphisms of~$P$ that are the identity over~$B'$ to the set~$\Gamma_{B'}(B,P\times_G G)$ of sections of~$P\times_G G$ that send~$b\in B'$ to~$[e,1]$ for some preimage $e \in P$ of $b \in B$.
To do so,  given a section~$s \colon B \to P\times_G G$  and~$e \in P$,  we write~$h(e,s) \in  G$ for the group element such that~$s(p(e))=[e,h(e,s)]$.
Since the~$G$-action on~$p^{-1}(p(e))$ is free and transitive, this~$h(e,s) \in G$ exists and is uniquely determined by the condition~$s(p(e))=[e,h(e,s)]$.

\begin{lemma}
\label{lem:BundleHomeosSection}
	The map $\Phi \colon \Gamma_{B'}(B,P\times_G G)\to \Aut_{B'}(P),s\mapsto (e\mapsto h(e,s)e)$ is a homeomorphism.
%Bundle automorphisms of a principal $G$-bundle $P \to B$ correspond bijectively to sections of $P\times_G G \to B$.
%%This is still a fibre bundle with fibre G.
%%The structure group isn't clear.
\end{lemma}	
\begin{proof}
We first verify that the assignment $\Phi$ is well defined, i.e.\ that $e\mapsto \Phi(s)(e)=h(e,s)e$ is a bundle automorphism that is the identity over $B'$. 
By definition of $\Gamma_{B'}(B,P\times_G G)$, $h(e,s)=1$ for every $e$ with $p(e)\in B'$. Hence this reduces to verifying that $\Phi(s)$ is $G$-equivariant.
Using the equalities~$s(p(ge))=s(p(e))=[e,h(e,s)]=[ge,gh(e,s)g^{-1}]$,  this follows promptly:
$$\Phi(s)(ge)=gh(e,s)g^{-1}ge=gh(e,s)e=g\Phi(s)(e).$$
We define an inverse to $\Phi$.
A principal~$G$-bundle automorphism~$\phi\colon P\to P$ determines a section~$\Psi(\phi) \colon  B \to P\times_G G$ by sending~$b\in B$ to~$\Psi(\phi)(b):=[e,h]$ where~$p(e)=b$ and~$h  \in G$ is the unique group element such that~$\phi(e)=he$. 
A verification now shows that $\Phi$ and $\Psi$ are inverses.
%\Psi(\Phi(s))(b)
%=(e,h)
%=(e,h(e,s)) (by def of h and \Phi)
%=s(p(e)) (by def of h(e,s).
%=s(b) by def of e as satisfying p(e)=b.
%%%%
%\Phi(\Psi(s))(e)
%=h(e,s)\cdot e where h(e,s) is defined by s(p(e))=[e,h(e,s)] with s:=\Psi(\phi)(e) 
%=\phi(e) by looking at the definition of \Psi(\phi)(p(e)).
%Hence the map~$\Gamma(B,P\times_G G)\to \Aut(P)$ is a bijection. 
The verification that $\Phi$ is a homeomorphism is omitted.
%%Don't delete it.
%%Proof of homeo.pdf
\end{proof}

The next remark focuses on the case $G=O(2)$.

\begin{remark}
\label{rem:ComponentsBundle}
The space~$P\times_{O(2)} O(2)$ has two components,  consisting respectively of those~$[(e,g)]$ with~$\det(g)=1$ (i.e.~$g \in SO(2)$) and those with~$\det(g)=-1$.
Here we used~$\det(h)=\det(ghg^{-1})$.
Denote the first of these two components by~$P\times_{O(2)}SO(2)$.
Since the conjugation action of $SO(2)$ on itself is trivial,
the projection $P \times SO(2)\to P/SO(2) \times SO(2)$ induces an isomorphism
$$P\times_{O(2)}SO(2)\xrightarrow{\cong} P/SO(2)\times_{\Z_2}SO(2),$$
where $\Z_2\cong O(2)/SO(2)$ acts on $SO(2)$ by conjugation by $\bsm 1&0\\0&-1\esm$.
Here any matrix with determinant $-1$ will work since all give the same conjugation action,  but $\bsm 1&0\\0&-1\esm$ is most convenient for our purposes below.
%{XX: Any matrix with determinant $-1$ works since the all give the same conjugation action. But this one fits best with the choices below.}
%and observe that a direct verification using $O(2)/SO(2)\cong \Z_2$ yields
%$$P\times_{O(2)}SO(2)\cong P/SO(2)\times_{\Z_2}SO(2).$$
%Here $P/SO(2)\times_{\Z_2}SO(2)$ is obtained as the orbit set of~$P/SO(2) \times SO(2)$ under the natural diagonal action of~$O(2)/SO(2) \cong \Z_2$.
%For~$G=O(2)$ we have~$\det(h)=\det(ghg^{-1})$ and hence~$P\times_{O(2)} O(2)\to B$ contains~$P\times_{O(2)}SO(2)\cong P/SO(2)\times_{\Z_2}SO(2)$ as a component. 
\end{remark}

In what follows,  given a vector bundle $E$, we write $\Aut^+_{B'}(E)$ for the subgroup of bundle automorphisms of $E$ that preserve local orientations and are the identity over $B'$.
\begin{construction}
Let $B$ be a connected space,  let $B'\subseteq B$ be a subspace, let $\xi \colon P\to B$ be a principal~$O(2)$-bundle and let $E:=P\times_{O(2)}\R^2\to B$ be the associated rank $2$ vector bundle. 
We construct a map 
$$
\Gamma_{B'}(B,P/SO(2)\times_{\Z_2} SO(2))\to \Aut^+_{B'}(E).
$$
To do so, we map a section~$s \colon B \to P/SO(2)\times_{\Z_2} SO(2)$ to the bundle automorphism
\begin{align*}
R(s) \colon E &\to E \\
[p,v]&\mapsto [p,h(\overline{p},s)^{-1}v].
\end{align*}
Here~$\overline{p}\in P/SO(2)$ is the image of~$p$ under the canonical projection~$\proj \colon P \to P/SO(2)$ and the group element~$h(\overline{p},s)\in SO(2)$ is such that~$s(\xi(p))=[\overline{p},h(\overline{p},s)]$.
One can verify that this assignment is well defined.
%%Don't delete. 
%%Have to show  for g \in Z_2=O(2)/SO(2), we have [pg,h(\ol{pg},s)gv]=[p,h(\ol{p},s)v]
%%%Let g be in O(2) and \ol{g} be the image in O(2)/SO(2) which is either the identity or (1 & 0 \\ 0 & -1). 
%We have [\ol{p},h(\ol{p},s)]=s(\xi(p))=s(\xi(gp))=[\ol{gp},h(\ol{gp},s)]=[\ol{p},\ol{g}h(\ol{gp},s)\ol{g}] and thus h(\ol{p},s)=\ol{g}h(\ol{gp},s)\ol{g}. 
%Using this we get
%[p,h(\ol{p},s)v]=[p,\ol{g}h(\ol{gp},s)\ol{g}v]=[gp,g\ol{g}h(\ol{gp},s)\ol{g}v]. Both g\ol{g} and h(\ol{gp},s) are in SO(2), so they commute and we get [gp,g\ol{g}h(\ol{gp},s)\ol{g}v]=[gp,h(\ol{gp},s)gv] after cancelling \ol{g}^2=1. Thus R(s)[p,v]=[p,h(\ol{p},s)v]=[gp,h(\ol{gp},s)gv]=R(s)[gp,gv].
\end{construction}

\begin{proposition}
\label{prop:BundleIsoHE}
Let $B$ be a connected space,  let $B'\subseteq B$ be a subspace, let $\xi \colon P\to B$ be a principal~$O(2)$-bundle and let $E:=P\times_{O(2)}\R^2\to B$ be the associated rank $2$ vector bundle.  
The assignment $s \mapsto R(s)$ gives rise to a homeomorphism
$$
\Gamma_{B'}(B,P/SO(2)\times_{\Z_2} SO(2))\xrightarrow{\cong} \Aut^+_{B'}(E).
$$
\end{proposition}
\begin{proof}
We first establish the bijection abstractly.
The equivalence of the categories of vector bundles and principal $O(2)$-bundles together with Lemma~\ref{lem:BundleHomeosSection} give a homeomorphism 
$$\Aut_{B'}(E) 
 \cong   \Aut_{B'}(P)
  \cong \Gamma_{B'}(B,P\times_{O(2)} O(2)).$$
 An automorphism of~$E$ preserves local orientations if and only if the associated section takes values in the subbundle~$P \times_{O(2)}SO(2).$ 
 We rewrite this subbundle as in Remark~\ref{rem:ComponentsBundle}.
 Namely,
 %Using Remark~\ref{rem:ComponentsBundle},  this subbundle can be rewritten as
recall from Remark~\ref{rem:ComponentsBundle} that the projection $\proj \times \id \colon P \times SO(2)\to P/SO(2) \times SO(2)$ induces an isomorphism
\begin{equation}
\label{eq:BundleIsoO2}
P \times_{O(2)}SO(2)
\cong P/SO(2)\times_{\Z_2}SO(2).
\end{equation}
%For the final homeomorphism, observe that since the principal $O(2)$-bundle~$P_E^{O(2)}$ is classified by~$B\to BO(2)$,  the quotient~$P_E^{O(2)}/SO(2)$ is classified by~$B\to BO(2)\to B(O(2)/SO(2))=B\Z_2$, i.e.\ by~$w_1(\xi) \colon B \to B\Z_2$.
%%Don't delete.
%Because the first map is the classifying map and the second is universal w_1.
Combining these facts leads to the required bijection:
$$
\Aut^+_{B'}(E)
\cong \Gamma_{B'}(B,P\times_{O(2)}SO(2))
 \cong \Gamma_{B'}(B,P/SO(2)\times_{\Z_2} SO(2)).
$$
It remains to verify that this bijection maps a section $s$ to the bundle automorphism~$R(s)$.
Start with a section~$s\in \Gamma_{B'}(B,P/SO(2)\times_{\Z_2} SO(2))$ with~$s(\xi(p))=[\overline{p},h(\overline{p},s)]$.
Here~$p \in P$ and~$\overline{p}:=\proj(p) \in P/SO(2)$ and $h(\overline{p},s)\in SO(2)$.
% is the group element such that this equality holds.}
Using the isomorphism from~\eqref{eq:BundleIsoO2}, the associated section~$s'\in \Gamma_{B'}(B,P\times_{O(2)}SO(2))$ is given by~$s'(\xi(p))=[p,h(\overline{p},s)]$. 
%%By construction note that h(\overline{p},s)=h(p,s').
Viewing~$P\times_{O(2)}SO(2)$ as a subbundle of~$P\times_{O(2)}O(2)$ and applying Lemma~\ref{lem:BundleHomeosSection}, the corresponding automorphism of~$P$ is given by~$p\mapsto h(\overline{p},s)p$. 
This is mapped to the automorphism of~$E$ given by
\[[p,v]\mapsto [h(\overline{p},s)p,v]=[p,h(\overline{p},s)^{-1}v].\]
This concludes the proof of the proposition.
\end{proof}

We can now conclude the proof of the last item of Proposition~\ref{prop:RotAlphaAxiomsNonori}.
As explained at the beginning of this section, this reduces to proving the following result.
\begin{proposition}
\label{prop:RealiseRotNonOri}
For every bundle automorphism $H \colon \Sigma \mathbin{\wt{\times}} \R^2 \to \Sigma \mathbin{\wt{\times}} \R^2$ 
%that preserves local orientations and 
%%Don't delete
%%Automatic because rel.\ boundary ->homeo is o-p <-> bundle aut preserves local orientations.
%%For the equivalence: true because both conditions are local.
that restricts to the identity over $\Sigma$, there is a unique a section~$s \colon \Sigma \to \Sigma\mathbin{\wt{\times}} S^1$ that is boring on $\partial \Sigma$ with $H=\Rot(s).$
\end{proposition}
\begin{proof}
Let $P$ be the principal $O(2)$-bundle associated with $\Sigma \mathbin{\wt{\times}} \R^2=\widehat{\Sigma} \times_{\Z_2} \R^2$.
In particular,  $P$ is the unique such bundle with $P \times_{O(2)} \R^2 \cong \widehat{\Sigma} \times_{\Z_2} \R^2$.
%%Don't delete
%%Because bijection
Since $\widehat{\Sigma} \times_{\Z_2} O(2)$ satisfies this property,  we deduce that~$P \cong \widehat{\Sigma} \times_{\Z_2} O(2)$.
%%Don't delete. 
%Observe the isomorphism
%\begin{align*}
%\wh\Sigma\times_{\Z_2} O(2)\times_{O(2)} \R^2 &\xrightarrow{\cong} \wh \Sigma\times_{\Z_2}\R^2=\Sigma \mathbin{\wt{\times}} \R^2 \\
%[b',a,v]&\mapsto [b',av].
%\end{align*}
%%Aka things cancel out as for tensor products.
It follows that $P/SO(2) \cong  (\widehat{\Sigma} \times_{\Z_2} O(2))/SO(2) \cong \widehat{\Sigma} \times_{\Z_2} \Z_2 \cong \widehat{\Sigma}$.
In summary, there are isomorphisms
\begin{align*}
P/SO(2) \times_{\Z_2} SO(2) \cong \widehat{\Sigma} \times_{\Z_2} SO(2) \\
P \times_{O(2)} \R^2 \cong \widehat{\Sigma} \times_{\Z_2} \R^2.
\end{align*}
Using these isomorphisms, Proposition \ref{prop:BundleIsoHE} implies that $\Gamma_{\partial \Sigma}(\Sigma,\Sigma \mathbin{\wt{\times}} S^1) \cong \Aut_{\partial \Sigma}(\Sigma \mathbin{\wt{\times}} \R^2).$
Here, recall that~$\Aut_{\partial \Sigma}(\Sigma \mathbin{\wt{\times}} \R^2)=\Aut_{\partial \Sigma}^+(\Sigma \mathbin{\wt{\times}} \R^2)$: a rel.\ boundary homeomorphism of~$\Sigma \mathbin{\wt{\times}} \R^2$ is orientation-preserving and so,  viewed as a bundle automorphism,  it preserves local orientations.

The proposition therefore reduces to proving that the bundle automorphism $R(s)$ associated to a section~$s\in \Gamma_{\partial \Sigma}(\Sigma,\Sigma \mathbin{\wt{\times}} S^1)$ by this proposition agrees with $\Rot(s)$.
But if we write the section~$s\colon \Sigma\to \Sigma \mathbin{\wt{\times}} S^1 \cong \wh \Sigma\times_{\Z_2}SO(2)$ as~$s(b)=[b',a]$, then
Proposition~\ref{prop:BundleIsoHE} shows that $R(s)$ is given by $[b',v]\mapsto [b',a^{-1}v].$
This is precisely the definition of $\Rot(s)$.
\end{proof}

\bibliographystyle{alpha}
\bibliography{BiblioHomotopyBoundary}
\end{document}